\documentclass[letterpaper,12pt]{report}

\usepackage{pifont}
\usepackage{tabularx} 
\usepackage{amsmath,environ}  
\usepackage{graphicx} 
\usepackage[margin=0.9in,letterpaper]{geometry} 
\usepackage{cite} 
\usepackage[final]{hyperref} 
\usepackage{cleveref}
\hypersetup{
	colorlinks=true,       
	linkcolor=blue,        
	citecolor=blue,        
	filecolor=magenta,     
	urlcolor=blue         
}
\usepackage{enumitem}
\usepackage{graphicx} 
\usepackage{amsfonts}
\usepackage{amssymb}
\usepackage{comment}

\usepackage{amsthm}
\usepackage{subfigure}
\usepackage{grffile}
\usepackage{multicol}
\usepackage{algorithm,algorithmic}
\usepackage{empheq}
\usepackage{mathtools}
\usepackage{xcolor}
\usepackage{tikz}
\usetikzlibrary{shapes.geometric, arrows}
\tikzstyle{process} = [rectangle, minimum width=1cm, minimum height=1cm, text centered, draw=black, fill=orange!30]
\tikzstyle{arrow} = [thick,->,>=stealth]
\usetikzlibrary{decorations.pathreplacing}
\usepackage[framemethod=tikz]{mdframed}

\usepackage{blindtext}

\usepackage[scr = dutchcal]{mathalpha}
\newcommand{\mcl}[1]{\mathcal{ #1}}
\newcommand{\mbf}[1]{\mathbf{ #1}}

\newcommand{\norm}[1]{\Vert #1\Vert}

\newcommand{\hinf}{\ensuremath{H_{\infty}}}
\newcommand{\ip}[2]{\left\langle{#1},{#2}\right\rangle}
\renewcommand{\th}{\ensuremath{\theta}}

\newcommand{\bmat}[1]{\begin{bmatrix} #1\end{bmatrix}}

\newcommand{\mat}[1]{\begin{matrix}#1\end{matrix}}

\newcommand{\sbmat}[1]{\left[\begin{smallmatrix}
		#1\end{smallmatrix} \right]}

\newcommand{\R}{\mathbb{R}}
\newcommand{\C}{\mathbb{C}}

\newcommand{\N}{\mathbb{N}}

\newcommand{\myint}{\int_{a}^{b}}
\newcommand{\myinta}[1]{\int_{a}^{#1}}
\newcommand{\myintb}[1]{\int_{#1}^{b}}

\newcommand{\precceq}{\preccurlyeq}

\let\bl\bigl
\let\bbl\Bigl
\let\bbbl\biggl
\let\bbbbl\Biggl
\let\br\bigr
\let\bbr\Bigr
\let\bbbr\biggr
\let\bbbbr\Biggr

\newtheorem{thm}{Theorem}
\newtheorem{defn}[thm]{Definition}
\newtheorem{lem}[thm]{Lemma}

\newtheorem{cor}[thm]{Corollary}

\newtheorem{ex}{\textbf{Example}}

\newenvironment{boxEnv}[1]
  {\mdfsetup{
    frametitle={\colorbox{red!30}{#1\space}},
    innertopmargin=10pt,
    frametitleaboveskip=-\ht\strutbox,
    roundcorner = 5pt,
    backgroundcolor = yellow!30,
    outerlinecolor = blue!50!black,
    }
  \begin{mdframed}%
  }
  {\end{mdframed}}
  \newenvironment{Statebox}[1]
  {\mdfsetup{
    frametitle={\colorbox{white}{#1\space}},
    innertopmargin=5pt,
    frametitleaboveskip=-\ht\strutbox,
    roundcorner = 0pt,
    backgroundcolor = black!2,
    outerlinecolor = black,
    }
  \begin{mdframed}%
  }
  {\end{mdframed}}

\newcounter{codebox}

\newenvironment{codebox}[1][]{%
    \refstepcounter{codebox}
  \mdfsetup{
    frametitle={\colorbox{blue!30}{Code Block \thecodebox\space}},
    innertopmargin=10pt,
    frametitleaboveskip=-\ht\strutbox,
    roundcorner = 5pt,
    backgroundcolor = green!20,
    outerlinecolor = blue!50!black,
    }
  \begin{mdframed}%
  }
  {\end{mdframed}}

\NewEnviron{codeMLB}{%
\fontsize{11}{13}
  \begin{flalign*}\tt
    \BODY
  \end{flalign*}
}

\newenvironment{matlab}
  {\par\noindent\normalfont\par\nopagebreak%
  \begin{mdframed}[frametitle=,
     linewidth=1pt,
     linecolor=black,
     bottomline=false,topline=false,rightline=false,
     innerrightmargin=0pt,innertopmargin=0pt,innerbottommargin=0pt,
     innerleftmargin=1em,
     skipabove=.5\baselineskip
   ] \fontsize{11}{13} \tt}
  {\end{mdframed}}

\newcommand{\fourpi}[4]{\ensuremath{\mcl{P}{\footnotesize\bmat{#1,& \hspace{-3mm}#2 \\ #3,& \hspace{-3mm} \left\{#4\right\} }}}}
\newcommand{\threepi}[1]{\mcl{P}_{\{#1\}}}

\allowdisplaybreaks

\begin{document}
	\vfill
	\title{ \includegraphics[width=0.35\linewidth]{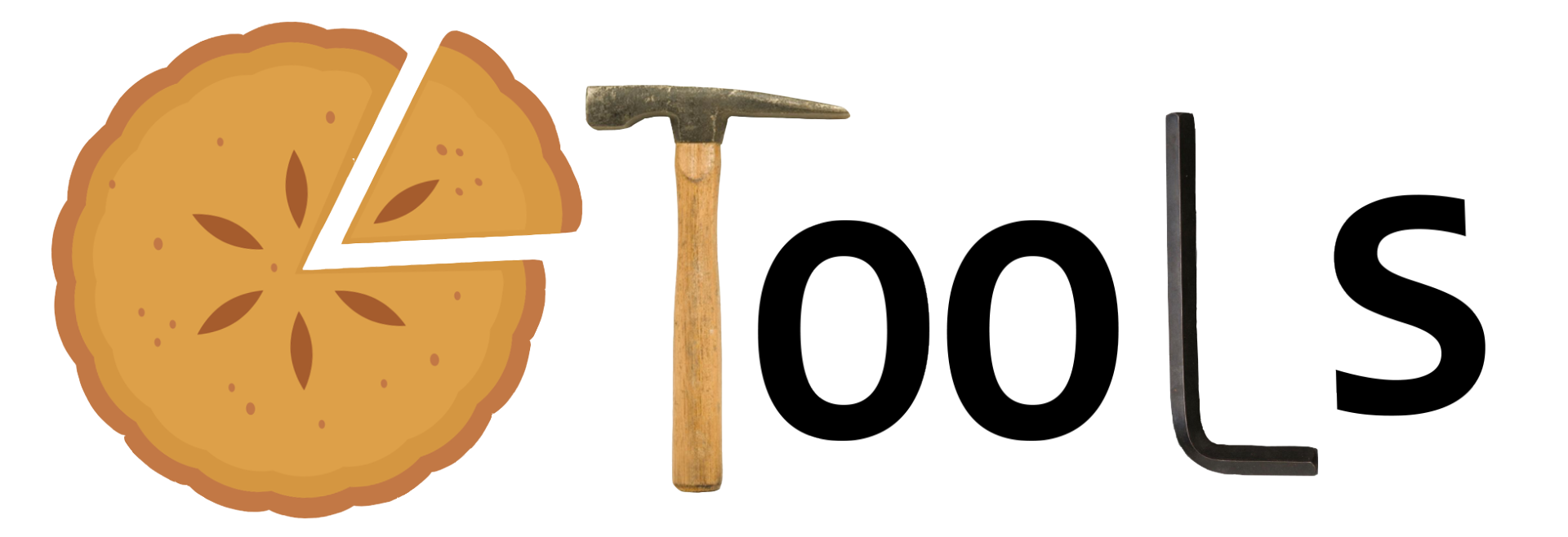}\\
		PIETOOLS 2022: User Manual}
	\author{S. Shivakumar \and A. Das \and D. Braghini \and D. Jagt \and Y. Peet \and M. Peet}
	\date{\today}
	\maketitle

	\chapter*{Copyrights and license information}
	PIETOOLS is free software: you can redistribute it and/or modify
	it under the terms of the GNU General Public License as published by
	the Free Software Foundation, either version 3 of the License, or
	(at your option) any later version.
	
	This program is distributed in the hope that it will be useful,
	but WITHOUT ANY WARRANTY; without even the implied warranty of
	MERCHANTABILITY or FITNESS FOR A PARTICULAR PURPOSE.  See the
	GNU General Public License for more details.
	
	You should have received a copy of the GNU General Public License along with this program; if not, write to the Free Software Foundation, Inc., 59 Temple Place, Suite 330, Boston, MA  02111-1307  USA. 
	\chapter*{Notation}
	\begin{tabular}{l l}
		$\R$ &Set of real numbers $(-\infty,\infty)$\\
		$\partial_s^i \mbf{x}$& $\frac{\partial^i\mbf{x}}{\partial s^i}$ where $s$ is in a compact subset of $\R$\\
		$\dot{\mbf{x}}$& $\frac{\partial\mbf{x}}{\partial t}$ where $t$ is in $[0,\infty)$\\
		$L_2^n[a,b]$& Set of Lebesgue-integrable functions from $[a,b]\to \R^n$\\
		$RL^{m,n}[a,b]$& $\R^m\times L_2^n[a,b]$\\
		$H_k^n[a,b]$& $\{f \in L_2^n[a,b]\mid \partial_s^if \in L_2^n[a,b] \forall i\le k\}$\\
		$0_{m\times n}$& Zero matrix of dimension $m\times n$\\
		$0_n$& Zero matrix of dimension $n\times n$\\
		$I_n$& Identity matrix of dimension $n\times n$\\
		$\Delta_a$ & Dirac operator on $f:C\to X$, $\Delta_a(f) = f(a)$, for $a\in C$\\
		$\mcl{B}(X,Y)$ & Space of bounded linear operators from $X$ to $Y$
	\end{tabular}
	
	\newpage

	\tableofcontents

	\chapter{About PIETOOLS}\label{ch:about}
	PIETOOLS is a free MATLAB toolbox for manipulating Partial Integral (PI) operators and solving Linear PI Inequalities (LPIs), which are convex optimization problems involving PI variables and PI constraints. 
	PIETOOLS can be used to:
	\begin{itemize}
		\item define PI operators in 1D and 2D
		\item declare PI operator decision variables (positive semidefinite or indefinite)
		\item add operator inequality constraints
		\item solve LPI optimization problems
	\end{itemize}  
	The interface is inspired by YALMIP and the program structure is based on that used by SOSTOOLS. By default the LPIs are solved using SeDuMi~\cite{sedumi}, however, the toolbox also supports use of other SDP solvers such as Mosek, sdpt3 and sdpnal.
	
	To install and run PIETOOLS, you need:
	\begin{itemize}
		\item MATLAB version 2014a or later (we recommend MATLAB 2020a or higher. Please note some features of PIETOOLS, for example PDE input GUI, might be unavailable if an older version of MATLAB is used)
		\item The current version of the MATLAB Symbolic Math Toolbox (This is installed in most default versions of Matlab.)
		\item An SDP solver (SeDuMi is included in the installation script.)
	\end{itemize}

\section{PIETOOLS 2022 Release Notes} 

PIETOOLS 2022 introduces several improvements to PIETOOLS, including a new interface for declaring 1D PDEs, and updating many functions to allow for representation and analysis of 2D PDEs. We list the main updates below:

\begin{enumerate}
    \item \textbf{Introduction of the command line user interface:} 
    In PIETOOLS 2022, 1D linear PDEs can be constructed on-the-fly using simple command line instructions. In particular, PIETOOLS 2022 adds the following functionality:

    \begin{itemize}
    \item A new function \texttt{state} that allows state variables, inputs and outputs to be easily declared.

    \item Built-in routines for performing addition, multiplication, differentiation, substitution, and integration of declared state variables and inputs, allowing for straightforward symbolic declaration of differential equations and boundary conditions.

    \item A new overarching structure \texttt{sys} for representing linear 1D ODE-PDE systems, to which declared equations can be easily added using \texttt{addequation}.

    \item A built-in routines \texttt{convert} to directly convert declared ODE-PDE systems to equivalent PIEs.

    \end{itemize}

    Using the new command line parser, a wide variety of linear 1D ODE-PDE systems, including systems with delay, can be easily declared from the Command Window, and converted to equivalent PIEs for further analysis. See Chapter~\ref{ch:PDE_DDE_representation} for more information.

    \item \textbf{Introduction of 2D PDE and PIE functionality:} Expanding upon the functionality for 1D ODE-PDEs, PIETOOLS 2022 now also allows coupled systems of linear ODEs, 1D PDEs and 2D PDEs to be declared and analysed. In particular, PIETOOLS 2022 adds the following functionality:

    \begin{itemize}
    \item A new terms-format that allows for representation of coupled linear ODE - 1D PDE - 2D PDE systems. See Section~\ref{sec:alt_PDE_input:terms_input_PDE}.

    \item \texttt{opvar2d} and \texttt{dopvar2d} functions for representing PI operators and PI operator decision variables in 2D. See Section~\ref{sec:PIs:2D}

    \item LPI programming functions for setting up and solving LPI programs involving 2D PI operators. See Chapter~\ref{ch:LPIs}.

    \end{itemize}

\end{enumerate}

\section{Installing PIETOOLS}

PIETOOLS 2022 is compatible with Windows, Mac or Linux systems and has been verified to work with MATLAB version 2020a or higher, however, we suggest to use the latest version of MATLAB.
	
Before you start, \textbf{make sure} that you have 
\begin{enumerate}
	\item MATLAB with version 2014a or newer. (MATLAB 2020a or newer for GUI input)
	\item MATLAB has permission to create and edit folders/files in your working directory.
\end{enumerate}

\subsection{Installation}
PIETOOLS 2022 can be installed in two ways.
\begin{enumerate}
	\item \textbf{Using install script:} The script installs the following files --- tbxmanager (skipped if already installed), SeDuMi 1.3 (skipped if already installed), SOSTOOLS 4.00 (always installed), PIETOOLS 2022 (always installed). Adds all the files to MATLAB path.
	\begin{itemize}
		\item Go to \url{https://github.com/CyberneticSCL/PIETOOLS} or \url{control.asu.edu/pietools/}. 
		\item Download the file \textbf{pietools\_install.m} and run it in MATLAB.
		\item \textbf{Run the script from the folder it is downloaded in to avoid path issues.}
	\end{itemize}
	\item \textbf{Setting up PIETOOLS 2022 manually:}
	\begin{itemize}
		\item Download and install C/C++ compiler for the OS. 
		\item Install an SDP solver. SeDuMi can be obtained from \href{http://sedumi.ie.lehigh.edu/?page_id=58}{\texttt{this link}}.
		\item Download SeDuMi and run \textbf{install\_sedumi.m} file. 
		\begin{itemize}
			\item Alternatively, install MOSEK, obtain license file and add to MATLAB path.
		\end{itemize}
		\item Download \textbf{PIETOOLS\_2022.zip} from \href{http://control.asu.edu/pietools/pietools.html}{\texttt{this link}}, unzip, and add to MATLAB path.
		\end{itemize}
\end{enumerate}

\chapter{Scope of PIETOOLS}\label{ch:scope}

In this chapter, we briefly talk about the need for a new computational tool for the analysis and control of ODE-PDE systems as well as DDEs. We lightly touch upon, without going into details, the class of problems PIETOOLS can solve. 

	\section{Motivation}
	Semidefinite programming (SDP) is a class of optimization problems that involve the optimization of a linear objective over the cone of positive semidefinite (PSD) matrices. The development of efficient interior-point methods for semidefinite programming (SDPs) problems made LMIs a powerful tool in modern control theory. As Doyle stated in \cite{doyle}, LMIs played a central role in postmodern control theory akin to the role played by graphical methods like Bode, Nyquist plots, etc in classical control theory. However, most of the applications of LMI techniques were restricted to finite dimensional systems, until the sum-of-squares method came into the limelight. The sum-of-squares (SOS) optimization methods found application in control theory, for example searching for Lyapunov functions or finding bounds on singular values. SOS polynomials were also used in constructing relaxations to some optimization problems with boolean decision variables. This gave rise to many toolboxes such as SOSTOOLS \cite{sostools}, SOSOPT \cite{sosopt} etc. that can handle SOS polynomials in MATLAB. However, unlike the use of LMIs for linear ODEs, SOS methods could not be used for the analysis and control of PDEs without ad-hoc interventions. For example, to search for a Lyapunov function that proves the stability of a PDE one would usually hit a roadblock in the form of boundary conditions which are typically resolved by using integration by parts, Poincare inequality, H\"older's Inequality, etc. 
	
	In an ideal world, we would prefer to define a PDE, specify the boundary conditions and let a computational tool take care of the rest. To resolve this problem, either we teach a computer to perform these ``ad-hoc'' interventions 
 or come up with a method that does not require such interventions, to begin with. To achieve the latter, we developed the Partial Integral Equation (PIE) representation of PDEs, which is an algebraic representation of dynamical systems using Partial Integral (PI) operators. The development of PIE representation led to a framework that can extend LMI-based methods to infinite-dimensional systems. The PIE representation encompasses a broad class of distributed parameter systems and is algebraic -- eliminating the use of boundary conditions and continuity constraints~\cite{shivakumar_2019CDC}, \cite{das_2019CDC}. The development of the toolbox, PIETOOLS, that can solve optimization problems involving these PI operators was an obvious consequence of the PIE representation.

	\section{PIETOOLS for Analysis and Control of ODE-PDE Systems}
Using PIETOOLS 2022 for controlling ODE-PDE models has been made intuitive, easy, and with very few mathematical details about the PIE operators and the maths behind them. Let's take a closer look at how this works. 
	
 \subsection{Defining Models in PIETOOLS}
Any control problem necessarily starts with declaring the model, and PIETOOLS makes it extremely simple to do so. Suppose that we are interested in modeling a coupled ODE-PDE system, such as a system with ODE dynamics given by
\begin{align}
 \label{part1:ode}
\dot{x}(t) = -x(t)+ u(t),
\end{align}
with controlled input $u$, and PDE dynamics given by
\begin{align}
    \label{part1:pde0}
    \ddot{\mbf{x}}(t,s) = c^2  \partial_s^2 \mbf x(t,s) -b \partial_s \mbf x(t,s) + s w(t), s\in (0,1), t\geq 0,
\end{align}
i.e, the one-dimensional wave equation with velocity $c$, added viscous damping coefficient $b$ and external disturbance $w$. Since the PDE has a second-order derivative in time, we should make a change of variables to appropriately define a state space. On this example, we do so by calling $\phi = (\partial_s \mbf{x}, \dot{\mbf x} )$. Thus the dynamic equation becomes
\begin{align}
    \label{part1:pde}
    \dot{\phi}(t,s) = \bmat{0 & 1 \\ c & 0}  \partial_s \phi (t,s) +\bmat{0 & 0\\0 & -b} \phi (t,s) + \bmat{0\\s} w(t), s\in (0,1), t\geq 0.
\end{align}
We can also add to our system a regulated output equation, for instance
\begin{align}
    \label{part1:output}
    z(t) = \bmat{r(t) \\ u(t)}, 
\end{align}
where $r(t) = \int_0^1 \bmat{1 & 0}\phi(t,s)ds$ gives $\mbf x(t,s=1) - \mbf x(t,s=0)$. Adding $u$ to the regulated output is a way of obtaining this signal from the simulations.

To define these models, we first create the following variables in MATLAB using the \texttt{state()} class:
\begin{matlab}
\begin{verbatim}
>> x = state('ode');   phi = state('pde',2);
>> w = state('in');    u = state('in');
>> z = state('out',2)
\end{verbatim}
\end{matlab}
Then, we can define/add equations to this \texttt{sys()} object using standard operators, such as \texttt{`+',`-',`*',`diff',`subs',`int', etc.}, as shown below.
\begin{matlab}
>> odepde = sys();\\
>> eq\_dyn = [diff(x,t,1) == -x+u
                diff(phi,t,1)==[0 1; c 0]*diff(phi,s,1)+[0;s]*w+[0 0;0 -b]*phi];\\
>>eq\_out= z ==[int([1 0]*phi,s,[0,1])
                                                    u];\\
>>odepde = addequation(odepde,[eq\_dyn;eq\_out]);
\begin{verbatim}
   Initialized sys() object of type "pde"
   
   5 equations were added to sys() object
\end{verbatim}
\end{matlab}
Whenever, equations are successfully added to the \texttt{sys()} object, a text message confirming the same is displayed in the command output window. To verify if PIETOOLS got the right equations, the user just needs to type the system variable (''odepde'' in this example) on the command window for PIETOOLS to display the added equations. We encourage the user to always check the equations before proceeding.

\subsection{Setting the control signal}
    Since our system has a controlled input, which could have any name, we must pass this information to PIETOOLS. This is done by the following command:
    
\begin{matlab}
    >> odepde= setControl(odepde,[u]);
    \begin{verbatim}
        1 inputs were designated as controlled inputs
    \end{verbatim}
\end{matlab}
 Similarly, observed outputs also need to be specified. For details, see chapter \ref{ch:PDE_DDE_representation}.
 
\subsection{Declaring Boundary Conditions}
A general PDE model is incomplete without boundary conditions, but in PIETOOLS, boundary conditions can be declared in much the same way as the system dynamics. For example, to declare the following Dirichlet,
\[
 \dot{\mbf x} (t, s= 0)= 0,
\]
and Neuman,
\[
 \partial_s \mbf x (t, s= 1)= x(t),
\]
boundary conditions, we can simply call
\begin{matlab}
>> bc1 = [0 1]*subs(phi,s,0) == 0;\\
>> bc2 = [1 0]*subs(phi,s,1) == x;\\
>> odepde = addequation(odepde,[bc1;bc2]);
\begin{verbatim}
    2 equations were added to sys() object
\end{verbatim}
\end{matlab}

\subsection{Simulating ODE-PDE Model}
One of the first things a practitioner might do is to simulate the system. If you look at traditional PDE literature, one of the biggest challenges is simulation. Every different kind of PDE requires different techniques to discretize. More importantly, there is no guarantee that the simulation result will stay bounded. In PIETOOLS, there is only one command to simulate any linear ODE-PDE coupled system of your choice:
\begin{matlab}
    >> solution = PIESIM(odepde, opts, uinput, ndiff);
\end{matlab}
For instance, suppose, we want to simulate the ODE-PDE model corresponding to \eqref{part1:ode} and \eqref{part1:pde} with constant velocity $c=1$, damping coefficient $b=0.1$, under the previously defined boundary conditions, and a specific choice of disturbance $w(t)$. In particular, we aim to investigate how $w(t)$ affects $\dot{\mbf{x}}$ and the value $\mbf x(t,s=1) = \mbf x(t,s=0)$ given by the previously defined output. 

To run this simulation in PIESIM, we use the following opts to the function, which commands PIESIM to don't automatically plot the solution, to use 8 Chebyshev polynomials, to simulate the solutions up to $t=1 s$ with a time-step of $10^{-2} s$, and to use Backward Differentiation Formula (BDF):

\begin{matlab}
    >>opts.plot = 'no';\\   
    >>opts.N = 8;\\         
    >>opts.tf = 10;\\        
    >>opts.dt = 1e-2;\\     
    >>opts.intScheme=1;\\   
    \end{matlab}
  Another important piece of information to PIESIM is regarding the initial conditions and external perturbations for the simulation. This is done as follows for the zero-state response of the system perturbed by an exponentially decaying sinusoidal signal:
\begin{matlab}
    >>uinput.ic.PDE = [0,0];\\  
    >>uinput.ic.ODE = 0;\\  
    >>uinput.u=0;\\
    >>uinput.w = sin(5*st)*exp(-st); 
\end{matlab}
    Note that the control input must also be zero since we want to simulate an open-loop response. The last argument is regarding the space differentiability of the states. In this example, the PDE state involves 2 first order differentiable state variables and this is passed through the input ndiff as:   
\begin{matlab}
       >>ndiff = [0,2,0];    
\end{matlab}

For details on PIESIM arguments, we refer the reader to chapter \ref{ch:PIESIM}. PIESIM gives us discretized time-dependent arrays corresponding to the time vector used in the simulations and the resulting state variables and output. The result is depicted on Figures \ref{fig:DEMO1_OL} and \ref{fig:DEMO1_OL_3D}.

\begin{figure}[htbp]
    \centering
    \includegraphics[width=0.6\textwidth]{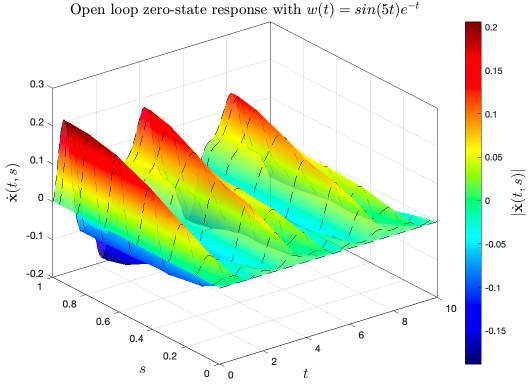}
    \caption{Transient response of the state variable $\dot{\mbf{x}}(t,s)$ by simulating the ODE-PDE model \eqref{part1:ode} and \eqref{part1:pde} with $u(t) = 0$ for external disturbance $w(t) = sin(5t) e^{-t}$.}
    \label{fig:DEMO1_OL_3D}
\end{figure}

\begin{figure}[htbp]
    \centering
    \includegraphics[width=0.6\textwidth]{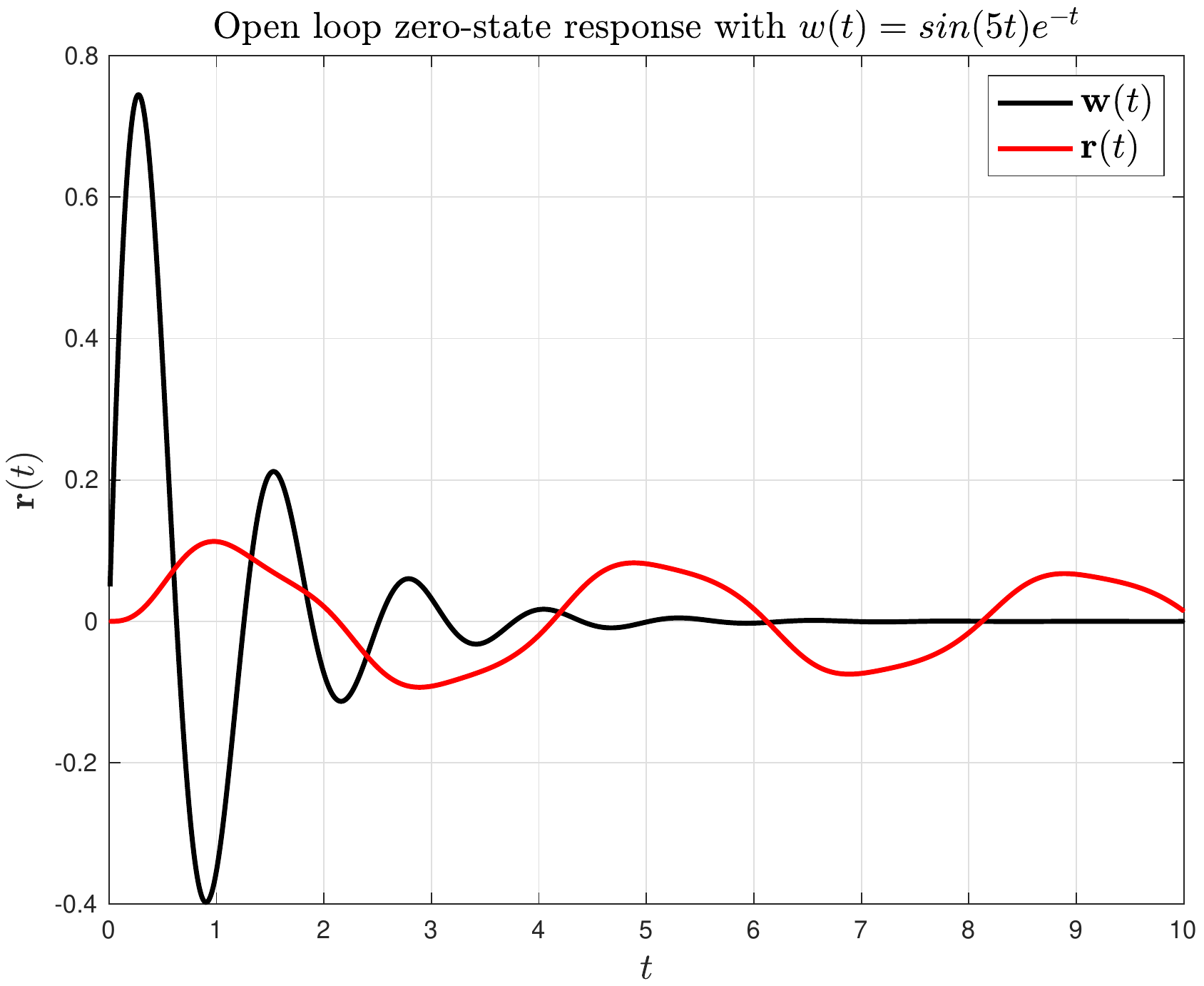}
    \caption{Transient response $r(t)$ of the ODE-PDE model \eqref{part1:ode} and \eqref{part1:pde} with $u(t) = 0$ for external disturbance $w(t) = sin(5t) e^{-t}$.}
    \label{fig:DEMO1_OL}
\end{figure}

\subsection{Analysis and Control of the ODE-PDE Model Using PIEs}
Apart from simulation, you may be interested in knowing whether the model is internally stable or not. Moreover, what would be a good control input such that the effect of external disturbances for a specific choice of output can be suppressed? In PIETOOLS, such an analysis and synthesis are typically performed by first converting the ODE-PDE model to a new representation called Partial Integral Equations (PIEs), which is parametrized by a special class of operators, and then solving convex optimization problems (see chapter.~\ref{ch:LPIs}) for more details. 

Thus, PIEs are an equivalent representation of ODE-PDE models which provide a convenient and efficient way to analyze ODE-PDE models by numerically treatable methods. The conversion from the original system to the PIE representation is simply done using the following command, resulting in the next output on the command window (considering that every previous step was correctly done): 
\begin{matlab}
    >>PIE = convert(odepde,'pie');  
    \begin{verbatim}
    --- Reordering the state components to allow for representation as PIE ---

    The order of the state components x has not changed.

    --- Converting ODE-PDE to PIE --- 
    Initialized sys() object of type ``pde''
    Conversion to pie was successful
    \end{verbatim}
\end{matlab}
Once the model is converted to a PIE, analysis, and control can be performed by calling one of the executive functions.
There are plenty of executive functions available, starting from stability, computing $H_{\infty}$ gain, $H_{\infty}$ optimal state estimator as well as state feedback controllers. 

For instance, the example of this chapter is asymptotically stable only when $b > 0$. This can be shown by calling the executive after one of the following predefined settings had been chosen: \texttt{extreme}, \texttt{stripped}, \texttt{light}, \texttt{heavy}, \texttt{veryheavy}, or \texttt{custom}. For details on the optimization settings, the reader is referred to section.~\ref{sec:executives-settings}. 

\begin{matlab}
    >>settings = lpisettings('heavy');\\
    >>[prog, P] = PIETOOLS\_stability(PIE,settings);
\end{matlab}

If the resultant optimization problem can be solved with these settings, the following message will be displayed after the optimization outputs:
\begin{matlab}
    \begin{verbatim}
        The System of equations was successfully solved.
    \end{verbatim}
\end{matlab}

, which provides an exponential stability certificate for the system. 

 Now, if we want to improve the system's rejection of disturbances, the optimal solution is to design a state-feedback controller that provides a control input $u(t)$ to be applied in \eqref{part1:ode}, which minimizes the $\hinf$ norm of the closed-loop system, provided that such a controller exists. To compute this performance measurement on the open-loop, we just need to call the executive:
 
 \begin{matlab}
    >> [prog, P, gamma] = PIETOOLS\_Hinf\_gain(PIE,settings);
\end{matlab}

 Provided, again, that the optimization problem can be solved, the command window output will display the $\hinf$ norm. In this example, we have:

 \begin{matlab}
    \begin{verbatim}
    The H-infty norm of the given system is upper bounded by:
    5.1631
    \end{verbatim}
\end{matlab}

 For PIETOOLS to synthesize a state feedback controller which minimizes this metric, we need to call a third executive:

 \begin{matlab}
    >> [prog, Kval, gam\_val] = PIETOOLS\_Hinf\_control(PIE, settings);
\end{matlab}

 , which will make PIETOOLS search for the operator $\mcl K$ stored in variable \texttt{Kval} corresponding to the controller, and display the closed loop $\hinf$ norm if succeeded. For this example, the result is a great increase in performance, in terms of this metric:
 
 \begin{matlab}
    \begin{verbatim}
    The closed-loop H-infty norm of the given system is upper bounded by:
    0.9779
    \end{verbatim}
\end{matlab}
 
 The controller is generally a 4-PI linear operator, as detailed described in Chapter.~\ref{ch:PIs}, which has an image parameterized by matrix-valued polynomials. The resultant controller can be displayed by entering its variable name on the command window. Keep in mind that PIETOOLS disregard the monomials with coefficients lower than an accuracy defaulted to $10^-4$.

 We can again use \texttt{PIESIM} to simulate the response of the resultant closed-loop system, as depicted in Figures \ref{fig:DEMO1_CL_3D} and \ref{fig:DEMO1_CL}. Moreover, we can compare the closed-loop with the open-loop performances, as shown by Figure \ref{fig:DEMO1_outputs}. We encourage the user to look at the file \texttt{DEMO1\_Simple\_Stability\_Simulation\_and\_Control.m}, included in the \textit{PIETOOLS\_demos} folder of PIETOOLS, which provides a complete guide to reproduce the results described and depicted on this section. 

\begin{figure}[htbp]
    \centering
    \includegraphics[width=0.6\textwidth]{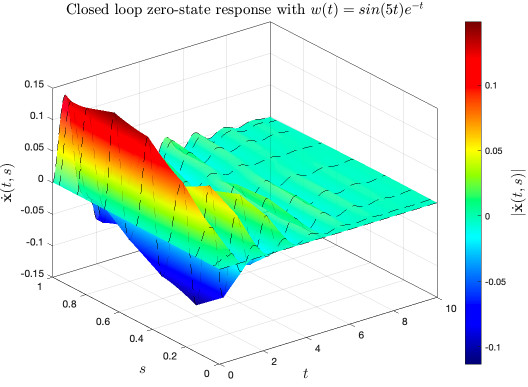}
    \caption{Transient response of the state variable $\dot{\mbf{x}}(t,s)$ on the closed-loop system for external disturbance $w(t) = sin(5t) e^{-t}$.}
    \label{fig:DEMO1_CL_3D}
\end{figure}

\begin{figure}[htbp]
\centering
\subfigure[]{\includegraphics[width=0.495\textwidth]{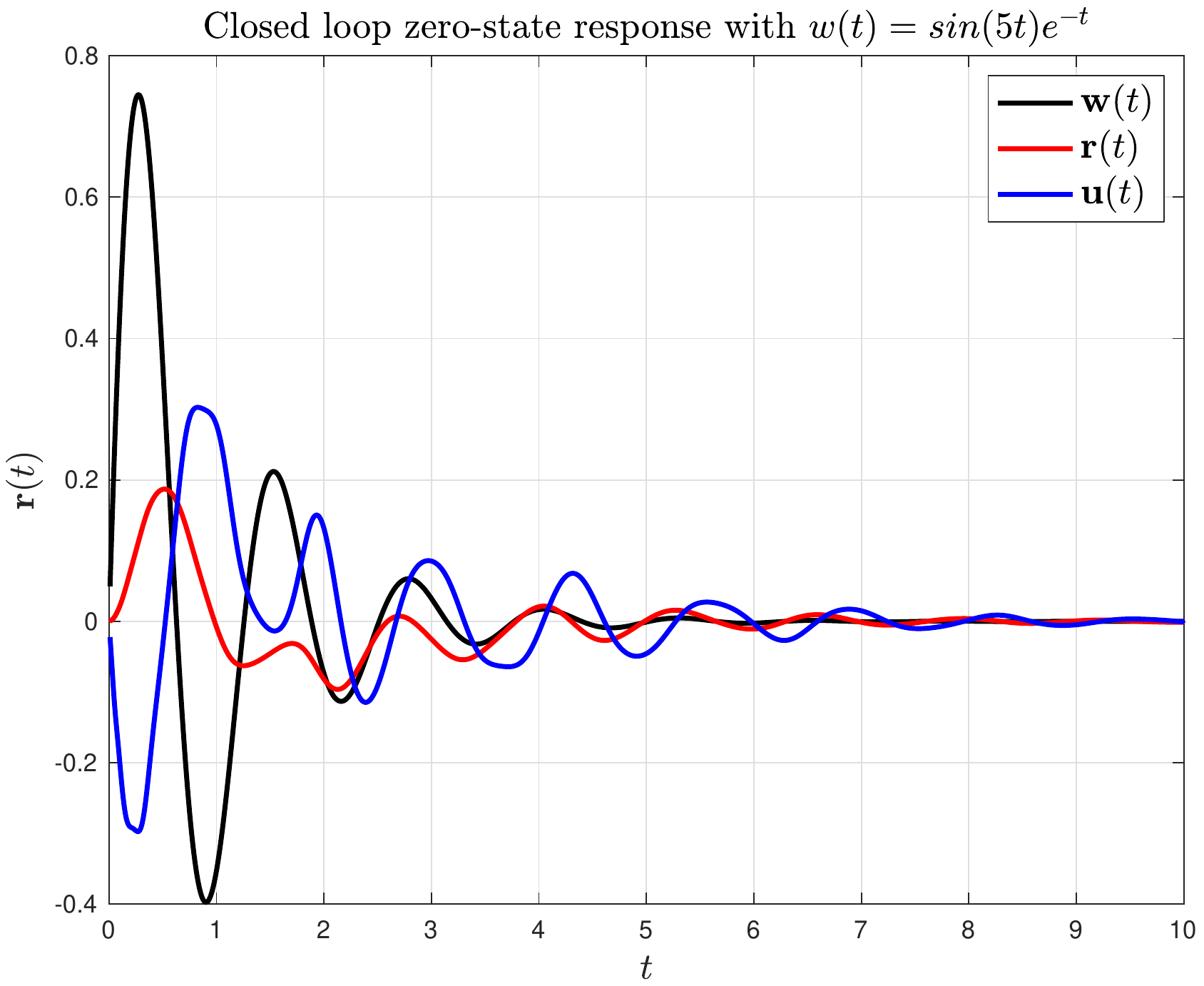}}\label{fig:DEMO1_CL}
\subfigure[]{\includegraphics[width=0.495\textwidth]{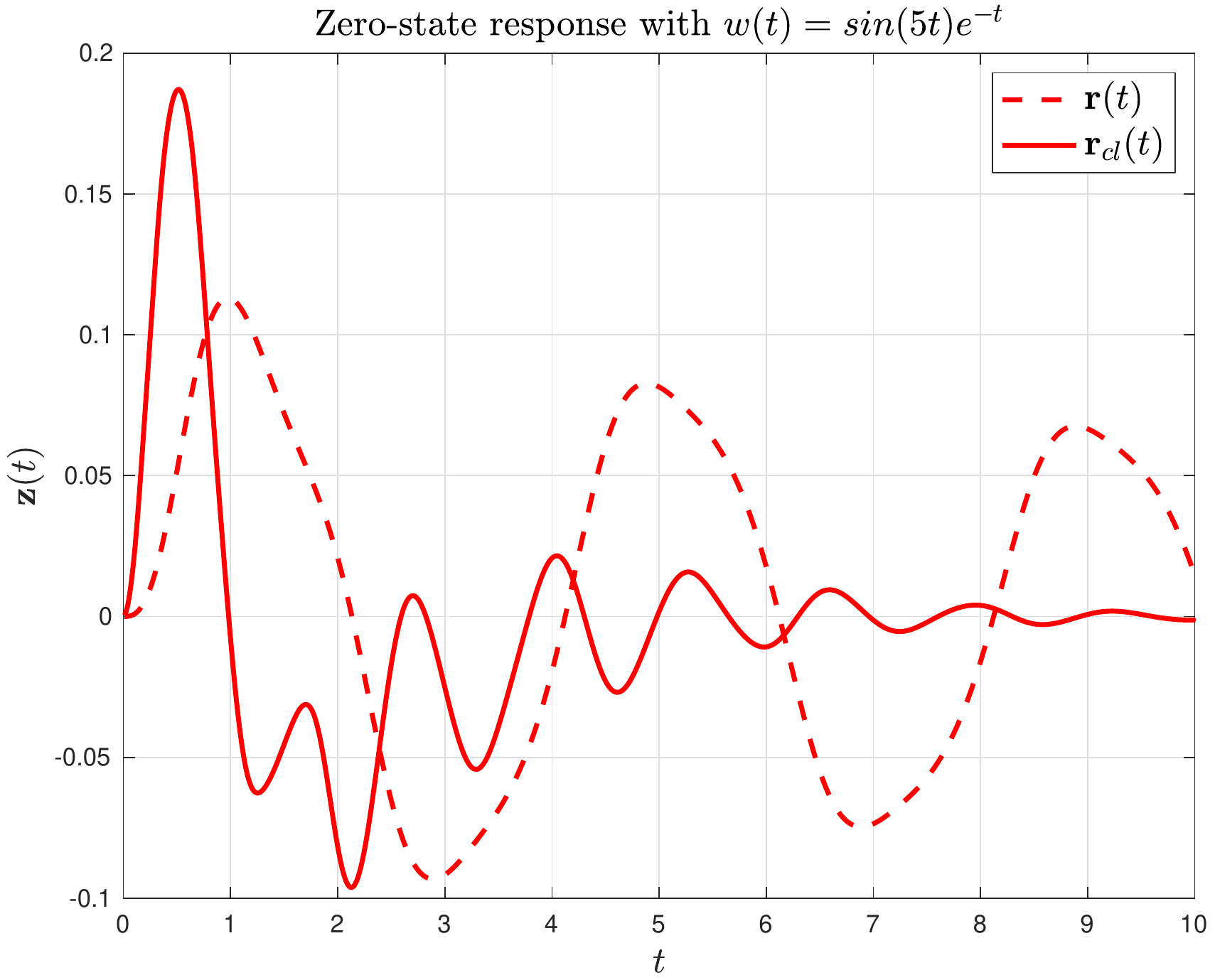}}\label{fig:DEMO1_outputs}
\caption{(a)Transient response of the output $r(t)$ and controlled input $u(t)$ of the closed-loop system for external disturbance $w(t) = sin(5t) e^{-t}$. (b)Comparison of the transient responses of the open-loop system-$r(t)$- with the closed-loop system-$r_{cl}(t)$.} 
\end{figure}

\section{Summary}

In this chapter, we gave an introduction to how PIETOOLS can be used to solve various control-relevant problems involving linear ODE-PDE models. The example depicted here was highly sensitive to disturbances infinite-dimensional system. Figures.~\ref{fig:DEMO1_OL_3D} and.~\ref{fig:DEMO1_OL} show that, even after the applied disturbance has ceased, the output signal $r(t)$ remains affected, taking more time than the final time of the presented simulation to reject the disturbance. This behavior is measured by the computed $\hinf$ norm of the open-loop system. 

On the other hand, with the synthesized feedback controller given by PIETOOLS, the closed-loop system quickly rejects the disturbance, as is clear from Figures.~\ref{fig:DEMO1_CL_3D} and.~\ref{fig:DEMO1_CL}. The increase in performance can be certified by the considerable reduction in the value of the $\hinf$ norm and by comparing the behavior of the outputs without and with the controller, in Figure.~\ref{fig:DEMO1_outputs}.

\chapter{PI Operators in PIETOOLS}\label{ch:PIs}

PIETOOLS primarily functions by manipulation of Partial Integral (PI) operators which is made simple by introduction of MATLAB classes that represent PI operators. 
In PIETOOLS 2022, there are two types of PI operators: PI operators with known parameters, \texttt{opvar/opvar2d} class objects, and PI operators with unknown parameters, \texttt{dopvar/dopvar2d} class objects. In this Chapter, we outline the classes used to represent PI operators with known parameters. The information in this chapter is divided as follows: Section~\ref{sec:PIs:1D} and Section~\ref{sec:PIs:2D} provide brief mathematical background, and corresponding MATLAB implementation, about PI operators in 1D and 2D, respectively. Section~\ref{sec:PIs:overview} provides an overview of the structure of \texttt{opvar/opvar2d} classes in PIETOOLS. For more theoretical background on PI operators, we refer to Appendix~\ref{appx:PI_theory}. For more information on operations that can be performed on \texttt{opvar/opvar2d} class objects, we refer to Chapter~\ref{ch:opvar}.

\section{Declaring PI Operators in 1D}\label{sec:PIs:1D}

In this Section, we illustrate how 1D PI operators can be represented in PIETOOLS using \texttt{opvar} class objects. Here, we say that an operator $\mcl{P}$ is a 1D PI operator if it acts on functions $\mbf{v}(s)$ depending on just one spatial variable $s$, and the operation it performs can be described using partial integrals. We further distinguish 3-PI operators, acting on functions $\mbf{v}\in L_2^{n}[a,b]$, and 4-PI operators, acting on functions $\sbmat{v_0\\\mbf{v}_1}\in\sbmat{\R^{n_0}\\L_2^{n_1}[a,b]}$. Both types of operators can be represented using \texttt{opvar} class objects, as we show in the remainder of this section.

\subsection{Declaring 3-PI Operators}

We first consider declaring a 3-PI operator in PIETOOLS. Here, for given parameters $R=\{R_0,R_1,R_2\}$, the associated 3-PI operator $\mcl{P}[R]:L_2^n[a,b]\rightarrow L_2^m[a,b]$ is given by
\begin{align}\label{eq:3PI_standard_form}
 \bl(\mcl{P}[R]\mbf{v}\br)(s)&=R_0(s)\mbf{v}(s)+\int_{a}^{s}R_1(s,\theta)\mbf{v}(\theta)d\theta + \int_{s}^{b}R_2(s,\theta)\mbf{v}(\theta)d\theta,   &   s&\in[a,b],
\end{align}
for any $\mbf{v}\in L_2^n[a,b]$. In PIETOOLS, we represent such 3-PI operators using \texttt{opvar} class objects. For example, suppose we wish to declare a very simply PI operator $\mcl{A}:L_2^2[-1,1]\rightarrow L_2^2[-1,1]$, defined by
\begin{align}\label{eq:ex_3PI_1}
    \bl(\mcl{A}\mbf{v}\br)(s)&=\int_{-1}^{s}\underbrace{\bmat{1&2\\3&4}}_{R_1}\mbf{v}(\theta)d\theta,   &   s&\in[-1,1].
\end{align}
To declare this operator, we first initialize an empty \texttt{opvar} object \texttt{A}, by simply calling \texttt{opvar} as:
\begin{matlab}
\begin{verbatim}
 >> opvar A
 A = 
    [] | []
    --------
    [] | []
 A.R = 
      [] | [] | []
 >> A.I = [-1,1];
\end{verbatim}    
\end{matlab}
Here, the first line initialize a $0\times 0$ \texttt{opvar} object with all empty parameters \texttt{[]}. The second line, \texttt{A.I=[-1,1]}, then sets the spatial interval associated to the operator equal to $[-1,1]$, indicating that it maps the function space $L_2[-1,1]$. 

Next, we set the parameters of the operator. For a 3-PI operator such as $\mcl{A}$, only the paramaters in the field \texttt{A.R} will be nonzero, where \texttt{A.R} itself has fields \texttt{R0}, \texttt{R1} and \texttt{R2}. For our simple operator, only the parameter $R_1$ in the 3-PI Expression~\eqref{eq:3PI_standard_form} is nonzero, so we only have to assign a value to the field \texttt{R1}:
\begin{matlab}
\begin{verbatim}
 >> A.R.R1 = [1,2; 3,4];
 A = 
    [] | []
    --------
    [] | []
 A.R = 
      [0,0] | [1,2] | [0,0]
      [0,0] | [3,4] | [0,0]
\end{verbatim}    
\end{matlab}
where the fields \texttt{A.R.R0} and \texttt{A.R.R2} automatically default to zero-arrays of the appropriate dimensions. With that, the \texttt{opvar} object \texttt{A} represents the PI operator $\mcl{A}$ as defined in~\eqref{eq:ex_3PI_1}.

Next, suppose we wish to implement a slightly more complicated operator $\mcl{B}:L_2[0,1]\rightarrow L_2^2[0,1]$, defined as
\begin{align*}
    \bl(\mcl{B}\mbf{x}\br)(s)&=\underbrace{\bmat{1\\s^2}}_{R_{0}}\mbf{v}(s)+\int_{0}^{s}\underbrace{\bmat{2s\\s(s-\theta)}}_{R_1}\mbf{x}(\theta)d\theta+\int_{s}^{1}\underbrace{\bmat{3\theta\\\frac{3}{4}(s^2-s)}}_{R_2}\mbf{x}(\theta)d\theta, &
    s&\in[0,1].
\end{align*}
For this operator, the parameters $R_i(s,\theta)$ are all polynomial functions. Such polynomial functions can be represented in PIETOOLS using the \texttt{polynomial} class (from the `multipoly' toolbox), for which operations such as addition, multiplication and concatenation have already been implemented. This means that polynomials such as the functions $R_i$ can be implemented by simply initializing polynomial variables $s$ and $\theta$, and then using these variables to define the desired functions:
\begin{matlab}
\begin{verbatim}
 >> pvar s theta
 >> R0 = [1; s^2]
 R0 = 
   [   1]
   [ s^2]
   
 >> R1 = [2*s; s*(s-theta)]
 R1 = 
   [           2*s]
   [ s^2 - s*theta]
   
 >> R2 = [3*theta; (3/4)*(s^2-s)]
 R2 = 
   [           3*theta]
   [ 0.75*s^2 - 0.75*s]
\end{verbatim}    
\end{matlab}
Here, the first line calls the function \texttt{pvar} to initialize the two polynomial variables \texttt{s} and \texttt{theta}, which we use to represent the spatial variable $s$ and dummy variable $\theta$ respectively. Then, we can add and multiply these variables to represent any desired polynomial in $(s,\theta)$, allowing us to implement the parameters $R_0(s)$, $R_1(s,\theta)$ and $R_2(s,\theta)$. Having defined these parameters, we can then represent the operator $\mcl{B}$ as an \texttt{opvar} object \texttt{b} as before:
\begin{matlab}
\begin{verbatim}
 >> opvar B;
 >> B.I = [0,1];
 >> B.var1 = s;      B.var2 = theta;
 >> B.R.R0 = R0;     B.R.R1 = R1;     B.R.R2 = R2
 B =
       [] | [] 
       ---------
       [] | B.R 

 B.R =
       [1] |         [2*s] |         [3*theta] 
     [s^2] | [s^2-s*theta] | [0.75*s^2-0.75*s] 
\end{verbatim}    
\end{matlab}
Note here that, in addition to specifying the spatial domain $[0,1]$ of the variables using the field \texttt{B.I}, we also have to specify the actual variables $s$ and $\theta$ that appear in the parameters, using the fields \texttt{B.var1} and \texttt{B.var2}. Here \texttt{var1} should correspond to the primary spatial variable, i.e. the variable $s$ on which the function $\mbf{u}(s):=\bl(\mcl{B}\mbf{v}\br)(s)$ will actually depend, and \texttt{B.var2} should correspond to the dummy variable, i.e. the variable $\theta$ which is used solely for integration.

\subsection{Declaring 4-PI Operators}

In addition to 3-PI operators, 4-PI operators can also be represented using the \texttt{opvar} structure. Here, for a given matrix $P$, given functions $Q_1,Q_2$, and 3-PI parameters $R=\{R_0,R_1,R_2\}$, we define the associated 4-PI operator $\mcl{P}\sbmat{P&Q_1\\Q_2&R}:\sbmat{\R^{n_0}\\L_2^{n_1}[a,b]}\rightarrow \sbmat{\R^{m_0}\\L_2^{m_1}[a,b]}$
\begin{align*}
    \bbl(\mcl{P}\sbmat{P&Q_1\\Q_2&R}\mbf{v}\bbr)(s)&=
    \left[\begin{array}{ll}
        Pv_0        \hspace*{-0.1cm}~& +\ \int_{a}^{b}Q_1(s)\mbf{v}_1(s)ds  \\
        Q_2(s)v_0   \hspace*{-0.1cm}& +\ \bl(\mcl{P}[R]\mbf{v}_1\br)(s)
    \end{array}\right],  &
    s&\in[a,b],
\end{align*}
for $\mbf{v}=\sbmat{v_0\\\mbf{v}_1}\in \sbmat{\R^{n_0}\\L_2^{n_1}[a,b]}$. To represent operators of this form, we use the same \texttt{opvar} structure as before, only now also specifying values of the fields \texttt{P}, \texttt{Q1} and \texttt{Q2}. For example, suppose we wish to declare a 4-PI operator $\mcl{C}:\sbmat{\R^2\\L_2[0,3]}\rightarrow \sbmat{\R\\L_2^2[0,3]}$ defined as
\begin{align*}
    \bl(\mcl{C}\mbf{x}\br)(s)&=
    \bbbbl[\begin{array}{ll}
        \overbrace{\sbmat{-1&2}}^{P} x_0        \hspace*{-0.1cm}~& +\ \int_{0}^{3}\overbrace{(3-s^2)}^{Q_1}\mbf{x}_1(s)ds  \\
        \underbrace{\sbmat{0&-s\\s&0}}_{Q_2}v_0   \hspace*{-0.1cm}& +\ \underbrace{\sbmat{1\\s^3}}_{R_0}\mbf{v}_1(s) + \int_{0}^{s}\underbrace{\sbmat{s-\theta\\\theta}}_{R_1}\mbf{v}_1(\theta)d\theta + \int_{s}^{3}\underbrace{\sbmat{s\\ \theta-s}}_{R_2}\mbf{v}_1(\theta)d\theta,
    \end{array}\bbbbr], &
    s&\in[0,3].
\end{align*}
for $\mbf{v}=\sbmat{v_0\\\mbf{v}_1}\in \sbmat{\R^{2}\\L_2^{1}[0,3]}$. To declare this operator, we first construct the polynomial functions defining the parameters $P$ through $R_2$, using \texttt{pvar} objects \texttt{s} and \texttt{tt} to represent $s$ and $\theta$:
\begin{matlab}
\begin{verbatim}
 >> pvar s tt
 >> P = [-1,2];
 >> Q1 = (3-s^2);        
 >> Q2 = [0,-s; s,0];
 >> R0 = [1; s^3];        R1 = [s-tt; tt];     R2 = [s; tt-s];
\end{verbatim}    
\end{matlab}
Having defined the desired parameters, we can then define the operator $\mcl{C}$ as
\begin{matlab}
\begin{verbatim}
 >> opvar C;
 >> C.I = [0,3];
 >> C.var1 = s;      C.var2 = tt;
 >> C.P = P:
 >> C.Q1 = Q1;
 >> C.Q2 = Q2;
 >> C.R.R0 = R0;     C.R.R1 = R1;     C.R.R2 = R2
 C =
     [-1,2] | [-s^2+3] 
     ------------------
     [0,-s] | C.R 
      [s,0] |   

 C.R =
      [1] | [s-tt] |     [s] 
    [s^3] |   [tt] | [-s+tt] 
\end{verbatim}    
\end{matlab}
using the field \texttt{R} to specify the 3-PI sub-component, and using the fields \texttt{P}, \texttt{Q1} and \texttt{Q2} to set the remaining parameters.

\section{Declaring PI Operators in 2D}\label{sec:PIs:2D}

In addition to PI operators in 1D, PI operators in 2D can also be represented in PIETOOLS, using the \texttt{opvar2d} data structure. Here, similarly to how we distinguish 3-PI operators and 4-PI operators for 1D function spaces, we will distinguish 2 classes of 2D operators. In particular, we distinguish the standard 9-PI operators, which act on just functions $\mbf{v}\in L_2\bl[[a,b]\times[c,d]\br]$, and the more general 2D PI operator, acting on coupled functions $\sbmat{v_0\\\mbf{v}_x\\\mbf{v}_y\\\mbf{v}_2}\in\sbmat{\R^{n_0}\\L_2^{n_x}[a,b]\\L_2^{n_y}[c,d]\\L_2^{n_2}[[a,b]\times[c,d]}$.

\subsection{Declaring 9-PI Operators}

For given parameters $N=\sbmat{N_{00}&N_{01}&N_{02}\\N_{10}&N_{11}&N_{12}\\N_{20}&N_{21}&N_{22}}$, the associated 9-PI operator $\mcl{P}[N]:L_2^n\bl[[a,b]\times[c,d]\br]\rightarrow L_2^m\bl[[a,b]\times[c,d]\br]$ is given by
{\small
\begin{align*}
    \left(\mcl{P}[N]\mbf{v}\right)(x,y)= N_{00}(x,y)\mbf{v}(x,y) &+\hspace*{0.0cm} \int_{c}^{y}\! N_{01}(x,y,\nu)\mbf{v}(x,\nu)d\nu + \int_{y}^{d}\! N_{02}(x,y,\nu)\mbf{v}(x,\nu)d\nu \nonumber\\
    +\int_{a}^{x}\! N_{20}(x,y,\theta)\mbf{v}(\theta,y)d\theta &+ \int_{a}^{x}\!\int_{c}^{y}\! N_{11}(x,y,\theta,\nu)\mbf{v}(\theta,\nu)d\nu d\theta + \int_{a}^{x}\!\int_{y}^{d}\! N_{12}(x,y,\theta,\nu)\mbf{v}(\theta,\nu)d\nu d\theta  \nonumber\\
    +\int_{x}^{b} N_{20}(x,y,\theta)\mbf{v}(\theta,y)d\theta &+ \int_{x}^{b}\!\int_{c}^{y}\! N_{21}(x,y,\theta,\nu)\mbf{v}(\theta,\nu)d\nu d\theta + \int_{x}^{b}\!\int_{y}^{d}\! N_{22}(x,y,\theta,\nu)\mbf{v}(\theta,\nu)d\nu d\theta,
 \end{align*}
}
for any $\mbf{v}\in L_2\bl[[a,b]\times[c,d]\br]$. In PIETOOLS 2022, we represent such operators using \texttt{opvar2d} class objects, which are declared in a similar manner to \texttt{opvar} objects. For example, to delcare a simple operator $\mcl{D}:L_2^2\bl[[0,1]\times[1,2]\br]\rightarrow L_2^2\bl[[0,1]\times[1,2]\br]$ defined as
\begin{align*}
    \bl[\mcl{D}\mbf{v}\br](s_1,s_2)&=\int_{0}^{s_1}\int_{s_2}^{2}\underbrace{\bmat{s_1^2&s_1s_2\\s_1s_2 &s_2^2}}_{N_{12}}\mbf{v}(\theta_1,\theta_2)d\theta_2 d\theta_1,   &   (s_1,s_2)&\in[0,1]\times[1,2],
\end{align*}
we first declare the parameter $N_{12}$ defining this operator by representing $s_1$ and $s_2$ by \texttt{pvar} objects \texttt{s1} and \texttt{s2}
\begin{matlab}
\begin{verbatim}
 >> pvar s1 s2
 >> N12 = [s1^2, s1*s2; s1*s2, s2^2];
\end{verbatim}    
\end{matlab}
Then, we initialize an empty \texttt{opvar2d} object \texttt{D} to represent $\mcl{D}$, and assign the variables $(s_1,s_2)$ and their domain $[0,1]\times[1,2]$ to this operator as
\begin{matlab}
\begin{verbatim}
 >> opvar2d D;
 >> D.var1 = [s1;s2];
 >> D.I = [0,1; 1,2];
\end{verbatim}    
\end{matlab}
Note here that, in \texttt{opvar2d} objects, \texttt{var1} is a column vector listing each of the spatial variables $(s_1,s_2)$ on which the result $\mbf{u}(s_1,s_2)=\bl(\mcl{D}\mbf{v}\br)(s_1,s_2)$ depends. Accordingly, the field \texttt{I} in an \texttt{opvar2d} object also has two rows, with each row specifying the interval on which the variable in the associated row of \texttt{var1} exists. Having initialized the operator, we then assign the parameter \texttt{N12} to the appropriate field. Here, the parameters defining a 9-PI operator are stored in the $3\times 3$ cell array \texttt{D.R22}, with \texttt{R22} referring to the fact that these parameters map 2D functions to 2D functions. Within this array, element \texttt{\{i,j\}} for $i,j\in\{1,2,3\}$ corresponds to parameter $N_{i-1,j-1}$ in the operator, and so we can specify parameter $N_{12}$ using element \texttt{\{2,3\}}:
\begin{matlab}
\begin{verbatim}
 >> D.R22{2,3} = N12
 D =
     [] |    [] |    [] |    [] 
     --------------------------
     [] | D.Rxx |    [] | D.Rx2 
     --------------------------
     [] |    [] | D.Ryy | D.Ry2 
     --------------------------
     [] | D.R2x | D.R2y | D.R22 

 D.Rxx =
     [] | [] | [] 

 D.Rx2 =
     [] | [] | [] 

 D.Ryy =
     [] | [] | [] 

 D.Ry2 =
     [] | [] | [] 

 D.R2x =
     [] | [] | [] 

 D.R2y =
     [] | [] | [] 

 D.R22 =
     [0,0] | [0,0] |        [0,0] 
     [0,0] | [0,0] |        [0,0] 
     ---------------------------- 
     [0,0] | [0,0] | [s1^2,s1*s2] 
     [0,0] | [0,0] | [s1*s2,s2^2] 
     ---------------------------- 
     [0,0] | [0,0] |        [0,0] 
     [0,0] | [0,0] |        [0,0] 
\end{verbatim}    
\end{matlab}
We note that, in the resulting structure, there are a lot of empty parameters, such as \texttt{D.Rxx}. As we will discuss in the next subsection, these parameters correspond to maps to and from other functions spaces, just like the parameters \texttt{P} and \texttt{Qi} in the \texttt{opvar} structure. Since the operator $\mcl{D}$ maps only functions $L_2^2\bl[[0,1]\times[1,2]\br]\rightarrow L_2^2\bl[[0,1]\times[1,2]\br]$, all parameters mapping different function spaces are empty for the object \texttt{D}.

Suppose now we want to declare a 9-PI operator $\mcl{E}:L_2\bl[[0,1]\times[-1,1]\br]\rightarrow L_2\bl[[0,1]\times[-1,1]\br]$ defined by
\begin{align*}
    \bl(\mcl{E}\mbf{v}\br)(s_1,s_2)&=\overbrace{x^2y^2}^{N_{00}}\mbf{v}(s_1,s_2) + \int_{-1}^{s_2}\overbrace{s_1(s_2-\theta_2)}^{N_{01}}\mbf{v}(s_1,\theta_2)d\theta_2   \\
    &+ \int_{s_1}^{1}\underbrace{(s_1-\theta_1)s_2}_{N_{20}}\mbf{v}(\theta_1,s_2)d\theta_1 + \int_{s_1}^{1}\int_{-1}^{s_2}\underbrace{(s_1-\theta_1)(s_2-\theta_2)}_{N_{21}}\mbf{v}(\theta_1,\theta_2)d\theta_2 d\theta_1
\end{align*}
As before, we first set the values of the parameters $N_{ij}$, using \texttt{s1}, \texttt{s2}, \texttt{th1} and \texttt{th2} to represent $s_1$, $s_2$, $\theta_1$ and $\theta_2$ respectively:
\begin{matlab}
\begin{verbatim}
 >> pvar s1 s2 th1 th2
 >> N00 = s1^2 * s2^2;
 >> N01 = s1*(s2-th2);
 >> N20 = (s1-th1)*s2;
 >> N21 = (s1-th1)*(s2-th2);
\end{verbatim}    
\end{matlab}
Next, we initialize an \texttt{opvar2d} object \texttt{E} with the appropriate variables and domain as
\begin{matlab}
\begin{verbatim}
 >> opvar2d E;
 >> E.var1 = [s1;s2];     E.var2 = [th1; th2];
 >> E.I = [0,1; -1,1];
\end{verbatim}    
\end{matlab}
where in this case we set both the primary variables, using \texttt{var1}, and the dummy variables, using \texttt{var2}. Note here that the domains of the first and second dummy variables are the same as those of the first and second primary variables, and are defined in the first and second row of \texttt{I} respectively. Finally, we assign the parameters $N_{ij}$ to the appropriate elements of \texttt{R22}
\begin{matlab}
\begin{verbatim}
 >> E.R22{1,1} = N00;     E.R22{1,2} = N01;
 >> E.R22{3,1} = N20;     E.R22{3,2} = N21
 E =
     [] |    [] |    [] |    [] 
     --------------------------
     [] | E.Rxx |    [] | E.Rx2 
     --------------------------
     [] |    [] | E.Ryy | E.Ry2 
     --------------------------
     [] | E.R2x | E.R2y | E.R22 

 E.R22 =

        [s1^2*s2^2] |                [s1*s2-s1*th2] | [0] 
     ---------------------------------------------------- 
                [0] |                           [0] | [0] 
     ---------------------------------------------------- 
     [s1*s2-s2*th1] | [s1*s2-s1*th2-s2*th1+th1*th2] | [0] 
\end{verbatim}    
\end{matlab}
so that \texttt{E} represents the desired operator.

\subsection{Declaring General 2D PI Operators}

The most general PI operators that can be represented in PIETOOLS 2022 are those mapping $\sbmat{\R^{n_0}\\L_2^{n_x}[a,b]\\L_2^{n_y}[c,d]\\L_2^{n_2}\bl[[a,b]\times[c,d]\br]}\rightarrow \sbmat{\R^{m_0}\\L_2^{m_x}[a,b]\\L_2^{m_y}[c,d]\\L_2^{m_2}\bl[[a,b]\times[c,d]\br]}$, defined by parameters $R=\sbmat{R_{00}&R_{0x}&R_{0y}&R_{02}\\R_{x0}&R_{xx}&R_{xy}&R_{x2}\\R_{y0}&R_{yx}&R_{yy}&R_{y2}\\R_{20}&R_{2x}&R_{2y}&R_{22}}$ as
{\small
\begin{align*}
    \bl(\mcl{P}[R]\mbf{x}\br)(s)=
    \left[\!\!\begin{array}{llll}
        R_{00}v_0 & \hspace*{-0.2cm}+\ \int_{a}^{b}R_{0x}(x)\mbf{v}_{x}(x)dx & \hspace*{-0.2cm}+\ \int_{c}^{d}R_{0y}(y)\mbf{v}_{y}(y)dy & \hspace*{-0.2cm}+\ \int_{a}^{b}\int_{c}^{d}R_{02}(x,y)\mbf{v}_{2}(x,y)dydx \\
        R_{x0}(x)v_0 & \hspace*{-0.2cm}+\ \bl(\mcl{P}[R_{xx}]\mbf{v}_{x}\br)(x) & \hspace*{-0.2cm}+\ \int_{c}^{d}R_{xy}(x,y)\mbf{v}_{y}(y)dy & \hspace*{-0.2cm}+\ \int_{c}^{d}\bl(\mcl{P}[R_{x2}]\mbf{v}_2\br)(x,y) dy \\
        R_{y0}(y)v_0 & \hspace*{-0.2cm}+\ \int_{a}^{b}R_{yx}(x,y)\mbf{v}_{x}(x)dx & \hspace*{-0.2cm}+\ \bl(\mcl{P}[R_{yy}]\mbf{v}_{y}\br)(y) & \hspace*{-0.2cm}+\ \int_{a}^{b}\bl(\mcl{P}[R_{y2}]\mbf{v}_2\br)(x,y) dx \\
        R_{20}(x,y)v_0 & \hspace*{-0.2cm}+\ \bl(\mcl{P}[R_{2x}]\mbf{v}_{x}\br)(x,y) & \hspace*{-0.2cm}+\ \bl(\mcl{P}[R_{2y}]\mbf{v}_{y}\br)(x,y) & \hspace*{-0.2cm}+\ \bl(\mcl{P}[R_{22}]\mbf{v}_2\br)(x,y) \\
    \end{array}\!\right]
\end{align*}
}
for $\mbf{v}=\sbmat{v_0\\\mbf{v}_x\\\mbf{v}_y\\\mbf{v}_2}\in \sbmat{\R^{n_0}\\L_2^{n_x}[a,b]\\L_2^{n_y}[c,d]\\L_2^{n_2}\bl[[a,b]\times[c,d]\br]}$, where $\mcl{P}[R_{xx}]$, $\mcl{P}[R_{yy}]$, $\mcl{P}[R_{x2}]$, $\mcl{P}[R_{y2}]$, $\mcl{P}[R_{2x}]$ and $\mcl{P}[R_{2y}]$ are 3-PI operators, and where $\mcl{P}[R_{22}]$ is a 9-PI operator. These types of PI operators are also represented using the \texttt{opvar2d} class, specifying each of the parameters $R_{ij}$ using the associated fields \texttt{Rij}. For example, suppose we want to implement a PI operator $\mcl{F}:\sbmat{\R\\L_2[0,2]\\L_2\bl[[0,2]\times[2,3]\br]}\rightarrow \sbmat{L_2^{2}[0,2]\\L_2\bl[[0,2]\times[2,3]\br]}$, defined as
\begin{align*}
    \bl(\mcl{F}\mbf{v}\br)(x,y)&=
    \bbbl[\begin{array}{lll}
         \overbrace{\sbmat{1\\x}}^{R_{x0}}v_0 & \hspace*{-0.2cm}+ \overbrace{\sbmat{x\\x^2}}^{R_{xx}^{0}}\mbf{v}_1(x) + \int_{x}^{2}\overbrace{\sbmat{1\\(\theta-x)}}^{R_{xx}^{2}}\mbf{v}_1(\theta)d\theta  & \hspace*{-0.2cm}+\int_{2}^{3}\int_{0}^{x}\overbrace{\sbmat{y\\y^2(x-\theta)}}^{R_{x2}^{1}}\mbf{v}_2(\theta,y)dy \\
         & \hspace*{0.35cm} \underbrace{y^2}_{R_{2x}^{0}}\mbf{v}_1(x) + \int_{0}^{x}\underbrace{y}_{R_{2x}^{1}}\mbf{v}_1(\theta)d\theta & \hspace*{-0.2cm}+ \int_{0}^{x}\int_{2}^{y}\underbrace{\theta\nu}_{R_{22}^{11}} \mbf{v}_2(\theta,\nu)d\nu d\theta
    \end{array}\bbbr],
\end{align*}
for $\mbf{v}=\sbmat{v_0\\\mbf{v}_1\\\mbf{v}_2}\in \sbmat{\R\\L_2[0,2]\\L_2\bl[[0,2]\times[2,3]\br]}$.  To declare this operator, we define the parameters as before as
\begin{matlab}
\begin{verbatim}
 >> pvar x y theta nu
 >> Rx0 = [1; x];
 >> Rxx_0 = [x; x^2];     Rxx_2 = [1; theta-x];
 >> Rx2_1 = [y; y^2 * (x-theta)];
 >> R2x_0 = y^2;         R2x_1 = y;
 >> R22_11 = theta*nu;
\end{verbatim}    
\end{matlab}
and then declare the \texttt{opvar2d} object as
\begin{matlab}
\begin{verbatim}
 >> opvar2d F;
 >> F.var1 = [x; y];      F.var2 = [theta; nu];
 >> F.I = [0,2; 2,3];
 >> F.Rx0 = Rx0;
 >> F.Rxx{1} = Rxx_0;     F.Rxx{3} = Rxx_2;
 >> F.Rx2{2} = Rx2_1;
 >> F.R2x{1} = R2x_0;     F.R2x{2} = R2x_1;
 >> F.R22{2,2} = R22_11;
\end{verbatim}    
\end{matlab}
yielding a structure 
\begin{matlab}
\begin{verbatim}
 >> F
 F =
 
      [] |    [] |    [] |    [] 
     ---------------------------
     [1] | F.Rxx |    [] | F.Rx2 
     [x] |       |       |       
     ---------------------------
      [] |    [] | F.Ryy | F.Ry2 
     ---------------------------
     [0] | B.R2x | F.R2y | F.R22 

 F.Rxx =
 
       [x] | [0] |       [1] 
     [x^2] | [0] | [theta-x] 

 F.Rx2 =
 
     [0] |                [y] | [0] 
     [0] | [-theta*y^2+x*y^2] | [0] 

 F.Ryy =
 
     [] | [] | [] 

 F.Ry2 =
 
     [] | [] | [] 

 F.R2x =
 
     [y^2] | [y] | [0] 

 F.R2y =
 
     [] | [] | [] 

 F.R22 =
 
     [0] |        [0] | [0] 
     ---------------------- 
     [0] | [nu*theta] | [0] 
     ---------------------- 
     [0] |        [0] | [0] 
\end{verbatim}
\end{matlab}
Representing the operator $\mcl{F}$.

In the following subsection, we provide an overview of how the \texttt{opvar} and \texttt{opvar2d} data structures are defined.

\newpage

\section{Overview of \texttt{opvar} and \texttt{opvar2d} Structure}\label{sec:PIs:overview}

\subsection{\texttt{opvar} class}

Let $\mcl{B}:\bmat{\R^{n_0}\\L_2^{n_1}[a,b]}\rightarrow \bmat{\R^{m_0}\\L_2^{m_1}[a,b]}$ be a 4-PI operator of the form
\begin{align}\label{eq:4PI_standard_form}
    \bl(\mcl{B}\mbf{x}\br)(s)=
    \left[\begin{array}{ll}
        Px_0        \hspace*{-0.1cm}~& +\ \int_{a}^{b}Q_1(s)\mbf{x}_1(s)ds  \\
        Q_2(s)x_0   \hspace*{-0.1cm}& +\ R_{0}(s)\mbf{x}_1(s) + \int_{a}^{s}R_{1}(s,\theta)\mbf{x}_1(\theta)d\theta + \int_{s}^{b}R_{2}(s,\theta)\mbf{x}_1(\theta)d\theta
    \end{array}\right]
\end{align}
for $\mbf{x}=\bmat{x_0\\\mbf{x}_1}\in \bmat{\R^{n_0}\\L_2^{n_1}[a,b]}$. Then, we can represent this operator as an \texttt{opvar} object \texttt{B} with fields as defined in Table~\ref{tab:opvar_fields}.

\begin{table}[!th]
\renewcommand{\arraystretch}{1.0}
\fontsize{11}{13}
 \begin{tabular}{p{1.0cm}p{2.00cm}p{12.75cm}}
 \hline
    \texttt{B.dim}    & \texttt{= [m0,n0; \hspace*{0.4cm} m1,n1]} 
    &  $2\times 2$ array of type \texttt{double} specifying the dimensions of the function spaces $\sbmat{\R^{m_0}\\L_2^{m_1}[a,b]}$ and $\sbmat{\R^{n_0}\\L_2^{n_1}[a,b]}$ the operator maps to and from;\\
    \texttt{B.var1} & \texttt{= s}    &  $1\times 1$ \texttt{pvar} (\texttt{polynomial} class) object specifying the spatial variable $s$; \\
    \texttt{B.var2} & \texttt{= theta}    &  $1\times 1$ \texttt{pvar} (\texttt{polynomial} class) object specifying the dummy variable $\theta$;   \\
    \texttt{B.I} & \texttt{= [a,b]}       &  $1\times 2$ array of type \texttt{double}, specifying the interval $[a,b]$ on which the spatial variables $s$ and $\theta$ exist; \\
    \texttt{B.P} & \texttt{= P} & $m_0\times n_0$ array of type \texttt{double} or \texttt{polynomial} defining the matrix $P$; \\
    \texttt{B.Q1} & \texttt{= Q1} & $m_0\times n_1$ array of type \texttt{double} or \texttt{polynomial} defining the function $Q_1(s)$; \\
    \texttt{B.Q2} & \texttt{= Q2} & $m_1\times n_0$ array of type \texttt{double} or \texttt{polynomial} defining the function $Q_2(s)$; \\
    \texttt{B.R.R0} & \texttt{= R0} & $m_1\times n_1$ array of type \texttt{double} or \texttt{polynomial} defining the function $R_0(s)$; \\
    \texttt{B.R.R1} & \texttt{= R1} & $m_1\times n_1$ array of type \texttt{double} or \texttt{polynomial} defining the function $R_1(s,\theta)$; \\
    \texttt{B.R.R2} & \texttt{= R2} & $m_1\times n_1$ array of type \texttt{double} or \texttt{polynomial} defining the function $R_2(s,\theta)$; \\
    \hline
 \end{tabular}
 \caption{Fields in an \texttt{opvar} object \texttt{B}, defining a general 4-PI operator as in Equation~\eqref{eq:4PI_standard_form}}
\label{tab:opvar_fields}
\end{table}

\subsection{\texttt{opvar2d} class}

Let $\mcl{D}:\sbmat{\R^{n_0}\\L_2^{n_x}[a,b]\\L_2^{n_y}[c,d]\\L_2^{n_2}\bl[[a,b]\times[c,d]\br]}\rightarrow \sbmat{\R^{m_0}\\L_2^{m_x}[a,b]\\L_2^{m_y}[c,d]\\L_2^{m_2}\bl[[a,b]\times[c,d]\br]}$ be a PI operator of the form
{\small
\begin{align}\label{eq:0112PI_standard_form}
    \bl(\mcl{D}\mbf{x}\br)(s)=
    \left[\!\!\begin{array}{llll}
        R_{00}v_0 & \hspace*{-0.2cm}+\ \int_{a}^{b}R_{0x}(x)\mbf{v}_{x}(x)dx & \hspace*{-0.2cm}+\ \int_{c}^{d}R_{0y}(y)\mbf{v}_{y}(y)dy & \hspace*{-0.2cm}+\ \int_{a}^{b}\int_{c}^{d}R_{02}(x,y)\mbf{v}_{2}(x,y)dydx \\
        R_{x0}(x)v_0 & \hspace*{-0.2cm}+\ \bl(\mcl{P}[R_{xx}]\mbf{v}_{x}\br)(x) & \hspace*{-0.2cm}+\ \int_{c}^{d}R_{xy}(x,y)\mbf{v}_{y}(y)dy & \hspace*{-0.2cm}+\ \int_{c}^{d}\bl(\mcl{P}[R_{x2}]\mbf{v}_2\br)(x,y) dy \\
        R_{y0}(y)v_0 & \hspace*{-0.2cm}+\ \int_{a}^{b}R_{yx}(x,y)\mbf{v}_{x}(x)dx & \hspace*{-0.2cm}+\ \bl(\mcl{P}[R_{yy}]\mbf{v}_{y}\br)(y) & \hspace*{-0.2cm}+\ \int_{a}^{b}\bl(\mcl{P}[R_{y2}]\mbf{v}_2\br)(x,y) dx \\
        R_{20}(x,y)v_0 & \hspace*{-0.2cm}+\ \bl(\mcl{P}[R_{2x}]\mbf{v}_{x}\br)(x,y) & \hspace*{-0.2cm}+\ \bl(\mcl{P}[R_{2y}]\mbf{v}_{y}\br)(x,y) & \hspace*{-0.2cm}+\ \bl(\mcl{P}[R_{22}]\mbf{v}_2\br)(x,y) \\
    \end{array}\!\right]
\end{align}
}
for $\mbf{v}=\sbmat{v_0\\\mbf{v}_x\\\mbf{v}_y\\\mbf{v}_2}\in \sbmat{\R^{n_0}\\L_2^{n_x}[a,b]\\L_2^{n_y}[c,d]\\L_2^{n_2}\bl[[a,b]\times[c,d]\br]}$, where $\mcl{P}[R_{xx}]$, $\mcl{P}[R_{yy}]$, $\mcl{P}[R_{x2}]$, $\mcl{P}[R_{y2}]$, $\mcl{P}[R_{2x}]$ and $\mcl{P}[R_{2y}]$ are 3-PI operators, and where $\mcl{P}[R_{22}]$ is a 9-PI operator.
We can represent the operator $\mcl{D}$ as an \texttt{opvar2d} object \texttt{D} with fields as defined in Table~\ref{tab:opvar2d_fields}.

\begin{table}[!ht]
\renewcommand{\arraystretch}{1.0}
\fontsize{11}{13}
 \begin{tabular}{p{1.0cm}p{2.00cm}p{12.75cm}}
 \hline
    \texttt{D.dim}    & \texttt{= [m0,n0; \hspace*{0.4cm} mx,nx; \hspace*{0.4cm} my,ny; \hspace*{0.4cm} m2,n2;]} 
    &  $4\times 2$ array of type \texttt{double} specifying the dimensions of the function spaces $\sbmat{\R^{m_0}\\L_2^{m_x}[a,b]\\L_2^{m_y}[c,d]\\L_2^{m_2}\bl[[a,b]\times[c,d]\br]}$ and $\sbmat{\R^{n_0}\\L_2^{n_x}[a,b]\\L_2^{n_y}[c,d]\\L_2^{n_2}\bl[[a,b]\times[c,d]\br]}$ the operator maps to and from;\\
    \texttt{D.var1} & \texttt{= [x; y]}    &  $2\times 1$ \texttt{pvar} (\texttt{polynomial} class) object specifying the spatial variables $(x,y)$; \\
    \texttt{D.var2} & \texttt{= [theta; \hspace*{0.4cm} nu]}    &  $2\times 1$ \texttt{pvar} (\texttt{polynomial} class) object specifying the dummy variables $(\theta,\nu)$;   \\
    \texttt{D.I} & \texttt{= [a,b; \hspace*{0.4cm} c,d]}       &  $2\times 2$ array of type \texttt{double}, specifying the domain $[a,b]\times[c,d]$ on which the spatial variables $(x,\theta)$ and $(y,\nu)$ exist; \\
    \texttt{D.R00} & \texttt{= R00} & $m_0\times n_0$ array of type \texttt{double} or \texttt{polynomial} defining the matrix $R_{00}$; \\
    \texttt{D.R0x} & \texttt{= R0x} & $m_0\times n_x$ array of type \texttt{double} or \texttt{polynomial} defining the function $R_{0x}(x)$; \\
    \texttt{D.R0y} & \texttt{= R0y} & $m_0\times n_y$ array of type \texttt{double} or \texttt{polynomial} defining the function $R_{0y}(y)$; \\
    \texttt{D.R02} & \texttt{= R02} & $m_0\times n_2$ array of type \texttt{double} or \texttt{polynomial} defining the function $R_{02}(x,y)$; \\
    \texttt{D.Rx0} & \texttt{= Rx0} & $m_x\times n_0$ array of type \texttt{double} or \texttt{polynomial} defining the function $R_{x0}(x)$; \\
    \texttt{D.Rxx} & \texttt{= Rxx} & $3\times 1$ cell array specifying the 3-PI parameters $R_{xx}$; \\
    \texttt{D.Rxy} & \texttt{= Rxy} & $m_x\times n_y$ array of type \texttt{double} or \texttt{polynomial} defining the function $R_{xy}(x,y)$; \\
    \texttt{D.Rx2} & \texttt{= Rx2} & $3\times 1$ cell array specifying the 3-PI parameters $R_{x2}$; \\
    \texttt{D.Ry0} & \texttt{= Ry0} & $m_y\times n_0$ array of type \texttt{double} or \texttt{polynomial} defining the function $R_{y0}(y)$; \\
    \texttt{D.Ryx} & \texttt{= Ryx} & $m_y\times n_x$ array of type \texttt{double} or \texttt{polynomial} defining the function $R_{yx}(x,y)$; \\
    \texttt{D.Ryy} & \texttt{= Ryy} & $1\times 3$ cell array specifying the 3-PI parameters $R_{yy}$; \\
    \texttt{D.Ry2} & \texttt{= Ry2} & $1\times 3$ cell array specifying the 3-PI parameters $R_{y2}$; \\
    \texttt{D.R20} & \texttt{= R20} & $m_2\times n_0$ array of type \texttt{double} or \texttt{polynomial} defining the function $R_{20}(x,y)$; \\
    \texttt{D.R2x} & \texttt{= R2x} & $3\times 1$ cell array specifying the 3-PI parameters $R_{2x}$; \\
    \texttt{D.R2y} & \texttt{= R2y} & $1\times 3$ cell array specifying the 3-PI parameters $R_{2y}$; \\
    \texttt{D.R22} & \texttt{= R22} & $3\times 3$ cell array specifying the 9-PI parameters $R_{22}$; \\\hline
 \end{tabular}
 \caption{Fields in an \texttt{opvar2d} object \texttt{D}, defining a general PI operator in 2D as in Equation~\eqref{eq:0112PI_standard_form}}
\label{tab:opvar2d_fields}
\end{table}

\part{PIETOOLS Workflow for ODE-PDE and DDE Models}

\chapter{Setup and Representation of PDEs and DDEs}\label{ch:PDE_DDE_representation}

Using PIETOOLS, a wide variety of linear differential equations and time-delay systems can be simulated and analysed by representing them as partial integral equations (PIEs). To facilitate this, PIETOOLS includes several input format to declare partial differential equations (PDEs) and delay-differential equations (DDEs), which can then be easily converted to equivalent PIEs using the PIETOOLS function \texttt{convert}, as we show in Chapter~\ref{ch:PIE}. In this chapter, we present two of these input formats, discussing in detail how linear 1D PDE and DDE systems can be easily implemented using the Command Line Parser for PDEs and Batch-Based input format for DDEs. We refer to Chapter~\ref{ch:alt_PDE_input} for information on two alternative input formats for PDEs, including a format to parse 2D PDEs, and we refer to Chapter~\ref{ch:alt_DDE_input} for two alternative input formats for time-delay systems, namely the Neutral Delay System (NDS) and Delay Differential Equation (DDF) formats.



\section{Command Line Parser for 1D ODE-PDEs}
In PIETOOLS 2022, by far the simplest and most intuitive of these input formats is the Command Line Parser format. Command Line Parser format utilizes MATLAB variables of class \texttt{state} to define symbols that can be freely manipulated to express PDEs as MATLAB expressions that are stored in the \texttt{sys} class object. \texttt{sys} class objects are used to store parameters of a coupled ODE-PDE and to perform analysis, control, and simulation of the stored system. Refer to \ref{subsec:sys} for more details.

\begin{boxEnv}{\textbf{Requirements}}
\begin{itemize}
    \item Command Line Parser format can be used to defined ODE-PDEs with ONLY one spatial dimension.
    \item Command Line Parser format supports ODEs/PDEs with time-delayed terms.
    \item Command Line Parser format requires MATLAB 2021a or higher.
    \item Variables \texttt{s}, \texttt{t}, and \texttt{theta} are protected and cannot be used for other purposes.
\end{itemize}

\end{boxEnv}

\subsection{Defining a coupled ODE-PDE system}
 For the purpose of demonstration, consider the following coupled ODE-PDE system in control theory framework 
\begin{align*}
\dot{x}(t) &= -5 x(t)+\int_0^1 \partial_s \mbf x(t,s) ds + u(t)\\
\dot{\mbf x}(t,s) &= 9 \mbf x(t,s)+ \partial_s^2 \mbf x(t,s) +sw(t)\\
\mbf x(t,0) &= 0, \quad \partial_s\mbf x(t,1) + x(t) = 2w(t)\\
z(t) &= \bmat{\int_0^1 \mbf x(t,s) ds\\ u(t)}\\
y(t) &= \mbf x(t,0).
\end{align*}

\begin{codebox}
    The above set of equations can be defined and converted to a PIE using the code shown below.
    \begin{matlab}
        >> pvar ~t ~s ~theta; \\
        >> x = state('ode'); ~X = state('pde');\\ 
        >> w = state('in'); z = state('out',2);\\
        >> u = state('in'); y = state('out');\\
        >> odepde= sys(); \\
        >> odepde = addequation(odepde, z==[int(X,s,[0,1]); u]);\\
        >> eqns = [diff(x,t)==-5*x+int(diff(X,s,1),s,[0,1])+u;\\
        diff(X,t)==9*X+diff(X,s,2)+s*w; \\
        subs(X,s,0)==0;\\ 
        subs(diff(X,s),s,1)==-x+2*w;\\ 
        y==subs(X,s,0)];\\
        >> odepde = addequation(odepde,eqns);\\
        >> odepde = setControl(odepde,[u]);\\
        >> odepde = setObserve(odepde,[y]); \\
        >> sys\_pie = convert(odepde,'pie');
    \end{matlab}
\end{codebox}

Next we will breakdown each step used in the code above and explain the action performed by each line of the above code block. Specifying any PDE system using the `Command Line Parser' format follows the steps listed below:
\begin{enumerate}
    \item Define independent variables ($s$, $t$) 
    \item Define dependent variables ($\mbf x$, $x$, $z$, $y$, $w$ and $u$)
    \item Define a \texttt{sys()} object to store the equations
    \item Add equations
    \item Specify control inputs and observed outputs
\end{enumerate}

\subsubsection{Define independent variables}

To define equations symbolically, first, the independent variables (spatial variable and time variable) and dependent variables (states, inputs, and outputs) have to be declared. For example, if the PDE is defined on space $s$ and time $t$, we would start by defining these variables as \texttt{polynomial} objects as shown below.

\begin{matlab}
>> pvar ~t ~s ~theta;  ~\% independent variables are polynomial objects
\end{matlab}

\noindent\textbf{Note} that we have defined additional variable \texttt{theta} which will be used as a dummy spatial variable if needed (for example, in operations involving integration).

\subsubsection{Define dependent variables}
After defining independent variables, we need to define dependent variables such as ODE/PDE states, inputs, and outputs (See \ref{ch:scope} for details). Dependent variables are defined as \texttt{state} class objects. For example:

\begin{matlab}
>> x = state('ode'); ~X = state('pde');\\ 
>> w = state('in'); z = state('out',2);\\
>> u = state('in'); y = state('out');
\end{matlab}

The above code, when executed in MATLAB, creates four symbolic variables, namely \texttt{x, X, w, u, z, y}, and assigns them the type ODE, PDE, input (\texttt{w, u}), and output (\texttt{z, y}) respectively. For more details on the usage of \texttt{state} class see \ref{ss-sec:state}. Now that the dependent and independent variables have been established, we define the ODE-PDE equations using these symbolic variables.

In this example, the output \texttt{z} must be of length 2 because
\[z(t) = \bmat{\int_0^1 \mbf x(t,s) ds\\ u(t)}.\] 
The length is specified using the second argument as shown in the code
\begin{matlab}
>> z = state('out',2); 
\end{matlab}
\noindent where the second argument to \texttt{state()} function always specifies length of the vector.

\subsubsection{Define a PDE object, \texttt{sys()}} Now that all the required symbols are defined, we can initialize a \texttt{sys} class object and add equations to the object.
\begin{matlab}
>> odepde= sys(); 
\begin{verbatim}
    Initialized sys() object of type ``pde''
\end{verbatim}
\end{matlab}
By default, \texttt{sys} objects are initialized as type `pde' on the domain $[0,1]$. The domain can be modified, if needed, by using the following command.
\begin{matlab}
>> odepde.dom = [0,3];
\end{matlab}
The above command changes the domain of all relevant PDE states to $[0,3]$. For the current example, we assume the domain is $[0,1]$ and proceed.

\begin{boxEnv}{Warning}
The domain parameter MUST be set before adding the equations to the system.    
\end{boxEnv}

\subsubsection{Define equations and add to PDE object}
We can individually add equations one at a time using \texttt{addequation()} method are group the equations into a column vector and add all the equations together. For example, we can add the following equation to the \texttt{sys()} 
\[z(t) = \bmat{\int_0^1 \mbf x(t,s) ds\\ u(t)}\] 
using the command
\begin{matlab}
>> odepde = addequation(odepde, z==[int(X,s,[0,1]); u]); \% one~equation~at~a~time
\begin{verbatim}
    2 equations were added to sys() object
\end{verbatim}
\end{matlab}
Alternatively, we can define all the equations together in a column vector and add them to the system in one command. We can add the remaining $5$ equations and boundary conditions listed below
\begin{align*}
\dot{x}(t) &= -5 x(t)+\int_0^1 \partial_s \mbf x(t,s) ds + u(t)\\
\dot{\mbf x}(t,s) &= 9 \mbf x(t,s)+ \partial_s^2 \mbf x(t,s) +sw(t),\qquad \mbf x(t,0) = 0, \quad \partial_s\mbf x(t,1) + x(t) = 2w(t)\\
y(t) &= \mbf x(t,0).
\end{align*}
using the code
\begin{matlab}
>> eqns = [diff(x,t)==-5*x+int(X,s,[0,1])+u; diff(X,t)==9*X+diff(X,s,2)+s*w;\\
\hspace{8mm}  subs(X,s,0)==0; subs(diff(X,s),s,1)==-x+2*w; y== subs(x,s,0)];\\
>> odepde = addequation(odepde,eqns);
\begin{verbatim}
    5 equations were added to sys() object
\end{verbatim}
\end{matlab}

Any expression passed to \texttt{addequation} function in the form \texttt{addequation(odepde,expr)} results in addition of the equation \texttt{`expr=0'} to the \texttt{odepde} object. 
\begin{boxEnv}{\textbf{Note}}
The $=$ symbol is not used while defining equations. Instead $==$ is used because MATLAB uses $=$ as a protected symbol for assignment operation. Thus, any symbolic expression that needs to be added takes the form \texttt{expr==0} or \texttt{exprA==exprB}.
\end{boxEnv}

\subsubsection{Specify control inputs and observed outputs}
By default, all inputs are defined as disturbance inputs and all outputs are defined as regulated outputs. To explicit assign, certain inputs as control inputs, one must use \texttt{setControl()} function. In the above example, we can specify \texttt{u} to be control input by using the following command.
\begin{matlab}
odepde = setControl(odepde,[u]); \% designate ~control ~inputs
\begin{verbatim}
    1 inputs were designated as controlled inputs
\end{verbatim}
\end{matlab}

Likewise, observed outputs can be added to the system as shown below.
\begin{matlab}
odepde = setObserve(odepde,[y]); 
\begin{verbatim}
    1 outputs were designated as observed outputs
\end{verbatim}
\end{matlab}
Functions such as \texttt{removeControl} and \texttt{removeObserve} are included to reset a particular input as disturbance (output as regulated output). While this is the last step in defining a system using the command line parser format, we can directly convert the \texttt{sys} objects to PIE objects using a function as shown the next subsection.

\subsubsection{Getting PIE from \texttt{sys} class objects}
Once the ODE-PDE class object has been defined, we can obtain the PIE class object, namely \texttt{pie\_struct}, using the convert function.
\begin{matlab}
sys\_pie = convert(odepde,'pie'); 
\begin{verbatim}
Conversion to 'pie' was successful
\end{verbatim}
\end{matlab}

After executing the code in the above `Code block', the PIE system parameters are stored under \texttt{params} property which can be accessed to obtain the following output
\begin{matlab}
\begin{verbatim}
    >> sys_pie.params

ans = 
  pie_struct with properties:

     dim: 1;
    vars: [1×2 polynomial];
     dom: [1×2 double];

       T: [2×2 opvar];     Tw: [2×1 opvar];     Tu: [2×1 opvar]; 
       A: [2×2 opvar];     B1: [2×1 opvar];     B2: [2×1 opvar]; 
      C1: [2×2 opvar];    D11: [2×1 opvar];    D12: [2×1 opvar]; 
      C2: [1×2 opvar];    D21: [1×1 opvar];    D22: [1×1 opvar]; 
\end{verbatim}
\end{matlab}

Once the PIE structure is obtained, we can proceed to perform analysis, control, and simulation, as discussed in detail in Chapters~\ref{ch:PIESIM} and~\ref{ch:LPIs}.

\subsection{More examples of command line parser format}
More details on the implementation of Command Line Parser format can be found in Section~\ref{subsec:additional-command-line}. In this subsection, we provide a few more examples to demonstrate the typical use of the command line parser. More specifically, we focus on examples involving inputs, outputs, delays, vector-valued PDEs, etc., to demonstrate the capabilities of command line parser.
\subsubsection{Example: Transport equation}
Consider the Transport equation which is modeled as a PDE with 1$^{st}$ derivatives in time and space given by 
\begin{align*}
    \dot{\mbf x}(t,s) &= 5\partial_s \mbf x(t,s)+u(t),\quad s\in[0,2]\\
    y(t) &= x(t,2),\\
    \mbf x(t,0) &= 0.
\end{align*}
Here, we use a control input in the domain and an observer at the right boundary with an intention to design an observer based controller. This system can be defined using the command line parser format as shown below.

\begin{codebox}
The above set of equations can be defined and converted to a PIE using the code shown below.
\begin{matlab}
>> pvar ~t ~s ~theta; \\
>> X = state('pde');\\ 
>> u = state('in'); y = state('out');\\
>> odepde= sys(); \\
>> odepde.dom = [0,2];\\
>> eqns = [diff(X,t)==5*diff(X,s)+u;\\
\hspace{5cm}  subs(X,s,0)==0; y==subs(x,s,2)];\\
>> odepde = addequation(odepde,eqns);\\
>> odepde = setControl(odepde,[u]);\\
>> odepde = setObserve(odepde,[y]); \\
>> sys\_pie = convert(odepde,'pie');
\end{matlab}
\end{codebox}
The steps are identical to the one specified in the previous section. The only exception is the change of domain to $[0,2]$. First, we define all the independent variables as pvar objects and dependent variables as state objects. Then we define the equations and add the equations to a system storage object. Finally, we specify which symbols are control inputs and observed outputs and convert the system to a PIE system. By accessing \texttt{params} property, we obtain the following output
\begin{matlab}
\begin{verbatim}
    >> sys_pie.params

ans = 
  pie_struct with properties:

     dim: 1;
    vars: [1×2 polynomial];
     dom: [1×2 double];

       T: [1×1 opvar];     Tw: [1×0 opvar];     Tu: [1×1 opvar]; 
       A: [1×1 opvar];     B1: [1×0 opvar];     B2: [1×1 opvar]; 
      C1: [0×1 opvar];    D11: [0×0 opvar];    D12: [0×1 opvar]; 
      C2: [1×1 opvar];    D21: [0×0 opvar];    D22: [0×1 opvar]; 
\end{verbatim}
\end{matlab}

Once the PIE system is obtained, we can use the LPIs in Chapter \ref{ch:LPIs} to design an observer and controller for this PDE.

\subsubsection{Example: PDE with delay terms}
Consider a reaction-diffusion equation which are modeled as a PDE with $2^{nd}$-order derivatives in both space and first order derivative in time. Let there be an dynamic ODE system coupled with the PDE through a channel that is delayed by an amount $\tau=2$. The coupled ODE-PDE model for this system is given by the equations
\begin{align*}
    \dot{x}(t) &= -5x(t),\\
    \dot{\mbf x}(t,s) &= 10\mbf x(t,s)+\partial_s^2 \mbf x(t,s)+x(t-2),\\
    \mbf x(t,0) &= 0 = \mbf x(t,1).
\end{align*}
\begin{codebox}
The above set of equations can be defined and converted to a PIE using the code shown below.
\begin{matlab}
>> x = state('ode');\\
>> X = state('pde');\\
>> pvar ~s ~t ~theta;\\
>> ss = sys();\\
>> eqns = [diff(x,t)==-5*x; diff(X,t)==diff(X,s,2)+subs(x,t,t-2);\\ \hspace{5cm} subs(X,s,0)==0;subs(X,s,1)==0;];\\
>> ss = addequation(ss,eqns);\\
>> ss = convert(ss,'pie');
\end{matlab}
\end{codebox}

In the above system, the delay term is converted to a new state, $\mbf v$, that is governed by a transport equation. The new set of equations is then given by
\begin{align*}
    \dot{x}(t) &= -5x(t),\\
    \dot{\mbf x}(t,s) &= 10\mbf x(t,s)+\partial_s^2 \mbf x(t,s)+\mbf v(t,-2),\\
    \dot{\mbf v}(t,\theta) &= \partial_{\theta}\mbf v(t,\theta),\\
    \mbf x(t,0) &= 0 = \mbf x(t,1)\quad \mbf v(t,0) = x(t).
\end{align*}
\textbf{Note:} The above conversion is performed internally, and the users only need to use the input format shown in the code block above.

Thus, the resulting PIE will have two distributed states and one finite dimensional state which can be verified by looking at the dimensions of the parameters stored in \texttt{ss.params} property:
\begin{matlab}
\begin{verbatim}
    >> ss.params

ans = 
  pie_struct with properties:

     dim: 2;
    vars: [2×2 polynomial];
     dom: [2×2 double];

       T: [3×3 opvar2d];     Tw: [3×0 opvar2d];     Tu: [3×0 opvar2d]; 
       A: [3×3 opvar2d];     B1: [3×0 opvar2d];     B2: [3×0 opvar2d]; 
      C1: [0×3 opvar2d];    D11: [0×0 opvar2d];    D12: [0×0 opvar2d]; 
      C2: [0×3 opvar2d];    D21: [0×0 opvar2d];    D22: [0×0 opvar2d]; 
\end{verbatim}
\end{matlab}

\subsubsection{Example: Beam equation}\label{ex:parser_timoshenko}
Here we consider the Timoshenko Beam equations which are modeled as a PDE with $2^{nd}$-order derivatives in both space and time. While this system cannot be directly input using the command line parser format, we can redefine the state variables to convert it to a PDE with first order temporal derivative as shown below.
\begin{align*}
&\ddot{w} = \partial_s (w_s-\phi), \quad \ddot{\phi} = \phi_{ss}+(w_s - \phi)\\
&\phi(0)=w(0)=0, \quad \phi_s(1)=0,\quad w_s(1)-\phi(1)=0.
\end{align*}
By choosing $\mbf x= [\dot{w}, w_s-\phi, \dot{\phi},\phi_s]$, we get
\begin{align*}
    &\dot{\mbf x}(t,s) = \bmat{0&0&0&0\\0&0&-1&0\\0&1&0&0\\0&0&0&0}\mbf x(t,s) + \bmat{0&1&0&0\\1&0&0&0\\0&0&0&1\\0&0&1&0}\partial_s \mbf x(t,s),\\
    &\bmat{1&0&0&0&0&0&0&0\\0&0&1&0&0&0&0&0\\0&0&0&0&0&0&0&1\\0&0&0&0&0&1&0&0}\bmat{\mbf x(t,0)\\\mbf x(t,1)} = 0,
\end{align*}
which is a vector-valued transport equation with a reaction term. Next, we define this system using the command line parser as shown below.

\begin{codebox}
The above set of equations can be defined and converted to a PIE using the code shown below.
\begin{matlab}
>> x = state('pde',4);\\
>> pvar ~s ~t ~theta;\\
>> ss = sys();\\
>> A0 = [0,0,0,0;0,0,-1,0;0,1,0,0;0,0,0,0];\\ 
>> A1 = [0,1,0,0;1,0,0,0;0,0,0,1;0,0,1,0];\\ 
>> B = [1,0,0,0,0,0,0,0;0,0,1,0,0,0,0,0;0,0,0,0,0,0,0,1;0,0,0,0,0,1,0,0];\\
>> eqns = [diff(x,t)==A0*x+A1*diff(x,s); B*[subs(x,s,0); subs(x,s,1)]==0];\\
>> ss = addequation(ss,eqns);\\
>> ss = convert(ss,'pie');
\end{matlab}
As seen above, the presence of vector-valued states does not change the typical workflow to define the PDE. As long as the dimensions of the parameters and vectors used in the equations match, the process and steps remain the same.
\end{codebox}

The above code should generate a PIE system with \texttt{params} as shown below:
\begin{matlab}
\begin{verbatim}
    >> ss.params

ans = 
  pie_struct with properties:

     dim: 1;
    vars: [1×2 polynomial];
     dom: [1×2 double];

       T: [4×4 opvar];     Tw: [4×0 opvar];     Tu: [4×0 opvar]; 
       A: [4×4 opvar];     B1: [4×0 opvar];     B2: [4×0 opvar]; 
      C1: [0×4 opvar];    D11: [0×0 opvar];    D12: [0×0 opvar]; 
      C2: [0×4 opvar];    D21: [0×0 opvar];    D22: [0×0 opvar]; 
\end{verbatim}
\end{matlab}

\section{Alternative Input Formats for PDEs}\label{sec:PDE_DDE_representation:alt_PDEs}

In addition to the command line parser input format, PIETOOLS 2022 offers two additional methods for declaring PDEs. In particular, PIETOOLS comes with a graphical user interface (GUI) that allows users to simultaneously visualize the PDE that they are specifying, as well as ``terms-based'' input format, which is the only input format to declare PDEs in multiple spatial variables in PIETOOLS 2022. We briefly introduce both of these input formats here, focusing on the GUI in Subsection~\ref{subsec:PDE_DDE_representation:alt_PDEs:GUI}, and the terms-based input format in Subsection~\ref{subsec:PDE_DDE_representation:alt_PDEs:2D}. Note that we provide only a brief introduction of each format here, refering to Chapter~\ref{ch:alt_PDE_input} for more details.

\subsection{A GUI for Declaring PDEs}\label{subsec:PDE_DDE_representation:alt_PDEs:GUI}

Aside from the command line parser, the GUI is the easiest way to declare linear 1D ODE-PDE systems in PIETOOLS, providing a simple, intuitive and interactive visual interface to directly input the model. The GUI can be opened by running \texttt{PIETOOLS\_PIETOOLS\_GUI} from the command line, opening a window like the one displayed in the picture below:
\begin{figure}[H]
	\centering
	\includegraphics[width=0.95\textwidth]{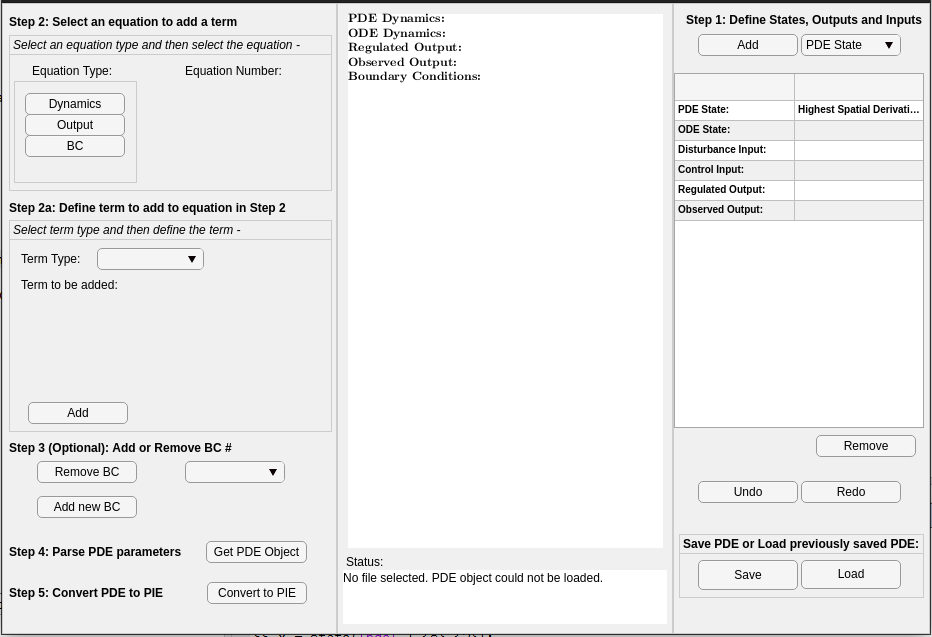}
	\caption{Example of empty GUI window.}
	\label{gui_empty}
\end{figure}
Then, the desired PDE can be declared following steps 1 through 4, first specifying the state variables, inputs and outputs, then declaring the different equations term by term, and finally adding any boundary conditions. It also allows PDE models to be saved and loaded, so that e.g. the system
\begin{align*}
    \dot{\mbf{x}}(t,s)&=\partial_{s}^2 \mbf{x}(t,s) + sw(t), &  s&\in[0,1]\\
    z(t)&=\int_{0}^{1}\mbf{x}(t,s)ds,   \\
    \mbf{x}(t,0)&=0,\qquad \mbf{x}(t,1)=0,
\end{align*}
can be retrieved by simply loading the file\\
\texttt{PIETOOLS\_PDE\_Ex\_Heat\_Eq\_with\_Distributed\_Disturbance\_GUI}\\
from the library of PDE examples, returning a window that looks like
\begin{figure}[H]
	\centering
	\includegraphics[width=0.95\textwidth]{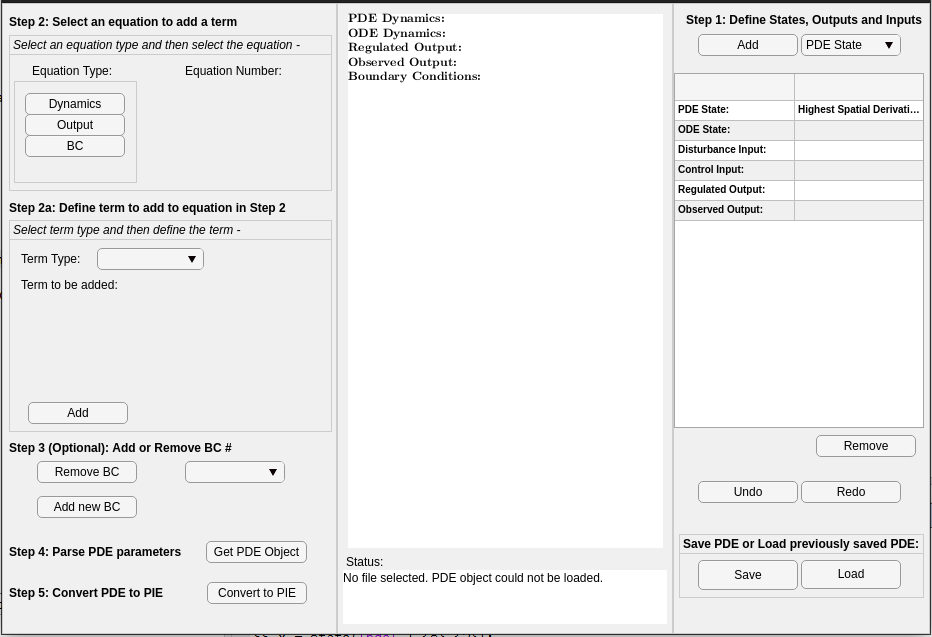}
	\caption{GUI window after loading the file\\ \texttt{PIETOOLS\_PDE\_Ex\_Heat\_Eq\_with\_Distributed\_Disturbance\_GUI} from the library of PDE examples.}
	\label{gui_loaded}
\end{figure}
Then, the system can be parsed by clicking \texttt{Get PDE Objects}, returning a structure \texttt{PDE\_GUI} in the MATLAB workspace that can be used for further analysis.

For more details on how to use the GUI, we refer to Section~\ref{sec:GUI}.

\subsection{An Input Format for 2D PDEs}\label{subsec:PDE_DDE_representation:alt_PDEs:2D}

In PIETOOLS 2022, the terms-based input format is the only way to declare 2D PDEs. In this format, a PDE is represented as a \texttt{pde\_struct} object. Each term in each equation in the PDE is then implemented separately.

For example, to declare a 2D heat equation,
\begin{align*}
    \dot{\mbf{x}}(t,s_1,s_2)&=\partial_{s_1}^2\mbf{x}(t,s_1,s_2) + \partial_{s_2}^{2}\mbf{x}(t,s_1,s_2) + w(t), &   (s_1,s_2)\in[-1,1]\times[0,1],  \\
    z(t)&=\int_{-1}^{1}\int_{0}^{1}\mbf{x}(t,s_1,s_2)ds_2 ds_1, \\
    \mbf{x}(t,-1,s_2)&=0,\qquad \mbf{x}(t,1,s_2)=0,  \\
    \mbf{x}(t,s_1,0)&=0,\qquad \mbf{x}(t,s_1,1)=0,
\end{align*}
we first declare spatial variables $(s_1,s_2)$ as \texttt{s1} and \texttt{s2}, and initialize an empty \texttt{pde\_struct} object \texttt{PDE} as
\begin{matlab}
\begin{verbatim}
 >> pvar s1 s2
 >> PDE = pde_struct();
\end{verbatim}
\end{matlab}
Then, we declare the state $x(t,s_1,s_2)$, input $w(t)$ and output $z(t)$ as
\begin{matlab}
\begin{verbatim}
 >> PDE.x{1}.vars = [s1;s2];
 >> PDE.x{1}.dom = [-1,1; 0,1];
 >> PDE.w{1}.vars = [];
 >> PDE.z{1}.vars = [];
\end{verbatim}
\end{matlab}
using the field \texttt{vars} to specify the spatial variables on which each component depends, and the field \texttt{dom} to specify the spatial domain on which these variables exist. Then, the PDE can be implemented one term at a time as
\begin{matlab}
\begin{verbatim}
 >> PDE.x{1}.term{1}.x = 1;       PDE.x{1}.term{2}.x = 1;       PDE.x{1}.term{3}.w = 1;
 >> PDE.x{1}.term{1}.D = [2,0];   PDE.x{1}.term{2}.D = [0,2];    
\end{verbatim}
\end{matlab}
using the fields \texttt{x} and \texttt{w} to indicate whether each term involves a state component or input, and the field \texttt{D} to specify the order of derivative of the state component in the term. The output equation can be similarly specified as
\begin{matlab}
\begin{verbatim}
 >> PDE.z{1}.term{1}.x = 1;
 >> PDE.z{1}.term{1}.I{1} = [-1,1];    PDE.z{1}.term{1}.I{2} = [0,1];
\end{verbatim}
\end{matlab}
using the field \texttt{I} to specify the desired domain of integration of the state component along each spatial direction. Finally, the boundary conditions can be declared as
\begin{matlab}
\begin{verbatim}
 >> PDE.BC{1}.term{1}.x = 1;            PDE.BC{2}.term{1}.x = 1;
 >> PDE.BC{1}.term{1}.loc = [-1,s2];    PDE.BC{2}.term{1}.loc = [1,s2];
 >> PDE.BC{3}.term{1}.x = 1;            PDE.BC{4}.term{1}.x = 1;
 >> PDE.BC{3}.term{1}.loc = [s1,0];     PDE.BC{4}.term{1}.loc = [s1,1];
\end{verbatim}
\end{matlab}
Then, the PDE can be initialized by calling \texttt{PDE = initialize(PDE)}, returning a structure that can be used for analysis and simulation. 

For more details on how to use the terms-based input format to declare (2D) PDEs, we refer to Section~\ref{sec:alt_PDE_input:terms_input_PDE}.

\section{Batch Input Format for DDEs}\label{sec:PDE_DDE_representation:DDEs}
 
The DDE data structure allows the user to declare any of the matrices in the following general form of Delay-Differential equation.
\begin{align}
	&\bmat{\dot{x}(t)\\z(t) \\ y(t)}=\bmat{A_0 & B_{1} & B_{2}\\ C_{1} & D_{11} &D_{12}\\ C_{2} & D_{21} &D_{22}}\bmat{x(t)\\w(t)\\u(t)}+\sum_{i=1}^K \bmat{A_i & B_{1i} & B_{2i}\\C_{1i} & D_{11i} & D_{12i}\\C_{2i} & D_{21i} & D_{22i}} \bmat{x(t-\tau_i)\\w(t-\tau_i)\\u(t-\tau_i)}\notag \\
& \hspace{2cm}+\sum_{i=1}^K \int_{-\tau_i}^0\bmat{A_{di}(s) & B_{1di}(s) &B_{2di}(s)\\C_{1di}(s) & D_{11di}(s) & D_{12di}(s)\\C_{2di}(s) & D_{21di}(s) & D_{22di}(s)} \bmat{x(t+s)\\w(t+s)\\u(t+s)}ds \label{eqn:DDE_ch4}
\end{align}
In this representation, it is understood that
\begin{itemize}
\item The present state is $x(t)$.\vspace{-2mm}
\item The disturbance or exogenous input is $w(t)$. These signals are not typically known or alterable. They can account for things like unmodelled dynamics, changes in reference, forcing functions, noise, or perturbations.\vspace{-2mm}
\item The controlled input is $u(t)$. This is typically the signal which is influenced by an actuator and hence can be accessed for feedback control. \vspace{-2mm}
\item The regulated output is $z(t)$. This signal typically includes the parts of the system to be minimized, including actuator effort and states. These signals need not be measured using senors.\vspace{-2mm}
\item The observed or sensed output is $y(t)$. These are the signals which can be measured using sensors and fed back to an estimator or controller.\vspace{-2mm}
\end{itemize}
To add any term to the DDE structure, simply declare its value. For example, to represent 
\[
\dot x(t)=-x(t-1),\qquad z(t)=x(t-2)
\]
we use
	\begin{matlab}
		>> DDE.tau = [1 2];\\
		>> DDE.Ai\{1\} = -1;\\
		>> DDE.C1i\{2\} = 1;
	\end{matlab}
All terms not declared are assumed to be zero. The exception is that we require the user to specify the values of the delay in \texttt{DDE.tau}. When you are done adding terms to the DDE structure, use the function \texttt{DDE=PIETOOLS\_initialize\_DDE(DDE)}, which will check for undeclared terms and set them all to zero. It also checks to make sure there are no incompatible dimensions in the matrices you declared and will return a warning if it detects such malfeasance. The complete list of terms and DDE structural elements is listed in Table~\ref{tab:DDE_parameters_ch4}.

\begin{table}[ht!]\vspace{-2mm}
\begin{center}{
\begin{tabular}{c|c||c|c||c|c}
  \multicolumn{6}{c}{\textbf{ODE Terms:}}\\
Eqn.~\eqref{eqn:DDE_ch4}  & \texttt{DDE.}  &   Eqn.~\eqref{eqn:DDE_ch4}  & \texttt{DDE.} &   Eqn.~\eqref{eqn:DDE_ch4}  & \texttt{DDE.}\\
\hline
$A_0$    & \texttt{A0} & $B_{1}$ & \texttt{B1} &$B_{2}$&\texttt{B2}\\ 
$C_{1}$ & \texttt{C1} & $D_{11}$ &\texttt{D11}&$D_{12}$&\texttt{D12}\\ 
$C_{2}$ & \texttt{C2} & $D_{21}$ &\texttt{D21}&$D_{22}$&\texttt{D22} \\ \hline \\
 \multicolumn{6}{c}{ \textbf{Discrete Delay Terms:}}\\
Eqn.~\eqref{eqn:DDE_ch4}  & \texttt{DDE.}  &   Eqn.~\eqref{eqn:DDE_ch4}  & \texttt{DDE.} &   Eqn.~\eqref{eqn:DDE_ch4}  & \texttt{DDE.}\\
\hline
$A_i$    & \texttt{Ai\{i\}} & $B_{1i}$ & \texttt{B1i\{i\}} &$B_{2i}$&\texttt{B2i\{i\}}\\ 
$C_{1i}$ & \texttt{C1i\{i\}} & $D_{11i}$ &\texttt{D11i\{i\}}&$D_{12i}$&\texttt{D12i\{i\}}\\ 
$C_{2i}$ & \texttt{C2i\{i\}} & $D_{21i}$ &\texttt{D21i\{i\}}&$D_{22i}$&\texttt{D22i\{i\}}\\
\hline \\  \multicolumn{6}{c}{\textbf{Distributed Delay Terms: May be functions of \texttt{pvar s}}}\\
Eqn.~\eqref{eqn:DDE_ch4}  & \texttt{DDE.}  &   Eqn.~\eqref{eqn:DDE_ch4}  & \texttt{DDE.} &   Eqn.~\eqref{eqn:DDE_ch4}  & \texttt{DDE.}\\
\hline
$A_{di} $   & \texttt{Adi\{i\}} & $B_{1di}$ & \texttt{B1di\{i\}} &$B_{2di}$&\texttt{B2di\{i\}}\\ 
$C_{1di} $& \texttt{C1di\{i\}} &$ D_{11di} $&\texttt{D11di\{i\}}&$D_{12di}$&\texttt{D12di\{i\}}\\ 
$C_{2di}$ & \texttt{C2di\{i\}} & $D_{21di} $&\texttt{D21di\{i\}}&$D_{22di}$&\texttt{D22di\{i\}}\\
\end{tabular}
}
\end{center}\vspace{-2mm}
\caption{ Equivalent names of Matlab elements of the \texttt{DDE} structure terms for terms in Eqn.~\eqref{eqn:DDE_ch4}. For example, to set term \texttt{XX} to \texttt{YY}, we use \texttt{DDE.XX=YY}.  In addition, the delay $\tau_i$ is specified using the vector element \texttt{DDE.tau(i)} so that if $\tau_1=1, \tau_2=2, \tau_3=3$, then \texttt{DDE.tau=[1 2 3]}. }\label{tab:DDE_parameters_ch4}\end{table}

\subsection{Initializing a DDE Data structure} The user need only add non-zero terms to the DDE structure. All terms which are not added to the data structure are assumed to be zero. Before conversion to another representation or data structure, the data structure will be initialized using the command
	\begin{flalign*}
		&\texttt{DDE = initialize\_PIETOOLS\_DDE(DDE)}&
	\end{flalign*}
This will check for dimension errors in the formulation and set all non-zero parts of the \texttt{DDE} data structure to zero. Not that, to make the code robust, all PIETOOLS conversion utilities perform this step internally.

\section{Alternative Input Formats for TDSs}\label{sec:PDE_DDE_representation:alt_TDSs}

Although the delay differential equation (DDE) format is perhaps the most intuitive format for representing time-delay systems (TDS), it is not the only representation of TDS systems, and not every TDS can be represented in this format. For this reason, PIETOOLS includes two additional input format for TDSs, namely the Neutral Type System (NDS) representation, and Differential-Difference Equation (DDF) representation. Here, the structure of a NDS is identical to that of a DDE except for 6 additional terms: 
\begin{align*}
	\bmat{\dot{x}(t)\\z(t) \\ y(t)}&=\bmat{A_0 & B_{1} & B_{2}\\ C_{1} & D_{11} &D_{12}\\ C_{2} & D_{21} &D_{22}}\bmat{x(t)\\w(t)\\u(t)}+\sum_{i=1}^K \bmat{A_i  & B_{1i} & B_{2i}& E_i\\C_{1i}& D_{11i} & D_{12i} & E_{1i}\\C_{2i} & D_{21i} & D_{22i}&E_{2i}} \bmat{x(t-\tau_i)\\w(t-\tau_i)\\u(t-\tau_i)\\ \dot x(t-\tau_i)}\notag\\[-3mm]
& +\hspace{-1mm}\sum_{i=1}^K \hspace{0mm}\int_{-\tau_i}^0\hspace{-1mm}\bmat{A_{di}(s) & \hspace{-1mm}B_{1di}(s) &\hspace{-1mm}B_{2di}(s)& \hspace{-1mm}E_{di}(s)\\C_{1di}(s) & \hspace{-1mm}D_{11di}(s) & \hspace{-1mm}D_{12di}(s)& \hspace{-1mm}E_{1di}(s)\\C_{2di}(s) &\hspace{-1mm}D_{21di}(s) & \hspace{-1mm}D_{22di}(s)& \hspace{-1mm}E_{2di}(s)} \hspace{-2mm}\bmat{x(t+s)\\w(t+s)\\u(t+s)\\ \dot x(t+s)}\hspace{-1mm}ds. 
\end{align*}
These new terms are parameterized by $E_i,E_{1i}$, and $E_{2i}$ for the discrete delays and by $E_{di},E_{1di}$, and $E_{2di}$ for the distributed delays, and should be included in a NDS object as, e.g. \texttt{NDS.E\{1\}=1}. On the other hand, the DDF representation is more compact but less transparent than the DDE and NDS representation, taking the form
\begin{align*}
	\bmat{\dot{x}(t)\\ z(t)\\y(t)\\r_i(t)}&=\bmat{A_0 & B_1& B_2\\C_1 &D_{11}&D_{12}\\C_2&D_{21}&D_{22}\\C_{ri}&B_{r1i}&B_{r2i}}\bmat{x(t)\\w(t)\\u(t)}+\bmat{B_v\\D_{1v}\\D_{2v}\\D_{rvi}} v(t) \notag\\
v(t)&=\sum_{i=1}^K C_{vi} r_i(t-\tau_i)+\sum_{i=1}^K \int_{-\tau_i}^0C_{vdi}(s) r_i(t+s)ds.
\end{align*}
In this representation, the output signal from the ODE part is decomposed into sub-components $r_i$, each of which is delayed by amount $\tau_i$. Identifying these sub-components is often challenging, so in most cases it will be preferable to use the NDS or DDE representation instead. However, the DDF representation is more general than either the DDE or NDS representation, so PIETOOLS also includes an input format for declaring DDF systems. For more information on how to declare systems in the DDF or NDS representation, and how to convert between different representations, we refer to Chapter~\ref{ch:alt_DDE_input}.


\chapter{Converting PDEs and DDEs to PIEs}\label{ch:PIE}

In the previous chapter, we showed how general linear 1D ODE-PDE and DDE systems can be declared in PIETOOLS. In order to analyze such systems, PIETOOLS represents each of them in a standardized format, as a partial integral equation (PIE). This format is parameterized by partial integral, or PI operators, rather than by differential operators, allowing PIEs to be analysed by solving optimization problems on these PI operators (see Chapter~\ref{ch:LPIs}).

In this chapter, we show how an equivalent PIE representation of PDE and DDE systems can be computed in PIETOOLS. In particular, in Section~\ref{sec:PIE:PDE2PIE:what_is_a_pie}, we first provide a simple illustration of what a PIE is. In Sections~\ref{sec:PIE:PDE2PIE:pde_2_pie} and~\ref{sec:PIE:PDE2PIE:dde_2_pie}, we then show how a PDE and a DDE can be converted to a PIE, and in Section~\ref{sec:PIE:PDE2PIE:io_pde_2_pie}, we show how a PDE with inputs and outputs can be converted to a PIE. To reduce notation, we demonstrate the PDE conversion only for 1D systems, though we note that the same steps also work for 2D PDEs.

\section{What is a PIE?}\label{sec:PIE:PDE2PIE:what_is_a_pie}

To illustrate the concept of partial integral equations, suppose we have a simple 1D PDE
\begin{align}\label{eq:1D_PDE_example}
 \dot{\mbf{x}}(t,s)&=2\partial_{s}\mbf{x}(t,s) + 10\mbf{x}(t,s),    &   s&\in[0,1],  \\
 \mbf{x}(t,0)&=0.   \nonumber
\end{align}
In this system, the PDE state $\mbf{x}(t)$ at any time $t\geq 0$ is a function of $s$, that has to satisfy the boundary condition (BC) $\mbf{x}(t,0)=0$. Moreover, the state must be at least first order differentiable with respect to $s$, for us to be able to evaluate the derivative $\partial_{s}\mbf{x}(t,s)$. As such, a more fundamental state would actually be this first order derivative $\partial_{s}\mbf{x}(t)$ of the state, which does not need to be differentiable, nor does it need to satisfy any boundary conditions. We therefore define $\mbf{x}_{\text{f}}(t,s):=\partial_{s}\mbf{x}(t,s)$ as the \textit{fundamental state} associated to this PDE. Using the fundamental theorem of calculus, we can then express the PDE state in terms of the fundamental state as
\begin{align*}
 \mbf{x}(t,s)&=\mbf{x}(t,0)+\int_{0}^{s}\partial_{s}\mbf{x}(t,\theta)d\theta=\mbf{x}(t,0)+\int_{0}^{s}\mbf{x}_{\text{f}}(t,\theta)d\theta = \int_{0}^{s}\mbf{x}_{\text{f}}(t,\theta)d\theta,
\end{align*}
where we invoke the boundary condition $\mbf{x}(t,0)=0$. Substituting this result into the PDE, we arrive at an equivalent representation of the system as
\begin{align}\label{eq:1D_PIE_example}
    \int_{0}^{s}\dot{\mbf{x}}_{\text{f}}(t,\theta)d\theta &= 2\mbf{x}_{\text{f}}(t,s) + \int_{0}^{s}10\mbf{x}_{\text{f}}(t,\theta)d\theta,   &   s&\in[0,1],
\end{align}
in which the fundamental state $\mbf{x}_{\text{f}}$ does not need to satisfy any boundary conditions, nor does it need to be differentiable with respect to $s$. We refer to this representation as the Partial Integral Equation, or PIE representation of the system, involving only partial integrals, rather than partial derivatives with respect to $s$. It can be shown that for any well-posed linear PDE -- meaning that the solution to the PDE is uniquely defined by the dynamics and the BCs -- there exists an equivalent PIE representation. In PIETOOLS, this equivalent representation can be obtained by simply calling \texttt{convert} for the desired PDE structure \texttt{PDE}, returning a structure \texttt{PIE} that corresponds to the equivalent PIE representation.

\section{Converting a PDE to a PIE}\label{sec:PIE:PDE2PIE:pde_2_pie}

Suppose that we have a PDE structure \texttt{PDE}, defining a 1D heat equation with integral boundary conditions:
\begin{align}\label{eq:PDE_ex_PDE2PIE_1}
     \dot{\mbf{x}}(t,s)&=\partial_{s}^{2}\mbf{x}(t,s),  &   s&\in[0,1]  \nonumber\\
     \text{with BCs}\hspace*{1.0cm} \mbf{x}(t,0)&+\int_{0}^{1}\mbf{x}(t,s)ds=0,    \hspace*{1.0cm}
     \mbf{x}(t,1)+\int_{0}^{1}\mbf{x}(t,s)ds=0
\end{align}
In this system, the state $\mbf{x}(t,s)$ at each time $t\geq 0$ must be at least second order differentiable with respect to $s$, so we define the associated fundamental state as $\mbf{x}_{\text{f}}(t,s)=\partial_{s}^2\mbf{x}(t,s)$.
We implement this system in PIETOOLS using the command line parser as follows:
\begin{matlab}
\begin{verbatim}
 >> pvar s t
 >> x = state(`pde');
 >> PDE_dyn = diff(x,t) == diff(x,s,2);
 >> PDE_BCs = [subs(x,s,0) + int(x,s,[0,1]) == 0;
               subs(x,s,1) + int(x,s,[0,1]) == 0];
 >> PDE = sys();
 >> PDE = addequation(PDE,[PDE_dyn; PDE_BCs]);
\end{verbatim}
\end{matlab}
Then, we can derive the associated PIE representation by simply calling
\begin{matlab}
\begin{verbatim}
 >> PIE = convert(PDE,`pie')
 PIE = 
   pie_struct with properties:

     dim: 1;
    vars: [1×2 polynomial];
     dom: [0 1];

       T: [1×1 opvar];     Tw: [1×0 opvar];     Tu: [1×0 opvar]; 
       A: [1×1 opvar];     B1: [1×0 opvar];     B2: [1×0 opvar]; 
      C1: [0×1 opvar];    D11: [0×0 opvar];    D12: [0×0 opvar]; 
      C2: [0×1 opvar];    D21: [0×0 opvar];    D22: [0×0 opvar]; 
\end{verbatim}
\end{matlab}
In this structure, the field \texttt{dim} corresponds to the spatial dimensionality of the system, with \texttt{dim=1} indicating that this is a 1D PIE. The fields \texttt{vars} and \texttt{dom} define the spatial variables in the PIE and their domain, with
\begin{matlab}
\begin{verbatim}
 >> PIE.vars
 ans = 
   [ s, theta]

 >> PIE.dom
 ans = 
      0     1
\end{verbatim}
\end{matlab}
indicating that \texttt{s} is the primary variable, \texttt{theta} the dummy variable, and both exist on the domain $s,\theta \in[0,1]$.
We note that the remaining fields in the \texttt{PIE} structure are all \texttt{opvar} objects, representing PI operators in 1D. Moreover, most of these operators are empty, being of dimension $1\times 0$, $0\times 1$ or $0\times 0$. This is because the PDE~\eqref{eq:PDE_ex_PDE2PIE_1} does not involve any inputs or outputs, and therefore its associated PIE has the simple structure
\begin{align*}
    \bl(\mcl{T}\dot{\mbf{x}}_{\text{f}}\br)(t,s)&=\bl(\mcl{A}\mbf{x}_{\text{f}}\br)(t,s),
\end{align*}
where the operator $\mcl{T}$ maps the fundamental state $\mbf{x}_{\text{f}}$ back to the PDE state $\mbf{x}$ as
\begin{align*}
    \bl(\mcl{T}\mbf{x}_{\text{f}}\br)(t,s)&=\mbf{x}(t,s).
\end{align*}
For the PDE~\eqref{eq:PDE_ex_PDE2PIE_1}, we know that $\mbf{x}_{\text{f}}(t,s)=\partial_{s}\mbf{x}(t,s)$. The associated operators $\mcl{T}$ and $\mcl{A}$ are represented by the \texttt{opvar} objects \texttt{T} and \texttt{A} in the \texttt{PIE} structure, for which we find that
\begin{matlab}
\begin{verbatim}
 >> T = PIE.T
 T =
     [] | [] 
     ---------
     [] | T.R 
     
 T.R = 
         [0] | [s^2 - 0.25*s - theta] | [0.75*(s^2 - s)]        
 >> A = PIE.A
 A =
     [] | [] 
     ---------
     [] | A.R 
 
 A.R = 
          [1] | [10] | [0]
\end{verbatim}
\end{matlab}
We conclude that the PDE~\eqref{eq:PDE_ex_PDE2PIE_1} is equivalently represented by the PIE
\begin{align*}
    \int_{0}^{s}\underbrace{\bbl(s^2-\frac{s}{4}-\theta\bbr)}_{\texttt{PIE.T.R.R1}}\dot{\mbf{x}}_{\textnormal{f}}(t,\theta)d\theta + \int_{s}^{1}\underbrace{\frac{3}{4}\bbl(s^2-s\bbr)}_{\texttt{PIE.T.R.R2}}\dot{\mbf{x}}_{\textnormal{f}}(t,\theta)d\theta = \underbrace{1}_{\texttt{PIE.A.R.R0}}\mbf{x}_{\textnormal{f}}(t,s).
\end{align*}

\section{Converting a DDE to a PIE}\label{sec:PIE:PDE2PIE:dde_2_pie}

Just like PDEs, DDEs (and other delay-differential equations) can also be equivalently represented as PIEs. For example, consider the following DDE
\begin{align*}
    \dot{x}(t)=\bmat{-1.5&0\\0.5&-1}x(t) + \int_{-1}^{0} \bmat{3 & 2.25\\0 &0.5}x(t+s)ds + \int_{-2}^{0}\bmat{-1&0\\0&-1}x(t+s)ds,
\end{align*}
where $x(t)\in\R^2$ for $t\geq 0$. We declare this system as a structure \texttt{DDE} in PIETOOLS as
\begin{matlab}
\begin{verbatim}
 >> DDE.A0 = [-1.5, 0; 0.5, -1];
 >> DDE.Adi{1} = [3, 2.25; 0, 0.5];     DDE.tau(1) = 1;
 >> DDE.Adi{2} = [-1, 0; 0, -1];        DDE.tau(2) = 2;
\end{verbatim}
\end{matlab}
We can then convert the DDE to a PIE by calling
\begin{matlab}
\begin{verbatim}
 >> PIE = convert_PIETOOLS_DDE(DDE,`pie')
 PIE = 
   pie_struct with properties:

     dim: 1;
    vars: [1×2 polynomial];
     dom: [1x2 double];
       T: [6×6 opvar];     Tw: [6×0 opvar];     Tu: [6×0 opvar]; 
       A: [6×6 opvar];     B1: [6×0 opvar];     B2: [6×0 opvar]; 
      C1: [0×6 opvar];    D11: [0×0 opvar];    D12: [0×0 opvar]; 
      C2: [0×6 opvar];    D21: [0×0 opvar];    D22: [0×0 opvar]; 
\end{verbatim}
\end{matlab}
In this structure, we note that \texttt{dim=1}, indicating that the PIE is 1D, even though the state $x(t)\in\R^2$ in the DDE is finite-dimensional. This is because, in order to incorporate the delayed signals, the state is augmented to $\mbf{x}(t)=\sbmat{x(t)\\\mbf{x}_1(t)\\\mbf{x}_2(t)}\in\sbmat{\R^2\\ L_2^2[-1,0]\\ L_2^2[-1,0]}$, where 
\begin{align*}
    \mbf{x}_1(t,s)&=x(t+\tau_{1}s)=x(t+s),  &   &\text{and,}    &
    \mbf{x}_{2}&=x(t+\tau_2 s)=x(t-2s)
\end{align*}
for $s\in[-1,0]$. Here, the artificial states $\mbf{x}_1(t)$ and $\mbf{x}_2(t)$ will have to satisfy
\begin{align*}
    \dot{\mbf{x}}_1(t,s)&=\dot{x}(t+s)=\partial_{r}x(r)=\partial_{s}x(t+s)=\partial_{s}\mbf{x}_1(t,s), \\
    \dot{\mbf{x}}_2(t,s)&=\dot{x}(t+2s)=\partial_{r}x(r)=\frac{1}{2}\partial_{s}x(t+2s)=\frac{1}{2}\partial_{s}\mbf{x}_2(t,s) &   s&\in[-1,0]
\end{align*}
and we can equivalently represent the DDE as a PDE
\begin{align*}
    \dot{x}(t)&=\bmat{-1.5&0\\0.5&-1}x(t) + \int_{-1}^{0} \bmat{3 & 2.25\\0 &0.5}\mbf{x}_1(t,s)ds + \int_{-2}^{0}\bmat{-1&0\\0&-1}\mbf{x}_2(t,s)dt, \\
    \dot{\mbf{x}}_1(t,s)&=\partial_{s}\mbf{x}_1(t,s)    \\
    \dot{\mbf{x}}_2(t,s)&=\frac{1}{2}\partial_{s}\mbf{x}_2(t,s) \\
    \text{with BCs}\qquad &\mbf{x}_1(t,-1)=x(t), \qquad  \mbf{x}_1(t,-2)=x(t).
\end{align*}
In this system, $\mbf{x}_1$ and $\mbf{x}_2$ must be first-order differentiable with respect to $s$, suggesting that the fundamental state associated to this PDE is given by $\mbf{x}_{\text{f}}(t)=\sbmat{x_{\text{f},0}(t)\\\mbf{x}_{\text{f},1}(t)\\\mbf{x}_{\text{f},2}(t)}=\sbmat{x(t)\\\partial_{s}\mbf{x}_1(t)\\\partial_{s}\mbf{x}_2(t)}$ for $t\geq 0$. The \texttt{PIE} structure derived from the \texttt{DDE} will describe the dynamics in terms of this fundamental state $\mbf{x}_{\text{f}}$, where we note that, indeed, the objects \texttt{T} and \texttt{A} are of dimension $6\times 6$. In particular, we find that
\begin{matlab}
\begin{verbatim}
 >> T = PIE.T
 T =
      [1,0] | [0,0,0,0] 
      [0,1] | [0,0,0,0] 
      ------------------
      [1,0] | T.R 
      [0,1] |   
      [1,0] |   
      [0,1] |   

 T.R =
     [0,0,0,0] | [0,0,0,0] | [-1,0,0,0] 
     [0,0,0,0] | [0,0,0,0] | [0,-1,0,0] 
     [0,0,0,0] | [0,0,0,0] | [0,0,-1,0] 
     [0,0,0,0] | [0,0,0,0] | [0,0,0,-1] 
\end{verbatim}
\end{matlab}
where \texttt{T.P} is simply a $2\times 2$ identity operator, as the first two state variables of the augmented state $\mbf{x}$ and the fundamental state $\mbf{x}_{\text{f}}$ are both identical, and equal to the finite-dimensional state $x(t)$. More generally, we find that the augmented state can be retrieved from the associated fundamental state as
\begin{align*}
    \mbf{x}(t,s)=\bl(\mcl{T}\mbf{x}_{\text{f}}\br)(t,s)
    =\left[\begin{array}{ll}
         I_{2\times 2} & \smallint_{-1}^{0}d\theta \sbmat{0_{2\times 2} & 0_{2\times 2}}  \\
         \sbmat{I_{2\times 2}\\I_{2\times 2}} & \smallint_{s}^{0}d\theta \sbmat{-I_{2\times 2}&0_{2\times 2}\\ 0_{2\times 2}&-I_{2\times 2}} 
    \end{array}\right]
    \left[\begin{array}{l}
        x_{\text{f},0}(t)\\ \mbf{x}_{\text{f},1}(t,\theta)\\ \mbf{x}_{\text{f},2}(t,\theta)
    \end{array}\right]
    =\left[\begin{array}{l}
        x_{\text{f},0}(t)\\ x_{\text{f},0}(t)-\int_{s}^{0}\mbf{x}_{\text{f},1}(t,\theta)d\theta\\ x_{\text{f},0}(t)-\int_{s}^{0}\mbf{x}_{\text{f},2}(t,\theta)d\theta
    \end{array}\right]
\end{align*}
Then, studying the value of the object \texttt{A}
\begin{matlab}
\begin{verbatim}
 >> A = PIE.A
 A =
      [-0.5,2.25] | [-3*s-3,-2.25*s-2.25,2*s+2,0] 
       [0.5,-2.5] | [0,-0.5*s-0.5,0,2*s+2] 
      --------------------------------------------
            [0,0] | A.R 
            [0,0] |   
            [0,0] |   
            [0,0] |   

 A.R =
          [1,0,0,0] | [0,0,0,0] | [0,0,0,0] 
          [0,1,0,0] | [0,0,0,0] | [0,0,0,0] 
     [0,0,0.5000,0] | [0,0,0,0] | [0,0,0,0] 
     [0,0,0,0.5000] | [0,0,0,0] | [0,0,0,0] 
\end{verbatim}
\end{matlab}
we find that the DDE can be equivalently represented by the PIE
\begin{align*}
    \bl(\mcl{T}\dot{\mbf{x}}_{\text{f}}\br)(t,s)&=
    \left[\begin{array}{l}
        \dot{x}_{\text{f},0}(t)\\ \dot{x}_{\text{f},0}(t)-\int_{s}^{0}\dot{\mbf{x}}_{\text{f},1}(t,\theta)d\theta\\ \dot{x}_{\text{f},0}(t)-\int_{s}^{0}\dot{\mbf{x}}_{\text{f},2}(t,\theta)d\theta
    \end{array}\right]  \\
    &=\left[\begin{array}{l}
        \sbmat{-0.5&2.25\\0.5&-2.5}x_{\text{f},0}(t) + \int_{-1}^{0}(s+1)\bbl(\sbmat{-3&-2.25\\0&-0.5}\mbf{x}_{\text{f},1}(t,s) + 2\mbf{x}_{\text{f},2}(t,s) \bbr)ds\\ 
        \mbf{x}_{\text{f},1}(t,s)\\
        \frac{1}{2}\mbf{x}_{\text{f},2}(t,s)
    \end{array}\right]
    =\bl(\mcl{A}\mbf{x}_{\text{f}}\br)(t,s)
\end{align*}

\section{Converting a System with Inputs and Outputs to a PIE}\label{sec:PIE:PDE2PIE:io_pde_2_pie}

In addition to autonomous differential systems, systems with inputs and outputs can also be represented as PIEs. In this case, the PIE takes a more general form
\begin{align}\label{eq:ioPIE_structure}
    \mcl{T}_u\dot{u}(t)+\mcl{T}_w\dot{w}+\mcl{T}\dot{\mbf{x}}_{\text{f}}(t)&=\mcl{A}\mbf{x}_{\text{f}}(t)+\mcl{B}_1 w(t)+\mcl{B}_2 u(t),   \nonumber\\
    z(t)&=\mcl{C}_1\mbf{x}_{\text{f}}(t) + \mcl{D}_{11}w(t) + \mcl{D}_{12}u(t), \nonumber\\
    y(t)&=\mcl{C}_2\mbf{x}_{\text{f}}(t) + \mcl{D}_{21}w(t) + \mcl{D}_{22}u(t),
\end{align}
where $w$ denotes the exogenous inputs, $u$ the actuator inputs, $z$ the regulated outputs, and $y$ the observed outputs. Here, the operator $\mcl{T}_{u}$, $\mcl{T}_{w}$ and $\mcl{T}$ define the map from the fundamental state $\mbf{x}_{\text{f}}$ back to the PDE state as
\begin{align*}
    \mbf{x}(t)=\mcl{T}_{u}u(t)+\mcl{T}_{w}w(t)+\mcl{T}\mbf{x}_{\text{f}}(t),
\end{align*}
where the operators $\mcl{T}_{u}$ and $\mcl{T}_{w}$ will be nonzero only if the inputs $u$ and $w$ contribute to the boundary conditions enforced upon the PDE state $\mbf{x}$. As such, the temporal derivatives $\dot{u}$ and $\dot{w}$ will also contribute to the PIE only if these inputs appear in the boundary conditions, which may be the case when performing e.g. boundary or delayed control.

In PIETOOLS, systems with inputs and outputs can be converted to PIEs in the same manner as autonomous systems. For example, consider a 1D heat equation with distributed disturbance $w$, and boundary control $u$, where we can observe the state at the upper boundary, and we wish to regulate the integral of the state over the entire domain:
\begin{align}\label{eq:ex_PDE2PIE_io}
    \dot{\mbf{x}}(t,s)&=\frac{1}{2}\partial_{s}^2\mbf{x}(t,s)+s(2-s)w(t),  &   s&\in[0,1]  \nonumber\\
    z(t)&=\int_{0}^{1}\mbf{x}(t,s)ds,    \nonumber\\
    y(t)&=\mbf{x}(t,1),  \nonumber\\
    \text{with BCs} \hspace*{1.0cm} 
    \mbf{x}(t,0)&=u(t), \hspace*{1.5cm} \partial_{s}\mbf{x}(t,1)=0,
\end{align}
This system too can be represented as a partial integral equation, describing the dynamics of the fundamental state $\mbf{x}_{f}=\partial_{s}^2\mbf{x}$. To arrive at this PIE representation, we once more implement the PDE using the command line parser as
\begin{matlab}
\begin{verbatim}
 >> pvar s t
 >> x = state('pde');    
 >> w = state('in');      u = state('in');
 >> z = state('out');     y = state('out');
 >> PDE = sys();
 >> PDE_dyn = diff(x,t) == 0.5*diff(x,s,2) + s*(2-s)*w;
 >> PDE_z   = z == int(x,s,[0,1]);
 >> PDE_y   = y == subs(x,s,1);
 >> PDE_BCs = [subs(x,s,0) == u; subs(diff(x,s),s,1) == 0];
 >> PDE = addequation(PDE,[PDE_dyn; PDE_z; PDE_y; PDE_BCs]);
 >> PDE = setControl(PDE,u);          PDE = setObserve(PDE,y);
\end{verbatim}
\end{matlab}
where we use the commands \texttt{setControl} and \texttt{setObserve} to indicate that $u$ and $y$ are a controlled input and observed output respectively. Then, we can convert this system to an equivalent PIE as before, finding a structure
\begin{matlab}
\begin{verbatim}
 >> PIE = convert(PDE,`pie')
 PIE = 
   pie_struct with properties:

     dim: 1;
    vars: [1×2 polynomial];
     dom: [0 1];
     
       T: [1×1 opvar];     Tw: [1×1 opvar];     Tu: [1×1 opvar]; 
       A: [1×1 opvar];     B1: [1×1 opvar];     B2: [1×1 opvar]; 
       C1: [1×1 opvar];    D11: [1×1 opvar];    D12: [1×1 opvar]; 
       C2: [1×1 opvar];    D21: [1×1 opvar];    D22: [1×1 opvar]; 
\end{verbatim}
\end{matlab}
In this structure, the fields \texttt{T} through \texttt{D22} describe the PI operators $\mcl{T}$ through $\mcl{D}_{22}$ in the PIE~\eqref{eq:ioPIE_structure}. Here, since the exogenous input $w$ does not contribute to the boundary conditions, it also will not contribute to the map $\mbf{x}=\mcl{T}_{u}u+\mcl{T}_{w}w+\mcl{T}\mbf{x}_{\text{f}}$ from the fundamental state $\mbf{x}_{\text{f}}$ to the PDE state $\mbf{x}$. As such, we also find that the associated \texttt{opvar} object \texttt{Tw} has all parameters equal to zero, whereas \texttt{Tu} and \texttt{T} are distinctly nonzero
\begin{matlab}
\begin{verbatim}
 >> Tw = PIE.Tw
 Tw = 
       [] | [] 
      -----------
      [0] | Tw.R 
      
 >> Tu = PIE.Tu
 Tu = 
       [] | [] 
      -----------
      [1] | Tu.R 

 >> T = PIE.T
 T = 
      [] | [] 
      ---------
      [] | T.R 
      
 T.R = 
       [0] | [-theta] | [-s] 
\end{verbatim}
\end{matlab}
Note here that only the parameter \texttt{Tu.Q2} is non-empty for \texttt{Tu}, and only \texttt{T.R} is nonempty for \texttt{T}, as $\mcl{T}_u$ maps a finite-dimensional state $u\in\R$ to an infinite dimensional state $\mbf{x}\in L_2[0,1]$, whilst $\mcl{T}$ maps an infinite-dimensional state $\mbf{x}_{\text{f}}\in L_2[0,1]$ to an infinite-dimensional state $\mbf{x}\in L_2[0,1]$. Studying the values of \texttt{Tu} and \texttt{T}, we find that we can retrieve the PDE state as
\begin{align*}
    \mbf{x}=\mcl{T}_{u}u+\mcl{T}\mbf{x}_{\text{f}}=u -\int_{0}^{s}\theta\mbf{x}_{\text{f}}(\theta)d\theta - \int_{s}^{1}s\mbf{x}_{\text{f}}(\theta)d\theta = u-\int_{0}^{s}\theta\partial_{\theta}^2\mbf{x}(\theta)d\theta - \int_{s}^{1}s\partial_{\theta}^2\mbf{x}(\theta)d\theta
\end{align*}
Next, we look at the operators $\mcl{A}$, $\mcl{B}_1$ and $\mcl{B}_{2}$. Here, $\mcl{B}_2$ will be zero, as the input $u$ does not appear in the equation for $\dot{\mbf{x}}$, nor does the value of $\mbf{x}_{\text{f}}=\partial_{s}^2\mbf{x}$ depend on $u$. For the remaining operators, we find that they are equal to
\begin{matlab}
\begin{verbatim}      
 >> A = PIE.A
 A = 
      [] | [] 
      ---------
      [] | T.R 

 A.R = 
       [0.5] | [0] | [0]  

 >> B1 = PIE.B1
 B1 =
              [] | [] 
      ------------------
      [-s^2+2*s] | B1.R   
\end{verbatim}
\end{matlab}
suggesting that the fundamental state $\mbf{x}_{\text{f}}$ must satisfy
\begin{align*}
    \dot{u}(t) -\int_{0}^{s}\theta\dot{\mbf{x}}_{\text{f}}(t,\theta)d\theta - \int_{s}^{1}s\dot{\mbf{x}}_{\text{f}}(t,\theta)d\theta &= \frac{1}{2}\mbf{x}_{\text{f}}(t) + [-s^2 + 2s]w(t),    &
    s&\in[0,1]
\end{align*}
This leaves only the output equations. Here, since there is no feed-through from $w$ into $z$ or $y$, the operators $\mcl{D}_{11}$ and $\mcl{D}_{21}$ will both be zero. However, despite the actuator input $u$ not appearing in the PDE equations for $z$ and $y$, the contribution of $u$ to the BCs means that the value of the PDE state $\mbf{x}=\mcl{T}\mbf{x}_{\text{f}}+\mcl{T}_u u$ also depends on the value of $u$, and therefore $\mcl{D}_{12}$ and $\mcl{D}_{22}$ are nonzero. In particular, we find that
\begin{matlab}
\begin{verbatim}      
 >> C1 = PIE.C1
 C1 = 
      [] | [0.5*s^2-s] 
      -----------------
      [] | C1.R 

 >> D12 = PIE.D12
 D12 = 
      [1] | [] 
      -----------
      []  | D12.R 

 >> D22 = PIE.D22
 D22 =
      [1] | [] 
      -----------
      []  | D22.R  
     
 >> C2 = PIE.C2
 C2 =
      [] | [-s] 
      ----------
      [] | C2.R                
\end{verbatim}
\end{matlab}
Here, only the parameters \texttt{Q1} of \texttt{C1} and \texttt{C2} are non-empty, as the operators $\mcl{C}_1$ and $\mcl{C}_2$ map infinite-dimensional states $\mbf{x}_{\text{f}}\in L_2[0,1]$ to finite-dimensional outputs $z,y\in\R$. Similarly, only the parameters \texttt{P} of \texttt{D12} and \texttt{D22} are non-empty, as $\mcl{D}_{12}$ and $\mcl{D}_{22}$ map the finite-dimensional input $u\in\R$ to finite-dimensional outputs $z,y\in\R$. Combining with the earlier results, we find that the PDE~\eqref{eq:ex_PDE2PIE_io} may be equivalently represented by the PIE
\begin{align*}
    \dot{u}(t)-\int_{0}^{s}\theta\dot{\mbf{x}}_{\text{f}}(t,\theta)d\theta - \int_{s}^{1}s\dot{\mbf{x}}_{\text{f}}(t,\theta)d\theta &= \frac{1}{2}\mbf{x}_{\text{f}}(t,s) + s(2-s)w(t),    \qquad   s\in[0,1]  \\
    z(t)&=\int_{0}^{1}\bbl(\frac{1}{2}s^2-s\bbr)\mbf{x}_{\text{f}}(t,s)ds + u(t),  \\
    y(t)&=-\int_{0}^{1}s\mbf{x}_{\text{f}}(t,s)ds + u(t),   \qquad \text{where $\mbf{x}_{\text{f}}(t,s)=\partial_{s}^{2}\mbf{x}(t,s)$}.
\end{align*}

\chapter{Simulating PDE, DDE and PIE Solutions with PIESIM}\label{ch:PIESIM}

In the previous chapter, we showed how an equivalent PIE representation of any well-posed, linear PDE or DDE can be computed in PIETOOLS. This partial integral representation offers several benefits to the partial differential and delay-differential formats, including analysis of e.g. stability properties using convex optimization as we show in Chapter~\ref{ch:LPIs}. In this chapter, we demonstrate another benefit of the PIE representation, namely the relative ease of numerical simulation of solutions in this format. In particular, in Section~\ref{sec:PIESIM:theory}, we provide some theoretical background on how solutions to PIEs can be numerically simulated. In Section~\ref{sec:PIESIM:functions}, we then show how the PIETOOLS function \texttt{PIESIM} can be used to simulate solutions to PDE, DDE and PIE systems. Finally, in Section~\ref{sec:PIESIM:demonstrations}, we show how solutions computed with \texttt{PIESIM} can be plotted, and several examples of how \texttt{PIESIM} can be used.

\section{Simulating Solutions in the PIE Representation}\label{sec:PIESIM:theory}
Consider a PIE system of the form
\begin{align}\label{eq:sim_pie}
    \mcl T \dot{\mbf v}(t) +\mcl T_w \dot{w}(t)+\mcl T_u \dot{u}(t)&= \mcl A \mbf v(t) + \mcl B_1 w(t)+\mcl B_2 u(t),\quad \mbf v(0,s) = \mbf v_0(s), \quad s\in[-1,1], \quad t\ge 0,\notag\\
    z(t) &= \mcl C \mbf v(t) + \mcl D_{11} w(t)+\mcl D_{12}u(t),
\end{align}
where $\mbf{v}(t)\in L_2^{n}[-1,1]$ for $t\geq 0$. To simulate solutions to this system, the state $\mbf{v}(t)$ at each time $t\geq 0$ is projected on to a finite dimensional vector space spanned by Chebyshev polynomials up to order $N$. The resulting projected solution, $\mbf v_N \approxeq \mbf v$, is of the form
\begin{align}\label{eq:sim_n_approx}
    \mbf v_N(t,s) &= \sum_{i=0}^N \alpha_i(t) P_i(s), \qquad s\in [-1,1], \quad t\ge 0,
\end{align}
where $P_i$ is a Chebyshev polynomial of degree $i$.

By substituting $\mbf v_N$ in the form of (\ref{eq:sim_n_approx}) into the PIE and taking an inner product with each basis Chebyshev polynomial function, we obtain an ODE approximation of the PIE as
\begin{align}\label{eq:sim_ode}
    T\dot{\alpha}(t) +T_w\dot{w}(t)+T_u\dot{u}(t)&= A \alpha(t) + B_1w(t)+ B_2u(t), \qquad \alpha = \bmat{\alpha_0&\alpha_1&\cdots&\alpha_N}^T,
\end{align}
where $T$, $T_w$, $T_u$, $A$ and $B_i$ are matrices obtained by spectral discretization of the PI operators $\mcl{T}$ through $\mcl{B}_i$. Given an initial condition $\alpha(0)=\alpha_{\text{I}}\in\R^{N+1}$, the resulting ODE~\eqref{eq:sim_ode} can then be solved analytically if $T$ is invertible, or numerically using any time-stepping scheme. Here, an initial value for $\alpha$ can be obtained from the initial conditions for $\mbf{v}$ in the PIE~\eqref{eq:sim_pie} using the identity
\begin{align*}
    \alpha_k(0) = \frac{2}{N+1}\sum_{i=0}^{N+1}\mbf v(0,s_i)P_K(s_i), \qquad s_i = \cos{\frac{i\pi}{N}}, \qquad k= 0, 1, \cdots, N.
\end{align*}
Given the ODE solution $\alpha(t)$ to~\ref{eq:sim_ode}, an approximate solution $\mbf{v}_N$ to the PIE can be reconstructed using the equation \eqref{eq:sim_n_approx}. If the PIE corresponds to a PDE or DDE system, the PIE solution $\mbf{v}_N$ can then be used to find an approximate solution $\mbf{x}_N$ of the original PDE or DDE using the identity
\begin{align}\label{eq:sim_original_sol}
    \mbf{x}(t,s) &= \mcl T\mbf v(t,s) + \mcl T_w w(t)+  \mcl T_u u(t)\approx T\mbf v_N(t,s) + T_w w(t)+T_u u(t).
\end{align}

\section{PIE Simulation Using PIETOOLS}\label{sec:PIESIM:functions}
PIETOOLS 2022 supports simulation of PIEs obtained by transforming a DDE/DDF or a PDE with spatial derivatives up to order $2$. Users can run PIE simulations either by using the script \texttt{solver\_PIESIM.m} located in the PIESIM folder or by directly calling the function \texttt{PIESIM()}. A guide to use both the methods will be presented in the following subsections.

\subsection{Following the \texttt{solver\_PIESIM} Template Script}
The script file is organized as a template for the user to demonstrate a typical workflow involved in simulation of the PIEs. The simulation procedure can be broken down into the following steps.
\begin{enumerate}
    \item Rescale the spatial dimension in case of PDEs (the delay interval in case of DDEs) to the interval $[-1,1]$, or alternatively use \texttt{rescalePIE()} function after conversion
    \item (Mandatory) Define a PDE/DDE model
    \begin{itemize}
        \item Alternatively, the above step can be skipped if the PIE is already known, however, that feature is restricted to executive function file
    \end{itemize}
    \item Convert PDE/DDE to a PIE (Use converter functions)
    \item (Optional) Define all the simulation settings listed below under \texttt{opts} structure 
    \begin{itemize}
        \item Order of approximation \texttt{N}
        \item Time of simulation \texttt{tf}
        \item Time integration scheme \texttt{intScheme}
        \item Order of time integration scheme \texttt{Norder} (only if \texttt{intScheme=1})
        \item System type \texttt{type}
        \item Time step \texttt{dt}
        \item Flag to turn plotting on or off \texttt{plot}
    \end{itemize}
    \item (Optional) Define all the system inputs listed below under \texttt{uinput} structure 
    \begin{itemize}
        \item Initial condition \texttt{ic.ODE} and \texttt{ic.ODE} (or \texttt{ic.DDE} for DDE model, \texttt{ic.PDE} for PIE model) as MATLAB symbolic expression in \texttt{sx} (space) and \texttt{st}
        \item Disturbance (MATLAB symbolic expression in time \texttt{st}) \texttt{w}
        \item Control input (MATLAB symbolic expression in time \texttt{st}) \texttt{u} 
        \item Flag for comparing with exact solution \texttt{ifexact}
        \item Exact solution as a MATLAB symbolic expression in time \texttt{st} and space \texttt{sx} under \texttt{exact}
    \end{itemize}
	\item Call \texttt{PIESIM(model,opts,uinput)} with the above inputs
    \item Reconstruct PDE/DDE solution (performed inside \texttt{PIESIM()})
\end{enumerate}

\begin{boxEnv}{Warning}
\texttt{uinput.ic.PDE} is used for both PDE and PIE inputs to \texttt{PIESIM()} function, however, when a \textbf{PDE} is passed \texttt{uinput.ic.PDE} stores initial conditions for the PDE, whereas when a \textbf{PIE} is passed \texttt{uinput.ic.PDE} stores initial conditions for the PIE!  
\end{boxEnv}




\subsection{Using the \texttt{PIESIM} Function}
By using the function, the user can directly employ simulation results in a other scripts. While the systems definitions for PDE/DDE/PIE and \texttt{uinput}, and \texttt{opts} have the same format. This function can be called using the syntax
\begin{matlab}
    >> solution = PIESIM(system, opts, uinput, n\_pde);
\end{matlab}
where \texttt{system} is a PDE, DDE, or PIE structure whereas \texttt{opts} is the simulation options structure and \texttt{uinput} is a structure as described in the previous subsection both of which are optional for PDE/DDE systems. Notice, the additional input \texttt{n\_pde} that is required only if the \texttt{system} defined is of the type `PIE'. In case a PIE is directly passed, the user has to provide order of differentiablity of the original PDE/DDE states that generated the PIE form as the fourth argument. Furthermore, the information in the \texttt{uinput}, such as IC and input, must now correspond to the PIE (and NOT the original system). The syntax to simulate a PIE directly using executive function is given by
\begin{matlab}
    >> solution = PIESIM(PIE, opts, uinput, n\_pde);
\end{matlab}
where \texttt{n\_pde} is a vector describing the number of continuous, differentiable and twice differentiable states in the original PDE/DDE, in the same order. For example, \texttt{n\_pde = [n0,n1,n2]} implies the original PDE/DDE has \texttt{n0} continuous states, \texttt{n1} differentiable states, and \texttt{n2} twice differentiable states.

This function returns an output structure \texttt{solution} with the fields
\begin{itemize}
\item \texttt{tf} - scalar - actual final time of the solution
\item \texttt{final.pde} - array of size $(N+1) \times n_s$, $n_s=n_0+n_1+n_2$ - PDE (distributed state) solution at a final time 
\item  \texttt{final.ode} - array of size $n_x$ - ODE solution at a final time 
\item  \texttt{final.observed} - array of size $n_y$ - final value of observed outputs
\item  \texttt{final.regulated} - array of size $n_z$ - final value of regulated outputs 
\item \texttt{timedep.dtime} - array of size $1 \times N_{steps}$ - array of temporal stamps (discrete time values) of the time-dependent solution (only if \texttt{intScheme=1})
\item \texttt{timedep.pde} - array of size $(N+1) \times n_s \times N_{steps}$ - time-dependent solution of $n_s$ PDE (only if \texttt{intScheme=1}) (distributed) states of the primary PDE system
\item \texttt{timedep.ode} - array of size $n_x \times N_{steps}$ - time-dependent solution of $n_x$ ODE states, where $N_{steps}= tf/dt$ (only if \texttt{intScheme=1})
\item \texttt{timedep.observed} - array of size $n_y \times N_{steps}$ - time-dependent value of observed outputs (only if \texttt{intScheme=1})
\item \texttt{timedep.regulated} - array of size $n_z \times N_{steps}$ - time-dependent value of regulated outputs (only if \texttt{intScheme=1})
\end{itemize}

\begin{boxEnv}{Warning}
    In contrast to \texttt{uinput.ic.pde}, irrespective of the input, the \texttt{solution.timedep.pde} always stores $\mcl T\mbf v+\mcl T_w w+\mcl T_u u$ as the value as described in Equation \ref{eq:sim_original_sol}.
\end{boxEnv}

\section{Plotting the solution}\label{sec:PIESIM:demonstrations}
Simulation of PIEs, either by using solver file or directly using the executive file, generates figures that plot time-varying ODE states (from 0 to final simulation time) for each ODE state. Further, a plot showing the spatial distribution (only at final simulation time) is generated for all distributed states in the PIE. Note, that plots correspond to the solution of the \textbf{original PDE/DDE} and not the PIE solution. In general, given a PIE of the form \eqref{eq:sim_pie}, the solution which is plotted is 
\eqref{eq:sim_original_sol}.
However, the value of time-varying distributed state at each simulation time step is stored under the solution output which is given by the executive file. The user can use this output to generate further plots to calculate outputs $z$ as defined in \eqref{eq:sim_pie}.

\subsection{PIESIM Demonstration A: PDE example}
In this section, we will demonstrate the standard process involved in simulation of PIEs using an example from the \texttt{examples\_pde\_library\_PIESIM} file.

\begin{codebox}
First, we load an example from the library file and then edit simulation parameters directly by accessing the corresponding properties. The full code for the simulation demonstrated in this section is given below. 
\begin{matlab}
    >> syms sx st;\\
    >> init\_option=1;\\
    >> [PDE,uinput]=examples\_pde\_library\_PIESIM(4);\\
    >> uinput.exact(1) = -2*sx*st-sx\^{}2;\\
    >> uinput.ifexact=true;\\
    >> uinput.w(1) = -4*st-4;\\
    >> uinput.ic.PDE=-sx\^{}2;\\
    >> opts.plot='yes';\\
    >> opts.N=8;\\
    >> opts.tf=0.1;\\
    >> opts.intScheme=1;\\
    >> opts.Norder = 2;\\
    >> opts.dt=1e-3;\\
    >> solution = PIESIM(PDE,opts,uinput);
\end{matlab}
Next, we will explain each line used in the above code.
\end{codebox}

First, to choose an example from the examples library, we set the library flag to one. Then, an example can be selected by specifying the example number (between 1 and 34) to load the example.
\begin{matlab}
    >> init\_option=1;\\
    >> [PDE,uinput]=examples\_pde\_library\_PIESIM(4);
\end{matlab}

\noindent In this demonstration, we choose the example
\begin{align*}
&\dot{\mbf x}(t,s) = s \partial_s^2 \mbf x(t,s), \qquad s\in[0,2], t\ge 0\\
&\mbf x(t,0) = 0, \qquad \mbf x(t,2) = w_1(t), \qquad \mbf x(0,s) = -s^2.
\end{align*}
where $w_1(t) = -4t-4$. For this PDE, the exact solution is known and is given by the expression $\mbf x(t,s) = -2st-s^2$ which can be specified under \texttt{uinput} structure for verification as shown below. 
\begin{matlab}
    >> uinput.exact(1) = -2*sx*st-sx\^{}2;\\
    >> uinput.ifexact=true;
\end{matlab}

\noindent Likewise, other input parameters such as, initial conditions and inputs at the boundary are specified as
\begin{matlab}
    >> uinput.w(1) = -4*st-4;\\
    >> uinput.ic.PDE=-sx\^{}2;
\end{matlab}

\noindent where \texttt{sx, st} are MATLAB symbolic objects. However, the example automatically defines the \texttt{uinput} structure and the above expressions are provided for demonstration only and not necessary when using a PDE from example library. Once the PDE and system inputs are defined, we have to specify simulation parameters under \texttt{opts} structure. First, we turn on the plotting by specifying the plotting flag as show below.
\begin{matlab}
    >> opts.plot='yes';\\
    >> opts.N=8;\\
    >> opts.tf=0.1;\\
    >> opts.intScheme=1;\\
    >> opts.Norder = 2;\\
    >> opts.dt=1e-3;
\end{matlab}

\noindent We specify the order of discretization (order of chebyshev polynomials to be used in approximation of PDE solution $N$) and time of simulation. Then, we select a time-integration scheme (backward difference scheme is used in this demonstration, however, one can chose symbolic integration by setting \texttt{intScheme=2}). In solver file, time step and order of truncation are automatically chosen for backward difference scheme as shown below, however, the user can modify these parameters as needed.

\noindent Now that we have defined all necessary parameters we can run the simulation using the command for the PDE example 
\begin{matlab}
    solution = PIESIM(PDE,opts,uinput);
\end{matlab}

\begin{figure}[ht]
	\centering
	\includegraphics[scale=0.75]{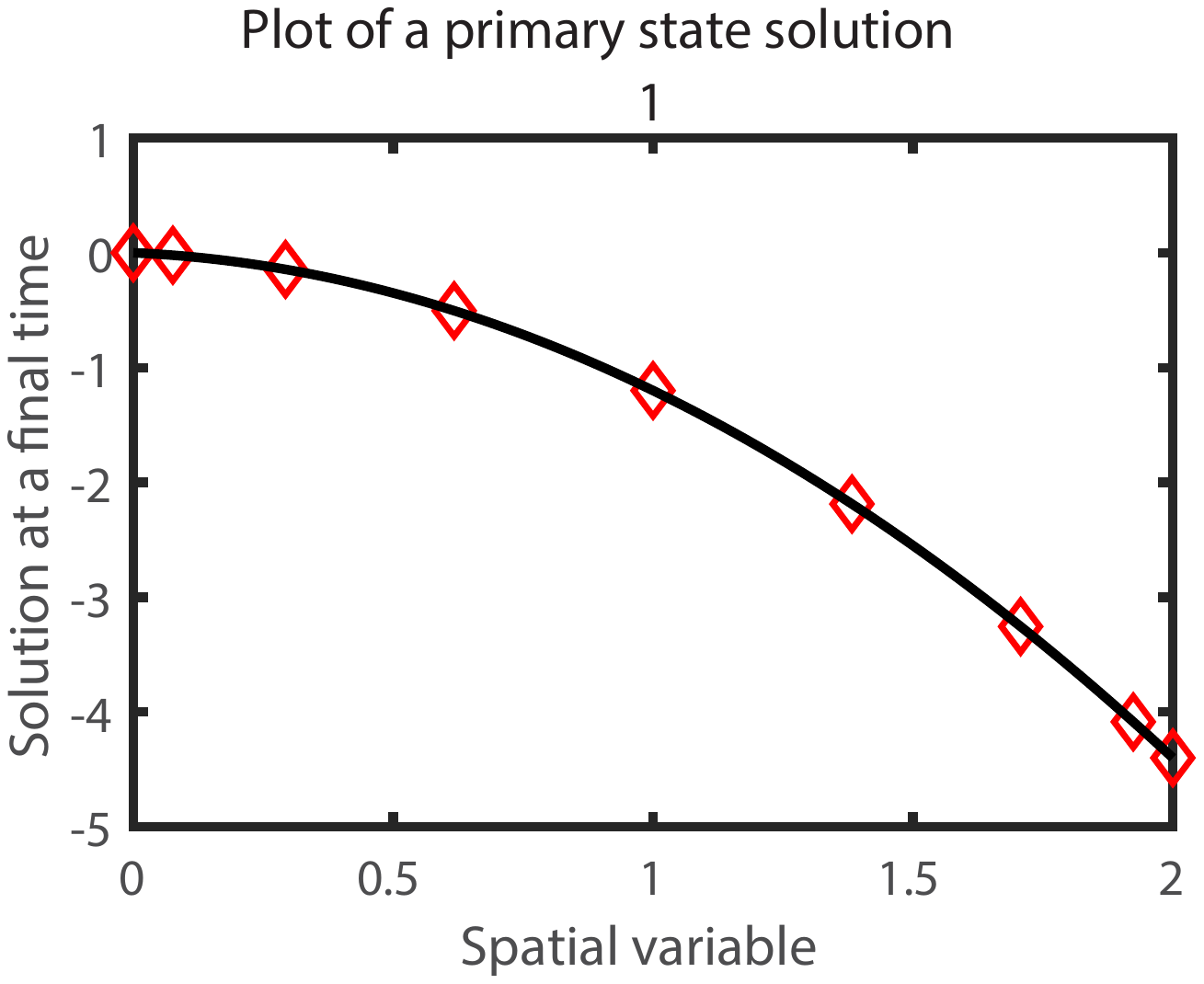}
	\caption{Default PIESIM plot: Final solution $\mbf x(t,\cdot)$ at $t= 0.1 s$ obtained by analytical expression (solid line) and by PIE simulation (dots) are plotted against space $[0,2]$}\label{fig:piesimdemoA}
\end{figure}
 which produces the plot Fig.~\ref{fig:piesimdemoA}, where we see that simulation result in dots whereas the analytical solution is plotted using the solid line. If the analytical solution is not passed, then only the dots are plotted.

\subsection{PIESIM Demonstration B: DDE example}
Simulation of DDEs can be performed using the same steps as the simulation of PDEs, however, there is a main difference which is the first argument to \texttt{PIESIM()} function a DDE model. Consider a DDE system,
\begin{align*}
    \dot{x}(t) &= \bmat{-1&2\\0&1}x(t) + \bmat{0.6&-0.4\\0&0}x(t-\tau_a) + \bmat{0&0\\0&-0.5}x(t-\tau_b)+\bmat{1\\1}w(t)+\bmat{0\\1}u(t)\\
    y(t) &= \bmat{1&0\\0&1\\0&0} x(t)+\bmat{0\\0\\0.1} u(t)
\end{align*}
where $\tau_a = 1$ and $\tau_b = 2$. This system can be stored in a \texttt{DDE} structure as shown below.
\begin{matlab}
    >> DDE.A0=[-1 2;0 1]; DDE.Ai{1}=[.6 -.4; 0 0]; \\
    >> DDE.Ai{2}=[0 0; 0 -.5]; DDE.B1=[1;1];\\
    >> DDE.B2=[0;1]; DDE.C1=[1 0;0 1;0 0];\\ 
    >> DDE.D12=[0;0;.1]; DDE.tau=[1,2];\\
    >> solution = PIESIM(DDE,opts,uinput);
\end{matlab}
We can use the same \texttt{opts} and \texttt{uinput} from previous section (except initial conditions, which we will leave undefined and default to zero). Note \texttt{uinput.u} must be set to zero since we are simulating without a controller.
\begin{matlab}
>> uinput.exact(1) = -2*sx*st-sx\^{}2;\\
    >> uinput.ifexact=true;\\
    >> uinput.w(1) = -4*st-4;\\
    >> uinput.u(1) =0;\\
    >> opts.plot='yes';\\
    >> opts.N=8;\\
    >> opts.tf=0.1;\\
    >> opts.intScheme=1;\\
    >> opts.Norder = 2;\\
    >> opts.dt=1e-3;\\
\end{matlab}
Then, use the command to simulate the system which gives us a default plot. However, using the code below, we can change the plot features by accessing the data in \texttt{solution} structure.
\begin{matlab}
>> plot(solution.timedep.dtime,solution.timedep.ode,'--o',...\\
   'MarkerIndices',1:50:length(solution.timedep.dtime));\\
>> ax = gca;\\
>> set(ax,'XTick',solution.timedep.dtime(1:150:end));\\
>> lgd1 = legend('$x_1$','$x_2$','Interpreter','latex'); lgd1.FontSize = 10.5; \\
>> lgd1.Location = 'northeast';\\
>> title('Time evolution of the Delay system states, x, ...\\
     without state feedback control');\\
>> ylabel('$x_1(t), ~~~x_2(t)$','Interpreter','latex','FontSize',15);\\
>> xlabel('t','FontSize',15,'Interpreter','latex');
\end{matlab}

\begin{figure}[ht]
	\centering
	\includegraphics[width=\textwidth]{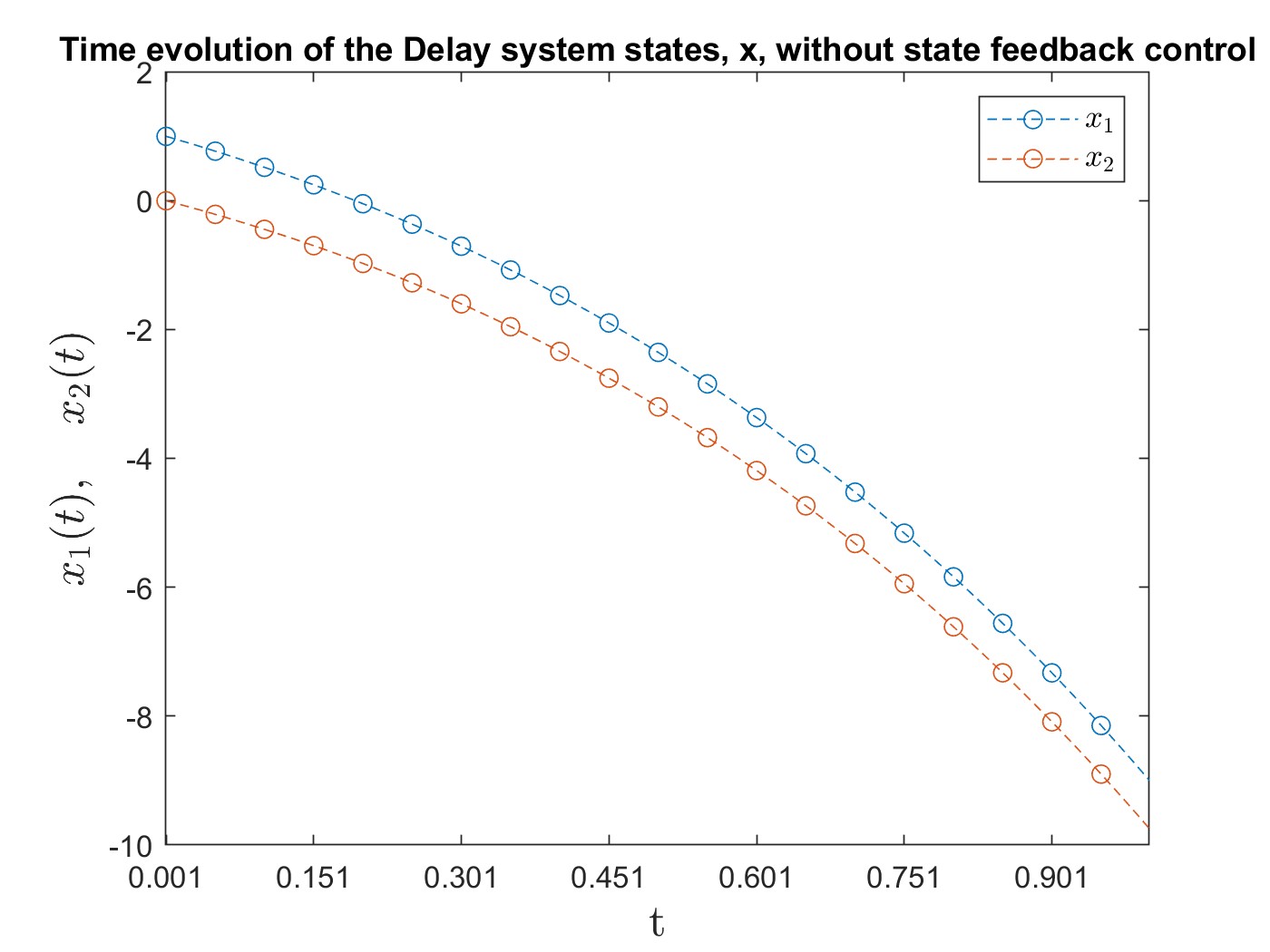}
	\caption{DDE solutions $x_1(t)$ and $x_2(t)$ obtained by by PIE simulation are plotted against time $t$}\label{fig:piesimdemo_ddeB}
\end{figure}

All the code presented here can be found in the demo file packaged with PIETOOLS named ``PIE\_simulation\_DEMO''.

\subsection{PIESIM Demonstration C: PIE example}
In the above DDE example, the solution is clearly unstable. For this system, we can design a stabilizing controller for the PIE form and then simulate the solution again to see the behaviour of the control system under the action of the controller. However, conversion of a PIE back to DDE/PDE format is often tricky. Fortunately, a PIE need not be converted back to DDE or PDE format to simulate its solutions, because \texttt{PIESIM()} can be used to simulate a system in PIE form directly. For example, for the given DDE, 
\begin{align*}
    \dot{x}(t) &= \bmat{-1&2;0&1}x(t) + \bmat{0.6&-0.4\\0&0}x(t-1) +\bmat{0&0\\0&-0.5}x(t-2)+\bmat{1\\1}w(t)+\bmat{0\\1}u(t)\\
    y(t) &= \bmat{1&0\\0&1\\0&0} x(t)+\bmat{0\\0\\0.1} u(t)
\end{align*}
we can find the controller by using the following code.
\begin{matlab}
    >> DDE.A0=[-1 2;0 1]; DDE.Ai{1}=[.6 -.4; 0 0]; \\
    >> DDE.Ai{2}=[0 0; 0 -.5]; DDE.B1=[1;1];\\
    >> DDE.B2=[0;1]; DDE.C1=[1 0;0 1;0 0];\\ 
    >> DDE.D12=[0;0;.1]; DDE.tau=[1,2];\\
    >> DDE=initialize\_PIETOOLS\_DDE(DDE);\\
    >> DDF=minimize\_PIETOOLS\_DDE2DDF(DDE);\\
    >> PIE=convert\_PIETOOLS\_DDF(DDF);\\
    >> [prog, K, gamma, P, Z] = PIETOOLS\_Hinf\_control(PIE);\\
    >> PIE = closedLoopPIE(PIE,K);\\
    >> ndiff = [0, PIE.T.dim(2,1)];
\end{matlab}
Then, we can define the \texttt{uinput} and \texttt{opt} to define simulation parameters for the PIESIM. Using the same inputs and options as before, we have
\begin{matlab}
>> uinput.exact(1) = -2*sx*st-sx\^{}2;\\
    >> uinput.ifexact=true;\\
    >> uinput.w(1) = -4*st-4;\\
    >> opts.plot='yes';\\
    >> opts.N=8;\\
    >> opts.tf=0.1;\\
    >> opts.intScheme=1;\\
    >> opts.Norder = 2;\\
    >> opts.dt=1e-3;
        >> solution=PIESIM(PIE,opts,uinput,ndiff);
\end{matlab}
See the file `DDE\_simulation\_demo.m' in the `GetStarted\_DOCS\_DEMOS'. Given the (optimal) controller gains $K$, we construct a closed loop PIE system using \texttt{closedLoopPIE(PIE,K)}. This closed system can then be simulated using the \texttt{PIESIM()} function. Finally, we plot the system behavior under the state feedback designed for the PIE using the following code.

\begin{matlab}
>> plot(solution.timedep.dtime,solution.timedep.ode,'--o',...\\
'MarkerIndices',1:50:length(solution.timedep.dtime));\\
>> ax = gca;\\
>> set(ax,'XTick',solution.timedep.dtime(1:150:end));\\
>> lgd1 = legend('$x_1$','$x_2$','Interpreter','latex'); lgd1.FontSize = 10.5; \\
>> lgd1.Location = 'northeast';\\
>> title('Time evolution of the Delay system states, x,...\\
with state feedback control');\\
>> ylabel('$x_1(t), ~~~x_2(t)$','Interpreter','latex','FontSize',15);\\
>> xlabel('t','FontSize',15,'Interpreter','latex');
\end{matlab}

The main difference from PDE/DDE simulations to PIE simulations is the new input argument \texttt{ndiff} which describes how many states are differentiable. In case of DDEs, all states with delays are at least once differentiable along the delay variable and hence are all placed under \texttt{n1}. The other arguments such as \texttt{opts} and \texttt{uinput} are same as the inputs described in earlier sections. Behaviour of the DDE solution (for same initial conditions) simulated using a PIE is shown in Figure \ref{fig:piesimdemo_pieB}.
\begin{figure}[H]
	\centering
	\includegraphics[width=\textwidth]{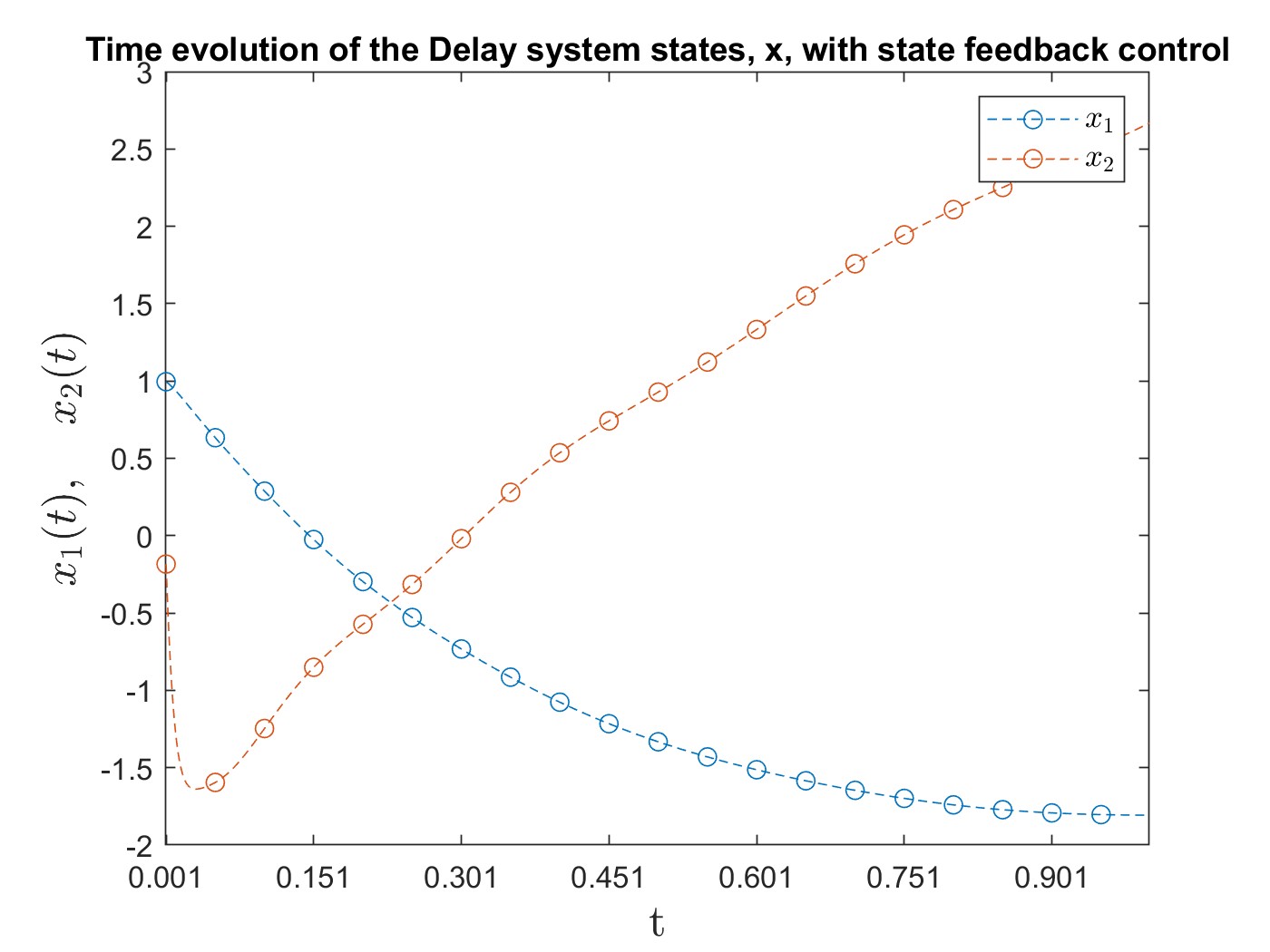}
	\caption{DDE solution $x_1(t)$ and $x_2(t)$ obtained by PIE simulation for the closed loop PIE of the previous DDE example is plotted against time $t$}\label{fig:piesimdemo_pieB}
\end{figure}

All the code presented here can be found in the demo file packaged with PIETOOLS named ``PIE\_simulation\_DEMO''.


%

\chapter{Declaring and Solving Convex Optimization Programs on PI Operators}\label{ch:LPIs}

In Chapter~\ref{ch:PIE} we showed that, using PIETOOLS, we can derive an equivalent PIE representation of any well-posed system of linear partial differential and delay-differential equations. This PIE representation is completely free of boundary conditions and continuity constraints that appear in any well-posed PDE, allowing analysis of PIEs to be performed without having to explicitly account for such additional constraints. In addition, PIEs are parameterized by PI operators, which can be added and multiplied, and for which concepts of e.g. positivity are well-defined. This allows us to impose positivity and negativity constraints on PI operators, referred to as linear PI inequalities (LPIs), to define convex optimization programs for testing properties (such as stability) of PIEs. 

In this chapter, we show how these convex optimization problems can be implemented in PIETOOLS. In particular, in Sections~\ref{sec:LPIs:prog_struct} and~\ref{sec:pivars}, we show how an LPI optimization program can be initialized, and how (PI operator) decision variables can be added to this program structure. Next, in Section~\ref{sec:LPIs:constraints} we show how PI operator equality and inequality constraints can be specified, followed by how an objective function can be set for the program in Section~\ref{sec:LPIs:obj_fun}. In Sections~\ref{sec:LPIs:solve} and~\ref{sec:LPIs:getsol}, we then show how the optimization program can be solved, and how the obtained solution can be extracted respectively. Finally, in Section~\ref{sec:executives-settings}, we show how pre-defined executive files can be used to solve standard LPI optimization programs for PIEs, and how properties in these optimization programs can be modified using the \texttt{settings} files.

We illustrate the use of each of the functions in this chapter with the following example. A demo file ``Hinf\_optimal\_estimator\_DEMO'' declaring and solving the LPI in this example has also been included in PIETOOLS; see Section~\ref{sec:demos:estimator}

\begin{Statebox}{\textbf{Example}}

Consider the problem of designing an $H_{\infty}$-optimal estimator of the form 
\begin{align}\label{eq:LPIs:PIE_example}
	\mcl T \dot{\hat{\mbf{v}}}(t) &=\mcl A\mbf{\hat{v}}(t)+\mathcal{L}\bl(y(t)-\hat{y}(t)\br), & & &
    \mcl T \dot{\mbf{v}}(t)&=\mcl A\mbf{v}(t)+\mcl{B}_1w(t), \notag\\
	\hat{z}(t) &= \mcl{C}_1\mbf{\hat{v}}(t),   &   &\text{for a PIE},  &
    z(t) &= \mcl{C}_1\mbf v(t) + \mcl{D}_{11}w(t),  \notag\\
	\hat{y}(t) &= \mcl{C}_2\mbf{\hat{v}}(t),   & & &
    y(t) &= \mcl{C}_2\mbf{v}(t) + \mcl{D}_{21}w(t),
\end{align}
aiming to find an operator $\mcl{L}$ that minimize the gain $\frac{\|\hat{z}-z\|_{L_2}}{\|w\|_{L_2}}$. 
To construct such an operator, we solve the LPI,
\begin{align}\label{eq:LPIs:LPI_example}
	&\min\limits_{\gamma,\mcl{P},\mcl{Z}} ~~\gamma&\notag\\
	&\mcl{P}\succ0, &
	&\hspace{-1ex}Q:=\bmat{-\gamma I& -\mcl D_{11}^{\top}&-(\mcl P\mcl B_1+\mcl Z\mcl D_{21})^*\mcl T\\(\cdot)^*&-\gamma I&\mcl C_1\\(\cdot)^*&(\cdot)^*&(\mcl P\mcl A+\mcl Z\mcl C_2)^*\mcl T+(\cdot)^*}\preccurlyeq 0&
\end{align}
so that, for any solution $(\gamma,\mcl{P},\mcl{Z})$ to this problem, letting $\mcl{L}:=\mcl{P}^{-1} \mcl{Z}$, the estimation error will satisfy $\norm{z-\hat{z}} \leq \gamma \norm{w}$. For more details on this LPI, and additional examples of LPIs and their applications, we refer to Chapter~\ref{ch:LPI_examples}. 

We will solve the LPI~\eqref{eq:LPIs:LPI_example} for the PIE associated to the PDE
\begin{align}\label{eq:LPIs:PDE_example}
    && \dot{\mbf{x}}(t,s)&=\partial_{s}^2\mbf{x}(t,s) + 4\mbf{x}(t,s) + w(t), & &&    s&\in[0,1] &&\nonumber\\
    \text{with BCs}& & 0&=\mbf{x}(t,0)=\partial_{s}\mbf{x}(t,1),  \nonumber\\
    \text{and outputs}& & z(t)&=\int_{0}^{1}\mbf{x}(t,s)ds + w(t),   \nonumber\\
    &&y(t)&=\mbf{x}(t,1). 
\end{align}
We declare this PDE using the command line parser as
\begin{matlab}
\begin{verbatim}
 >> pvar s t
 >> PDE = sys();
 >> x = state('pde');    w = state('in');
 >> y = state('out');    z = state('out');
 >> eqs = [diff(x,t) == diff(x,s,2) + 4*x + w;
           z == int(x,s,[0,1]) + w;
           y == subs(x,s,1);
           subs(x,s,0) == 0;
           subs(diff(x,s),s,1) == 0];
 >> PDE = addequation(PDE,eqs);
 >> PDE = setObserve(PDE,y);
\end{verbatim}  
\end{matlab}
We convert the PDE to an equivalent PIE using \texttt{convert}, and extract the defining PI operators so that these can be used to declare the LPI~\eqref{eq:LPIs:LPI_example}
\begin{matlab}
\begin{verbatim}
 >> PIE = convert(PDE,'pie');         PIE = PIE.params;
 >> T = PIE.T;    
 >> A = PIE.A;      C1 = PIE.C1;      C2 = PIE.C2;
 >> B1 = PIE.B1;    D11 = PIE.D11;    D21 = PIE.D21;
\end{verbatim}
\end{matlab}

\end{Statebox}

\section{Initializing an Optimization Problem Structure}\label{sec:LPIs:prog_struct}

In PIETOOLS, optimization programs are stored as program structures \texttt{prog}. These structures keep track of the free variables in the optimization program, the decision variables in the optimization program, the constraints imposed upon these decision variables, and the objective function in terms of these decision variables. To initialize an optimization program structure, call the function \texttt{sosprogram}, passing as primary argument a vector of free variables that appear in the problem, and possibly passing decision variables to appear in the program as second argument:
\begin{matlab}
\begin{verbatim}
 >> pvar s1 s2;                            % Declare free (polynomial) variables
 >> dpvar d1 d2;                           % Initialize decision variables
 >> prog = sosprogram([s1;s2], [d1;d2]);   % Initialize the program structure
\end{verbatim}
\end{matlab}
When initializing a program structure for an LPI, it is crucial that the free variables \texttt{poly\_vars} include the spatial variables and dummy variables that appear in the relevant PI operators.

\begin{Statebox}{\textbf{Example}}

For the LPI~\eqref{eq:LPIs:LPI_example}, the free variables that appear in the optimization program will be $s$ and $\theta$, the spatial variables that appear in the PI operators. We initialize the optimization program structure as
\begin{matlab}
\begin{verbatim}
 >> vars = PIE.vars(:)
 vars = 
   [     s]
   [ theta]
   
 >> prog = sosprogram(vars)
 prog = 

   struct with fields:

             var: [1×1 struct]
            expr: [1×1 struct]
        extravar: [1×1 struct]
       objective: []
         solinfo: [1×1 struct]
        vartable: {2×1 cell}
          varmat: [1×1 struct]
     decvartable: {}
\end{verbatim}
\end{matlab}
This initialize an empty optimization program in the free variables \texttt{vars=[s;theta]}. The field \texttt{vartable} will then contain the names of the free variables that appear in the optimization program,
\begin{matlab}
\begin{verbatim}
 >> prog.vartable
 ans =
   2×1 cell array

    {'s'    }
    {'theta'}
\end{verbatim}
\end{matlab}
matching the variables that appear in the \texttt{PIE} structure.
\end{Statebox}



\begin{boxEnv}{\textbf{Note:}}
To represent LPI optimization programs, PIETOOLS utilizes the \texttt{sosprogram} optimization program structure from SOSTOOLS 4.00~\cite{sostools}. For additional options allowed by SOSTOOLS not mentioned here, we refer to the SOSTOOLS 4.00 manual.
\end{boxEnv}

\section{Declaring Decision Variables}\label{sec:pivars}
Having discussed how to initialize an optimization program structure \texttt{prog}, in this section, we show how decision variables can be added to the optimization program structure. For the purposes of implementing LPIs, we distinguish three types of decision variables: standard scalar decision variables (Subsection~\ref{sec:sosdecvar}), positive semidefinite PI operator decision variables (Subsection~\ref{sec:poslpivar}), and indefinite PI operator decision variables (Subsection~\ref{sec:lpivar}).

\subsection{\texttt{sosdecvar}}\label{sec:sosdecvar}
The simplest decision variables in LPI programs are represented by scalar \texttt{dpvar} objects, and can be declared using \texttt{dpvar}. The function \texttt{sosdecvar} must then be used to add the decision variable to the optimization program structure:
\begin{matlab}
\begin{verbatim}
 >> dpvar d1;                   % Initialize a new decision variable
 >> prog = sosdecvar(prog,d1);  % Add the decision variable to the program
\end{verbatim}
\end{matlab}

\begin{Statebox}{\textbf{Example}}

In the LPI~\eqref{eq:LPIs:LPI_example}, $\gamma$ is a scalar decision variables, that appears both in the constraints and the objective function. To declare this variable, we simply call
\begin{matlab}
\begin{verbatim}
 >> dpvar gam;
 >> prog = sosdecvar(prog, gam)
 prog = 

   struct with fields:
 
             var: [1×1 struct]
            expr: [1×1 struct]
        extravar: [1×1 struct]
       objective: 0
         solinfo: [1×1 struct]
        vartable: {2×1 cell}
          varmat: [1×1 struct]
     decvartable: {'gam'}
\end{verbatim}
\end{matlab}
first initializing the variable \texttt{gam} as a \texttt{dpvar} object representing $\gamma$, and subsequently declaring it as decision variables in the optimization program, adding it to the field \texttt{decvartable}.

\end{Statebox}

\subsection{\texttt{poslpivar}}\label{sec:poslpivar}

In PIETOOLS, positive semidefinite PI operator decision variables $\mcl{P}\succcurlyeq 0$ are parameterized by positive matrices $P\succcurlyeq 0$ as $\mcl{P}=\mcl{Z}_d^* T\mcl{Z}_d\succeq 0$, where $\mcl{Z}_{d}$ is a PI operator parameterized by monomials of degree of at most $d$ (see Theorem~\ref{th:positivity} in Appendix~\ref{appx:PI_theory}). Such PI operators can be declared using the function \texttt{poslpivar}:
\begin{matlab}
 >> [prog,P] = poslpivar(prog,n,I,d,opts);
\end{matlab}
This function takes three mandatory inputs.
\begin{itemize}
	\item \texttt{prog}: An sosprogram structure to which to add the PI operator decision variable.
	\item \texttt{n}: A $2\times 1$ vector \texttt{[n0;n1]} specifying the dimensions of $\mcl{P}:\sbmat{\R^{n_0}\\L_2^{n_1}[a,b]}\rightarrow\sbmat{\R^{n_0}\\L_2^{n_1}[a,b]}$ for a 1D operator, or a $4\times 1$ vector \texttt{[n0;n1;n2;n3]} specifying the dimensions for a 2D operator $\mcl{P}:\sbmat{\R^{n_0}\\L_2^{n_1}[a,b]\\L_2^{n_2}[c,d]\\L_2^{n_3}[[a,b]\times[c,d]]}\rightarrow \sbmat{\R^{n_0}\\L_2^{n_1}[a,b]\\L_2^{n_2}[c,d]\\L_2^{n_3}[[a,b]\times[c,d]]}$.
	\item \texttt{I}: a $1\times 2$ array \texttt{[a,b]} specifying the interval $[a,b]$ of the variables in \texttt{P} in 1D, or a $2\times 2$ array \texttt{[a,b;c,d]} specifying the spatial domain $[a,b]\times[c,d]$ of the variables in \texttt{P} in 2D.
	\item \texttt{d} (optional): 
    \begin{itemize}
     \item 1D: A cell structure of the form $\texttt{\{a,[b0,b1,b2]\}}$, specifying the degree \texttt{a} of $s$ in $Z_1(s)$, the degree \texttt{b0} of $s$ in $Z_2(s,\theta)$, the degree \texttt{b1} of $\theta$ in $Z_2(s,\theta)$, and the degree \texttt{b2} of $s$ and $\theta$ combined in $Z_2(s,\theta)$ (see Thm.~\ref{th:positivity}).

     \item 2D: A structure with fields  \texttt{d.dx,d.dy,d.d2}, specifying degrees for operators along $x\in[a,b]$, along $y\in[c,d]$, and along both $(x,y)\in[a,b]\times[c,d]$; call \texttt{help poslpivar\_2d} for more information.
    \end{itemize}

	\item \texttt{opts} (optional): This is a structure with fields 
	\begin{itemize}
		\item \texttt{exclude}: $4\times 1$ vector with 0 and 1 values. Excludes the block $T_{ij}$ (see Thm.~\ref{th:positivity}) if $i$-th value is 1. Binary $16\times1$ array in 2D; call \texttt{help poslpivar\_2d} for more information.
		\item \texttt{psatz}: Sets $g(s)=1$ if set to 0, and $g(s) = (b-s)(s-a)$ if set to 1 in 1D. In 2D, sets $g(x,y)=(b-x)(x-a)(d-y)(y-c)$ if \texttt{psatz=1}.
		\item \texttt{sep}: Binary scalar value, constrains $\texttt{P.R.R1=P.R.R2}$ if set to 1. In 2D, this field is a $6\times 1$ array; call \texttt{help poslpivar\_2d} for more information.
  \end{itemize}
\end{itemize}
The output is a program structure \texttt{prog} with new decision variables and a \texttt{dopvar} class object \texttt{P} representing a positive semidefinite PI operator decision variable. The \texttt{poslpivar} function has other experimental features to impose sparsity constraints on the $T$ matrix of Theorem \ref{th:positivity}, which should be used with caution. Use \texttt{help poslpivar} for more information. 

Note that, to enforce $\mcl{P}\succeq 0$ only for $s\in[a,b]$ (or $(x,y)\in[a,b]\times[c,d]$), the option \texttt{psatz} can be used as
\begin{matlab}
\begin{verbatim}
 >> [prog,P] = poslpivar(prog,n,I,d);
 >> opts.psatz = 1;
 >> [prog,P2] = poslpivar(prog,n,I,d,opts);
 >> P = P+P2;
\end{verbatim}
\end{matlab}
This will allow $P$ to be negative definite outside of the specified domain \texttt{I}, allowing for more freedom in the optimization problem. However, since it involves declaring a second PI operator decision variable \texttt{P2}, it may also substantially increase the computational complexity associated with setting up and solving the optimization problem.

Note also that the output of operator \texttt{P} of \texttt{poslpivar} is only positive semidefinite, i.e. $\mcl{P}\succeq 0$. To ensure strict positive definiteness, a (small) positive constant $\epsilon$ can be added to the operator as e.g.
\begin{matlab}
\begin{verbatim}
 >> ep = 1e-5;
 >> P = P + ep;
\end{verbatim}
\end{matlab}
This amounts to adding a matrix $\epsilon I$ to the parameters \texttt{P.P} and \texttt{P.R.R0} in 1D, or \texttt{P.R00}, \texttt{P.Rxx\{1\}}, \texttt{P.Ryy\{1\}} and \texttt{P.R22\{1,1\}} in 2D, ensuring that $\mcl{P}\succ 0$.

\begin{Statebox}{\textbf{Example}}
In the LPI~\eqref{eq:LPIs:LPI_example}, we have one positive definite decision variable $\mcl{P}\succeq 0$. This operator $\mcl{P}$ should have the same dimensions $\sbmat{n_0\\n_1}$ as the operator $\mcl{T}:\sbmat{\R^{n_0}\\L_2^{n_1}[0,1]}\rightarrow \sbmat{\R^{n_0}\\L_2^{n_1}[0,1]}$, which in our case are $n_0=0$ and $n_1=1$:
 \begin{matlab}
\begin{verbatim}
 >> Pdim = T.dim(:,1)
 Pdim =

      0
      1      
\end{verbatim}
\end{matlab}
Similarly, we require the spatial domain of the operator $\mcl{P}$ to match the spatial domain of the PIE as
\begin{matlab}
\begin{verbatim}
 >> Pdom = PIE.dom
 Pdom =

      0     1
\end{verbatim}
\end{matlab}
We then specify the degrees of the monomials in the operator $\mcl{Z}_d$ defining $\mcl{P}=\mcl{Z}^*_{d} P\mcl{Z}_{d}$ as
\begin{matlab}
\begin{verbatim}   
 >> Pdeg = {6,[2,3,5],[2,3,5]}
 Pdeg =
  
   1×3 cell array

     {[6]}    {[2 3 5]}    {[2 3 5]}
\end{verbatim}
\end{matlab}
These degrees were chosen based on some trial and error, finding that they produce an accurate solution, whilst still being relatively small.

To reduce the size of the optimization program, we also require \texttt{P.R.R1=P.R.R2} using the option \texttt{sep},
\begin{matlab}
\begin{verbatim}
 >> opts.sep = 1;
\end{verbatim}
\end{matlab}
decreasing the freedom in our choice of $\mcl{P}$, but also decreasing computational complexity.
Then, we can declare a positive PI operator $\mcl{P}$ of the proposed specifications as
\begin{matlab}
\begin{verbatim}
 >> [prog,P] = poslpivar(prog,Pdim,Pdom,Pdeg,opts)
 prog = 

   struct with fields:

             var: [1×1 struct]
            expr: [1×1 struct]
        extravar: [1×1 struct]
       objective: [362×1 double]
         solinfo: [1×1 struct]
        vartable: {2×1 cell}
          varmat: [1×1 struct]
     decvartable: {362×1 cell}

 P =
      [] | [] 
      ---------
      [] | P.R 

 P.R =
     Too big to display. Use ans.R.R0 | Too big to display. Use ans.R.R1 |
\end{verbatim}
\end{matlab}
The output object \texttt{P} is a \texttt{dopvar} objects, representing a PI operator decision variable rather than a fixed PI operator. As such, the parameters \texttt{P}, \texttt{Q1}, \texttt{Q2} and \texttt{R} defining this PI operator are \texttt{dpvar} class objects, parameterized by decision variables \texttt{coeff\_i}. These decision variables \texttt{coeff\_i} are collected in the field \texttt{decvartable} of the program structure (along with the previously declared decision variable \texttt{gam}), and represent the matrix $T$ in the expansion $\mcl{P}=\mcl{Z}_{d}^* T\mcl{Z}_{d}$, constrained to satisfy $T\succcurlyeq 0$.

In the LPI~\eqref{eq:LPIs:LPI_example}, the operator $\mcl{P}$ is required to be strictly positive definite, satisfying $\mcl{P}\succeq\epsilon I$ for some $\epsilon>0$. To enforce this, we let $\epsilon=10^{-6}$, and ensure strict positivity of $\mcl{P}$ by calling
\begin{matlab}
\begin{verbatim}
 >> eppos = 1e-6;
 >> P.R.R0 = P.R.R0 + eppos*eye(size(P));
\end{verbatim}
\end{matlab}
\end{Statebox}

\subsection{\texttt{lpivar}}\label{sec:lpivar}
A general (indefinite) PI operator decision variable $\mcl{Z}$ can be declared in PIETOOLS using the \texttt{lpivar} function as shown below.
\begin{matlab}
 >> [prog,Z] = lpivar(prog,n,I,d);
\end{matlab}
This function takes three mandatory inputs.
\begin{itemize}
	\item \texttt{prog}: An sosprogram structure to which to add the PI operator decision variable.
	\item \texttt{n}: a $2\times 2$ array \texttt{[m0,n0;m1,n1]} specifying the dimensions of $\mcl{Z}:\sbmat{\R^{n_0}\\L_2^{n_1}[a,b]}\rightarrow\sbmat{\R^{m_0}\\L_2^{m_1}[a,b]}$ for a 1D operator, or $4\times 2$ array \texttt{[m0,n0;m1,n1;m2,n2;m3,n3]} specifying the dimensions for a 2D operator $\mcl{Z}:\sbmat{\R^{n_0}\\L_2^{n_1}[a,b]\\L_2^{n_2}[c,d]\\L_2^{n_3}[[a,b]\times[c,d]]}\rightarrow \sbmat{\R^{m_0}\\L_2^{m_1}[a,b]\\L_2^{m_2}[c,d]\\L_2^{m_3}[[a,b]\times[c,d]]}$.
	\item \texttt{I}: a $1\times 2$ array \texttt{[a,b]} specifying the interval $[a,b]$ of the variables in \texttt{Z} in 1D, or a $2\times 2$ array \texttt{[a,b;c,d]} specifying the spatial domain $[a,b]\times[c,d]$ of the variables in \texttt{Z} in 2D.
	\item \texttt{d} (optional):
    \begin{itemize}
	    \item 1D: An array  structure of the form $\texttt{[b0,b1,b2]}$, specifying the degree \texttt{b0} of $s$ in \texttt{Q1, Q2, R0}, the degree \texttt{b1} of $\theta$ in \texttt{Z.R.R1, Z.R.R2}, and the degree \texttt{b2} of $s$ in \texttt{Z.R.R1, Z.R.R2}.
        \item 2D: A structure with fields  \texttt{d.dx,d.dy,d.d2}, specifying degrees for operators along $x\in[a,b]$, along $y\in[c,d]$, and along both $(x,y)\in[a,b]\times[c,d]$; call \texttt{help poslpivar\_2d} for more information.
	\end{itemize}
\end{itemize}
The output is a \texttt{dopvar} object \texttt{Z} representing an indefinite PI decision variable, and an updated program structure to which the decision variable has been added. 

\begin{boxEnv}{\textbf{Note:}}
Indefinite PI operator decision variables $\mcl{Z}$ need not be symmetric. As such, the second argument \texttt{n} to the function \texttt{lpivar} must specify both the output (row) dimensions of the operator $\mcl{Z}$ (\texttt{n(:,1)}), and the input (column) dimensions of the operator $\mcl{Z}$ (\texttt{n(:,2)}).
\end{boxEnv}

\begin{Statebox}{\textbf{Example}}

For the LPI~\eqref{eq:LPIs:LPI_example}, we need to declare a PI operator decision variable $\mcl{Z}$ which need not be positive or negative definite. Here, if $\mcl{C}_{2}:\sbmat{\R^{n_0}\\L_2^{n_1}[a,b]}\rightarrow \sbmat{\R^{m_0}\\L_2^{m_1}[a,b]}$, then $\mcl{Z}:\sbmat{\R^{m_0}\\L_2^{m_1}[a,b]}\rightarrow \sbmat{\R^{n_0}\\L_2^{n_1}[a,b]}$, so we specify the dimensions and domain of $\mcl{Z}$ as
\begin{matlab}
\begin{verbatim}
 >> Zdim = C2.dim(:,[2,1])
 Zdim =

      0     1
      1     0
      
 >> Zdom = PIE.dom;
\end{verbatim}
\end{matlab}
so that in our case $\mcl{Z}:\R\rightarrow L_2[0,1]$. As such, only the parameter \texttt{Z.Q2} will be non-empty, and we will allow this parameter to be a quartic function of $s$, specifying degrees
\begin{matlab}
\begin{verbatim}
 >> Zdeg = [4,0,0];
\end{verbatim}
\end{matlab}
Using the function \texttt{lpivar}, we then declare an indefinite PI operator decision variable as
\begin{matlab}
\begin{verbatim}
 >> [prog,Z] = lpivar(prog,Zdim,Zdom,Zdeg)
 prog = 

   struct with fields:

             var: [1×1 struct]
            expr: [1×1 struct]
        extravar: [1×1 struct]
       objective: [367×1 double]
         solinfo: [1×1 struct]
        vartable: {2×1 cell}
          varmat: [1×1 struct]
     decvartable: {367×1 cell}
 
 Z =
                                                                     [] | [] 
      -----------------------------------------------------------------------
      [coeff_362+coeff_363*s+coeff_364*s^2+coeff_365*s^3+coeff_366*s^4] | Z.R 

 Z.R =
     [] | [] | [] 
\end{verbatim}
\end{matlab}
where we note that \texttt{Z.Q2} is indeed a quartic function of $s$. The resulting \texttt{dopvar} object \texttt{Z} only involves five decision variables \texttt{coeff}, increasing the total number of decision variables in \texttt{prog.decvartable} to 367.
    
\end{Statebox}

\section{Imposing Constraints}\label{sec:LPIs:constraints}
Constraints form a crucial aspect of most optimization problems. In LPIs, constraints often appear as inequality or equality conditions on 4-PI objects (e.g. $\mcl{Q}\preceq 0$ or $\mcl{Q}=0$). These constraints can be set up using the functions \texttt{lpi\_ineq} and \texttt{lpi\_eq} respectively, as we show in the next subsections.
	
\subsection{\texttt{lpi\_ineq}}\label{sec:ineq_pi}
Given a program structure \texttt{prog} and \texttt{dopvar} object \texttt{Q} (representing a PI operator variable $\mcl{Q}$), an inequality constraint $\mcl{Q}\succeq 0$ can be added to the program by calling
\begin{matlab}
 >> prog = lpi\_ineq(prog,Q,opts);
\end{matlab}
This function takes three input arguments
\begin{itemize}
	\item \texttt{prog}: An sosprogram structure to which to add the constraint $\mcl{Q}\succeq 0$.
	\item \texttt{Q}: A \texttt{dopvar} or \texttt{dopvar2d} object representing the PI operator $\mcl{Q}$. 
	\item \texttt{opts} (optional): This is a structure with fields 
	\begin{itemize}
		\item \texttt{psatz}: Binary scalar indicating whether to enforce the constraint only locally. If \texttt{psatz=0} (default), a constraint $\mcl{Q}=\mcl{R}_1$ will be enforced, where $\mcl{R}\succeq 0$ will be declared as a \texttt{dopvar} or \texttt{dopvar2d} object
        \begin{matlab}
         >> R1=poslpivar(prog,n,I,d,opts1)
        \end{matlab}
        with \texttt{opts1.psatz=0}. If \texttt{psatz=1}, a constraint $\mcl{Q}=\mcl{R}_1+\mcl{R}_2$ will be enforced, with $\mcl{R}_1$ as before, and $\mcl{R}_2\succeq 0$ declared as a \texttt{dopvar} or \texttt{dopvar2d} object 
        \begin{matlab}
         >> R2=poslpivar(prog,n,I,d,opts2)
        \end{matlab}
        with \texttt{opts2.psatz=0}. Using \texttt{psatz=1} allows the constraint $\mcl{Q}\succeq 0$ to be violated outside of the spatial domain \texttt{I}, but may also substantially increase the computational effort in solving the problem.
  \end{itemize}
\end{itemize}
Calling \texttt{lpi\_ineq}, a modified optimization structure \texttt{prog} is returned with the constraints \texttt{Q>=0} included. Note that the constraint imposed by \texttt{lpi\_ineq} is always \textbf{non-strict}. For strict positivity, an offset $\epsilon>0$ may be introduced, enforcing $\mcl{Q}-\epsilon\succeq 0$ to ensure $\mcl{Q}\succeq\epsilon I\succ 0$. This may be implemented as \texttt{lpi\_ineq(prog,Q-ep)} where \texttt{ep} is a small positive number.

\begin{Statebox}{\textbf{Example}}
For the LPI~\eqref{eq:LPIs:LPI_example}, we impose the constraint $\mcl{Q}\preccurlyeq 0$ by calling
\begin{matlab}
\begin{verbatim}
 >> nw = size(B1,2);       nz = size(C1,1);
 >> Q = [-gam*eye(nw),     -D11',        -(P*B1+Z*D21)'*T;
         -D11,             -gam*eye(nz), C1;
        -T'*(P*B1+Z*D21),   C1',         (P*A+Z*C2)'*T+T'*(P*A+Z*C2)];
 >> prog = lpi_ineq(prog,-Q);
 prog = 
 
   struct with fields:

             var: [1×1 struct]
            expr: [1×1 struct]
        extravar: [1×1 struct]
       objective: [20248×1 double]
         solinfo: [1×1 struct]
        vartable: {2×1 cell}
          varmat: [1×1 struct]
     decvartable: {20248×1 cell}
\end{verbatim}
\end{matlab}
We note that \texttt{lpi\_ineq} enforce the constraint $-\mcl{Q}\succcurlyeq 0$ by introducing a new PI operator decision variable $\mcl{R}\succcurlyeq 0$, and imposing the equality constraint $-\mcl{Q}=\mcl{R}$. In doing so, \texttt{lpi\_ineq} tries to ensure that the degrees of the polynomial parameters defining $\mcl{R}$ match those of the parameters defining $-\mcl{Q}$, in this case parameterizing $\mcl{R}$ by $20248-367 = 19881$ decision variables (check the size of \texttt{decvartable}). An operator $\mcl{R}$ involving more or fewer decision variables can also be declared manually, at  which point the constraint $-\mcl{Q}=\mcl{R}$ can be enforced using \texttt{lpi\_eq}, as we show next.

\end{Statebox}

\subsection{\texttt{lpi\_eq}}\label{sec:eq_pi}
Given a program structure \texttt{prog} and \texttt{dopvar} object \texttt{Q} (representing a PI operator decision variable $\mcl{Q}$), an equality constraint $\mcl{Q}=0$ can be added to the program by calling
\begin{matlab}
 >> prog = lpi\_eq(prog,Q);
\end{matlab}
This returns a modified optimization structure \texttt{prog} with the constraints \texttt{Q=0} included. Note that equality constraints such as $\mcl{Q}_1=\mcl{Q}_2$ may be equivalently written as e.g $\mcl{Q}_1-\mcl{Q}_2=0$, which may be implemented as \texttt{prog = lpi\_eq(prog,Q-R)}. 

\begin{Statebox}{\textbf{Example}}
 For the LPI~\eqref{eq:LPIs:LPI_example}, we can enforce the constraint $\mcl{Q}\precceq 0$ by declaring a new positive semidefinite PI operator decision variable $\mcl{R}\succeq 0$, and enforcing $\mcl{Q}=-\mcl{R}\preccurlyeq 0$ as
\begin{matlab}
\begin{verbatim}
 >> nw = size(B1,2);       nz = size(C1,1);
 >> Q = [-gam*eye(nw),     -D11',        -(P*B1+Z*D21)'*T;
         -D11,             -gam*eye(nz), C1;
        -T'*(P*B1+Z*D21),   C1',         (P*A+Z*C2)'*T+T'*(P*A+Z*C2)];
 >> [prog_alt,R] = poslpivar(prog,Q.dim(:,1),Q.I);
 >> prog_alt = lpi_eq(prog_alt,Q+R)
 prog_alt = 

   struct with fields:

             var: [1×1 struct]
            expr: [1×1 struct]
        extravar: [1×1 struct]
       objective: [467×1 double]
         solinfo: [1×1 struct]
        vartable: {2×1 cell}
          varmat: [1×1 struct]
     decvartable: {467×1 cell}
\end{verbatim}
\end{matlab}
The new program structure will include the constraints $\mcl{R}\succcurlyeq 0$ and $\mcl{Q}+\mcl{R}=0$. The resulting number of decision variables (467) is substantially smaller than in the program obtained using \texttt{lpi\_ineq}, as the default degrees used by \texttt{poslpivar} are relatively small. Specifying higher degrees \texttt{d} in declaring \texttt{R = poslpivar(prog,Q.dim(:,1),Q.I,d)}, we can increase the freedom in the optimization problem, potentially increasing accuracy but also increasing computational complexity.
\end{Statebox}

\section{Defining an Objective Functions}\label{sec:LPIs:obj_fun}
Aside from constraints, many optimization problems also involve an objective function, aiming to minimize or maximize some function of the decision variables. To \textbf{minimize} the value of a \textbf{linear} objective function $f(d_1,\hdots,d_1)$, where $d_1,\hdots,d_n$ are decision variables, call \texttt{sossetobj} with as first argument the program structure, and as second argument the function $f(d)$:
\begin{matlab}
\begin{verbatim}
 >> prog = sossetobj(prog, f);
\end{verbatim}
\end{matlab}
where \texttt{f} must be a \texttt{dpvar} object representing the objective function. For example, a function $f(\gamma_1,\gamma_2)=\gamma_1+5\gamma_2$ could be specified as objective function using
\begin{matlab}
\begin{verbatim}
 >> dpvar gam1 gam2;
 >> f = gam1 + 5*gam2;
 >> prog = sosdecvar(prog,[gam1;gam2]);
 >> prog = sossetobj(prog, f);
\end{verbatim}
\end{matlab}
\begin{boxEnv}{\textbf{Note:}}
    \begin{itemize}
        \item The objective function must always be linear in the decision variables.
        \item Only one (scalar) objective function can be specified.
        \item In solving the optimization program, the value of the objective function will always be minimized. Thus, to maximize the value of the objective function $f(d)$, specify $-f(d)$ as objective function. 
    \end{itemize}
\end{boxEnv}

\begin{Statebox}{\textbf{Example}}
 In the LPI~\eqref{eq:LPIs:LPI_example}, the value of the decision variable $\gamma$ is minimized. As such, the objective function in this problem is just $f(\gamma)=\gamma$, which we declare as
\begin{matlab}
\begin{verbatim}
 >> prog = sossetobj(prog, gam);
 prog = 

   struct with fields:

             var: [1×1 struct]
            expr: [1×1 struct]
        extravar: [1×1 struct]
       objective: [23776×1 double]
         solinfo: [1×1 struct]
        vartable: {2×1 cell}
          varmat: [1×1 struct]
     decvartable: {23776×1 cell}
\end{verbatim} 
\end{matlab}
The field \texttt{objective} in the resulting structure \texttt{prog} will be a vector with all elements equal to zero, except the first element equal to 1, specifying that the objective function is equal to 1 times the first decision variable in \texttt{decvartable}, which is $\gamma$.
\end{Statebox}

\section{Solving the Optimization Problem}\label{sec:LPIs:solve}

Once an optimization program has been specified as a program structure \texttt{prog}, it can be solved by calling \texttt{sossolve}
\begin{matlab}
\begin{verbatim}
 >> prog_sol = sossolve(prog,opts);
\end{verbatim}
\end{matlab}
Here \texttt{opts} is an optional argument to specify settings in solving the optimization program, with fields
\begin{itemize}
    \item \texttt{opts.solve}: \texttt{char} type object specifying which semidefinite programming (SDP) solver to use. Options include `\texttt{sedumi}' (default), `\texttt{mosek}', `\texttt{sdpt3}', and `\texttt{sdpnalplus}'. Note that these solvers must be separately installed in order to use them.

    \item \texttt{opts.simplify}: Binary value indicating whether the solver should attempts to simplify the SDP before solving. The simplification process may take additional time, but may reduce the time of actually solving the SDP.
\end{itemize}
After calling \texttt{sossolve}, it is important to check whether the problem was actually solved. Using the solver SeDuMi, this can be established looking at e.g. the value of \texttt{feasratio}, which will be close to $+1$ if the problem was successfully solved, and the values of \texttt{pinf} and \texttt{dinf}, which should both be zero if the problem is primal and dual solvable. If \texttt{sossolve} is unsuccessful in solving the problem, the problem may be infeasible, or different settings must be used in declaring the decision variables and constraints (e.g. include higher degree monomials in the positive operators).

\begin{Statebox}{\textbf{Example}}
Having declared the full LPI~\eqref{eq:LPIs:LPI_example}, we can finally solve the problem as
\begin{matlab}
\begin{verbatim}
 >> opts.solver = 'sedumi';
 >> opts.simplify = true;
 >> prog_sol = sossolve(prog,opts);

 Residual norm: 1.1879e-05
 
          iter: 19
     feasratio: 0.8308
          pinf: 0
          dinf: 0
        numerr: 1
        timing: [0.1475 5.4485 0.0113]
       wallsec: 5.6073
        cpusec: 6.3700
\end{verbatim}
\end{matlab}
We note that the problem was not found to be either primal or dual infeasible, and the value of \texttt{feasratio} is fairly close to 1. However, the solver did run into numerical errors, as signified by \texttt{numerr: 1}. Nevertheless, we will refrain from performing an additional test (e.g. rerunning with greater degrees for the monomials defining $\mcl{P}$) for the purposes of this demonstration.

\end{Statebox}

\section{Extracting the Solution}\label{sec:LPIs:getsol}

Calling \texttt{prog\_sol=sossolve(prog)}, a program structure \texttt{prog\_sol} is returned that is very similar to the input structure \texttt{prog}, defining the solved optimization problem. From this solved structure, solved values of the decision variables can be extracted using \texttt{sosgetsol} and \texttt{getsol\_lpivar}, as we show next.

\subsection{\texttt{sosgetol}}\label{sec:sosgetsol}
To obtain the (optimal) value of a decision variable after an LPI optimization program has been solved, the function \texttt{sosgetsol} can be used, passing the solved optimization program structure as first argument, and the considered decision variable as second argument:
\begin{matlab}
\begin{verbatim}
 >> f_val = sosgetsol(prog_sol,f);
\end{verbatim}
\end{matlab}
Here, \texttt{f} must be a \texttt{dpvar} object, representing any polynomail function that is affine in the decision variables that appear in the optimization program. 

\begin{Statebox}

To extract the optimal value of $\gamma$ found when solving the LPI~\eqref{eq:LPIs:LPI_example}, we call
\begin{matlab}
\begin{verbatim}
 >> gam_val = sosgetsol(prog_sol,gam)
 gam_val = 
   1.0028
\end{verbatim}
\end{matlab}
We find that, through proper choice of the operator $\mcl{L}$, the estimator can achieve an $H_{\infty}$-norm $\frac{\|\tilde{z}\|_{L_2}}{\|w\|_{L_2}}\leq 1.0028$.
    
\end{Statebox}

\subsection{\texttt{getsol\_lpivar}}\label{sec:getsol_lpivar}
To retrieve the values of \texttt{dopvar} decision variables after an optimization program has been solved, the function \texttt{getsol\_lpivar} can be used, passing the solved optimization program structure \texttt{prog\_sol} as first argument, and the considered \texttt{dopvar} decision variable as second argument:
\begin{matlab}
\begin{verbatim}
 >> Pop_val = sosgetsol(prog_sol,Pop);
\end{verbatim}
\end{matlab}
Note that sosprogram \texttt{prog\_sol} must be in a solved state (\texttt{sossolve} must have been called) to retrieve the solution for the input \texttt{dopvar} or \texttt{dopvar2d} object \texttt{Pop}. The output \texttt{Pop\_val} will then be \texttt{opvar} or \texttt{opvar2d} class object, representing a fixed PI operator, with the solved (optimal) values of the decision variables substituted into the associated parameters.

\begin{Statebox}{\textbf{Example}}
To extract the values of the operators $\mcl{P}$ and $\mcl{Z}$ for which the LPI~\eqref{eq:LPIs:LPI_example} has been found feasible, we call
\begin{matlab}
\begin{verbatim}
 >> Pval = getsol_lpivar(prog_sol,P);  
 >> Zval = getsol_lpivar(prog_sol,Z);
\end{verbatim}
\end{matlab}    
The resulting object \texttt{Pval} and \texttt{Zval} are \texttt{opvar} objects, representing the values of the operator $\mcl{P}$ and $\mcl{Z}$ for which the LPI~\eqref{eq:LPIs:LPI_example} holds. Using these values, we can compute an operator $\mcl{L}$ such that the Estimator~\eqref{eq:LPIs:PIE_example} satisfies $\frac{\|\tilde{z}\|_{L_2}}{\|w\|_{L_2}}\leq \gamma=1.0028$, as we show next.
\end{Statebox}

When performing estimator or controller synthesis (see also Chapter~\ref{ch:LPI_examples}), the optimal estimator or controller associated to a solved problem has to be constructed from the solved PI operator decision variables. For example, for the estimator in~\eqref{eq:LPIs:PIE_example}, the value of $\mcl{L}$ is determined by the values of $\mcl{P}$ and $\mcl{Z}$ in the LPI~\eqref{eq:LPIs:LPI_example} as $\mcl{L}=\mcl{P}^{-1}\mcl{Z}$. To facilitate this post-processing of the solution, PIETOOLS includes several utility functions. We note that these functions have not been included for 2D operators in PIETOOLS 2022.

\subsubsection{\underline{\texttt{getObserver}}}
For a solution $(\mcl{P},\mcl{Z})$ to the optimal estimator LPI~\eqref{eq:LPIs:LPI_example}, the operator $\mcl{L}$ in the Estimator~\eqref{eq:LPIs:PIE_example} can be computed as
\begin{matlab}
\begin{verbatim}
 >> Lval = getObserver(Pval,Zval);
\end{verbatim}
\end{matlab}
where \texttt{Pval} and \texttt{Zval} are \texttt{opvar} objects representing the (optimal) values of $\mcl{P}$ and $\mcl{Z}$ in the LPI, and \texttt{Lval} is an \texttt{opvar} object representing the associated (optimal) value of $\mcl{L}$ in the estimator.

\begin{Statebox}{\textbf{Example}}
Given the \texttt{opvar} objects \texttt{Pval} and \texttt{Zval}, we can finally construct an optimal observer operator $\mcl{L}$ for the System~\eqref{eq:LPIs:PIE_example} by calling
\begin{matlab}
\begin{verbatim}
 >> Lval = getObserver(Pval,Zval)
 Lval =
                                [] | [] 
    ----------------------------------------
    Too big to display. Use ans.Q2 | Lval.R 

 Lval.R =
     [] | [] | [] 
\end{verbatim}
\end{matlab}
where the expression for \texttt{Lval.Q2} is rather complicated. Nevertheless, using this value for the operator $\mcl{L}$, an $H_{\infty}$ norm $\frac{\|\tilde{z}\|_{L_2}}{\|w\|_{L_2}}\leq \gamma=1.0028$ can be achieved. Increasing the freedom in our optimization problem, e.g. by increasing the degrees of the monomaials defining $\mcl{P}$ and $\mcl{Z}$, we may be able to achieve a tighter bound on the $H_{\infty}$ norm for the obtained operator $\mcl{L}$, or find another operator $\mcl{L}$ achieving a smaller value of the norm.
\end{Statebox}

\subsubsection{\underline{\texttt{getController}}}
For a solution $(\mcl{P},\mcl{Z})$ to the optimal control LPI~\eqref{eq:cont_lpi} (see Section~\ref{sec:LPI_examples:control}, the operator $\mcl{K}=\mcl{Z}\mcl{P}^{-1}$ defining the feedback law $u=\mcl{K}\mbf{v}$ for optimal control of the PIE~\eqref{eq:LPI_examples:control_PIE} can be computed as
\begin{matlab}
\begin{verbatim}
 >> Kval = getController(Pval,Zval);
\end{verbatim}
\end{matlab}
where \texttt{Pval} and \texttt{Zval} are \texttt{opvar} objects representing the (optimal) values of $\mcl{P}$ and $\mcl{Z}$ in the LPI, and \texttt{Kval} is an \texttt{opvar} object representing the associated (optimal) value of $\mcl{K}$ in the feedback law $u=\mcl{K}\mbf{v}$. Note that this feedback law is described in terms of the PIE state $\mbf{v}$, not the state of the associated PDE or TDS. Deriving an optimal controller for the associated PDE or TDS system will require careful consideration of how the PIE state relates to the PDE or TDS state.

\subsubsection{\underline{\texttt{closedLoopPIE}}}
For a PIE~\eqref{eq:LPI_examples:control_PIE} and an operator $\mcl{K}$ defining a feedback law $u=\mcl{K}\mbf{v}$, a PIE corresponding to the closed-loop system for the given feedback law can be computed as
\begin{matlab}
\begin{verbatim}
 >> PIE_CL = getController(PIE_OL,Kval);
\end{verbatim}
\end{matlab}
where \texttt{Kval} is an \texttt{opvar} object representing the (optimal) value of $\mcl{K}$ in the feedback law $u=\mcl{K}\mbf{v}$, \texttt{PIE\_OL} is a \texttt{pie\_struct} object representing the PIE system without feedback, and \texttt{PIE\_CL} is a \texttt{pie\_struct} object representing the closed-loop PIE system with the feedback law $u=\mcl{K}\mbf{v}$ enforced. Note that the resulting system takes no more actuator inputs $u$, so that operators \texttt{PIE\_CL.Tu}, \texttt{PIE\_CL.B2}, \texttt{PIE\_CL.D12}, and \texttt{PIE\_CL.D22} are all empty.

\begin{Statebox}{\textbf{Example}}
 A full code declaring and solving the optimal estimator LPI~\eqref{eq:LPIs:LPI_example} for the PDE~\eqref{eq:LPIs:PDE_example} has been included as a demo file ``Hinf\_optimal\_observer\_DEMO'' in PIETOOLS. See Section~\ref{sec:demos:estimator} for more information. Section~\ref{sec:demos:estimator} also shows how the estimated PDE state associated to the obtained operator $\mcl{L}$ can be simulated with PIESIM.
\end{Statebox}

\section{Running Pre-Defined LPIs: Executives and Settings}\label{sec:executives-settings}

Combining the steps from the previous sections, we find that the $H_{\infty}$-optimal estimator LPI~\eqref{eq:LPIs:LPI_example} can be declared and solved for any given PIE structure \texttt{PIE} using roughly the same code. Therefore, to facilitate solving the $H_{\infty}$-optimal estimator problem, the code has been implemented in an \texttt{executive} file \texttt{PIETOOLS\_Hinf\_estimator}, that is structured roughly as
\begin{matlab}
\begin{verbatim}
 >> % Extract the operators and initialize the LPI program
 >> T = PIE.T;     A = PIE.A;       B1 = PIE.B1;
 >> C1 = PIE.C1;   D11 = PIE.D11;   C2 = PIE.C2;   D12 = PIE.D12;
 >> prog = sosprogram([PIE.vars(:,1);PIE.vars(:,2)]);

 >> % Declare the objective function min{gamma}
 >> dpvar gam;
 >> prog = sosdecvar(prog, gam);
 >> prog = sossetobj(prog, gam);

 >> % Declare the positive operator P>=0
 >> [prog,P] = poslpivar(prog,T.dim(:,1),PIE.dom,dd1,options1);
 >> if override1==0
 >>     % Allow P<=0 outside domain PIE.dom
 >>     [prog,P2] = poslpivar(prog,T.dim(:,1),PIE.dom,dd12,options12);
 >>     P = P + P2;
 >> end
 >> % Enforce strict positivity P>0
 >> P.P = P.P + eppos*eye(size(P.P));
 >> P.R.R0 = eppos2*eye(size(P.R.R0));
 
 >> % Declare the indefinite operator Z
 >> [prog,Z] = lpivar(prog,C2.dim(:,[2,1]),PIE.dom,ddZ);

 >> % Enforce the negativity constraint Q<=0
 >> Q = [-gam*eye(nw),     -D11’,        -(P*B1+Z*D21)’*T;
         -D11,             -gam*eye(nz),  C1;
         -T’*(P*B1+Z*D21),  C1’,          (P*A+Z*C2)’*T+T’*(P*A+Z*C2) + epneg*T'*T];
 >> if sosineq
 >>     % Enforce using lpi_ineq
 >>     prog = lpi_ineq(prog,-Q,opts);
 >> else
 >>     % Enforce using lpi_eq
 >>     [prog,R] = poslpivar(prog,Q.dim(:,1),Q.I,dd2,options2);
 >>     if override2==0
 >>         % Allow R<=0 outside of domain Q.I
 >>         [prog,R2] = poslpivar(prog,Q.dim(:,1),Q.I,dd3,options3);
 >>         R = R+R2;
 >>     end
 >>     % Enforce Q=-R<=0
 >>     prog = lpi_eq(prog,Q+R);
 >> end
 
 >> % Solve the optimization program and extract the solution
 >> prog_sol = sossolve(prog, sos_opts);
 >> gam_val = sosgetsol(prog_sol, gam);
 >> Pval = getsol_lpivar(prog_sol, P);     
 >> Zval = getsol_lpivar(prog_sol, Z);
 >> Lval = getObserver(Pval, Zval);
\end{verbatim}
\end{matlab}
We note that, in this code, that there are several parameters that can be set, including what degrees to use for the PI operator decision variables (\texttt{dd1}, \texttt{ddZ}, \texttt{dd2}, etc.), whether or not to enforce positivity/negativity strictly and/or locally (\texttt{eppos}, \texttt{epneg}, \texttt{override1}, etc.), and what options to use in calling each of the different functions (\texttt{options1}, \texttt{opts}, \texttt{sos\_opts}, etc.). To specify each of the options, the executive files such as \texttt{PIETOOLS\_Hinf\_estimator} can be called with an optional argument \texttt{settings}, as we describe in the following subsection.

\subsection{Settings in PIETOOLS Executives}

When calling an executive function such as \texttt{PIETOOLS\_Hinf\_estimator} in PIETOOLS, a second (optional) argument can be used to specify settings to use in declaring the LPI program. This argument should be a MATLAB structure with fields as defined in Table~\ref{tab:executive_settings_fields}, specifying a value for each of the different options to be used in declaring the LPI program.

\begin{table}[ht]
\begin{tabular}{l|l}
 \texttt{settings} Field  &  \textbf{Application} \\\hline
 \texttt{eppos} & Nonnegative (small) scalar to enforce strict positivity of \texttt{Pop.P}\\
 \texttt{eppos2} & Nonnegative (small) scalar to enforce strict positivity of \texttt{Pop.R0}\\
 \texttt{epneg} & Nonnegative (small) scalar to enforce strict negativity of \texttt{Dop}\\
 \texttt{sosineq} & Binary value, set 1 to use \texttt{sosineq}\\
 \texttt{override1} & Binary value, set 1 to let $\texttt{P2op}=0$\\
 \texttt{override2} & Binary value, set 1 to let $\texttt{De2op}=0$\\
 \texttt{dd1} & 1x3 cell structure defining monomial degrees for \texttt{Pop}\\
 \texttt{dd12} & 1x3 cell structure defining monomial degrees for \texttt{P2op}\\
 \texttt{dd2} & 1x3 cell structure defining monomial degrees for \texttt{Deop}\\
 \texttt{dd3} & 1x3 cell structure defining monomial degrees for \texttt{De2op}\\
 \texttt{ddZ} & 1x3 array defining monomial degrees for \texttt{Zop}\\
 \texttt{options1} & Structure of \texttt{poslpivar} options for \texttt{Pop}\\
 \texttt{options12} & Structure of \texttt{poslpivar} options for \texttt{P2op}\\
 \texttt{options2} & Structure of \texttt{poslpivar} options for \texttt{Deop}\\
 \texttt{options3} & Structure of \texttt{poslpivar} options for \texttt{De2op}\\
 \texttt{opts} & Structure of \texttt{lpi\_ineq} options for enforcing $\texttt{Dop}\leq0$\\
 \texttt{sos\_opts} & Structure of \texttt{sossolve} options for solving the LPI
\end{tabular}
\caption{Fields of settings structure passed on to PIETOOLS executive files}
\label{tab:executive_settings_fields}
\end{table}

To help in declaring settings for the executive files, PIETOOLS includes several pre-defined \texttt{settings} structures, allowing LPI programs of varying complexity to be constructed. In particular, we distinguish \texttt{extreme}, \texttt{stripped}, \texttt{light}, \texttt{heavy} and \texttt{veryheavy} settings, corresponding to LPI programs of increasing complexity. A \texttt{settings} structure associated to each can be extracted by calling the function \texttt{lpisettings}, using
\begin{matlab}
\begin{verbatim}
 >> settings = lpisettings(complexity, epneg, simplify, solver);
\end{verbatim}
\end{matlab}
This function takes the following arguments:
\begin{itemize}
    \item \texttt{complexity}: A \texttt{char} object specifying the complexity for the settings. Can be one of `\texttt{extreme}', `\texttt{stripped}', `\texttt{light}', `\texttt{heavy}', `\texttt{veryheavy}' or `\texttt{custom}'.

    \item epneg: (optional) Positive scalar $\epsilon$ indicating how strict the negativity condition $Q\precceq \epsilon\|\mcl{T}\|^2$ would need to be in e.g. the LPI for stability. Defaults to 0, enforcing $Q\precceq 0$.

    \item simplify: (optional) A \texttt{char} object set to `\texttt{psimplify}' if the user wishes to simplify the SDP produced in the executive before solving it, or set to `\texttt{}' if not. Defaults to `\texttt{}'.

    \item solver: (optional) A \texttt{char} object specifying which solver to use to solve the SDP in the executive. Options include `\texttt{sedumi}' (default), `\texttt{mosek}', `\texttt{sdpt3}', and `\texttt{sdpnalplus}'. Note that these solvers must be separately installed in order to use them.
\end{itemize}
Note that, using higher-complexity settings, the number of decision variables in the optimization problem will be greater. This offers more freedom in solving the optimization program, thereby allowing for (but not guaranteeing) more accurate results, but also (substantially) increasing the computational effort. We therefore recommend initially trying to solve with e.g. \texttt{stripped} or \texttt{light} settings, and only using heavier settings if the executive fails to solve the problem. Note also that PIETOOLS includes a \texttt{custom} settings file, which can be used to declare custom settings for the executives. 

Once settings have been specified, the desired LPI can be declared and solved for a PIE represented by a structure \texttt{PIE} by simply calling the corresponding \texttt{executive} file, solving e.g. the $H_\infty$-optimal estimator LPI~\eqref{eq:LPIs:LPI_example} by calling
\begin{matlab}
\begin{verbatim}
 >> settings = lpisettings('light');
 >> [prog, Lop, gam, P, Z] = PIETOOLS_Hinf_estimator(PIE, settings);
\end{verbatim}
\end{matlab}
If successful, this returns the program structure \texttt{prog} associated to the solved problem, as well as a \texttt{opvar} object \texttt{Lop} and scalar \texttt{gam} respectively corresponding to the operator $\mcl{L}$ in the estimator~\eqref{eq:LPIs:PIE_example} and associated estimation error gain $\gamma$. The function also returns \texttt{dopvar} objects \texttt{P} and \texttt{Z}, corresponding to the unsolved PI operator decision variables in the LPI.

\subsection{Executive Functions Available in PIETOOLS}

In addition to the $H_{\infty}$-optimal estimation LPI, PIETOOLS includes several other executive files to run standard LPI programming tests for a provided PIE. For example, stability of the PIE defined by \texttt{PIE} when $w=0$ can be tested by calling
\begin{matlab}
\begin{verbatim}
 >> settings = lpisettings('light');
 >> [prog] = PIETOOLS_stability(PIE, settings);
\end{verbatim}
\end{matlab}
returning the optimization program structure \texttt{prog} associated to the solved program, and displaying a message of whether the system was found to be stable or not in the command window.

Table~\ref{tab:executive_functions} lists the different executive functions that have already been implemented in PIETOOLS. For each executive, a brief description of its purpose is provided, along with a mathematical description of the LPI that is solved. Each executive can be called for a \texttt{PIE} structure with fields
\begin{matlab}
\begin{verbatim}
       T:  [nx × nx opvar];     Tw: [nx × nw opvar];     Tu: [nx × nu opvar]; 
       A:  [nx × nx opvar];     B1: [nx × nw opvar];     B2: [nx × nu opvar]; 
       C1: [nz × nx opvar];    D11: [nz × nw opvar];    D12: [nz × nu opvar]; 
       C2: [ny × nx opvar];    D21: [ny × nw opvar];    D22: [ny × nu opvar]; 
\end{verbatim}
\end{matlab}
representing a PIE of the form
\begin{align*}
    \mcl{T}_u\dot{u}(t)+\mcl{T}_w\dot{w}+\mcl{T}\dot{\mbf{x}}_{\text{f}}(t)&=\mcl{A}\mbf{x}_{\text{f}}(t)+\mcl{B}_1 w(t)+\mcl{B}_2 u(t),   \nonumber\\
    z(t)&=\mcl{C}_1\mbf{x}_{\text{f}}(t) + \mcl{D}_{11}w(t) + \mcl{D}_{12}u(t), \nonumber\\
    y(t)&=\mcl{C}_2\mbf{x}_{\text{f}}(t) + \mcl{D}_{21}w(t) + \mcl{D}_{22}u(t).
\end{align*}
For more information on the origin and application of each LPI, see the references provided in the table, as well as Chapter~\ref{ch:LPI_examples}.

\begin{table}[ht]
\hspace*{-1.0cm}
\setlength{\abovedisplayskip}{-6pt}
\setlength{\belowdisplayskip}{-2pt}
\begin{tabular}{m{6.5cm} | m{10cm}}
    \textbf{Problem} & \textbf{LPI} \\
    \hline\hline
    \multicolumn{2}{l}{\texttt{[prog, P] = PIETOOLS\_stability(PIE, settings)}} \\\hline
    Test stability of the PIE for $w=0$ and $u=0$, by verifying feasibility of the primal LPI~\cite{shivakumar_2019CDC}.  &
    {\begin{flalign*} 
    &   \mcl{P}\succ 0 & \\
	& \mcl{T}^*\mcl{P}\mcl{A}+\mcl{A}^*\mcl{P}\mcl{T}\preccurlyeq 0&
    \end{flalign*} } \\
    \hline\hline
    \multicolumn{2}{l}{\texttt{[prog, P] = PIETOOLS\_stability\_dual(PIE, settings)}} \\\hline
    Test stability of the PIE for $w=0$ and $u=0$ by verifying feasibility of the dual LPI~\cite{shivakumar_2020CDC}.         & 
    {\begin{flalign*}
	&\mcl{P}\succ0&\\
	&\mcl{T}\mcl{P}\mcl{A}^*+\mcl{A}\mcl{P}\mcl{T}^*\preccurlyeq 0&
	\end{flalign*}}\\
    \hline\hline
    \multicolumn{2}{l}{\texttt{[prog, P, gam] = PIETOOLS\_Hinf\_gain(PIE, settings)}} \\\hline
    Determine an upper bound $\gamma$ on the $\mcl{H}_{\infty}$-norm $\sup_{w,z\in L_2}\frac{\|z\|_{L_2}}{\|w\|_{L_2}}$ of the PIE for $u=0$, by solving the primal LPI~\cite{shivakumar_2019CDC}. &
    {\begin{flalign*}
    &\min\limits_{\gamma,\mcl{P}} ~~\gamma&\\
	&\mcl{P}\succ0&\\
	&\bmat{-\gamma I & \mcl{D}_{11}^*&\mcl{B}_1^*\mcl{PT}\\(\cdot)^*&-\gamma I&\mcl{C}_1\\(\cdot)^*&(\cdot)^*&\mcl{T}^*\mcl{P}\mcl{A}+\mcl{A}^*\mcl{P}\mcl{T}}\preccurlyeq 0&\end{flalign*}}\\
    \hline\hline    
    \multicolumn{2}{l}{\texttt{[prog, P, gam] = PIETOOLS\_Hinf\_gain\_dual(PIE, settings)}} \\\hline
    Determine an upper bound $\gamma$ on the $\mcl{H}_{\infty}$-norm $\sup_{w,z\in L_2}\frac{\|z\|_{L_2}}{\|w\|_{L_2}}$ of the PIE for $u=0$, by solving the dual LPI~\cite{shivakumar_2020CDC}. &
    {\begin{flalign*}
	&\min\limits_{\gamma,\mcl{P}} ~~\gamma&\\
	&\mcl{P}\succ0&\\
	&\bmat{-\gamma I & \mcl{D}_{11}&\mcl{T}\mcl{P}\mcl{C}_1\\(\cdot)^*&-\gamma I&\mcl{B}_1^*\\(\cdot)^*&(\cdot)^*&\mcl{T}\mcl{P}\mcl{A}^*+\mcl{A}\mcl{P}\mcl{T}^*}\preccurlyeq 0&\end{flalign*}}\\
    \hline\hline  
    \multicolumn{2}{l}{\texttt{[prog, L, gam, P, Z] = PIETOOLS\_Hinf\_estimator(PIE, settings)}} \\\hline
    Establish an $\mcl{H}_{\infty}$-optimal observer $\mcl{T}\dot{\hat{\mbf{x}}}_{\text{f}}=A\hat{\mbf{x}}_{\text{f}}+\mcl{L}(\mcl{C}_1\hat{\mbf{x}}_{\text{f}}-y)$ for the PIE with $u=0$ by solving the LPI and returning $\mcl{L}=\mcl{P}^{-1}\mcl{Z}$~\cite{das_2019CDC}.  &
    {\begin{flalign*}
    &\min\limits_{\gamma,\mcl{P},\mcl{Z}} ~~\gamma&\\
	&\mcl{P}\succ0&\\
	&\bmat{-\gamma I & \mcl{D}_{11}^*&(\mcl{P}\mcl{B}_1+\mcl{Z}\mcl{D}_{21})^*\mcl{T}\\(\cdot)^*&-\gamma I&\mcl{C}_1\\(\cdot)^*&(\cdot)^*&\mcl{T}^*(\mcl{P}\mcl{A}+\mcl{Z}\mcl{C}_{2})+(\mcl{P}\mcl{A}+\mcl{Z}\mcl{C}_{2})^*\mcl{P}\mcl{T}}\preccurlyeq 0&\end{flalign*}}\\
    \hline\hline  
    \multicolumn{2}{l}{\texttt{[prog, K, gam, P, Z] = PIETOOLS\_Hinf\_control(PIE, settings)}} \\\hline
    Establish an $\mcl{H}_{\infty}$-optimal controller $u=\mcl{K}\mbf{x}_{\text{f}}$ for the PIE by solving the LPI and returning $\mcl{K}=\mcl{Z}\mcl{P}^{-1}$~\cite{shivakumar_2020CDC}.  &
    {\begin{flalign*}
	&\min\limits_{\gamma,\mcl{P},\mcl{Z}} ~~\gamma&\\
	&\mcl{P}\succ0&\\
	&\bmat{-\gamma I & \mcl{D}_{11}&\mcl{T}(\mcl{P}\mcl{C}_1+\mcl{Z}\mcl{D}_{12})\\(\cdot)^*&-\gamma I&\mcl{B}_1^*\\(\cdot)^*&(\cdot)^*&\mcl{T}(\mcl{A}\mcl{P}+\mcl{B}_2\mcl{Z})^*+(\mcl{A}\mcl{P}+\mcl{B}_2\mcl{Z})\mcl{T}^*}\preccurlyeq 0&\end{flalign*}}\\
    \hline\hline  
\end{tabular}
\caption{List of pre-defined executives for analysis and control of PIEs. See also Chapter~\ref{ch:LPI_examples}.}
\label{tab:executive_functions}
\end{table}


\part{Additional PIETOOLS Functionality}

\chapter{Alternative Input Formats for ODE-PDE Systems}\label{ch:alt_PDE_input}

In PIETOOLS, coupled ODE-PDE systems can be defined using three different approaches: Command Line Parser (via command line or MATLAB scripts), graphical user interface (a MATLAB app), and `terms input' format structure. The former two input methods are useful when the system is defined by using entire equations, whereas the last method (terms input format method) requires the user to identify the parameters from each term of set of equations and define them in a PIETOOLS compatible class called \texttt{pde\_struct}.

As such, Command Line Parser and GUI input methods are the simplest and most intuitive methods to define a coupled ODE-PDE. However, in this chapter we describe only the GUI and terms-based input methods, referring to Chapter~\ref{ch:PDE_DDE_representation} for information on the Command Line Parser input format. In particular, in Section~\ref{sec:GUI}, we show how any well-posed, linear, coupled 1D ODE-PDE system can be declared in PIETOOLS using the GUI. In Section~\ref{sec:alt_PDE_input:terms_input_PDE}, we then show how any well-posed, linear, coupled ODE - 1D PDE - 2D PDE system can be declared in PIETOOLS using the terms-based input format. We note that, although most ODE-PDE coupled systems can be defined by using any of the three input methods, certain type of PDEs (for example, PDEs with 2 spatial dimensions and PDEs with second order time differential terms) can only be defined using the terms input format in PIETOOLS 2022.

\section{A GUI for Defining PDEs}\label{sec:GUI}

In addition to the Command Line Parser, PIETOOLS 2022 also allows PDEs to be delcared using a graphical user interface (GUI), that provides a simple, intuitive and interactive visual interface to directly input the model. It also allows declared PDE models to be saved and loaded, so that the same system can be used in different sessions without having to declare the model from scratch each time.

To open the GUI, simply call \texttt{PIETOOLS\_PIETOOLS\_GUI} from the command line. You will see something similar to the picture below:
\begin{figure}[H]
	\centering
	\includegraphics[width=0.95\textwidth]{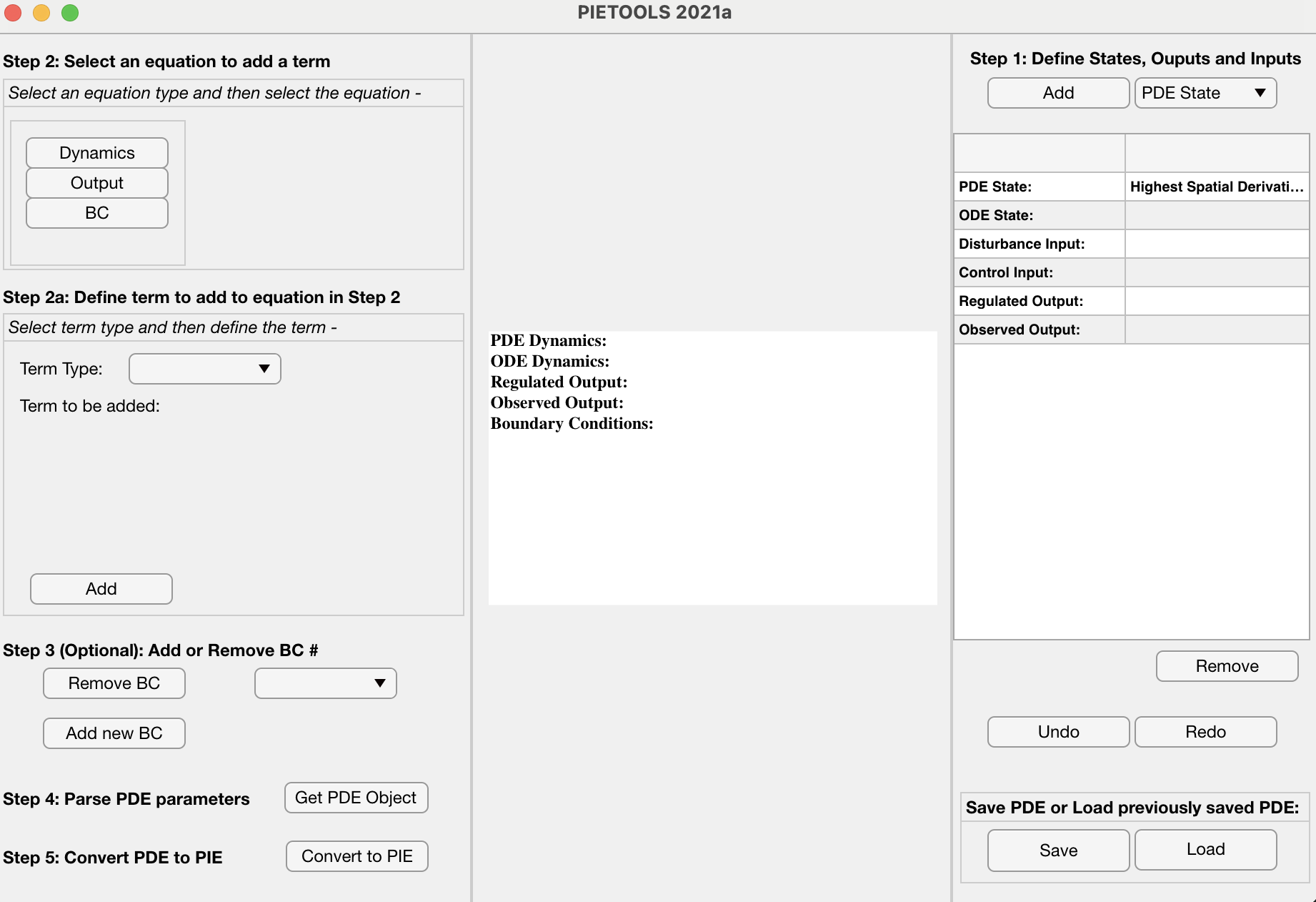} 
	\caption{GUI overview.}
	\label{gui1}
\end{figure}

Now we will go over the GUI step-by-step to demonstrate how to define your own linear, 1D ODE-PDE model in PIETOOLS. 

\subsection{Step 1: Define States, Outputs and Inputs}
First, we start with the right side of the screen as follows:
	\begin{figure}[H]
	\centering
	\includegraphics[width=0.95\textwidth]{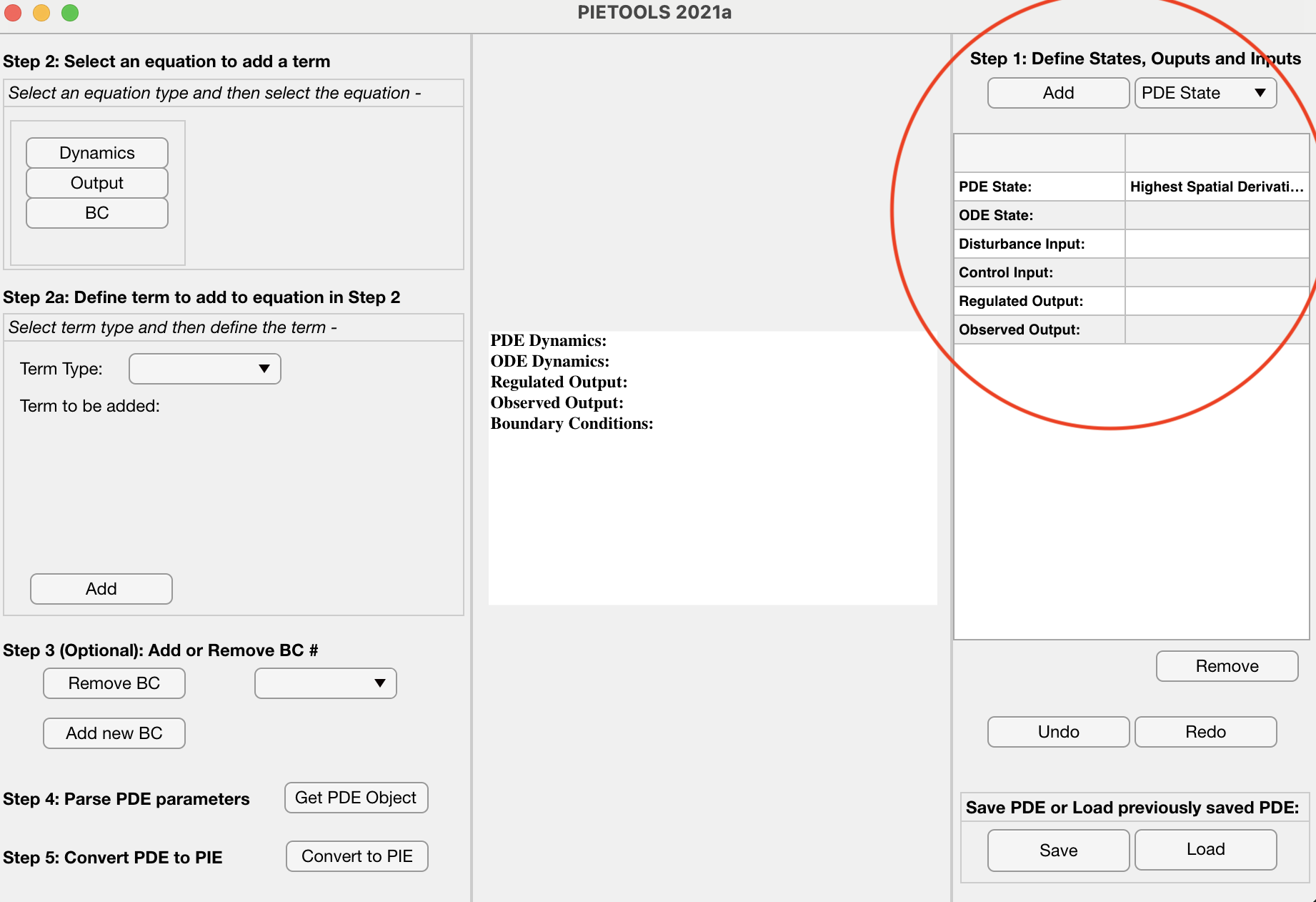} 
	\caption{Step 1: Define States, Outputs and Inputs}
	\label{gui2}
\end{figure}

\begin{enumerate}
    \item The drop-down menu \texttt{PDE State} provides a list of all the possible variables to be defined on your model. Clicking on the \texttt{PDE State} menu reveals the list
    
    	\begin{figure}[H]
	\centering
	\includegraphics[width=0.30\textwidth]{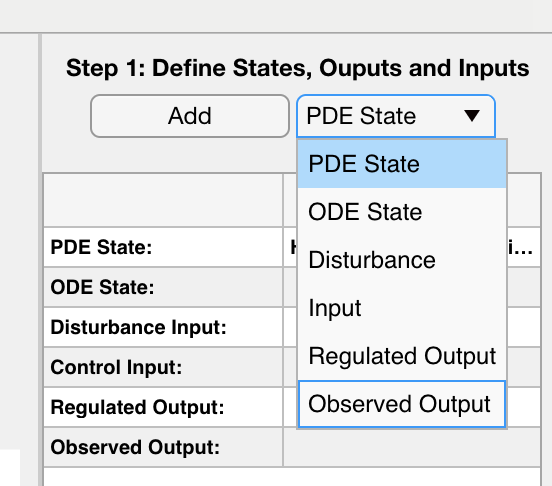} 
	\caption{Adding variables your model}
	\label{gui3}
\end{figure}

\item After selecting your intended variable, you can add them by clicking on the \texttt{Add} button.

\item Specifically, when you select \texttt{PDE State} from the drop-down menu and attempt to  \texttt{PDE State}, you also have to specify what is highest order of derivative the particular state admits.

	\begin{figure}[H]
	\centering
	\includegraphics[width=0.30\textwidth]{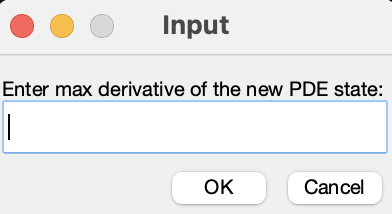} 
	\caption{Enter the highest order of derivative the particular state admits}
	\label{gui4}
\end{figure}
\item Once the variables are added, it automatically gets displayed in the display panel in the middle. As, no dynamics is specified to the model so far, all the variables are set to the default setting temporarily. 
\begin{figure}[H]
	\centering
	\includegraphics[width=0.80\textwidth]{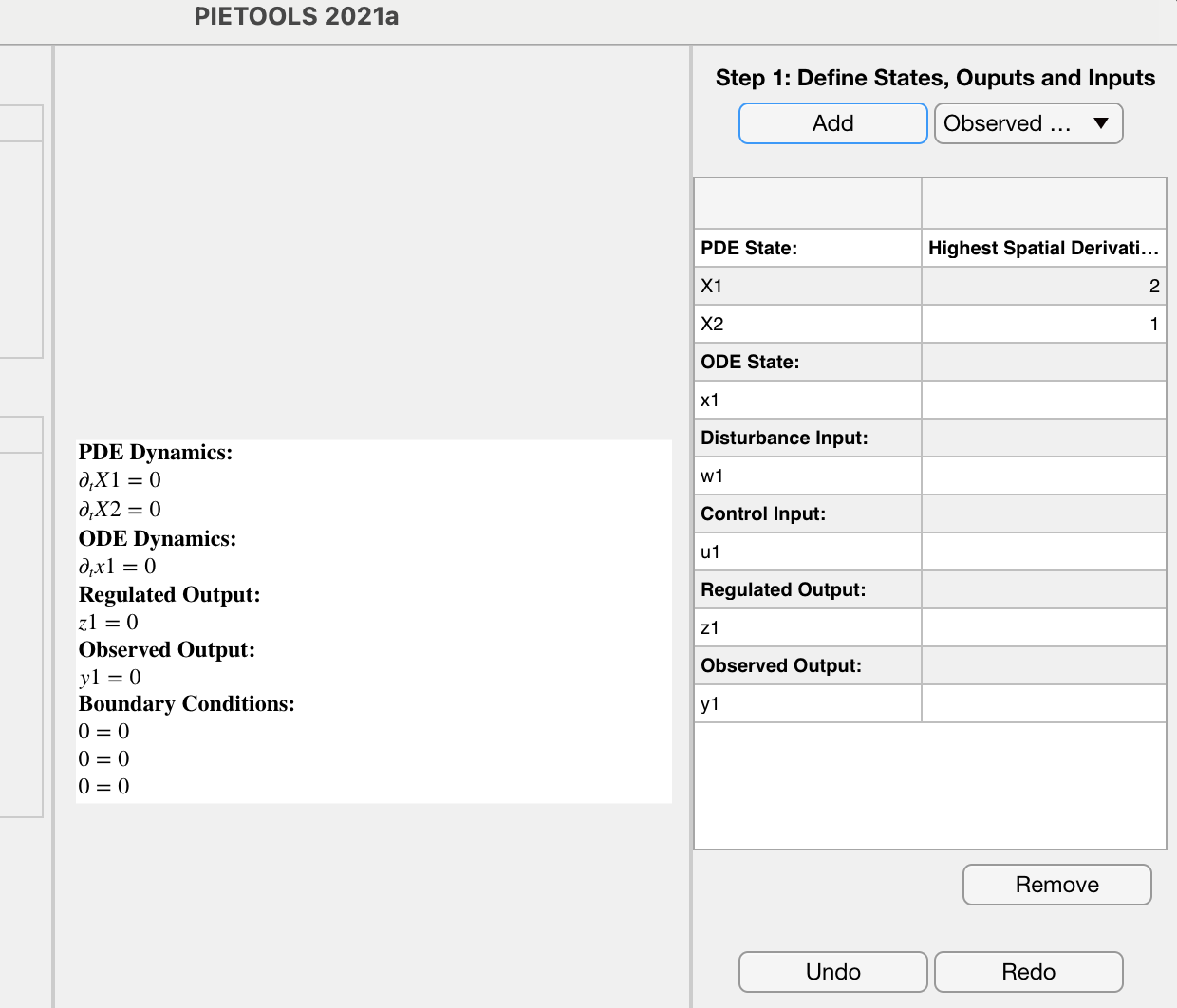} 
	\caption{After adding the variables to the model}
	\label{gui5}
\end{figure}

\item At the bottom there are options of \texttt{Remove}, \texttt{Undo}, \texttt{Redo} to delete or recover variables. 
\end{enumerate}

\subsection{Step 2: Select an Equation to Add a Term}
Now we specify the dynamics and the terms corresponding to each variables defined in Step 1.  This is located on the left hand side of the GUI.

    	\begin{figure}[H]
	\centering
	\includegraphics[width=0.80\textwidth]{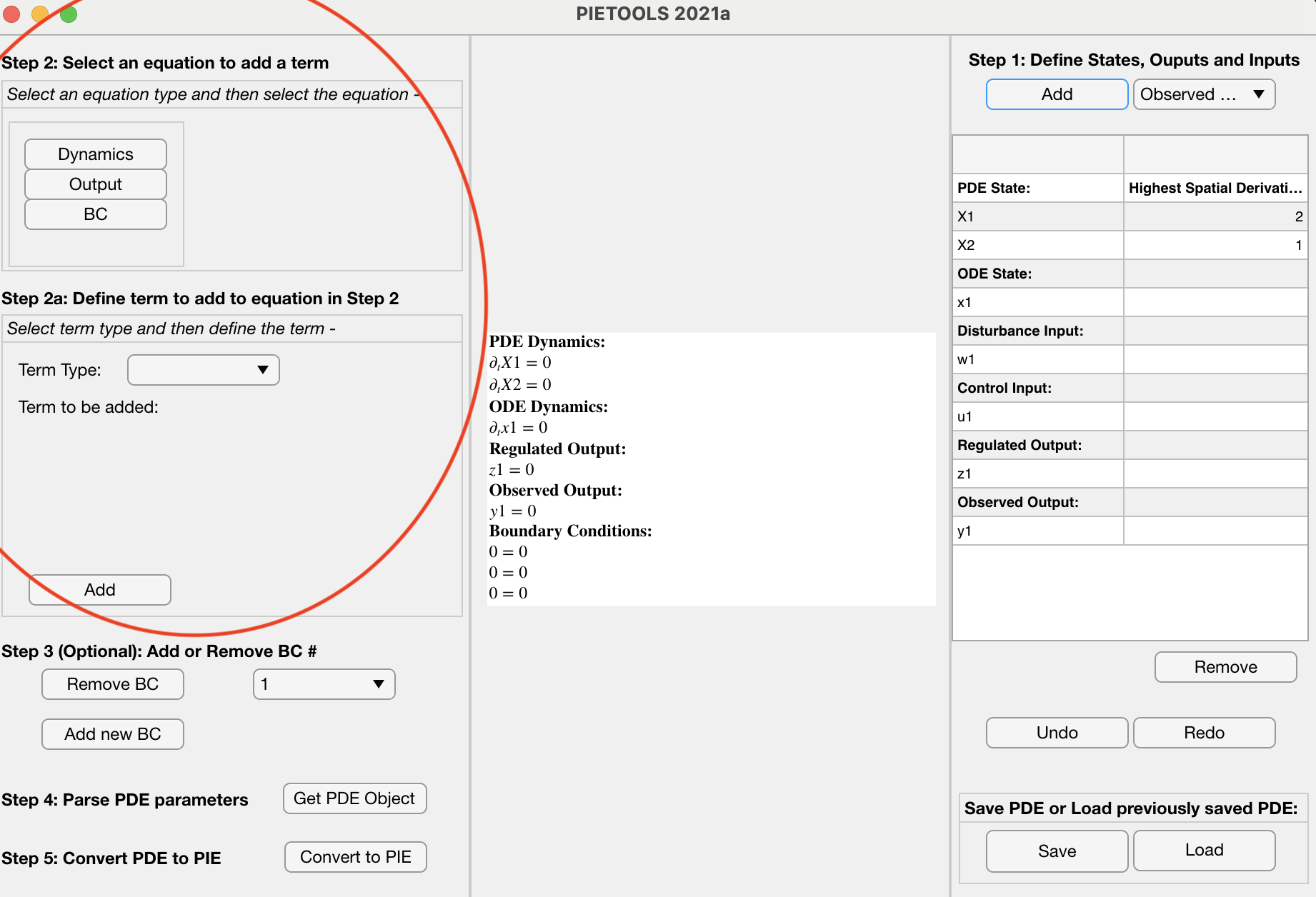} 
	\caption{Step 2: Select an Equation to Add a Term}
	\label{gui2_1}
\end{figure}
This has two parts. On the top, we have a panel for \texttt{Select an Equation type and the select the equation-} to choose which part of the model to be defined. Then, for each of the part, there an another  panel below that has to be used to \texttt{Select term type and then define the term-}.

\begin{enumerate}
    \item In the panel titled \texttt{Select an Equation type and the select the equation-}, select either \texttt{Dynamics}, \texttt{Output} or \texttt{BC} (i.e. Boundary Conditions). 
    
\item If you select \texttt{Dynamics}, all the PDE and ODE states that you specified in Step 1 appears. 
    	\begin{figure}[H]
	\centering
	\includegraphics[width=0.30\textwidth]{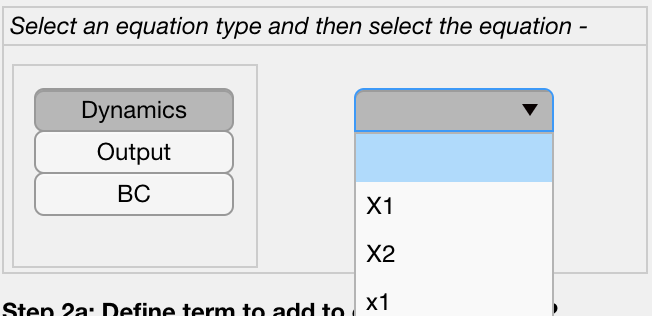} 
	\label{gui2_2}
\end{figure}

\item Once you select the desired state for which the dynamics term is needed to be added to, go down to the Step 2.a \texttt{Select term type and then define the term-}.
	\begin{figure}[H]
	\centering
	\includegraphics[width=0.30\textwidth]{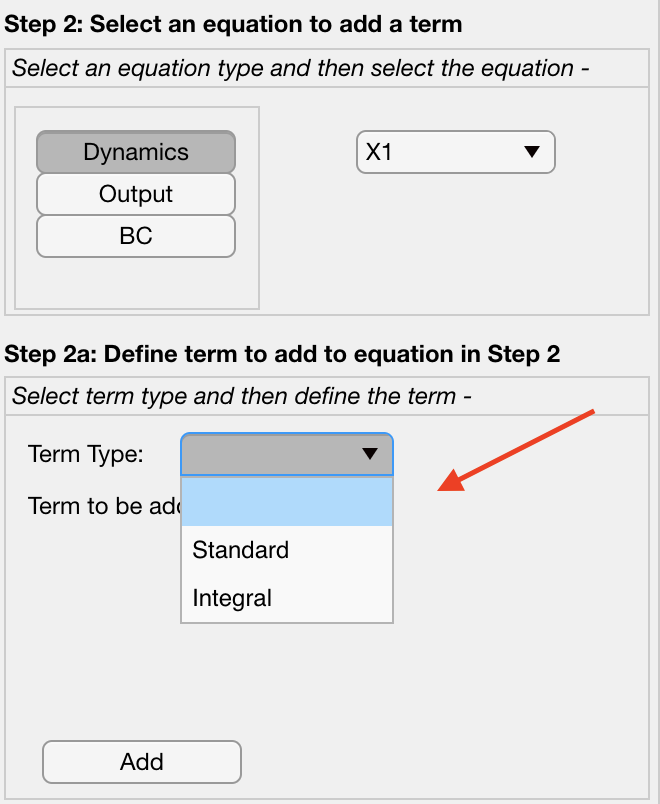} 
	\label{gui2_3}
\end{figure}

\item For individual PDE state, you may have two kinds of terms, a \texttt{Standard}, a multiplier coefficient associated with a state or a \texttt{Integral}, an integral term associated with a state.

\item The \texttt{Standard} option allows to define all the multiplier terms associated with each variables. On the other hand, the \texttt{Integral} option is only available for the PDE states. 
	\begin{figure}[H]
	\centering
	\includegraphics[width=0.30\textwidth]{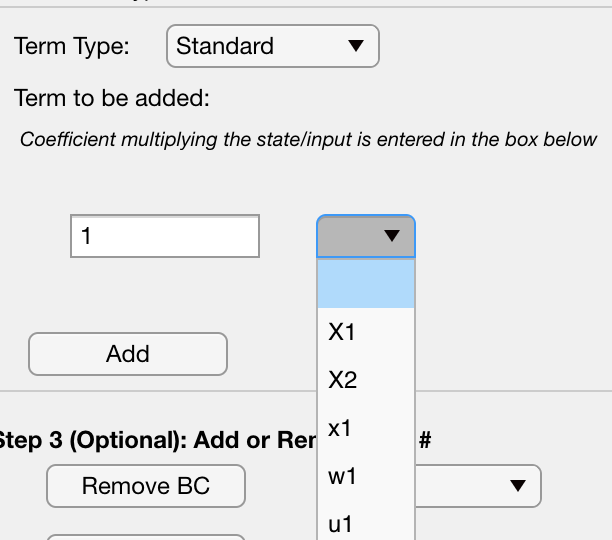}
	\includegraphics[width=0.35\textwidth]{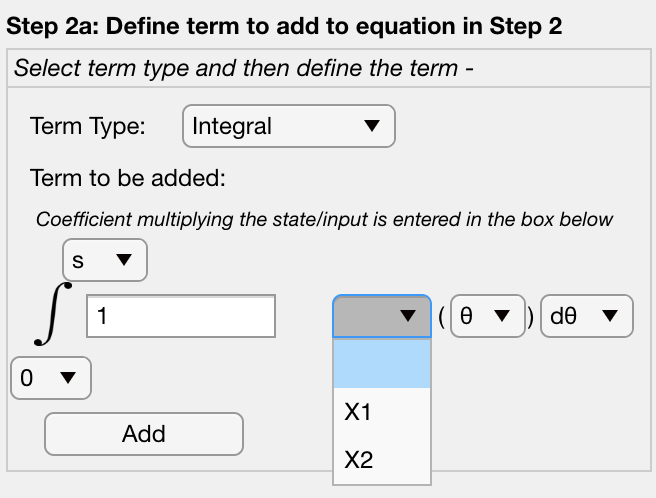}
\end{figure}

\item Now in \texttt{Standard} option, one has to select the variable and add the coefficient in the adjacent panel. Moreover, the PDE states may also contain its derivatives. If you select PDE state, you can input the order of derivative (from $0$ up to the highest order derivative for that state), the independent variable with respect to which the function is defined (it is $s$ for in-domain, $0, 1$ for boundary), and the corresponding coefficient terms. Then, by clicking on the \texttt{Add} button, we can add that term to the model and it gets shown in the display panel readily. 

	\begin{figure}[H]
	\centering
	\includegraphics[width=0.70\textwidth]{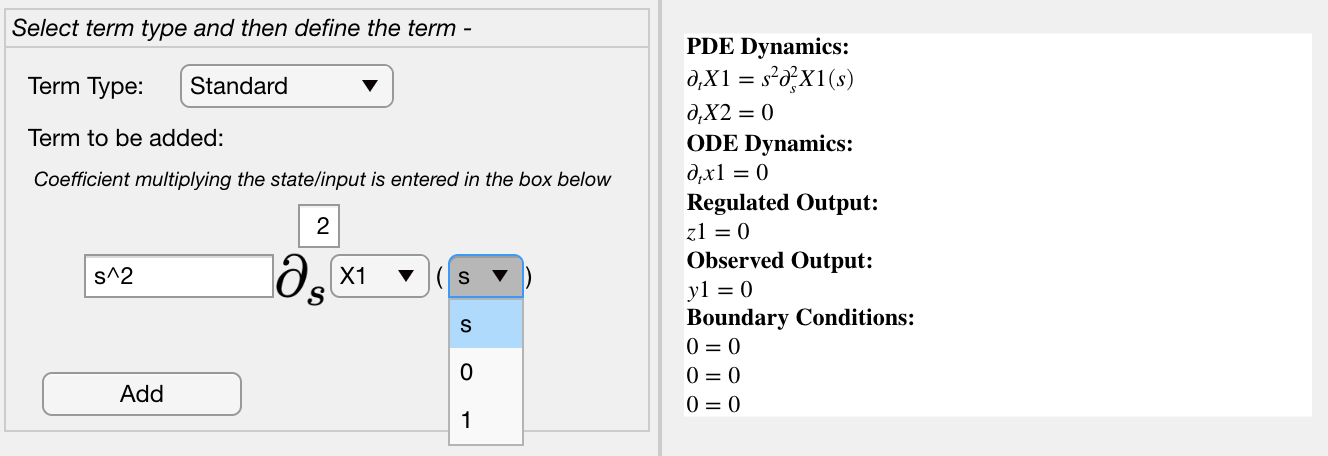}
\end{figure}

\begin{boxEnv}{\textbf{Note:}}
Only the PDE states can be a function of `$s$'. For other terms option of adding $'s'$ as an independent variable is not available.
\end{boxEnv} 

The order of derivative can not exceed the highest order derivative for that state. If the input value exceed that, while adding it throws the following error

	\begin{figure}[H]
	\centering
	\includegraphics[width=0.30\textwidth]{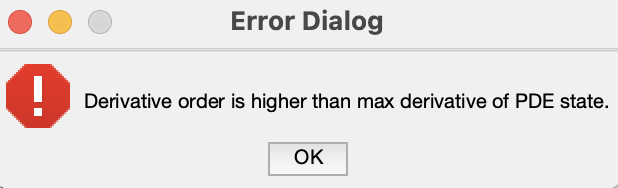}
\end{figure}

\item One can also add an integral term by selecting \texttt{Integral}. Here, identical to the \texttt{Standard} option, we can define the order or derivative, the coefficient and the limits of integral which can be $0$ or $s$ for lower limit and $s$ or $1$ for upper limit. The functions are always with respect to $\theta$. 

	\begin{figure}[H]
	\centering
	\includegraphics[width=0.80\textwidth]{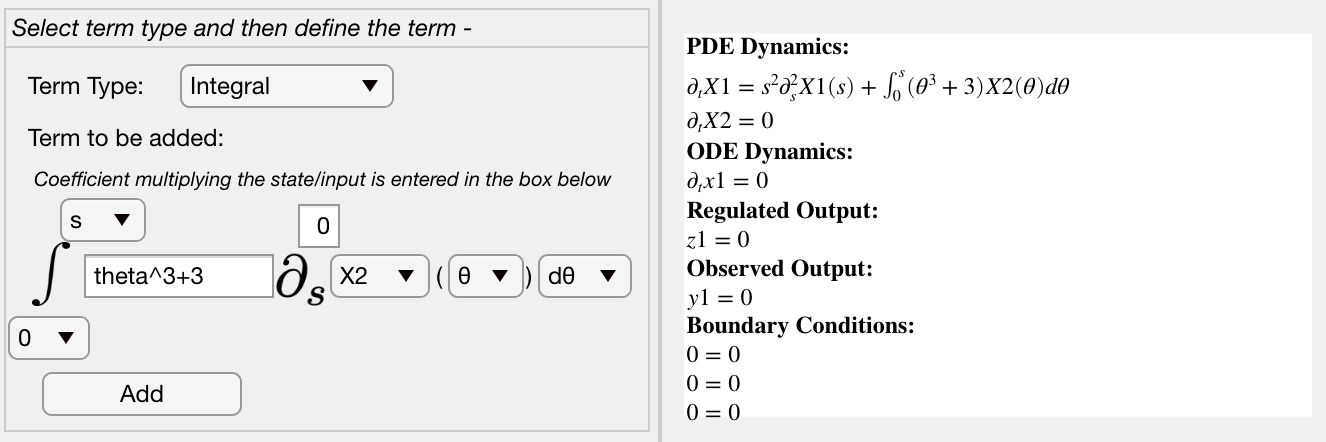}
\end{figure}

	\begin{figure}[H]
	\centering
	\includegraphics[width=0.80\textwidth]{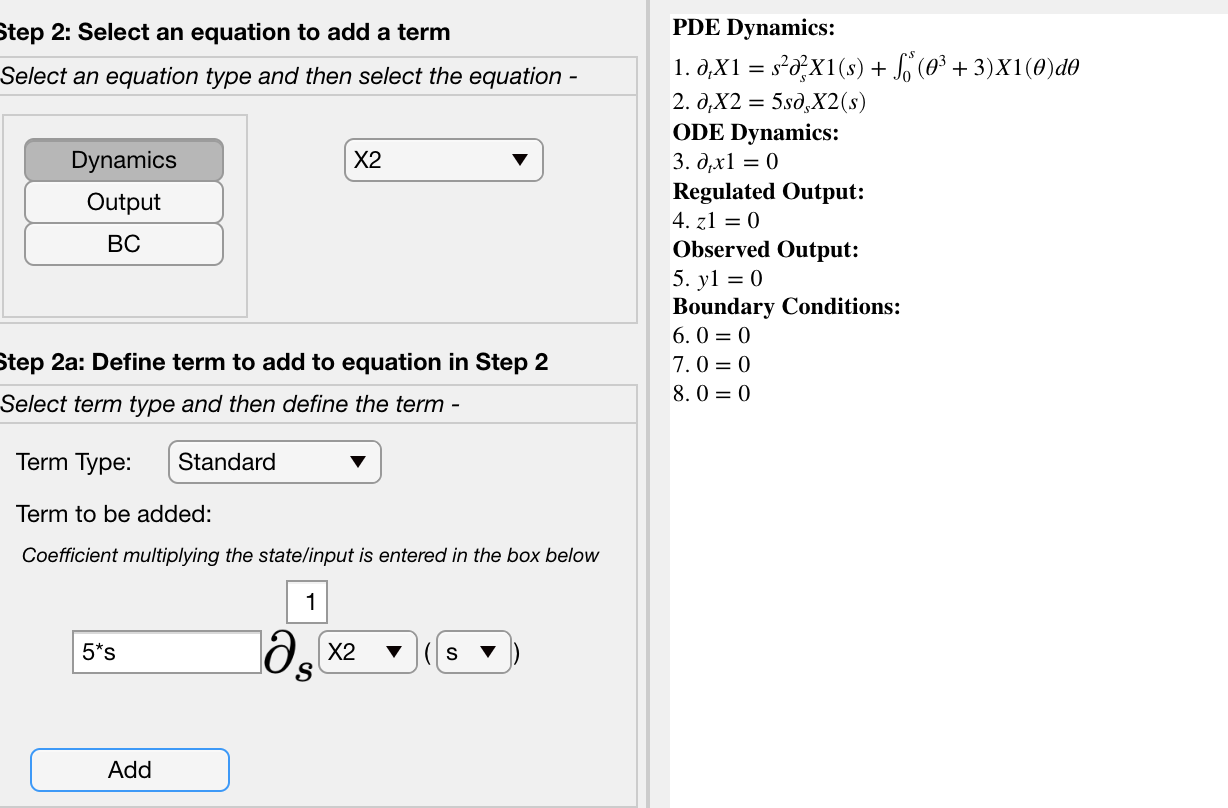}
\end{figure}

\item To define the outputs and boundary conditions, one must follow the same steps as above. 
	\begin{figure}[H]
	\centering
	\includegraphics[width=0.35\textwidth]{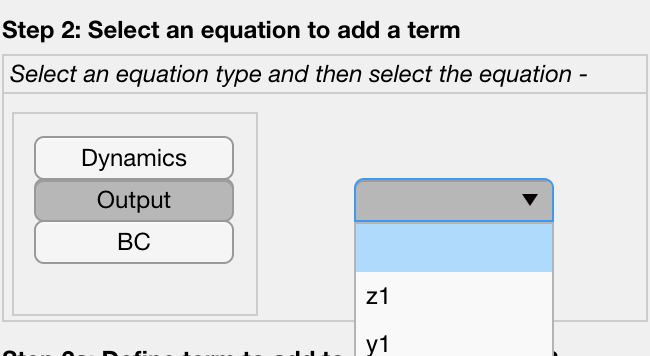}
	\includegraphics[width=0.35\textwidth]{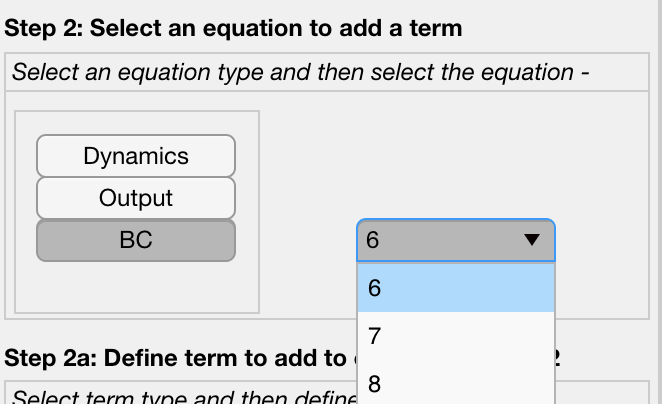}
\end{figure}
\end{enumerate}

\begin{boxEnv}{Additional Remarks:}
A) The integral term of PDE states can only be a function of $'\theta'$. Moreover, the order of derivative can not exceed the specified order of differentiability of that state. If the input value exceeds this order of differentiability, and error is thrown when trying to add the term.

\noindent B) Terms can be specified and added only for one variable at a time. Once a desired variable and one of the options (\texttt{Dynamics}, \texttt{Output}, \texttt{BC}) has been selected at the top, the term can be added following the instructions at the bottom (Step 2a). In order to select another variable for which to add a term, the above steps must be repeated.
\end{boxEnv}

After adding all the corresponding terms for dynamics, outputs and boundary conditions, the complete description of an example model looks something like below:
	\begin{figure}[H]
	\centering
	\includegraphics[width=0.40\textwidth]{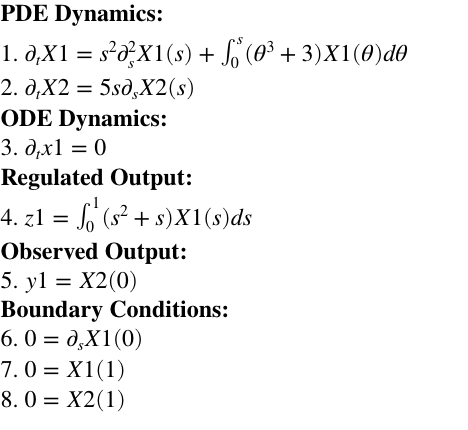}
	\caption{An example of a complete model as displayed in the GUI}
\end{figure}

\subsection{Step 3: (Optional) Add or Remove BC}
As an option, the user can either add a new boundary condition or remove one. Note that the number of boundary condition has to be coherent with the number of state variables. 
	\begin{figure}[H]
	\centering
	\includegraphics[width=0.50\textwidth]{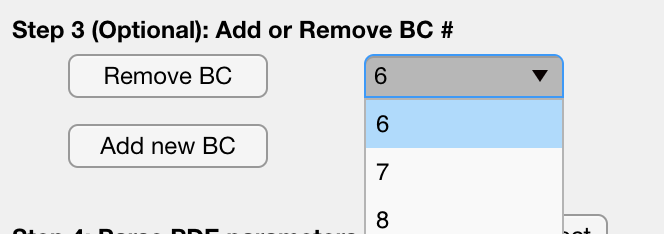}
\end{figure}

\subsection{Step 4-5: Parse PDE Parmeters and Convert Them to PIE}
\begin{enumerate}
    \item Now clicking \texttt{Get PDE Object}, you can extract all the parameters related to your model and store them in an object called \texttt{PDE\_GUI} which directly gets loaded into the MATLAB workspace. 
    
   \item Now clicking \texttt{convert to PIE}, you can convert your model to PIE and store them in an object called \texttt{PIE\_GUI} which directly gets loaded into the MATLAB workspace. 
\end{enumerate}

	\begin{figure}[H]
	\centering
	\includegraphics[width=0.50\textwidth]{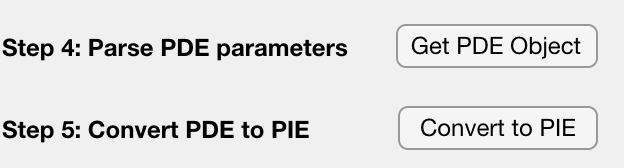}
\end{figure}

\section{The Terms-Based Input Format}\label{sec:alt_PDE_input:terms_input_PDE}

In addition to the Command Line Parser and GUI input methods, PIETOOLS also allows ODE-PDE systems to be specified in a terms-based format. Although this format is not quite as intuitive as the alternative input methods, it is the most general input format, used internally by PIETOOLS for e.g. the PDE to PIE conversion. In PIETOOLS 2022, it is also the only format that allows for the implementation of PDEs involving 2 spatial variables.

In the terms-based input format, each term in each equation defining the system is specified as a separate cell element in \texttt{pde\_struct} class object. These \texttt{pde\_struct} objects collect all information on the state variables, inputs and outputs, as well as the equations defining their dynamics, through the following fields:

\begin{Statebox}{\texttt{pde\_struct}}
\begin{flalign*}
 &\texttt{PDE.x}	&	&\text{a cell with each element $i$ specifying a state component $\mbf{x}_i$ in the system;}	\\
 &\texttt{PDE.w}	&	&\text{a cell with each element $i$ specifying an exogenous input $\mbf{w}_{i}$;}	\\
 &\texttt{PDE.u}	&	&\text{a cell with each element $i$ specifying an actuator input $\mbf{u}_{i}$;}	\\
 &\texttt{PDE.z}	&	&\text{a cell with each element $i$ specifying a regulated output $\mbf{z}_{i}$;}	\\
 &\texttt{PDE.y}	&	&\text{a cell with each element $i$ specifying an observed output $\mbf{y}_{i}$;}	\\
 &\texttt{PDE.BC}	&	&\text{a cell with each element $i$ specifying a boundary condition for the PDE.}	&	&	
\end{flalign*}
\end{Statebox}

Each element of each of the cells in the above list is yet another struct, with fields
\begin{flalign*}
&\texttt{size}	&	&\text{a scalar integer specifying the size of the state component, input or output;} \\
&\texttt{vars}	&	&\text{a $p\times 2$ pvar (polynomial) array (for $p\leq 2$), specifying the spatial variables} \\
&   &   &\text{of the state component, input, or output;}	&	&\\
&\texttt{dom}	&	&\text{a $p\times 2$ array specifying the interval on which each spatial variable exists;}	&	&\\
&\texttt{term}	&	&\text{a cell defining the (differential) equation associated to the state component,}\\
&   &   &\text{output, or boundary condition,}	&	&
\end{flalign*}
specifying e.g. a state component $\mbf{x}_{i}(t)\in L_2^{n}[[a,b]\times[c,d]]$ on variables $(s_1,s_2)\in[a,b]\times[c,d]$ and an exogenous input $w_k(t)\in\R^m$ as
\begin{matlab}
 \texttt{PDE.x\{i\}.size = n;}	        \hspace*{2.5cm} \texttt{PDE.w\{k\}.size = m;} \\
 \texttt{PDE.x\{i\}.vars = [s1; s2];}	\hspace*{2.0cm} \texttt{PDE.w\{k\}.vars = [];}	\\
 \texttt{PDE.x\{i\}.dom = [a,b; c,d];}	
\end{matlab}
The remaining field \texttt{term} can then be used to specify the differential equation for the state component $\mbf{x}_i$, output $\mbf{y}_i$ or $\mbf{z}_i$, or the $i$th boundary condition $0=\mbf{F}_i$. Here, the $j$th term in each of these equations is specified through the fields
\begin{flalign*}
& \texttt{term\{j\}.x;}	 & &\text{integer specifying which state component,}	&	\\
& \hspace*{0.5cm} \textbf{or} \quad \texttt{term\{j\}.w;}	&	&\hspace*{4.0cm} \text{\textbf{or} exogenous input,}	\\
& \hspace*{1.0cm} \textbf{or} \quad \texttt{term\{j\}.u;}	&	&\hspace*{4.5cm} \text{\textbf{or} actuator input appears in the term;}	\\	
& \texttt{term\{j\}.D}	& &\text{$1\times p$ integer array specifying the order of the derivative of the \textbf{state}} \\
&   &   &\text{\textbf{component} \texttt{term\{j\}.x} in each variable;}	\\
& \texttt{term\{j\}.loc} & &\text{$1\times p$ polynomial or ``double'' array specifying the spatial position}\\
&   &   &\text{at which to evaluate the \textbf{state component} \texttt{term\{j\}.x};}	\\
& \texttt{term\{j\}.I} & &\text{$p\times 1$ cell array specifying the lower and upper limits of the integral}\\
&   &   &\text{to take of the state or input;}	\\
& \texttt{term\{j\}.C}	& &\text{polynomial or constant factor with which to multiply the state or input,} 
\end{flalign*}
where $p$ denotes the number of spatial variables on which the component $x$ or input $w$ or $u$ depends. For example, a term involving a state component $\mbf{x}_k(s_1,s_2)$ depending on two spatial variables of the form
{\small
\begin{align}\label{PDE_term}
&\underbrace{\int_{L_1}^{U_1} \int_{L_2}^{U_2}}_{\texttt{I}}\left(
\underbrace{C(s_1,s_2,\theta_1,\theta_2)}_{\texttt{C}}\
\overbrace{\partial_{\theta_1}^{d_1}\partial_{\theta_2}^{d_2}}^{\texttt{D}}\thinspace
\underbrace{\mbf{x}_{k}}_{\texttt{x}}(t-\tau,\overbrace{\theta_1,\theta_2}^{\texttt{loc}})\right) d\theta_2 d\theta_1.
\end{align}
}
can be added to the equation for $\dot{\mbf{x}}_i$ as
\begin{matlab}
\begin{verbatim}
 PDE.x{i}.term{j}.x = k;
 PDE.x{i}.term{j}.D = [d1,d2];
 PDE.x{i}.term{j}.loc = [theta1, theta2];
 PDE.x{i}.term{j}.I = {[L1,U1]; [L2,U2]};
 PDE.x{i}.term{j}.C = C;
 PDE.x{i}.term{j}.delay = tau;
\end{verbatim}
\end{matlab}
Similarly, a term allowing a term involving an input as
{\small
\begin{align}\label{PDE_term_input}
&\underbrace{\int_{L_1}^{U_1} \int_{L_2}^{U_2}}_{\texttt{I}}\left(\underbrace{C(s_1,s_2,\theta_1,\theta_2)}_{\texttt{C}}\ \underbrace{\mbf{w}_{k}}_{\texttt{w}}(t-\tau,\theta_1,\theta_2)\right)d\theta_2 d\theta_1,	\qquad \text{or}   &\nonumber\\	
&\hspace*{3.0cm}\underbrace{\int_{L_1}^{U_1} \int_{L_2}^{U_2}}_{\texttt{I}}\left(\underbrace{C(s_1,s_2,\theta_1,\theta_2)}_{\texttt{C}}\ \underbrace{\mbf{u}_{k}}_{\texttt{u}}(t-\tau,\theta_1,\theta_2)\right)d\theta_2 d\theta_1,   &   &
\end{align}
}
can be added to the PDE of the $i$th state component as
\begin{matlab}
 \texttt{PDE.x\{i\}.term\{j\}.w = k;} \hspace*{2.5cm} \text{or}  \hspace*{2.0cm} \texttt{PDE.x\{i\}.term\{j\}.u = k;}	\\
\texttt{PDE.x\{i\}.term\{j\}.I = \{[L1,U1]; [L2,U2]\};}  \\
\texttt{PDE.x\{i\}.term\{j\}.C = C;}	\\
\texttt{PDE.x\{i\}.term\{j\}.delay = tau;}	
\end{matlab}
Here, in any term, at most one of the fields \texttt{x}, \texttt{w}, and \texttt{u} can be specified, indicating whether the term involves a state component or input, and specifying which state component or input this would be. If the term involves an input, no field \texttt{D} or \texttt{loc} can be specified, as PIETOOLS cannot take a derivative of an input, nor evaluate it at any particular position. Using a similar structure, terms in the equations for the outputs and the boundary conditions can also be specified. \\

In the remainder of this section, we will illustrate through several examples how this terms-based input format can be used to declare general linear ODE-PDE systems in PIETOOLS. In particular:
\bigskip
\begin{enumerate}
	
	
	\item In Section~\ref{subsec:terms:multi_state}, we show how a system of coupled 2D PDEs can be implemented.
	\medskip
	
	\item In Section~\ref{subsec:terms:ODE_PDE}, we show how a coupled ODE - 1D PDE - 2D PDE system can be implemented.
	\medskip
	
	\item In Section~\ref{subsec:terms:input_output}, we show how systems with inputs and outputs can be implemented.
	\medskip
	
	\item Finally, in Section~\ref{subsec:terms:additional_opts}, we outline some additional options to implement more complicated systems, involving partial integrals, kernel functions, and stricter continuity constraints.
\end{enumerate}

%
%

\subsection{Declaring a System of Coupled 2D PDEs}\label{subsec:terms:multi_state}

Suppose we want to implement a PDE given by
\begin{align}\label{eq:terms_example_multi_state}
\dot{\mbf{x}}_{1}(t,s_1,s_2)&=\bmat{1\\2}\mbf{x}_2(t,s_1,s_2) + \bmat{10&0\\0&-1}\partial_{s_1}\partial_{s_2}\mbf{x}_1(t,s_1,s_2),	&	(s_1,s_2)&\in[0,3]\times[-1,1],	\nonumber\\
\dot{\mbf{x}}_{2}(t,s_1,s_2)&=\bmat{1&1}\partial_{s_1}^2 \mbf{x}_1(t,s_1,s_2) + \bmat{2&-1}\partial_{s_2}^2 \mbf{x}_1(t,s_1,s_2),		&	t&\geq 0,  \nonumber\\
0&=\mbf{x}_1(t,0,s_2)		&    
&\hspace*{-4.0cm} 0=\partial_{s_1}\mbf{x}_1(t,0,s_2),	\nonumber\\
0&=\mbf{x}_1(t,s_1,-1),	    &
&\hspace*{-4.0cm} 0=\partial_{s_2}\mbf{x}_1(t,s_1,-1),	
\end{align}
where $\mbf{x}_1\in L_2^{2}\bl[[0,3]\times[-1,1]\br]$ and $\mbf{x}_2\in L_2\bl[[0,3]\times[-1,1]\br]$. To implement this PDE, we first initialize an empty \texttt{pde\_struct} object as
\begin{matlab}
\begin{verbatim}
 >> PDE = pde_struct();
\end{verbatim}
\end{matlab}
Next, we declare the state components that appear in the system. In~\eqref{eq:terms_example_multi_state}, we have two PDEs, involving two state components $\mbf{x}=\bmat{\mbf{x}_1\\\mbf{x}_2}$. Each component and associated PDE is implemented using a separate field \texttt{PDE.x\{i\}}, for which we specify the spatial variables and their domain as
\begin{matlab}
\begin{verbatim}
 >> pvar s1 s2
 >> PDE.x{1}.vars = [s1;s2];        PDE.x{2}.vars = [s1;s2];
 >> PDE.x{1}.dom = [0,3;-1,1];      PDE.x{2}.dom = [0,3;-1,1];
\end{verbatim}
\end{matlab}
Here, we use the first line to initialize polynomial variables $s1$ and $s2$. With the second line, we then indicate that both $\mbf{x}_1$ and $\mbf{x}_2$ depend on these spatial variables, and we set the domain of these variables using the third line. Note that the variables should be specified as a column vector, so that each row in \texttt{dom} defines the domain of the variable(s) in the associated row of \texttt{vars}. 

Having declared the state components, we now define the PDEs associated to each component. Here, for $i\in\{1,2\}$, the PDE for state component $\mbf{x}_i$ is defined one term at a time using the cell structure \texttt{PDE.x\{i\}.term}. For example, to specify the first term in the first PDE,
\begin{align*}
 \underbrace{\bmat{1\\2}}_{\texttt{C}}\underbrace{\mbf{x}_2}_{\texttt{x}}(t,s_1,s_2),
\end{align*}
we specify the following elements
\begin{matlab}
\begin{verbatim}
 >> PDE.x{1}.term{1}.x = 2;
 >> PDE.x{1}.term{1}.C = [1;2];
\end{verbatim}
\end{matlab}
Here, we set \texttt{x\{1\}.term\{1\}.x=2} to indicate that the first term in the PDE of the first state component $\mbf{x}_1$ is expressed in terms of the second state component $\mbf{x}_2$. Setting \texttt{C=[1;2]}, we pre-multiply this term with the matrix $\sbmat{1\\2}$ before adding it to the equation. 

In a similar manner, we declare the second term in the first PDE,
\begin{align*}
\underbrace{\bmat{10&0\\0&-1}}_{\texttt{C}}\overbrace{\partial_{s_1}^1\partial_{s_2}^1}^{\texttt{D}}\underbrace{\mbf{x}_1}_{\texttt{x}}(t,s_1,s_2),
\end{align*}
by calling
\begin{matlab}
\begin{verbatim}
 >> PDE.x{1}.term{2}.x = 1;
 >> PDE.x{1}.term{2}.D = [1,1];
 >> PDE.x{1}.term{2}.C = [10,0; 0,-1];
\end{verbatim}
\end{matlab}
Here, in addition to specifying the considered state component and coefficient matrix using \texttt{term\{2\}.x} and \texttt{term\{2\}.C}, we also specify a derivative to be taken of the state using the field \texttt{term\{2\}.D}. This field \texttt{term\{j\}.D} specifies the order of the derivative that should be taken of the state component \texttt{term\{j\}.x} with respect to each of the spatial variables on which this state component depends. Setting \texttt{D=[1,1]}, a first order derivative will be taken with respect to each of the two variables $(s_1,s_2)$ on which state component $\mbf{x}_1$ depends. Note that if no field \texttt{D} is specified for a term, it will default to $0$, assuming that no differentiation is desired for this term.

Now, to declare the terms in the PDE for the second state component $\mbf{x}_2$,
\begin{align*}
    \underbrace{\bmat{1&1}}_{\texttt{C}}\overbrace{\partial_{s_1}^2\partial_{s_2}^{0}}^{\texttt{D}} \underbrace{\mbf{x}_1}_{\texttt{x}}(t,s_1,s_2) + \underbrace{\bmat{2&-1}}_{\texttt{C}}\overbrace{\partial_{s_1}^{0}\partial_{s_2}^2}^{\texttt{D}} \underbrace{\mbf{x}_1}_{\texttt{x}}(t,s_1,s_2)
\end{align*}
we can specify each term separately as we did the first equation, by calling
\begin{matlab}
\begin{verbatim}
 >> PDE.x{2}.term{1}.x = 1;         PDE.x{2}.term{2}.x = 1;
 >> PDE.x{2}.term{1}.D = [2,0];     PDE.x{2}.term{2}.D = [0,2];
 >> PDE.x{2}.term{1}.C = [1,1];     PDE.x{2}.term{2}.C = [2,-1];
\end{verbatim}
\end{matlab}
Here, since both terms depend on the same state component $\mbf{x}_1$, we set the field \texttt{x} in both terms equal to $1$. However, in such a situation, we can also combine the terms to represent them as
\begin{align*}
\bmat{1&1}\partial_{s_1}^2 \mbf{x}_1(t,s_1,s_2) + \bmat{2&-1}\partial_{s_2}^2 \mbf{x}_1(t,s_1,s_2)
=\underbrace{\bmat{1&1&2&-1}}_{\texttt{C}}\bbbbl[\overbrace{\mat{\partial_{s_1}^2\partial_{s_2}^{0}\\ \partial_{s_1}^{0}\partial_{s_2}^{2}}}^{\texttt{D}}\underbrace{\mat{\mbf{x}_1\\\mbf{x}_1}}_{\texttt{x}}\!\!\mat{(t,s_1,s_2)\\(t,s_1,s_2)}\bbbbr],
\end{align*}
which we can implement using a single term
\begin{matlab}
\begin{verbatim}
 >> PDE.x{2}.term{1}.x = 1;
 >> PDE.x{2}.term{1}.D = [2,0; 0,2];
 >> PDE.x{2}.term{1}.C = [1,1, 2,-1];
\end{verbatim}
\end{matlab}

Having declared the PDE, it remains only to specify the boundary conditions,
\begin{align*}
    0&=\underbrace{\mbf{x}_1}_{\texttt{x}}(t,\underbrace{0,s_2}_{\texttt{loc}})		&    
    0&=\underbrace{\mbf{x}_1}_{\texttt{x}}(t,\underbrace{s_1,-1}_{\texttt{loc}}) ,	\nonumber\\
    0&=\overbrace{\partial_{s_1}^{1}\partial_{s_2}^{0}}^{\texttt{D}}\underbrace{\mbf{x}_1}_{\texttt{x}}(t,\underbrace{0,s_2}_{\texttt{loc}}),	    &
    0&=\overbrace{\partial_{s_1}^{0}\partial_{s_2}^{1}}^{\texttt{D}}\underbrace{\mbf{x}_1}_{\texttt{x}}(t,\underbrace{s_1,-1}_{\texttt{loc}})
\end{align*}
For these, in addition to specifying e.g.~the desired state component and order of the derivative, we also have to specify at which boundary to evaluate the state component. For this, we use the field \texttt{loc}, specifying the first (and only) term in each of the four boundary conditions as
\begin{matlab}
\begin{verbatim}
 >> PDE.BC{1}.term{1}.x = 1;            PDE.BC{2}.term{1}.x = 1;
 >> PDE.BC{1}.term{1}.loc = [0,s2];     PDE.BC{2}.term{1}.loc = [s1,-1];
 
 >> PDE.BC{3}.term{1}.x = 1;            PDE.BC{4}.term{1}.x = 1;
 >> PDE.BC{3}.term{1}.D = [1,0];        PDE.BC{4}.term{1}.D = [0,1];
 >> PDE.BC{3}.term{1}.loc = [0,s2];     PDE.BC{4}.term{1}.loc = [s1,-1];
\end{verbatim}    
\end{matlab}
Note here that, by not specifying a field \texttt{C}, PIETOOLS assumes the term is not pre-multiplied by any matrix, and the value of \texttt{C} defaults to an identity matrix of the appropriate size. 

Having declared the full PDE, we use the function \texttt{initialize} to check for errors and fill in any gaps, which produces the output
\begin{matlab}
\begin{verbatim}
>> PDE = initialize(PDE);

Encountered 2 state components: 
 x1(t,s1,s2), of size 2, differentiable up to order (2,2) in variables (s1,s2);
 x2(t,s1,s2), of size 1, differentiable up to order (0,0) in variables (s1,s2);

Encountered 4 boundary conditions: 
 F1(t,s2) = 0, of size 2;
 F2(t,s1) = 0, of size 2;
 F3(t,s2) = 0, of size 2;
 F4(t,s1) = 0, of size 2;
\end{verbatim}
\end{matlab}
matching the desired specifications of our state components and boundary conditions.

\subsection{Declaring an ODE-PDE System}\label{subsec:terms:ODE_PDE}

Having illustrated how to implement general linear 2D PDEs, we now illustrate how systems of coupled ODEs and PDEs can be implemented in PIETOOLS. For example, suppose we have a system of the form
\begin{align*}
\dot{x}_1(t)&=-x_1(t) + \int_{0}^{3}2\mbf{x}_2(t,s_1)ds_1 - \int_{-1}^{1}3\mbf{x}_{3}(t,3,s_2)ds_2,	&		(s_1,s_2)&\in[0,3]\times[-1,1]	\\
\dot{\mbf{x}}_2(t,s_1)&=s_1 x_1(t) + \partial_{s_1}\mbf{x}_2(t,s_1) - \int_{-1}^{1}\mbf{x}_3(t,s_1,s_2) ds_2,	&	t&\geq 0	\\
\dot{\mbf{x}}_3(t,s_1,s_2)&=s_1s_2 x_1(t) + \mbf{x}_2(t,s_1) + \partial_{s_2}\mbf{x}_3(t,s_1,s_2),		\\
0&=\mbf{x}_2(t,3)-x_1(t)	\\
0&=\mbf{x}_3(t,3,-1) - x_1(t-3)	\\
0&=\partial_{s_1}\mbf{x}_3(t,s_1,-1) - \mbf{x}_2(t,s_1)
\end{align*}
where $x_1\in\R$ is a finite-dimensional state, $\mbf{x}_2\in L_2[0,3]$ varies only along the first spatial dimension, and $\mbf{x}_3\in L_2\bl[[0,3]\times[-1,1]\br]$ varies along both spatial directions. As before, we first initialize an empty PDE structure as
\begin{matlab}
\begin{verbatim}
 >> PDE = pde_struct; 
\end{verbatim}
\end{matlab}
Next, we implement each state component and their associated differential equation using a separate field \texttt{PDE.x\{i\}}. In this case, different states vary in (different numbers of) variables, which we indicate through the fields \texttt{vars} as
\begin{matlab}
\begin{verbatim}
 >> pvar s1 s2
 >> PDE.x{1}.vars = [];     PDE.x{2}.vars = s1;     PDE.x{3}.vars = [s1; s2];
 >> PDE.x{1}.dom = [];      PDE.x{2}.dom = [0,3];   PDE.x{3}.dom = [0,3; -1,1];
\end{verbatim}
\end{matlab}
Note here that, for the finite-dimensional state component $x_1$, we can indicate that this state does not vary in space by specifying an empty set of variables \texttt{vars} and/or specifying an empty spatial domain \texttt{dom}.

Having initialized the different state components, we now specify the PDE for each. Here, the first term in the first PDE,
\begin{align*}
 \underbrace{-}_{\texttt{C}}\underbrace{x_1}_{\texttt{x}}(t),
\end{align*}
can be easily implemented as
\begin{matlab}
\begin{verbatim}
 >> PDE.x{1}.term{1}.x = 1;
 >> PDE.x{1}.term{1}.C = -1;
\end{verbatim}
\end{matlab}
However, for the second term, we have to perform integration, to remove the dependence of the state $\mbf{x}_2(s_1)$ on the spatial variable $s_1$. To declare this, we use the field $\texttt{I}$, which is a cell array of which the number of elements is equal to the number of variables that the state component depends on. Each element of this cell array should specify the spatial domain over which to integrate the state component along the associated spatial direction. Since $\mbf{x}_2(s_1)$ depends only on $s_1$, the field \texttt{I} will have just one element, which we set equal to the domain of integration $[0,3]$. That is, we implement the second term in the first PDE,
\begin{align*}
\underbrace{2}_{\texttt{C}}\underbrace{\int_{0}^{3}}_{\texttt{I\{1\}}}\underbrace{\mbf{x}_2}_{\texttt{x}}(t,s_1)ds_1,
\end{align*}
as
\begin{matlab}
\begin{verbatim}
 >> PDE.x{1}.term{2}.x = 2;
 >> PDE.x{1}.term{2}.C = 2;
 >> PDE.x{1}.term{2}.I{1} = [0,3];
\end{verbatim}
\end{matlab}
Finally, for the third term, we also integrate the state, but this time along the second spatial variable $s_2$ on which the state component $\mbf{x}_3(s_1,s_2)$ depends -- meaning we need to specify this integral using the second element \texttt{I\{2\}} of the field \texttt{I}. To also remove the dependence of $\mbf{x}_3(s_1,s_2)$ on $s_1$, the state is evaluated at the boundary $s_1=3$, which we can specify using the field \texttt{loc}. Doing so, we implement the third term in the first PDE,
\begin{align*}
\underbrace{-3}_{\texttt{C}} \underbrace{\int_{-1}^{1}}_{\texttt{I\{2\}}}\underbrace{\mbf{x}_{3}}_{\texttt{x}}(t,\underbrace{3,s_2}_{\text{loc}})ds_2
\end{align*}
as
\begin{matlab}
\begin{verbatim}
 >> PDE.x{1}.term{3}.x = 3;             PDE.x{1}.term{3}.C = -3;
 >> PDE.x{1}.term{3}.I{2} = [-1,1];     PDE.x{1}.term{3}.loc = [3,s2];
\end{verbatim}
\end{matlab}
Note here that we do not specify a value for \texttt{I\{1\}}. Leaving this field empty, we indicate that no integral should be taken along the direction of the first spatial variable $s_1$ on which the state component depends.

Now, to specify the second PDE,
\begin{align*}
 \dot{\mbf{x}}_2(t,s_1)&=\underbrace{s_1}_{\texttt{C}}\underbrace{x_1}_{\texttt{x}}(t) + \overbrace{\partial_{s_1}^{1}}^{\texttt{D}}\underbrace{\mbf{x}_2}_{\texttt{x}}(t,s_1) \underbrace{-}_{\texttt{C}} \underbrace{\int_{-1}^{1}}_{\texttt{I\{2\}}}\underbrace{\mbf{x}_3}_{\texttt{x}}(t,s_1,s_2) ds_2
\end{align*}
 we set
\begin{matlab}
\begin{verbatim}
 >> PDE.x{2}.term{1}.x = 1;         PDE.x{2}.term{1}.C = s1;  
 >> PDE.x{2}.term{2}.x = 2;         PDE.x{2}.term{2}.D = 1
 >> PDE.x{2}.term{3}.x = 3;         PDE.x{2}.term{3}.C = -1;                                               
 >> PDE.x{2}.term{3}.I{2} = [-1,1];
\end{verbatim}
\end{matlab}
where we note that, since the state component $\mbf{x}_2(s_1)$ depends only on a single spatial variable $s_1$, we need to specify only one order of the derivative \texttt{D=1}. Similarly, no order of a derivative should be specified for state component $x_1$ under any circumstance, as this state does not vary in space at all. 
Finally, to implement the third PDE,
\begin{align*}
 \dot{\mbf{x}}_3(t,s_1,s_2)&=\underbrace{s_1s_2}_{\texttt{C}} \underbrace{x_1}_{\texttt{x}}(t) + \underbrace{\mbf{x}_2}_{\texttt{x}}(t,s_1) + \overbrace{\partial_{s_1}^{0}\partial_{s_2}^{1}}^{\texttt{D}}\underbrace{\mbf{x}_3}_{\texttt{x}}(t,s_1,s_2),
\end{align*}
we set
\begin{matlab}
\begin{verbatim}
 >> PDE.x{3}.term{1}.x = 1;      PDE.x{3}.term{1}.C = s1*s2;     
 >> PDE.x{3}.term{2}.x = 2;
 >> PDE.x{3}.term{3}.x = 3;      PDE.x{3}.term{3}.D = [0,1];
\end{verbatim}
\end{matlab}
With that, we have specified the PDE, and it remains only to set the boundary conditions. We implement the first condition
\begin{align*}
0&=\underbrace{\mbf{x}_2}_{\texttt{x}}(t,\overbrace{3}^{\texttt{loc}})\underbrace{-}_{\texttt{C}}\underbrace{x_1}_{\texttt{x}}(t),
\end{align*}
as
\begin{matlab}
\begin{verbatim}
 >> PDE.BC{1}.term{1}.x = 2;       PDE.BC{1}.term{2}.x = 1;
 >> PDE.BC{1}.term{1}.loc = 3;     PDE.BC{1}.term{2}.C = -1;
\end{verbatim}
\end{matlab}
where we note that only a single position needs to be indicated in \texttt{loc} for state component $\mbf{x}_2(s_1)$, as this component depends on only one spatial variable. On the other hand, for the second boundary condition,
\begin{align*}
0&=\underbrace{\mbf{x}_3}_{\texttt{x}}(t,\overbrace{3,-1}^{\texttt{loc}}) \underbrace{-}_{\texttt{C}}\underbrace{x_1}_{\texttt{x}}(t-\underbrace{3}_{\texttt{delay}}),
\end{align*}
we do need to specify a spatial position for both $s_1$ and $s_2$ in state component $\mbf{x}_3$,
\begin{matlab}
\begin{verbatim}
 >> PDE.BC{2}.term{1}.x = 3;            PDE.BC{2}.term{2}.x = 1;
 >> PDE.BC{2}.term{1}.loc = [3,-1];     PDE.BC{2}.term{2}.C = -1;
 >>                                     PDE.BC{2}.term{2}.delay = 3;
\end{verbatim}
\end{matlab}
where we use the field \texttt{delay} to indicate that we wish to evaluate $x_{1}$ at time $t-1$.
Finally, for the third boundary condition
\begin{align*}
0&=\overbrace{\partial_{s_1}^{1}\partial_{s_2}^{0}}^{\texttt{D}}\underbrace{\mbf{x}_3}_{\texttt{x}}(t,\overbrace{s_1,-1}^{\texttt{loc}}) \underbrace{-}_{\texttt{C}}\underbrace{\mbf{x}_2}_{\texttt{x}}(t,\overbrace{s_1}^{\texttt{loc}}),
\end{align*}
we specify a location at which to evaluate both $\mbf{x}_1(s_1)$ and $\mbf{x}_2(s_1,s_2)$,
\begin{matlab}
\begin{verbatim}
 >> PDE.BC{3}.term{1}.x = 3;             PDE.BC{3}.term{2}.x = 2;
 >> PDE.BC{3}.term{1}.loc = [s1,-1];     PDE.BC{3}.term{2}.loc = s1;
 >> PDE.BC{3}.term{1}.D = [1,0];         PDE.BC{3}.term{2}.C = -1;
\end{verbatim}
\end{matlab}
Note that, although we are not evaluating $\mbf{x}_2(t,s_1)$ here at any boundary, we do specify its position as \texttt{loc=s1}, as it is good practice to always specify the position at which to evaluate (non-ODE) state components when declaring the boundary conditions.

Finally, having declared the full PDE, we run \texttt{initialize},
\begin{matlab}
\begin{verbatim}
>> PDE = initialize(PDE);

Encountered 3 state components: 
 x1(t),       of size 1, finite-dimensional;
 x2(t,s1),    of size 1, differentiable up to order (1) in variables (s1);
 x3(t,s1,s2), of size 1, differentiable up to order (1,1) in variables (s1,s2);

Encountered 3 boundary conditions: 
 F1(t) = 0,    of size 1;
 F2(t) = 0,    of size 1;
 F3(t,s1) = 0, of size 1;
\end{verbatim}
\end{matlab}
indicating that the state components and boundary conditions have been properly implemented.


\subsection{Declaring a System with (Delayed) Inputs and Outputs}\label{subsec:terms:input_output}

Suppose now we want to implement the following system
\begin{align*}
 \dot{\mbf{x}}(t,s_1,s_2)&=\partial_{s_1}\partial_{s_2}\mbf{x}(t,s_1,s_2)+(3-s_1)(1-s_2)(s_2+1)w(t)	&	(s_1,s_2)&\in[0,3]\times[-1,1]	\\
 z(t)&=10\int_{0}^{3}\int_{-1}^{1}\mbf{x}(t,s_1,s_2)ds_2 ds_1	\\
 \mbf{y}_1(t,s_1)&=\int_{-1}^{1}\mbf{x}(t,s_1,s_2)ds_2 + s_1 w(t)	\\
 \mbf{y}_2(t,s_2)&=\mbf{x}(t,0,s_2)	\\
 0&=\mbf{x}(t,0,-1)	\\
 0&=\partial_{s_1}\mbf{x}(t,s_1,-1) - u_1(t-0.5)	\\
 0&=\partial_{s_2}\mbf{x}(t,0,s_2) - \mbf{u}_2(t,s_2)	
\end{align*}
In this system, we have a finite-dimensional disturbance $w$, finite-dimensional regulated output $z$, infinite dimensional actuator inputs $\mbf{u}(t,s_1,s_2)=\bmat{u_1(t)\\\mbf{u}_2(t,s_2)}$, and infinite-dimensional observed outputs $\mbf{y}(t,s_1,s_2)=\bmat{\mbf{y}_1(t,s_1)\\\mbf{y}_2(t,s_2)}$, in addition to our infinite dimensional state $\mbf{x}(t,s_1,s_2)$. Here, we initialize the PDE structure as before as
\begin{matlab}
\begin{verbatim}
 >> PDE = pde_struct();
\end{verbatim}
\end{matlab}
and then implement the state as
\begin{matlab}
\begin{verbatim}
 >> pvar s1 s2
 >> PDE.x{1}.vars = [s1,s2];
 >> PDE.x{1}.dom = [0,3; -1,1];
\end{verbatim}
\end{matlab}
Similarly, we indicate the spatial variation of the different inputs and outputs as
\begin{matlab}
\begin{verbatim}
 >> PDE.z{1}.vars = [];      PDE.w{1}.vars = [];
 
 >> PDE.y{1}.vars = s1;      PDE.u{1}.vars = [];
 >> PDE.y{1}.dom = [0,3];
 >> PDE.y{2}.vars = s2;      PDE.u{2}.vars = s2;
 >> PDE.y{2}.dom = [-1,1];   PDE.u{2}.dom = [-1,1];
\end{verbatim}
\end{matlab}
where we set \texttt{PDE.z\{1\}.vars = [];} and \texttt{PDE.w\{1\}.vars = [];} to indicate that the regulated output and exogenous input are finite-dimensional.

Next, we implement the PDE. The first term here,
\begin{align*}
\overbrace{\partial_{s_1}^{1}\partial_{s_2}^{1}}^{\texttt{D}}\underbrace{\mbf{x}}_{\texttt{x}}(t,s_1,s_2)
\end{align*}
can be implemented simply as
\begin{matlab}
\begin{verbatim}
 >> PDE.x{1}.term{1}.x = 1;
 >> PDE.x{1}.term{1}.D = [1,1];
\end{verbatim}
\end{matlab}
where we use the field \texttt{x} to indicate that the term involves the first (and only) state component. Similarly, to add the term
\begin{align*}
 \underbrace{(3-s_1)(1-s_2)(s_2+1)}_{\texttt{C}}\underbrace{w}_{\texttt{w}}(t),
\end{align*}
we indicate that we are adding an input by using the field \texttt{w} as
\begin{matlab}
\begin{verbatim}
 >> PDE.x{1}.term{2}.w = 1;
 >> PDE.x{1}.term{2}.C = (3-s1)*(1-s2)*(s2+1);
\end{verbatim}
\end{matlab}
Note that, in any term, we can only specify one of the fields \texttt{x}, \texttt{w} and \texttt{u}, as any one term cannot involve both a state and an input, or both an exogenous and an actuator input. PIETOOLS will check whether a field \texttt{x}, \texttt{w}, or \texttt{u} is specified, and act accordingly. 
Note also that, since the inputs may describe arbitrary external forcings, PIETOOLS cannot take the derivative of an input or evaluate it at any particular position. As such, in any term involving an input \texttt{w} or \texttt{u}, we cannot specify an order of a derivative (\texttt{D}) or spatial position (\texttt{loc}).

Now, to specify the regulated output
\begin{align*}
 z&=\underbrace{10}_{\texttt{C}}\underbrace{\int_{0}^{3}}_{\texttt{I\{1\}}}\underbrace{\int_{-1}^{1}}_{\texttt{I\{2\}}}\underbrace{\mbf{x}}_{\texttt{x}}(t,s_1,s_2)ds_2 ds_1
\end{align*}
we simply add one term to \texttt{PDE.z\{1\}} as
\begin{matlab}
\begin{verbatim}
 >> PDE.z{1}.term{1}.x = 1;            PDE.z{1}.term{1}.C = 10;
 >> PDE.z{1}.term{1}.I{1} = [0,3];     PDE.z{1}.term{1}.I{2} = [-1,1];
\end{verbatim}
\end{matlab}
where we use \texttt{I\{1\}} and \texttt{I\{2\}} to define the domain of integration for the first and second variable on which the state component depends.

Similarly, to specify the equation for $\mbf{y}_1$,
\begin{align*}
\mbf{y}_1(s_1)&=\underbrace{\int_{-1}^{1}}_{\texttt{I\{2\}}}\underbrace{\mbf{x}}_{\texttt{x}}(t,s_1,s_2)ds_2 + \underbrace{s_1}_{\texttt{C}}\underbrace{w}_{\texttt{w}}(t),
\end{align*}
we set
\begin{matlab}
\begin{verbatim}
 >> PDE.y{1}.term{1}.x = 1;             PDE.y{1}.term{2}.w = 1;
 >> PDE.y{1}.term{1}.I{2} = [-1,1];     PDE.y{1}.term{2}.C = s1;
\end{verbatim}
\end{matlab}
this time specifying only an integral along the second spatial direction through \texttt{I\{2\}}.

Finally, we implement the second regulated output equation,
\begin{align*}
 \mbf{y}_2(s_2)&=\underbrace{\mbf{x}}_{\texttt{x}}(t,\overbrace{0,s_2}^{\texttt{loc}}),
\end{align*}
as
\begin{matlab}
\begin{verbatim}
 >> PDE.y{2}.term{1}.x = 1;
 >> PDE.y{2}.term{1}.loc = [0,s2];
\end{verbatim}
\end{matlab}
This leaves only the boundary conditions, the first of which
\begin{align*}
 0&=\underbrace{\mbf{x}}_{\texttt{x}}(t,\overbrace{0,-1}^{\texttt{loc}}),
\end{align*}
can be easily implemented as
\begin{matlab}
\begin{verbatim}
 >> PDE.BC{1}.term{1}.x = 1;
 >> PDE.BC{1}.term{1}.loc = [0,-1];
\end{verbatim}
\end{matlab}
Next, for the second boundary condition
\begin{align*}
 0&=\overbrace{\partial_{s_1}^1\partial_{s_2}^{0}}^{\texttt{D}}\underbrace{\mbf{x}}_{\texttt{x}}(t,\overbrace{s_1,-1}^{\texttt{loc}}) \underbrace{-}_{\texttt{C}}\underbrace{u_1}_{\texttt{u}}(t-\underbrace{0.5}_{\texttt{delay}})
\end{align*}
we specify the two terms as
\begin{matlab}
\begin{verbatim}
 >> PDE.BC{2}.term{1}.x = 1;           PDE.BC{2}.term{2}.u = 1;
 >> PDE.BC{2}.term{1}.D = [1,0];       PDE.BC{2}.term{2}.C = -1;
 >> PDE.BC{2}.term{1}.loc = [s1,-1];   PDE.BC{2}.term{2}.delay = 0.5;
\end{verbatim}
\end{matlab}
Here, although we have to specify the spatial position at which to evaluate the state $\mbf{x}$, we do not specify the spatial position of the input $\mbf{u}_1$, as we cannot evaluate an input at any specific position. Similarly, the third boundary condition
\begin{align*}
 0&=\overbrace{\partial_{s_1}^{0}\partial_{s_2}^{1}}^{\texttt{D}}\underbrace{\mbf{x}}_{\texttt{x}}(t,\overbrace{0,s_2}^{\texttt{loc}}) \underbrace{-}_{\texttt{C}}\underbrace{\mbf{u}_2}_{\texttt{u}}(t,s_2),
\end{align*}
can be implemented as
\begin{matlab}
\begin{verbatim}
 >> PDE.BC{3}.term{1}.x = 1;           PDE.BC{3}.term{2}.u = 2;
 >> PDE.BC{3}.term{1}.D = [0,1];       PDE.BC{3}.term{2}.C = -1;
 >> PDE.BC{3}.term{1}.loc = [0,s2];
\end{verbatim}
\end{matlab}
Having declared the full system, we can then call \texttt{initialize},
\begin{matlab}
\begin{verbatim}
>> PDE = initialize(PDE);

Encountered 1 state component: 
 x(t,s1,s2),    of size 1, differentiable up to order (1,1) in variables (s1,s2);

Encountered 2 actuator inputs: 
 u1(t),          of size 1;
 u2(t,s2),       of size 1;

Encountered 1 exogenous input: 
 w(t),    of size 1;

Encountered 2 observed outputs: 
 y1(t,s1),       of size 1;
 y2(t,s2),       of size 1;

Encountered 1 regulated output: 
 z(t),    of size 1;

Encountered 3 boundary conditions: 
 F1(t) = 0,      of size 1;
 F2(t,s1) = 0,   of size 1;
 F3(t,s2) = 0,   of size 1;
\end{verbatim}
\end{matlab}
indicating that the state, inputs, and outputs have been correctly declared.

\begin{boxEnv}{\textbf{Note:}}
Although the terms-based input format allows infinite-dimensional inputs and outputs to be specified in the PDE structure (as we did here for $\mbf{u}$ and $\mbf{y}$), PIETOOLS 2022 may not always be able to convert such PDEs to PIEs. This functionality will be included in a later release.
\end{boxEnv}


\subsection{Additional Options}\label{subsec:terms:additional_opts}

Using the methods from the previous sections, a wide variety of ODE-PDE input-output systems can be implemented in PIETOOLS. However, for those users interested in implementing even more complicated (linear) systems, there are a few additional options for PDEs that PIETOOLS allows. In particular, in defining linear systems, PIETOOLS also allows partial integrals of state variables and inputs to be added to the dynamics, using the field \texttt{I}. In addition, PDEs with higher-order temporal derivatives can be declared using the field \texttt{tdiff}. 
Finally, the inherent continuity constraints imposed for PDEs in PIETOOLS can also be adjusted, using the field \texttt{diff}. 

\bigskip

\subsubsection{Partial Integrals and Kernel Functions}

In the previous subsections, we showed that an integral $\int_{a_k}^{b_k}\mbf{x}(s_1,\hdots,s_n)ds_k$ of a state $\mbf{x}$ over $s_k\in[a_k,b_k]$ could be specified in the PDE structure by setting \texttt{I\{k\}=[ak,bk]}. Integrating over the entire domain of the variable $s_k\in[a_k,b_k]$, this integral will remove the dependence of the state $\mbf{x}$ on the variable $s_k$. However, rather than performing a full integral, we can also specify a partial integral, integrating over either $[a_k,s_k]$ or $[s_k,b_k]$. For example, to add a term
\begin{align*}
 \underbrace{\int_{s_1}^{3}}_{\texttt{I\{1\}}}\underbrace{\int_{-1}^{s_2}}_{\texttt{I\{2\}}}\underbrace{\mbf{x}_{j}}_{\texttt{x}}(t,\theta_1,\theta_2)d\theta_2 d\theta_1
\end{align*}
to the PDE of the $i$th state variable, we can set
\begin{matlab}
\begin{verbatim}
 >> PDE.x{i}.term{l}.x = j;
 >> PDE.x{i}.term{l}.I{1} = [s1,3];
 >> PDE.x{i}.term{l}.I{2} = [-1,s2];
\end{verbatim}
\end{matlab}
Similarly, a term 
\begin{align*}
\underbrace{10s_2}_{\texttt{C}}\underbrace{\int_{s_2}^{1}}_{\texttt{I\{2\}}}\overbrace{\partial_{s_1}^{1}\partial_{s_2}^{0}}^{\texttt{D}}\underbrace{\mbf{x}_{j}}_{\texttt{x}}(t,\overbrace{3,\theta_2}^{\texttt{loc}})d\theta_2
\end{align*}
can be added to the $i$th boundary condition as
\begin{matlab}
\begin{verbatim}
 >> PDE.BC{i}.term{l}.x = j;
 >> PDE.BC{i}.term{l}.D = [1,0];
 >> PDE.BC{i}.term{l}.loc = [3,s2];
 >> PDE.BC{i}.term{l}.I{2} = [s2,1];
 >> PDE.BC{i}.term{l}.C = 10*s2;
\end{verbatim}
\end{matlab}
Note here that, although a dummy variable $\theta_2$ is introduced in the mathematical representation of the integral $\int_{s_2}^{1}\mbf{x}(t,3,\theta_2) d\theta$, the location \texttt{loc=[3,s2]} at which we evaluate the state can still be specified using the primary variable $s_2$. However, for the factor $C(s_2)=10s_2$ with which to multiply the state, the use of the primary variable $s_2$ ensures that this factor will not be integrated itself. Suppose now that we wish for $C$ to act as a kernel instead, adding e.g. a term
\begin{align*}
\int_{s_2}^{1}10\theta_2 \partial_{s_1}^{1}\mbf{x}_{j}(t,3,\theta_2)d\theta_2
\end{align*}
To implement this term, we first have to specify dummy variables for the state component. This can be done using the second column of the field \texttt{x\{j\}.vars}, specifying dummy variables $(\theta_1,\theta_2)$ for the $j$th state component as
\begin{matlab}
\begin{verbatim}
 >> pvar s1 s2 theta1 theta2
 >> PDE.x{j}.vars = [s1,theta1; s2,theta2];
 >> PDE.x{j}.dom = [0,3; -1,1]
\end{verbatim}
\end{matlab}
Here, each dummy variable $\theta_i$ is linked to a particular spatial variable $s_i$ by specifying the two variables in the same \textbf{row} of the field \texttt{vars}. The domain of the dummy variable will then be the same as that of its associated primary variable, as specified in the corresponding row of \texttt{dom}. 

Having specified the dummy variables, we can now add the term
\begin{align*}
\underbrace{\int_{s_2}^{1}}_{\texttt{I\{2\}}}\underbrace{10\theta_2}_{\texttt{C}}\overbrace{\partial_{s_1}^{1}\partial_{s_2}^{0}}^{\texttt{D}}\underbrace{\mbf{x}_{j}}_{\texttt{x}}(t,\overbrace{3,\theta_2}^{\texttt{loc}})d\theta_2
\end{align*}
to the $i$th boundary condition by specifying
\begin{matlab}
 \texttt{>> PDE.BC\{i\}.term\{l\}.x = j;}	\\
 \texttt{>> PDE.BC\{i\}.term\{l\}.D = [1,0];}	\\
 \texttt{>> PDE.BC\{i\}.term\{l\}.loc = [3,theta2];}	\hspace*{0.45cm} \text{or} \hspace*{0.45cm}
 \texttt{PDE.BC\{i\}.term\{l\}.loc = [3,s2];}	\\
 \texttt{>> PDE.BC\{i\}.term\{l\}.I\{2\} = [s2,1];}	\\
 \texttt{>> PDE.BC\{i\}.term\{l\}.C = 10*theta2;}	
\end{matlab}
Here, we can choose to set either \texttt{loc=[3,theta2]} or \texttt{loc=[3,s2]}, as the field \texttt{I\{2\}=[s2,1]} ensures that both options will be recognized as integrating the state over the domain $[s_2,1]$. However, dummy variables can only be used (in the fields \texttt{loc} and \texttt{C}) if integration is performed along the associated spatial direction. For example, a term
\begin{align*}
\underbrace{\int_{0}^{s_1}}_{\texttt{I}\{1\}}\underbrace{s_2(s_1-\theta_1)}_{\texttt{C}}\overbrace{\partial_{s_1}^{1}\partial_{s_2}^{2}}^{\texttt{D}}\underbrace{\mbf{x}_{j}}_{\texttt{x}}(t,\overbrace{\theta_1,-1}^{\texttt{loc}})d\theta_1
\end{align*}
could also be added to e.g. the PDE of the $i$th state component as
\begin{matlab}
 \texttt{>> PDE.BC\{i\}.term\{l\}.x = j;}	\\
 \texttt{>> PDE.BC\{i\}.term\{l\}.D = [1,2];}	\\
 \texttt{>> PDE.BC\{i\}.term\{l\}.loc = [theta1,-1];}	\hspace*{0.45cm} \text{or}	\hspace*{0.45cm}
 \texttt{PDE.BC\{i\}.term\{l\}.loc = [s1,-1];}	\\
 \texttt{>> PDE.BC\{i\}.term\{l\}.I\{1\} = [0,s1];}	\\
 \texttt{>> PDE.BC\{i\}.term\{l\}.C = s2*(s1-theta1);}	
\end{matlab}
in which case $\theta_2$ cannot be used, since no integration is performed along the second spatial direction. However, dummy variables can also be used to implement kernels for full integrals, adding e.g. a term
\begin{align*}
\underbrace{\int_{0}^{s_1}}_{\texttt{I\{1\}}}\underbrace{\int_{-1}^{1}}_{\texttt{I\{2\}}}\underbrace{\theta_2(s_1-\theta_1)}_{\texttt{C}}\overbrace{\partial_{s_1}^{1}\partial_{s_2}^{2}}^{\texttt{D}}\underbrace{\mbf{x}_{j}}_{\texttt{x}}(t,\theta_1,\theta_2)d\theta_2 d\theta_1
\end{align*}
as
\begin{matlab}
\begin{verbatim}
 >> PDE.BC{i}.term{l}.x = j;
 >> PDE.BC{i}.term{l}.D = [1,2];
 >> PDE.BC{i}.term{l}.I{1} = [0,s1];
 >> PDE.BC{i}.term{l}.I{2} = [-1,1];
 >> PDE.BC{i}.term{l}.C = theta2*(s1-theta1);
\end{verbatim}
\end{matlab}
where in this case, no location needs to be specified, as the state is not evaluated at any boundary.

\begin{boxEnv}{\textbf{Note:}}
In specifying the domain of integration for a variable $s_k\in[a_k,b_k]$, only the full domain \texttt{I\{k\}=[ak,bk]} or the partial domains \texttt{I\{k\}=[ak,sk]} or \texttt{I\{k\}=[sk,bk]} are supported. Specifying any other domain for the integral will produce an error.
\end{boxEnv}

\subsubsection{Higher Order Temporal Derivatives}

In the previous sections, only PDEs involving first-order temporal derivatives have been implemented, using the element of \texttt{x\{i\}.term} to declare the terms in the equation for $\dot{\mbf{x}}_i$. However, PDEs involving higher order temporal derivatives can also be specified, by adjusting the value of \texttt{x\{i\}.tdiff}. For example, suppose we want to declare the wave equation
\begin{align*}
    \ddot{\mbf{x}}(t,s)&=\partial_{s}^2\mbf{x}(t,s), &   s&\in[0,1],\quad t\geq 0,   \\
    \mbf{x}(t,0)&=0,\qquad \mbf{x}(t,1)&=0,
\end{align*}
where $\mbf{x}\in L_2[0,1]$. We initialize the PDE structure and the state $\mbf{x}$ as before,
\begin{matlab}
\begin{verbatim}
 >> PDE = pde_struct();
 >> pvar s
 >> PDE.x{1}.vars = s;
 >> PDE.x{1}.dom = [0,1];
\end{verbatim}
\end{matlab}
Then, before specifying the terms in the PDE for this state component $\mbf{x}_1=\mbf{x}$, we indicate that this PDE concerns a second order temporal derivative of $\mbf{x}$, using the field \texttt{tdiff}:
\begin{matlab}
\begin{verbatim}
 >> PDE.x{1}.tdiff = 2;
\end{verbatim}
\end{matlab}
Note that if no field \texttt{tdiff} is specified (as in the previous sections), its value will default to 1. Having specified the desired order of the temporal derivative, we can declare the PDE and BCs as
\begin{matlab}
\begin{verbatim}
 >> PDE.x{1}.term{1}.D = 2;
 >> PDE.BC{1}.term{1}.loc = 0;
 >> PDE.BC{2}.term{1}.loc = 1;
\end{verbatim}
\end{matlab}
and we can initialize
\begin{matlab}
\begin{verbatim}
 >> PDE = initialize(PDE);
 
Encountered 1 state component: 
 x(t,s),    of size 1, differentiable up to order (2) in variables (s);

Encountered 2 boundary conditions: 
 F1(t) = 0, of size 1;
 F2(t) = 0, of size 1;
\end{verbatim}
\end{matlab}
We note that, although higher order temporal derivatives can be specified in the PDE format, the PIE representation cannot represent such higher order temporal derivatives. Therefore, when calling \texttt{convert(PDE,'pie')}, the higher order temporal derivatives will first be expanded, expressing the PDE as a system involving only first order temporal derivatives. This is done using the function \texttt{expand\_tderivatives}, which can also be called to manually expand the delays
\begin{matlab}
\begin{verbatim}
 >> PDE = expand_tderivatives(PDE)
 
 Added 1 state component: 
    x2(t,s) := (d/dt) x1(t,s)
\end{verbatim}
\end{matlab}
Introducing the new state component $x_{2}(t,s)=\dot{x}_1(t,s)$, PIETOOLS then expresses the system as a PDE involving only first-order temporal derivatives as
\begin{align*}
    \dot{\mbf{x}}_1(t,s)&=\mbf{x}_2(t,s),       &   s&\in[0,1],\quad t\geq 0,   \\
    \dot{\mbf{x}}_2(t,s)&=\partial_{s}^2\mbf{x}_1(t,s),     \\
    \mbf{x}_1(t,0)&=0,\qquad \mbf{x}_1(t,1)=0.
\end{align*}
\begin{boxEnv}{\textbf{Note:}}
 Expanding higher order temporal derivatives using \texttt{expand\_tderivatives}, no boundary conditions will be imposed upon the newly introduced state components (such as $x_{2}(t,s)=\dot{x}_1(t,s)$). Boundary conditions can be imposed upon the new states by manually declaring them in terms-based format. Results from e.g. stability tests may vary depending on whether or not boundary conditions are imposed upon the newly introduced state components.
\end{boxEnv}



\subsubsection{Continuity Constraints}\label{sec:subsec_diff}

In any PDE, there are inherent continuity constraints imposed upon the PDE state, requiring the state to be differentiable in each spatial variable up to some order. For example, for a PDE
\begin{align*}
 \dot{\mbf{x}}(t,s_1,s_2)&=\partial_{s_1}^2\mbf{x}(t,s_1,s_2) +\partial_{s_2}\mbf{x}(t,s_1,s_2)
\end{align*}
the state $\mbf{x}$ must be at least second order differentiable with respect to $s_1$, and first order differentiable with respect to $s_2$. Specifying this system in PIETOOLS, a continuity constraint will automatically be imposed upon the PDE state as
\begin{align*}
 \partial_{s_1}^2\partial_{s_2}\mbf{x}(t,s_1,s_2)\in L_2.
\end{align*}
Accordingly, a well-posed solution to the PDE will require (at least) two boundary conditions to be specified along boundaries in the first spatial dimension (e.g. $s_1=0$ and/or $s_1=3$), and one boundary condition to be specified along boundaries in the second spatial dimension (e.g. $s_2=-1$ or $s_2=1$). Moreover, at the boundaries along the first spatial direction, we can only evaluate first order derivatives of the state with respect to $s_1$, and at the boundaries along the second spatial direction, the we cannot differentiate the state with respect to $s_2$ at all. That is, where the boundary conditions
\begin{align*}
0&=\partial_{s_1} \mbf{x}(t,0,s_2)	&
0&=\mbf{x}(t,s_1,-1)
\end{align*}
are permitted, the boundary conditions
\begin{align*}
 0&=\partial_{s_1}^2 \mbf{x}(t,0,s_2)	&
 0&=\partial_{s_2}\mbf{x}(t,s_1,-1)
\end{align*}
are prohibited.

In some cases, however, it may be desirable for a state component to be differentiable up to a greater order than the maximal order of the derivative appearing in the PDE. In PIETOOLS, this can be indicated in the PDE structure through the optional field \texttt{PDE.x\{i\}.diff}, allowing the order of spatial differentiability of the $i$th PDE state component to be explicitly specified as
\begin{matlab}
\begin{verbatim}
 >> PDE.x{i}.diff = [dmax1, dmax2];
\end{verbatim}
\end{matlab}
where \texttt{dmax1} denotes the maximal order of differentiability in $s_1$, and \texttt{dmax2} the maximal order of differentiability in $s_2$. Note that this order of differentiability cannot be smaller than the maximal order of the spatial derivatives of the $i$th PDE state in any term of the system. That is, specifying a term such as
\begin{matlab}
\begin{verbatim}
 >> PDE.x{k}.term{j}.x = i;   PDE.x{k}.term{j}.D = [dmax1+1, dmax2];
\end{verbatim}
\end{matlab}
PIETOOLS will automatically update the order of differentiability of the $i$th state component to \texttt{PDE.x\{i\}.diff=[dmax1+1, dmax2]}. Furthermore, the number of columns of the field \texttt{diff} should always match the number of rows of the field \texttt{vars}, specifying an order of differentiability for each spatial variable on which the state component depends. For example, if a state component $\mbf{x}_j(s_2)$ varies only in a single variable $s_2$, only a single order of differentiability can be specified as
\begin{matlab}
\begin{verbatim}
 >> PDE.x{j}.vars = s2;       PDE.x{j}.diff = dmax2;
\end{verbatim}
\end{matlab}

\section{The Command Line Parser format Format}\label{subsec:additional-command-line}
In this subsection, we explain how defining and manipulating ODE-PDE systems revolves around two main classes: \texttt{state} and \texttt{sys}. Furthermore, we will also specify valid modes of manipulating these objects in MATLAB and potential caveats while using these objects.
\subsection{\texttt{state} class objects}\label{ss-sec:state}
All symbols used to define a systems are either \texttt{polynomial} type (part of SOSTOOLS) or \texttt{state} type (part of PIETOOLS). Here, we will focus on \texttt{state} class objects and methods defined for such objects. First, any \texttt{state} class object has the following properties that can be freely accessed (but \textbf{should not} be modified directly).
\begin{Statebox}{\texttt{state}}
This class has the following properties:
\begin{enumerate}
\item \texttt{type}: Type of variable; It is a cell array of strings that can take values in $\{'ode'$ $,'pde'$ $,'out'$ $,'in'\}$
\item \texttt{veclength}: Positive integer 
\item \texttt{var}: Cell array of polynomial row vectors (Multipoly class object)
\item \texttt{diff\_order}: Cell array of non-negative integers (same size as var)
\end{enumerate}
\end{Statebox}

The first independent variable stored in each row of the \texttt{state.var} cell structure is always the time variable \texttt{t}. Spatial variables are stored in location 2 and on-wards. For example,
\begin{matlab}
>> X = state('pde'); x = state('ode');\\
>> X.var
\begin{verbatim}
ans =
  [ t, s]
\end{verbatim}
>> x.var
\begin{verbatim}
ans =
  [ t ]
\end{verbatim}
\end{matlab}

Differentiation information is stored as a cell array where the cell structure has the same size as \texttt{state.var} with non-negative integers specifying order of differentiation w.r.t. the independent variable based on the location.
For the above example, we have
\begin{matlab}
>> X.diff\_order
\begin{verbatim}
ans =
  [ 0, 0]
\end{verbatim}
>> y = diff(X,s,2);\\
>> y.diff\_order
\begin{verbatim}
ans =
  [ 0, 2]
\end{verbatim}
\end{matlab}

Note, user can indeed edit these properties directly by assignment. For example, the code
\begin{matlab}
>> x = state('pde');\\
>> x.diff\_order = [0,2];
\end{matlab}
defines the symbol \texttt{x} as a function $x(t,s)$, and converts it to the second derivative $\partial_s^2 x(t,s)$. This is same as the code
\begin{matlab}
>> x = state('pde'); \\
>> x = diff(x,s,2);
\end{matlab}
Since, this permanently changes \texttt{x} to its second spatial derivative in the workspace, such direct manipulation of the properties should be avoided at all costs. 

\paragraph{Declaring/initializing state variables} 
The initialization function \texttt{state()} takes two input arguments (both are optional):
\begin{itemize}
    \item type: The argument is reserved to specification of the type of the state object (defaults to 'ode', if not specified)
    \item veclength: The size of the vector-valued state (defaults to one, if not specified)
\end{itemize}
\begin{matlab}
d = state('pde',3);
\end{matlab}
Alternatively, multiple states can be defined collectively using the command shown below, however, all such states will default to the type 'ode' and length 1.
\begin{matlab}
state ~a ~b ~c;
\end{matlab}

\paragraph{Operations on state class objects} All of the following operations should give us a \texttt{terms} (an internal class that cannot be accessed or modified by users) class object which is defined by some PI operator times a vector of states. Operators/functions that are used to manipulate \texttt{state} objects are:
\begin{enumerate}
\item addition: \texttt{x+y} or  \texttt{x-y}  
\item multiplication: \texttt{K*x}
\item vertical concatentation: \texttt{[x;y]}
\item differentiation: \texttt{diff(x,s,3)}
\item integration: \texttt{int(x,s,[0,s])}
\item substitution: \texttt{subs(x,s,0)}
\end{enumerate}

\vfill
\paragraph{Caveats in operations on state class objects} While manipulation of state class objects, the users must adhere the following rules stated in the table \ref{tab:invalid_state_operations}. All the operations listed in the table are invalid.

Addition of time derivatives is not allowed, since that usually leads to a descriptor dynamical PDE system which is not supported by PIETOOLS.
For example, consider the following PDE.
\begin{align*}
    \dot{\mbf x}(t) + \dot{\mbf y}(t) &=\partial_s^2\mbf x(t,s)\\
    2\dot{\mbf y}(t) &=5\partial_s^2\mbf y(t,s).
\end{align*}
\textbf{This PDE cannot be implemented directly using the command line parser.}
Since, the left hand side of the equation has a coefficient different from identity, the user needs to first separate it as
\begin{align*}
    \dot{\mbf x}(t) &=\partial_s^2\mbf x(t,s)-2.5\partial_s^2\mbf y(t,s)\\
    \dot{\mbf y}(t) &=2.5\partial_s^2\mbf y(t,s).
\end{align*}
Now, we can define this PDE using the following code:
\begin{verbatim}
\begin{matlab}
>> pvar ~t ~s ~theta; \\
>> x = state('pde'); ~y = state('pde');\\ 
>> odepde= sys(); \\
>> odepde = addequation(odepde, diff(x,t)-diff(x,s,2)-2.5*diff(y,s,2));\\
>> odepde = addequation(odepde, diff(y,t)-2.5*diff(y,s,2));
\end{matlab}
\end{verbatim}

Likewise, we do not permit adding outputs with outputs, outputs with time derivatives, or right multiplication which also lead to descriptor type systems. Coupling on left hand side of these equations must be manually resolved before defining the PDE in PIETOOLS.

Other limitations to note are, PIETOOLS does not support temporal-spatial mixed derivatives, integration in time, and evaluation of functions at specific time or inside the spatial domain. For example, for a state $x(t,s)$ with $s\in[0,1]$ we cannot find $x(t=2,s)$ or $x(t,s=0.5)$. $x$ can only be evaluated at the boundary $s=0$ or $s=1$. 

\begin{boxEnv}{Note:}
While PIETOOLS terms format (introduced in Section~\ref{sec:alt_PDE_input:terms_input_PDE}) supports higher order temporal derivatives, the command line parser does not support it currently.
\end{boxEnv}

\begin{table}
\renewcommand{\arraystretch}{0.3}
\begin{tabular}{ | m{8em} | m{12cm} |} 
\hline 
\vspace*{0.15cm} Operation type & 
\vspace*{0.15cm} Incorrect or `not-permitted' operations\\[0.6em]
  \hline
 Addition & \begin{enumerate}[label=\textcolor{red}{\ding{54}}]
\item Adding two time derivatives: \texttt{ diff(x,t)+diff(x,t)}
\item Adding two outputs: \texttt{z1+z2}
\item Adding time derivative and outptu: \texttt{ diff(x,t)+z}
\end{enumerate}\\
  \hline
 Multiplication&\begin{enumerate}[label=\textcolor{red}{\ding{54}}]
\item Multiplying two states: \texttt{x*x}
\item Multiplying non-identity with time derivative/output: \texttt{ 2*diff(x,t)} or \texttt{ -1*z}
\item Right multiplication: \texttt{x*3} 
\end{enumerate}\\
  \hline
  Differentiation&\begin{enumerate}[label=\textcolor{red}{\ding{54}}]
\item Higher order time derivatives:\texttt{ diff(x,t,2)}
\item Mixed derivatives of space and time: \texttt{ diff(diff(x,t),s,2)}
\end{enumerate}\\
  \hline
  Substitution&\begin{enumerate}[label=\textcolor{red}{\ding{54}}]
\item Substituting a double for time variable: \texttt{ subs(x,t,2)}
\item Substituting positive time delay: \texttt{ subs(x,t,t+5)}
\item Substitution values other than \texttt{pvar} variable or boundary values
\end{enumerate}\\
  \hline
  Integration&\begin{enumerate}[label=\textcolor{red}{\ding{54}}]
\item Integration of time variable: \texttt{ int(x,t,[0,5])}
\item both limits being non-numeric: \texttt{ int(x,s,[theta,eta])}
\item limit same as variable of integration: \texttt{ int(x,s,[s,1])}
\end{enumerate}\\
  \hline
  Concatenation&
  \begin{enumerate}[label=\textcolor{red}{\ding{54}}]
    \item Horizontal concatenation: \texttt{ [x,x]}
    \item Blank spaces in vertical concatentation: \texttt{[x + y; z]}
\end{enumerate}\\
\hline
\end{tabular}
\caption{This table lists all the invalid forms of operations on \texttt{state} class objects. The left column specifies the type of operation whereas the right column lists the operations that are \textbf{INVALID} for that `type' of operation.}
\label{tab:invalid_state_operations}
\end{table}

\subsection{\texttt{sys} class objects}\label{subsec:sys}

\begin{Statebox}{\texttt{sys}}
This class has following accessible properties:
\begin{itemize}
    \item \texttt{equation}: stores all the equations added to the system object in a column vector where every row is an equation with zero on the right hand side (i.e., \texttt{row(i)=0} for every \texttt{i})
    \item \texttt{type}: type of the system (currently supports `pde' and `pie')
    \item \texttt{params}: either a \texttt{pde\_struct} or \texttt{pie\_struct} object
    \item \texttt{dom}: a $1\times 2$ vector double (value of first element must be strictly smaller than that of second element)
    \item Other hidden properties:
    \begin{enumerate}
        \item \texttt{states}: a vector of all states, inputs, outputs appearing in the equation property    
        \item \texttt{ControlledInputs}: A vector with length same as the states property with 0 or 1 value. This vector specifies whether a state is a controlled input or not.
        \item \texttt{ObservedOutputs}: A vector with length same as the states property with 0 or 1 value. This specifies whether a state is an observed output or not.
    \end{enumerate}
\end{itemize}
\end{Statebox}

\paragraph{\texttt{sys} class methods} Methods used to modify a \texttt{sys()} object are listed below.
\begin{itemize}
\item \texttt{addequation}: adds an equation to the \texttt{obj.equation} property; syntax \texttt{addequation(obj, eqn)}
\item \texttt{removeequation}: removes equation in row \texttt{i} from the \texttt{obj.equation} property; syntax \texttt{removeequation(obj,i)}
\item \texttt{setControl}: sets a chosen state \texttt{x} as a control input; syntax \texttt{setControl(obj,x)}
\item \texttt{setObserve}: sets a chosen state \texttt{x} as an observed output; syntax \texttt{setObserve(obj,x)}
\item \texttt{removeControl}: removes a chosen state \texttt{x} from the set of control inputs; syntax

\texttt{removeControl(obj,x)}
\item removeObserve: removes a chosen state \texttt{x} from the set of observed outputs; syntax \texttt{removeObserve(obj,x)}
\item getParams: parses symbolic equations from \texttt{obj.equation} property to get \texttt{pde\_struct} object which is stored in \texttt{obj.params}; syntax \texttt{getParams(obj)}
\item convert: converts \texttt{obj.params} from \texttt{pde\_struct} to \texttt{pie\_struct} object; syntax 

\texttt{convert(obj,'pie')}
\end{itemize}

\begin{boxEnv}{\textbf{WARNING:}}
\texttt{sys} class object properties should not be modified directly (unless you know what you are doing); Use the methods provided above.
\end{boxEnv}

\chapter{Batch Input Formats for Time-Delay Systems}\label{ch:alt_DDE_input}

In Chapter~\ref{ch:PDE_DDE_representation}, we showed how time-delay systems (TDSs) can be implemented as delay differential equations (DDEs) in PIETOOLS. In that chapter, we further hinted at the fact that PIETOOLS also allows TDSs to be declared in two alternative representations: as neutral type systems (NDSs) and as differential difference equations (DDFs). In this chapter, we will provide more details on how to work with such NDS and DDF systems in PIETOOLS. In particular, in Section~\ref{sec:alt_DDE_input:TDS_formats}, we recall the DDE representation, and show what NDS and DDF systems look like, and how systems of each type can be declared in PIETOOLS. In Section~\ref{ch:alt_DDE_inputs:convert}, we then show how NDS and DDE systems can be converted to the DDF representation in PIETOOLS, and how each type of TDS can be converted to a PIE.

\section{Representing Systems with Delay}\label{sec:alt_DDE_input:TDS_formats}
In this section, we show how time-delay systems in DDE, NDS and DDF representation can be declared in PIETOOLS, focusing on DDE systems in Subsection~\ref{subsec:DDE_format}, NDS systems in Subsection~\ref{subsec:NDS_format}, and DDF systems in Subsection~\ref{subsec:DDF_format}. For more information on how to declare systems in DDE representation in PIETOOLS, we refer to Section~\ref{sec:PDE_DDE_representation:DDEs}

\subsection{Input of Delay Differential Equations}\label{subsec:DDE_format}

The DDE data structure allows the user to declare any of the matrices in the following general form of Delay-Differential equation.
\begin{align}
	&\bmat{\dot{x}(t)\\z(t) \\ y(t)}=\bmat{A_0 & B_{1} & B_{2}\\ C_{1} & D_{11} &D_{12}\\ C_{2} & D_{21} &D_{22}}\bmat{x(t)\\w(t)\\u(t)}+\sum_{i=1}^K \bmat{A_i & B_{1i} & B_{2i}\\C_{1i} & D_{11i} & D_{12i}\\C_{2i} & D_{21i} & D_{22i}} \bmat{x(t-\tau_i)\\w(t-\tau_i)\\u(t-\tau_i)}\notag \\
& \hspace{2cm}+\sum_{i=1}^K \int_{-\tau_i}^0\bmat{A_{di}(s) & B_{1di}(s) &B_{2di}(s)\\C_{1di}(s) & D_{11di}(s) & D_{12di}(s)\\C_{2di}(s) & D_{21di}(s) & D_{22di}(s)} \bmat{x(t+s)\\w(t+s)\\u(t+s)}ds \label{eqn:DDE}
\end{align}
In this representation, it is understood that
\begin{itemize}
\item The present state is $x(t)$.\vspace{-2mm}
\item The disturbance or exogenous input is $w(t)$. These signals are not typically known or alterable. They can account for things like unmodelled dynamics, changes in reference, forcing functions, noise, or perturbations.\vspace{-2mm}
\item The controlled input is $u(t)$. This is typically the signal which is influenced by an actuator and hence can be accessed for feedback control. \vspace{-2mm}
\item The regulated output is $z(t)$. This signal typically includes the parts of the system to be minimized, including actuator effort and states. These signals need not be measured using senors.\vspace{-2mm}
\item The observed or sensed output is $y(t)$. These are the signals which can be measured using sensors and fed back to an estimator or controller.\vspace{-2mm}
\end{itemize}
To add any term to the DDE structure, simply declare is value. For example, to represent 
\[
\dot x(t)=-x(t-1),\qquad z(t)=x(t-2)
\]
we use
	\begin{matlab}
		>> DDE.tau = [1 2];\\
		>> DDE.Ai\{1\} = -1;\\
		>> DDE.C1i\{2\} = 1;
	\end{matlab}
All terms not declared are assumed to be zero. The exception is that we require the user to specify the values of the delay in \texttt{DDE.tau}. When you are done adding terms to the DDE structure, use the function \texttt{DDE=PIETOOLS\_initialize\_DDE(DDE)}, which will check for undeclared terms and set them all to zero. It also checks to make sure there are no incompatible dimensions in the matrices you declared and will return a warning if it detects such malfeasance. The complete list of terms and DDE structural elements is listed in Table~\ref{tab:DDE_parameters}.

\begin{table}[ht!]\vspace{-2mm}
\begin{center}{
\begin{tabular}{c|c||c|c||c|c}
  \multicolumn{6}{c}{\textbf{ODE Terms:}}\\
Eqn.~\eqref{eqn:DDE}  & \texttt{DDE.}  &   Eqn.~\eqref{eqn:DDE}  & \texttt{DDE.} &   Eqn.~\eqref{eqn:DDE}  & \texttt{DDE.}\\
\hline
$A_0$    & \texttt{A0} & $B_{1}$ & \texttt{B1} &$B_{2}$&\texttt{B2}\\ 
$C_{1}$ & \texttt{C1} & $D_{11}$ &\texttt{D11}&$D_{12}$&\texttt{D12}\\ 
$C_{2}$ & \texttt{C2} & $D_{21}$ &\texttt{D21}&$D_{22}$&\texttt{D22} \\ \hline \\
 \multicolumn{6}{c}{ \textbf{Discrete Delay Terms:}}\\
Eqn.~\eqref{eqn:DDE}  & \texttt{DDE.}  &   Eqn.~\eqref{eqn:DDE}  & \texttt{DDE.} &   Eqn.~\eqref{eqn:DDE}  & \texttt{DDE.}\\
\hline
$A_i$    & \texttt{Ai\{i\}} & $B_{1i}$ & \texttt{B1i\{i\}} &$B_{2i}$&\texttt{B2i\{i\}}\\ 
$C_{1i}$ & \texttt{C1i\{i\}} & $D_{11i}$ &\texttt{D11i\{i\}}&$D_{12i}$&\texttt{D12i\{i\}}\\ 
$C_{2i}$ & \texttt{C2i\{i\}} & $D_{21i}$ &\texttt{D21i\{i\}}&$D_{22i}$&\texttt{D22i\{i\}}\\
\hline \\  \multicolumn{6}{c}{\textbf{Distributed Delay Terms: May be functions of \texttt{pvar s}}}\\
Eqn.~\eqref{eqn:DDE}  & \texttt{DDE.}  &   Eqn.~\eqref{eqn:DDE}  & \texttt{DDE.} &   Eqn.~\eqref{eqn:DDE}  & \texttt{DDE.}\\
\hline
$A_{di} $   & \texttt{Adi\{i\}} & $B_{1di}$ & \texttt{B1di\{i\}} &$B_{2di}$&\texttt{B2di\{i\}}\\ 
$C_{1di} $& \texttt{C1di\{i\}} &$ D_{11di} $&\texttt{D11di\{i\}}&$D_{12di}$&\texttt{D12di\{i\}}\\ 
$C_{2di}$ & \texttt{C2di\{i\}} & $D_{21di} $&\texttt{D21di\{i\}}&$D_{22di}$&\texttt{D22di\{i\}}\\
\end{tabular}
}
\end{center}\vspace{-2mm}
\caption{ Equivalent names of Matlab elements of the \texttt{DDE} structure terms for terms in Eqn.~\eqref{eqn:DDE}. For example, to set term \texttt{XX} to \texttt{YY}, we use \texttt{DDE.XX=YY}.  In addition, the delay $\tau_i$ is specified using the vector element \texttt{DDE.tau(i)} so that if $\tau_1=1, \tau_2=2, \tau_3=3$, then \texttt{DDE.tau=[1 2 3]}. }\label{tab:DDE_parameters}\end{table}

\subsubsection{Initializing a DDE Data structure} The user need only add non-zero terms to the DDE structure. All terms which are not added to the data structure are assumed to be zero. Before conversion to another representation or data structure, the data structure will be initialized using the command
	\begin{flalign*}
		&\texttt{DDE = initialize\_PIETOOLS\_DDE(DDE)}&
	\end{flalign*}
This will check for dimension errors in the formulation and set all non-zero parts of the \texttt{DDE} data structure to zero. Not that, to make the code robust, all PIETOOLS conversion utilities perform this step internally.

\subsection{Input of Neutral Type Systems}\label{subsec:NDS_format}
The input format for a Neutral Type System (NDS) is identical to that of a DDE except for 6 additional terms: 
\begin{align}
	\bmat{\dot{x}(t)\\z(t) \\ y(t)}&=\bmat{A_0 & B_{1} & B_{2}\\ C_{1} & D_{11} &D_{12}\\ C_{2} & D_{21} &D_{22}}\bmat{x(t)\\w(t)\\u(t)}+\sum_{i=1}^K \bmat{A_i  & B_{1i} & B_{2i}& E_i\\C_{1i}& D_{11i} & D_{12i} & E_{1i}\\C_{2i} & D_{21i} & D_{22i}&E_{2i}} \bmat{x(t-\tau_i)\\w(t-\tau_i)\\u(t-\tau_i)\\ \dot x(t-\tau_i)}\notag\\[-3mm]
& +\hspace{-1mm}\sum_{i=1}^K \hspace{0mm}\int_{-\tau_i}^0\hspace{-1mm}\bmat{A_{di}(s) & \hspace{-1mm}B_{1di}(s) &\hspace{-1mm}B_{2di}(s)& \hspace{-1mm}E_{di}(s)\\C_{1di}(s) & \hspace{-1mm}D_{11di}(s) & \hspace{-1mm}D_{12di}(s)& \hspace{-1mm}E_{1di}(s)\\C_{2di}(s) &\hspace{-1mm}D_{21di}(s) & \hspace{-1mm}D_{22di}(s)& \hspace{-1mm}E_{2di}(s)} \hspace{-2mm}\bmat{x(t+s)\\w(t+s)\\u(t+s)\\ \dot x(t+s)}\hspace{-1mm}ds\label{eqn:NDS}
\end{align}
These new terms are parameterized by $E_i,E_{1i}$, and $E_{2i}$ for the discrete delays and by $E_{di},E_{1di}$, and $E_{2di}$ for the distributed delays. As for the DDE case, these terms should be included in a NDS object as, e.g. \texttt{NDS.E\{1\}=1}. 

\begin{table}[hb]\vspace{-2mm}
\begin{center}{\small
\begin{tabular}{c|c||c|c||c|c||c|c}
  \multicolumn{8}{c}{\textbf{ODE Terms:}}\\
Eqn.~\eqref{eqn:NDS}  & \texttt{NDS.}  &   Eqn.~\eqref{eqn:NDS}  & \texttt{NDS.} &   Eqn.~\eqref{eqn:NDS}  & \texttt{NDS.}\\
\hline
$A_0$    & \texttt{A0} & $B_{1}$ & \texttt{B1} &$B_{2}$&\texttt{B2}&&\\ 
$C_{1}$ & \texttt{C1} & $D_{11}$ &\texttt{D11}&$D_{12}$&\texttt{D12}&&\\ 
$C_{2}$ & \texttt{C2} & $D_{21}$ &\texttt{D21}&$D_{22}$&\texttt{D22}&& \\ \hline \\
 \multicolumn{8}{c}{ \textbf{Discrete Delay Terms:}}\\
Eqn.~\eqref{eqn:NDS}  & \texttt{NDS.}  &   Eqn.~\eqref{eqn:NDS}  & \texttt{NDS.} &   Eqn.~\eqref{eqn:NDS}  & \texttt{NDS.}&   Eqn.~\eqref{eqn:NDS}  & \texttt{NDS.}\\
\hline
$A_i$    & \texttt{Ai\{i\}} & $B_{1i}$ & \texttt{B1i\{i\}} &$B_{2i}$&\texttt{B2i\{i\}}&$E_i$& \texttt{Ei\{i\}}\\ 
$C_{1i}$ & \texttt{C1i\{i\}} & $D_{11i}$ &\texttt{D11i\{i\}}&$D_{12i}$&\texttt{D12i\{i\}}&$E_{1i}$& \texttt{E1i\{i\}}\\ 
$C_{2i}$ & \texttt{C2i\{i\}} & $D_{21i}$ &\texttt{D21i\{i\}}&$D_{22i}$&\texttt{D22i\{i\}}&$E_{2i}$& \texttt{E2i\{i\}}\\
\hline \\  \multicolumn{8}{c}{\textbf{Distributed Delay Terms: May be functions of \texttt{pvar s}}}\\
Eqn.~\eqref{eqn:NDS}  & \texttt{NDS.}  &   Eqn.~\eqref{eqn:NDS}  & \texttt{NDS.} &   Eqn.~\eqref{eqn:NDS}  & \texttt{NDS.}&   Eqn.~\eqref{eqn:NDS}  & \texttt{NDS.}\\
\hline
$A_{di} $   & \texttt{Adi\{i\}} & $B_{1di}$ & \texttt{B1di\{i\}} &$B_{2di}$&\texttt{B2di\{i\}}&$E_{di}$& \texttt{Edi\{i\}}\\ 
$C_{1di} $& \texttt{C1di\{i\}} &$ D_{11di} $&\texttt{D11di\{i\}}&$D_{12di}$&\texttt{D12di\{i\}}&$E_{1di}$& \texttt{E1di\{i\}}\\ 
$C_{2di}$ & \texttt{C2di\{i\}} & $D_{21di} $&\texttt{D21di\{i\}}&$D_{22di}$&\texttt{D22di\{i\}}&$E_{2di}$& \texttt{E2di\{i\}}\\
\end{tabular}
}
\end{center}\vspace{-2mm}
\caption{ Equivalent names of Matlab elements of the \texttt{NDS} structure terms for terms in Eqn.~\eqref{eqn:NDS}. For example, to set term \texttt{XX} to \texttt{YY}, we use \texttt{NDS.XX=YY}. In addition, the delay $\tau_i$ is specified using the vector element \texttt{NDS.tau(i)} so that if $\tau_1=1, \tau_2=2, \tau_3=3$, then \texttt{NDS.tau=[1 2 3]}. }\label{tab:NDS_parameters}\end{table}

\subsubsection{Initializing a NDS Data structure} The user need only add non-zero terms to the NDS structure. All terms which are not added to the data structure are assumed to be zero. Before conversion to another representation or data structure, the data structure will be initialized using the command
\begin{matlab}
\begin{verbatim}
 >> NDS = initialize_PIETOOLS_NDS(NDS);
\end{verbatim}
\end{matlab}
This will check for dimension errors in the formulation and set all non-zero parts of the \texttt{NDS} data structure to zero. Not that, to make the code robust, all PIETOOLS conversion utilities perform this step internally.


\subsection{The Differential Difference Equation (DDF) Format}\label{subsec:DDF_format}
A Differential Difference Equation (DDF) is a more general representation than either the DDE or NDS. Most importantly, unlike the DDE or NDS, it allows one to represent the structure of the delayed channels. As such, it is the only representation for which the minimal realization features of PIETOOLS are defined. Nevertheless, the general form of DDF is more compact that of the DDE or NDS. The distinguishing feature of the DDF is decomposition of the output signal from the ODE part into sub-components, $r_i$, each of which is delayed by amount $\tau_i$. Identification of these $r_i$ is often challenging and hence most users will input the system as an ODE or NDS and then convert to a minimal DDF representation. The form of a DDF is given as follows.
	
\begin{align}
	\bmat{\dot{x}(t)\\ z(t)\\y(t)\\r_i(t)}&=\bmat{A_0 & B_1& B_2\\C_1 &D_{11}&D_{12}\\C_2&D_{21}&D_{22}\\C_{ri}&B_{r1i}&B_{r2i}}\bmat{x(t)\\w(t)\\u(t)}+\bmat{B_v\\D_{1v}\\D_{2v}\\D_{rvi}} v(t) \notag\\
v(t)&=\sum_{i=1}^K C_{vi} r_i(t-\tau_i)+\sum_{i=1}^K \int_{-\tau_i}^0C_{vdi}(s) r_i(t+s)ds.\label{eqn:DDF}
\end{align}

As for a DDE or NDS, each of the non-zero parameters in Eqn.~\eqref{eqn:DDF} should be added to the \texttt{DDF} structure, along with the vector of values of the delays \texttt{DDF.tau}. The elements of the \texttt{DDF} structure which can be defined by the user are included in Table~\ref{tab:DDF_parameters}.

\begin{table}[hb]\vspace{-2mm}
\begin{center}{
\begin{tabular}{c|c||c|c||c|c||c|c}
  \multicolumn{6}{c}{\textbf{ODE Terms:}}\\
Eqn.~\eqref{eqn:DDF}  & \texttt{DDF.}  &   Eqn.~\eqref{eqn:DDF}  & \texttt{DDF.} &   Eqn.~\eqref{eqn:DDF}  & \texttt{DDF.}&   Eqn.~\eqref{eqn:DDF}  & \texttt{DDF.}\\
\hline
$A_0$    & \texttt{A0} & $B_{1}$ & \texttt{B1} &$B_{2}$&\texttt{B2}&$B_v$&\texttt{Bv}\\ 
$C_{1}$ & \texttt{C1} & $D_{11}$ &\texttt{D11}&$D_{12}$&\texttt{D12}&$D_{1v}$&\texttt{D1v}\\ 
$C_{2}$ & \texttt{C2} & $D_{21}$ &\texttt{D21}&$D_{22}$&\texttt{D22}&$D_{2v}$&\texttt{D2v} \\
$C_{ri}$ & \texttt{Cri\{i\}} & $B_{r1i}$ &\texttt{Br1i\{i\}}&$B_{r2i}$&\texttt{Br2i\{i\}}&$D_{rvi}$&\texttt{Drvi\{i\}} \\
\hline \\
 \multicolumn{6}{c}{ \textbf{Discrete Delay Terms:}}\\
Eqn.~\eqref{eqn:DDF}  & \texttt{DDF.}  &   &  &    & \\
\hline
$C_{vi}$    & \texttt{Cvi\{i\}} & & &&\\ 
\hline \\  \multicolumn{6}{c}{\textbf{Distributed Delay Terms: May be functions of \texttt{pvar s}}}\\
Eqn.~\eqref{eqn:DDF}  & \texttt{DDF.}  &    &  &    & \\
\hline
$C_{vdi}(s) $   & \texttt{Cvdi\{i\}} & &  &&\\ 
\end{tabular}
}
\end{center}\vspace{-2mm}
\caption{ Equivalent names of Matlab elements of the \texttt{DDF} structure terms for terms in Eqn.~\eqref{eqn:DDF}. For example, to set term \texttt{XX} to \texttt{YY}, we use \texttt{DDF.XX=YY}.  In addition, the delay $\tau_i$ is specified using the vector element \texttt{DDF.tau(i)} so that if $\tau_1=1, \tau_2=2, \tau_3=3$, then \texttt{DDF.tau=[1 2 3]}. }\label{tab:DDF_parameters}\end{table}

\subsubsection{Initializing a DDF Data structure} The user need only add non-zero terms to the DDF structure. All terms which are not added to the data structure are assumed to be zero. Before conversion to another representation or data structure, the data structure will be initialized using the command
\begin{matlab}
\begin{verbatim}
 >> DDF = initialize_PIETOOLS_DDF(DDF);
\end{verbatim}
\end{matlab}
This will check for dimension errors in the formulation and set all non-zero parts of the \texttt{DDF} data structure to zero. Not that, to make the code robust, all PIETOOLS conversion utilities perform this step internally.


\section{Converting between DDEs, NDSs, DDFs, and PIEs}\label{ch:alt_DDE_inputs:convert}
For a given delay system, there are several alternative representations of that system. For example, a DDE can be represented in the DDE, DDF, or PIE format. However, only the DDF and PIE formats allow one to specify structure in the delayed channels, which are infinite-dimensional. For that reason, it is almost always preferable to efficiently convert the DDE or NDS to either a DDF or PIE - as this will dramatically reduce computational complexity of the analysis, control, and simulation problems (assuming you have tools for analysis, control and simulation of DDFs and PIEs - which we do!). However, identifying an efficient DDF or PIE representation of a given DDF/NDS is laborious for large systems and requires detailed understanding of the DDF format. For this reason, we introduce a set of functions for automating this conversion process.

\subsection{DDF to PIE}\label{sec:DDF2pie2} To convert a DDF data structure to an equivalent PIE representation, we have two utilities which are typically called sequentially. The first uses the SVD to identify and eliminate unused delay channels. The second na\"ively converts a DDF to an equivalent PIE. 

\subsubsection{Minimal DDF Realization of a DDF} The typical first step in analysis, simulation and control of a DDF is elimination of unused delay channels. This is accomplished using the SVD to identify such channels in a DDF structure and output a smaller, equivalent DDF structure. To use this utility, simply declare your \texttt{DDF} and enter the command
\begin{matlab}
\begin{verbatim}
 >> DDF = minimize_PIETOOLS_DDF(DDF);
\end{verbatim}
\end{matlab}

\subsubsection{Converting a DDF to a PIE} Having constructed a minimal (or not) DDF representation of a DDE, NDS or DDF, the next step is conversion to an equivalent PIE. To use this utility, simply declare your \texttt{DDF} structure and enter the command
\begin{matlab}
\begin{verbatim}
 >> PIE = convert_PIETOOLS_DDF(DDF,'pie');
\end{verbatim}
\end{matlab}

\subsection{DDE to DDF or PIE} We next address the problem of converting a DDE data structure to a DDF or PIE data structure. Because the DDE representation does not allow one to represent structure, this conversion is almost always performed by first computing a minimized DDF representation using \texttt{minimize\_PIETOOLS\_DDE2DDF}, followed possibly by converting this DDF representation to a PIE. Both steps are included in the function \texttt{convert\_PIETOOLS\_DDE}, allowing the minimal DDF representation and PIE representation of a DDE structure \texttt{DDE} to be computed as
\begin{matlab}
\begin{verbatim}
 >> [DDF, PIE] = convert_PIETOOLS_DDE(DDE);
\end{verbatim}
\end{matlab}
Here, the function \texttt{convert\_PIETOOLS\_DDE} computes the minimal DDF representation by calling \texttt{minimal\_PIETOOLS\_DDE2DDF}, which uses the SVD to eliminate unused delay channels in the DDF - resulting in a much more compact representation of the same system. As such, the minimal DDF representation can be computed by calling \texttt{minimal\_PIETOOLS\_DDE2DDF} directly as
\begin{matlab}
\begin{verbatim}
 >> DDF = minimal_PIETOOLS_DDE2DDF(DDF);
\end{verbatim}
\end{matlab}
or by calling \texttt{convert\_PIETOOLS\_DDE} with a second argument \texttt{'ddf'}
\begin{matlab}
\begin{verbatim}
 >> DDF = convert_PIETOOLS_DDE(DDE,'dde');
\end{verbatim}
\end{matlab}
Similarly, if only the PIE representation is desired, the user can also call
\begin{matlab}
\begin{verbatim}
 >> PIE = convert_PIETOOLS_DDE(DDE,'pie');
\end{verbatim}
\end{matlab}
though the procedure for computing the PIE will still involve computing the DDF representation first.

\subsubsection{DDE direct to PIE [NOT RECOMMENDED!]} 
Although it should never be used in practice, we also include a utility to construct the equivalent na\"ive PIE representation of a DDE. This is occasionally useful for purposes of comparison. To use this utility, simply declare your \texttt{DDE} and enter the command
\begin{matlab}
\begin{verbatim}
 >> DDF = convert_PIETOOLS_DDE2PIE_legacy(DDE);
\end{verbatim}
\end{matlab}
Because of the limited utility of the unstructured representation, we have not included a na\"ive DDE to DDF utility.

\subsection{NDS to DDF or PIE} 
Finally, we next address the problem of converting a NDS data structure to a DDF or PIE data structure. Like the DDE, the NDS representation does not allow one to represent structure and so the typical process is involves 3 steps: direct conversion of the NDS to a DDF; constructing a minimal representation of the resulting DDF using \texttt{minimize\_PIETOOLS\_DDF}; and conversion of the reduced DDF to a PIE. These three steps have been combined into a single function \texttt{convert\_PIETOOLS\_NDS}, computing the DDF represenation, minimizing this representation, and converting this representation to a PIE. All three resulting structures can be returned by calling
\begin{matlab}
\begin{verbatim}
 >> [DDF\_max, DDF, PIE] = convert\_PIETOOLS\_NDS(NDS);
\end{verbatim}
\end{matlab}
where now \texttt{DDF\_max} corresponds to the non-minimized DDF representation of \texttt{NDS}, and \texttt{DDF} corresponds to the minimized representation. If the user only want to compute this DDF representation, it is computationally cheaper to call the function with only two outputs,
\begin{matlab}
\begin{verbatim}
 >> [DDF\_max, DDF] = convert\_PIETOOLS\_NDS(NDS);
\end{verbatim}
\end{matlab}
or to call the function with a second argument \texttt{'ddf'},
\begin{matlab}
\begin{verbatim}
 >> [DDF] = convert\_PIETOOLS\_NDS(NDS,'ddf');
\end{verbatim}
\end{matlab}
Similarly, if only the non-minimized DDF representation is desired, the function should called with a single output,
\begin{matlab}
\begin{verbatim}
 >> DDF\_max = convert\_PIETOOLS\_NDS(NDS,'ddf_max');
\end{verbatim}
\end{matlab}
or with a second argument \texttt{'ddf\_max'},
\begin{matlab}
\begin{verbatim}
 >> DDF\_max = convert\_PIETOOLS\_NDS(NDS,'ddf_max');
\end{verbatim}
\end{matlab}
It is also possible to pass an argument \texttt{'pie'}, calling
\begin{matlab}
\begin{verbatim}
 >> PIE = convert\_PIETOOLS\_NDS(NDS,'pie');
\end{verbatim}
\end{matlab}
returning only the PIE representation, even though the DDF representations will also be computed.


\chapter{Operations on PI Operators in PIETOOLS: opvar and dopvar}\label{ch:opvar}

In Chapter~\ref{ch:PIE}, we showed how PI operators could be declared as \texttt{opvar} and \texttt{opvar2d} objects in PIETOOLS. In Chapter~\ref{ch:LPIs}, we showed how the similar class of \texttt{dopvar} (and \texttt{dopvar2d}) objects can be used to represent PI operator decision variables in convex optimization programs. In this Chapter, we detail some features of these \texttt{opvar} and \texttt{dopvar} classes, showing how standard operations on PI operators can be easily performed using the \texttt{opvar} classes in PIETOOLS.
In particular, we first recall the structure of \texttt{opvar} and \texttt{opvar2d} objects in Setion~\ref{sec:opvar_declare}, also showing how such objects can be declared in PIETOOLS.. In Section~\ref{sec:opvar_binary_ops}, we then show algebraic operations such as addition of \texttt{opvar} objects can be performed in PIETOOLS, after which we show how matrix operations such as concatenation can be performed on \texttt{opvar} objects in Section~\ref{sec:opvar_matrix_ops}. Finally, in Section~\ref{sec:additional_methods}, we outline a few additional operations for \texttt{opvar} objects. 
For more information on the theory behind these operations, we refer to Appendix~\ref{appx:PI_theory}, as well as papers such as~\cite{}.

\begin{boxEnv}{\textbf{Note}}
Unless stated otherwise, the operations on \texttt{opvar} objects presented in the following sections can also be performed on \texttt{opvar2d}, \texttt{dopvar} and \texttt{dopvar2d} class objects. To reduce notation, these operations will be illustrated only for \texttt{opvar} class objects.
\end{boxEnv}

\section{Declaring \texttt{opvar} and \texttt{dopvar} Objects}\label{sec:opvar_declare}

In this section, we briefly recall how \texttt{opvar} objects are structured, and how they represent PI operators in 1D. We also briefly introduce the \texttt{dopvar} class, showing how such objects can be declared in a similar manner to \texttt{opvar} objects. For more information on the \texttt{opvar2d} structure for PI operators in 2D, we refer to Chapter~\ref{ch:PIs}.

\subsection{The \texttt{opvar} Class}

In PIETOOLS, 4-PI operators are represented by \texttt{opvar} objects, which are structures with 8 accessible fields. In particular, letting $\mcl{T}:\bmat{\R^{n_0}\\L_2^{n_1}[a,b]}\rightarrow \bmat{\R^{m_0}\\L_2^{m_1}[a,b]}$ be a 4-PI operator of the form
\begin{align}\label{eq:4PI_standard_form_2}
    \bl(\mcl{T}\mbf{x}\br)(s)=
    \left[\begin{array}{ll}
        Px_0        \hspace*{-0.1cm}~& +\ \int_{a}^{b}Q_1(s)\mbf{x}_1(s)ds  \\
        Q_2(s)x_0   \hspace*{-0.1cm}& +\ R_{0}(s)\mbf{x}_1(s) + \int_{a}^{s}R_{1}(s,\theta)\mbf{x}_1(\theta)d\theta + \int_{s}^{b}R_{2}(s,\theta)\mbf{x}_1(\theta)d\theta
    \end{array}\right]
\end{align}
for $\mbf{x}=\bmat{x_0\\\mbf{x}_1}\in \bmat{\R^{n_0}\\L_2^{n_1}[a,b]}$, we can declare $\mcl{T}$ as an \texttt{opvar} object \texttt{T} with the following fields

\begin{Statebox}{\texttt{opvar} fields}

{\small
\centering
\hspace*{-0.425cm}
\begin{tabular}{p{0.75cm}p{1.85cm}p{12.75cm}}
    \texttt{dim}    & \texttt{= [m0,n0; \hspace*{0.4cm} m1,n1]} 
    &  $2\times 2$ array of type \texttt{double} specifying the dimensions of the function spaces $\sbmat{\R^{m_0}\\L_2^{m_1}[a,b]}$ and $\sbmat{\R^{n_0}\\L_2^{n_1}[a,b]}$ the operator maps to and from;\\
    \texttt{var1} & \texttt{= s}    &  $1\times 1$ \texttt{pvar} (\texttt{polynomial} class) object specifying the spatial variable $s$; \\
    \texttt{var2} & \texttt{= theta}    &  $1\times 1$ \texttt{pvar} (\texttt{polynomial} class) object specifying the dummy variable $\theta$;   \\
    \texttt{I} & \texttt{= [a,b]}       &  $1\times 2$ array of type \texttt{double}, specifying the interval $[a,b]$ on which the spatial variables $s$ and $\theta$ exist; \\
    \texttt{P} & \texttt{= P} & $m_0\times n_0$ array of type \texttt{double} or \texttt{polynomial} defining the matrix $P$; \\
    \texttt{Q1} & \texttt{= Q1} & $m_0\times n_1$ array of type \texttt{double} or \texttt{polynomial} defining the function $Q_1(s)$; \\
    \texttt{Q2} & \texttt{= Q2} & $m_1\times n_0$ array of type \texttt{double} or \texttt{polynomial} defining the function $Q_2(s)$; \\
    \texttt{R.R0} & \texttt{= R0} & $m_1\times n_1$ array of type \texttt{double} or \texttt{polynomial} defining the function $R_0(s)$; \\
    \texttt{R.R1} & \texttt{= R1} & $m_1\times n_1$ array of type \texttt{double} or \texttt{polynomial} defining the function $R_1(s,\theta)$; \\
    \texttt{R.R2} & \texttt{= R2} & $m_1\times n_1$ array of type \texttt{double} or \texttt{polynomial} defining the function $R_2(s,\theta)$; \\
 \end{tabular}
 }

\end{Statebox}

As an example, suppose we want to declare the 4-PI operator $\mcl{T}:\sbmat{\R^2\\ L_2^{3}[2,3]}\rightarrow \sbmat{\R^{3}\\L_2[2,3]}$ defined as\\[-2.5em]
\begin{align*}
\bl(\mcl{T}\mbf{x}\br)(r)&=
\bbbbl[\!\begin{array}{ll}
  \overbrace{\sbmat{1&0\\0&2\\3&4}}^{P}x_0 & \!+\ \int_{2}^{3}\overbrace{\sbmat{r^2& 0& 0\\ 3& r^3& 0\\ 0& r+2*r^2 & 0}}^{Q_1(r)}\mbf{x}_1(r)dr \\
  \underbrace{\sbmat{-5r & 6}}_{Q_2(s)}x_0       &\!+\ \int_{2}^{r}\underbrace{\sbmat{r&2\nu & 3(r-\nu)}}_{R_1(r,\nu)}\mbf{x}_1(\nu)d\nu
\end{array}\!\bbbbr],   &   r&\in[2,3], \\[-2.0em]
\end{align*}
for any $\mbf{x}=\sbmat{x_0\\\mbf{x}_1}\in\sbmat{\R^2\\ L_2^{3}[2,3]}$. To declare this operator, we first initialize an opvar object \texttt{T}, using the syntax
\begin{matlab}
\begin{verbatim}
 >> opvar T
 ans =
       [] | [] 
       -----------
       [] | ans.R 

 ans.R =
     [] | [] | [] 
\end{verbatim}     
\end{matlab}
This command creates an empty \texttt{opvar} object \texttt{T} with all dimensions 0. Consequently, the parameters \texttt{P,Qi,Ri} are initialized to $0\times 0$ matrices. Since we know the operator $\mcl{T}$ to map $\sbmat{\R^2\\ L_2^{3}[2,3]}\rightarrow \sbmat{\R^{3}\\L_2[2,3]}$, we can specify the desired dimension of the \texttt{opvar} object \texttt{T} using the command
\begin{matlab}
\begin{verbatim}
 >> T.dim = [3 2; 1 3]
 T =
       [0,0] | [0,0,0] 
       [0,0] | [0,0,0] 
       [0,0] | [0,0,0] 
       ----------------
       [0,0] | T.R 

 T.R =
     [0,0,0] | [0,0,0] | [0,0,0] 
\end{verbatim}    
\end{matlab}
Here, by assigning a value to \texttt{dim}, the parameters are adjusted to zero matrices of appropriate dimensions. We note that, this command is not strictly necessary, as the dimensions of \texttt{T} will also be automatically adjusted once we specify the values of the parameters.

Next, we assign the interval $[2,3]$ on which the functions $\mbf{x}_1\in L_2^3[0,3]$ are defined, by setting the field \texttt{I} as
\begin{matlab}
 >> T.I = [2,3];
\end{matlab}
Since the parameters defining $\mcl{T}$ also depend on $r,\nu\in[2,3]$, we have to assign these variables as well. For this, we represent them by \texttt{pvar} objects \texttt{r} and \texttt{nu}, and set the values of \texttt{T.var1} and \texttt{T.var2} as
\begin{matlab}
\begin{verbatim}
 >> pvar r nu
 >> T.var1 = r;     T.var2 = nu;
\end{verbatim}
\end{matlab}
Note that, if the parameters \texttt{Qi} and \texttt{R.Ri} are constant, there is no need to declare the variables \texttt{var1} or \texttt{var2}, in which case these fields will default to \texttt{var1=s} and \texttt{var2=theta}. Having declared the variables, we finally set the values of the parameters, by assigning them to the appropriate fields of \texttt{T}
\begin{matlab}
\begin{verbatim}
 >> T.P = [1,0; 0,2; 3,4];
 >> T.Q1 = [r^2, 0, 0; 3, r^3, 0; 0, r+2*r^2, 0];
 >> T.Q2 = [-5*r, 6];
 >> T.R.R1 = [r, 2*nu, 3*r-nu]
 T =
          [1,0] | [r^2,0,0] 
          [0,2] | [3,r^3,0] 
          [3,4] | [0,2*r^2+r,0] 
       ----------------------
       [-5*r,6] | T.R 

 T.R =
     [0,0,0] | [r,2*nu,-nu+3*r] | [0,0,0] 
\end{verbatim}
\end{matlab}


\begin{boxEnv}{\textbf{Note}}
	\texttt{dim} is dependent on size of the 6 parameters \texttt{P, Qi} and \texttt{R.Ri}. Modifying those parameters automatically changes the value stored in \texttt{dim} property. If the dimensions of the parameters are incompatible, \texttt{dim} will store \texttt{Nan} as its value to alert the user about the discrepancy.
\end{boxEnv}

\subsection{\texttt{dpvar} objects and the \texttt{dopvar} class}

In addition to polynomial functions, polynomial decision variables can also be declared in PIETOOLS. Such decision variables are used in polynomial optimization programs, optimizing over the values of the coefficients defining these polynomials. To declare such polynomial decision variables, we use the \texttt{dpvar} class, extending the \texttt{polynomial} class to represent polynomials with decision variables. For example, to represent a variable quadratic polynomial
\begin{align*}
    p(c_0,c_1,c_2;x)=c_0+c_1x+c_2x^2,
\end{align*}
with unknown coefficient $\{c_0,c_1,c_2\}$, we use
\begin{matlab}
\begin{verbatim}
 >> pvar x
 >> dpvar c0 c1 c2
 >> p = c0 + c1*x + c2*x^2
 p = 
   c0 + c1*x + c2*x^2
\end{verbatim}
\end{matlab}
Here, the second line decision variables $c_0$, $c_1$ and $c_2$, and the third line uses these decision variables to declare the decision variable polynomial $p(c_0,c_1,c_2;x)$ as a \texttt{dpvar} object \texttt{p}. Crucially, this variable is affine in the coefficients $c_0$, $c_1$ and $c_2$, as \texttt{dpvar} objects can only be used to represent decision variable polynomials that are affine with respect to their coefficients. This is because PIETOOLS cannot tackle polynomial optimization programs that are nonlinear in the decision variables, and accordingly, the \texttt{dpvar} structure has been built to exploit the linearity of the decision variables to minimize computational effort.

Using polynomial decision variables, we can also define PI operator decision variables, defining the parameters $Q_1$ through $R_2$ by polynomial decision variables rather than polynomials. For example, suppose we have an operator $\mcl{D}:L_2^{2}[0,1]\rightarrow L_2^{2}[0,1]$ defined as
\begin{align*}
 \bl(\mcl{D}\mbf{x}\br)(s)&=\underbrace{\bmat{c_1&c_2s\\c_2s&c_3s^2}}_{R_0(s)}\mbf{x}(s)+\int_{s}^{1}\underbrace{\bmat{c_4s^2&c_5s\theta\\ -c_6s\theta&c_7\theta^2}}_{R_2(s,\theta)}\mbf{x}(\theta)d\theta, &
 s&\in[0,1]
\end{align*}
for $\mbf{x}\in L_2^2[0,1]$, where $c_1$ through $c_7$ are unknown coefficients. Then, we can declare the parameters $R_1$ and $R_2$ as \texttt{dpvar} class objects by calling
\begin{matlab}
\begin{verbatim}
 >> pvar s th
 >> dpvar c1 c2 c3 c4 c5 c6 c7
 >> R0 = [c1, c2*s; c2*s, c3*s^2];
 >> R2 = [c4, c5*s*th; c6*s*th, c7*th^2];
\end{verbatim}
\end{matlab}
Then, we can declare $\mcl{D}$ as a \texttt{dopvar} object \texttt{D}. Such \texttt{dopvar} objects have a structure identical to that of \texttt{opvar} objects, with the only difference being that the parameters \texttt{P} through \texttt{R.R2} can be declared as \texttt{dpvar} objects. Accordingly, we declare the \texttt{dopvar} object \texttt{D} in a similar manner as we would and \texttt{opvar} objects:
\begin{matlab}
\begin{verbatim}
 >> dopvar D;
 >> D.I = [0,1];
 >> D.var1 = s;     D.var2 = th;
 >> D.R.R0 = R0;    D.R.R1 = R2
 D =
       [] | [] 
       ---------
       [] | D.R 

 D.R =
         [c1,c2*s] | [0,0] |      [c4,c5*s*th] 
     [c2*s,c3*s^2] | [0,0] | [c6*s*th,c7*th^2] 
\end{verbatim}
\end{matlab}

\section{Algebraic Operations on \texttt{opvar} Objects}\label{sec:opvar_binary_ops}
	
In this section, we go over various methods that help in manipulating and handling of \texttt{opvar} objects in PIETOOLS. In particular, in Subsection~\ref{sec:add_pi}, we show how the sum of two PI operators can be computed in PIETOOLS, followed by the composition of PI operators in Subsection~\ref{sec:comp_pi}. In Subsection~\ref{sec:adj_pi}, we then show how to take the adjoint of a PI operator, and finally, in Subsection~\ref{sec:inv_pi}, we show how a numerical inverse of a PI operator can be computed. 
To illustrate each of these operations, we use the PI operators $\mcl{A},\mcl{B}:\sbmat{\R^2\\L_2[-1,1]}\rightarrow\sbmat{\R^2\\L_2[-1,1]}$ defined as
\begin{align}\label{eq:opvar_ops_ex}
    \bl(\mcl{A}\mbf{x}\br)(s)&=
    \bbbbl[\begin{array}{ll}
      \sbmat{1&0\\2&-1}x_0   &\!+\ \int_{-1}^{1}\sbmat{1-s\\s+1}\mbf{x}_1(s)ds \\
      \sbmat{10s&-1}x_0   &\!+\ 2\mbf{x}_1(s) + \int_{-1}^{s}(s-\theta)\mbf{x}_1(\theta)d\theta + \int_{s}^{1}(s-\theta)\mbf{x}_1(\theta) d\theta
    \end{array}\bbbbr],  \\
    \bl(\mcl{B}\mbf{x}\br)(s)&=
    \bbbbl[\begin{array}{ll}
      \sbmat{1&0\\0&3}x_0   &  \nonumber\\
      \sbmat{5s&-s}x_0   &\!+\ s^2\mbf{x}_1(s) + \int_{s}^{1}\theta\mbf{x}_1(\theta) d\theta
    \end{array}\bbbbr], &   s&\in[-1,1],
\end{align}
for $\mbf{x}=\sbmat{x_0\\\mbf{x}_1}\in \sbmat{\R^2\\L_2[-1,1]}$. We declare these operators as
\begin{matlab}
\begin{verbatim}
 >> pvar s th
 >> opvar A B;
 >> A.I = [-1,1];         B.I = [-1,1];
 >> A.var1 = s;           B.var1 = s;
 >> A.var2 = th;          B.var2 = th;
 >> A.P = [1,0;2,-1];     B.P = [1,0;0,3];
 >> A.Q1 = [1-s;s+1];
 >> A.Q2 = [10*s,-1];     B.Q2 = [5*s,-s];
 >> A.R.R0 = 2;           B.R.R0 = s^2;
 >> A.R.R1 = (s-th);
 >> A.R.R2 = (s-th);      B.R.R2 = th;
\end{verbatim}
\end{matlab}

\subsection{Addition (\texttt{A+B})}\label{sec:add_pi}
\texttt{opvar} objects, \texttt{A} and \texttt{B}, can be added simply by using the command
\begin{matlab}
 >> A+B
\end{matlab}
For two \texttt{opvar} objects to be added, they \textbf{must} have same dimensions (\texttt{A.dim=B.dim}), domains (\texttt{A.I=B.I}), and variables (\texttt{A.var1=B.var1}). Furthermore, if \texttt{A} (or \texttt{B}) is a scalar then PIETOOLS considers that as adding \texttt{A*I} (or \texttt{B*I}) where \texttt{I} is an identity matrix. Again, this operation is appropriate if and only if dimensions match. Similarly, if \texttt{A} (or \texttt{B}) is a matrix with matching dimension, it can be added to \texttt{opvar} \texttt{B} (or  \texttt{A}) using the same command.

\paragraph*{Example}
Adding the \texttt{opvar} objects \texttt{A} and \texttt{B} corresponding to operators $\mcl{A},\mcl{B}$ defined as in Equation~\eqref{eq:opvar_ops_ex}, we find
\begin{matlab}
\begin{verbatim}
 >> C = A+B
 C =
            [2,0] | [-s+1] 
            [2,2] | [s+1] 
        ------------------
      [15*s,-s-1] | C.R 

 C.R =
     [s^2+2] | [s-th] | [s]
\end{verbatim}
\end{matlab}
suggesting that, for $\mbf{x}=\sbmat{x_0\\\mbf{x}_1}\in \sbmat{\R^2\\L_2[-1,1]}$,
\begin{align*}
    \bl(\mcl{A}\mbf{x}\br)(s)+\bl(\mcl{B}\mbf{x}\br)(s)&=
    \bbbbl[\begin{array}{ll}
      \sbmat{2&0\\2&2}x_0   &\!+\ \int_{-1}^{1}\sbmat{1-s\\s+1}\mbf{x}_1(s)ds \\
      \sbmat{15s&-s-1}x_0   &\!+\ (s^2+2)\mbf{x}_1(s) + \int_{-1}^{s}(s-\theta)\mbf{x}_1(\theta)d\theta + \int_{s}^{1}s\mbf{x}_1(\theta) d\theta
    \end{array}\bbbbr].
\end{align*}

\subsection{Multiplication (\texttt{A*B})}\label{sec:comp_pi}
\texttt{opvar} objects, \texttt{A} and \texttt{B}, can be composed simply by using the command
\begin{matlab}
>> A*B
\end{matlab}
For two opvar objects to be composed, they \textbf{must} have the same domains (\texttt{A.I=A.B}), the same variables (\texttt{A.var1=B.var1} and \texttt{A.var2=B.var2}), and the output dimension of \texttt{B} must match the input dimension of \texttt{A} (\texttt{A.dim(:,2)=B.dim(:,1)}). Furthermore, if \texttt{A} (or \texttt{B}) is a scalar then PIETOOLS considers that as a scalar multiplication operation, thus multiplying all parameters of \texttt{B} (or \texttt{A}) by that value. 

\paragraph*{Example}
Composing the \texttt{opvar} objects \texttt{A} and \texttt{B} corresponding to operators $\mcl{A},\mcl{B}$ defined as in Equation~\eqref{eq:opvar_ops_ex}, we find
\begin{matlab}
\begin{verbatim}
 >> C = A*B
 C = 
              [-2.3333,0.66667] | [-1.5*s^3+2*s^2+1.5*s] 
               [5.3333,-3.6667] | [1.5*s^3+2*s^2+0.5*s] 
      -------------------------------------------
      [20*s-3.3333,-2*s-2.3333] | C.R 

 C.R =
     [2*s^2] | [2*s*th^2-1.5*th^3+s*th+0.5*th] | [2*s*th^2-1.5*th^3+s*th+2.5*th] 
\end{verbatim}
\end{matlab}
suggesting that, for $\mbf{x}=\sbmat{x_0\\\mbf{x}_1}\in \sbmat{\R^2\\L_2[-1,1]}$,
\begin{align*}
    \bbl(\mcl{A}\bl(\mcl{B}\mbf{x}\br)\bbr)(s)&=
    \left[\begin{array}{ll}
      \sbmat{-2\frac{1}{3}&\frac{2}{3}\\5\frac{1}{3}&-3\frac{2}{3}}x_0   &\!+\ \int_{-1}^{1}\sbmat{-1\frac{1}{2}s^3+2s^2+1\frac{1}{2}s\\1\frac{1}{2}s^3+2s^2+\frac{1}{2}s}\mbf{x}_1(s)ds \\
      \sbmat{20s-3\frac{1}{3}&-2s-2\frac{1}{3}}x_0   &\!+\ 2s^2\mbf{x}_1(s) + \int_{-1}^{s}(2s\theta^2-1\frac{1}{2}\theta^3+s\theta+\frac{1}{2}\theta)\mbf{x}_1(\theta)d\theta \\
      &\quad + \int_{s}^{1}(2s\theta^2-1\frac{1}{2}\theta^3+s\theta+2\frac{1}{2}\theta)\mbf{x}_1(\theta) d\theta
    \end{array}\right].
\end{align*}

\begin{boxEnv}{\textbf{Note}}
 Although \texttt{dopvar} objects can be multiplied with \texttt{opvar} objects and vice versa, producing a \texttt{dopvar} object in both cases, it is not possible to compute the composition of two \texttt{dopvar} objects. This is because \texttt{dopvar} objects depend linearly (affinely) on the decision variables, and the composition of two \texttt{dopvar} objects would require taking the product of decision variables. Similarly for \texttt{dopvar2d} objects.
\end{boxEnv}

\subsection{Adjoint (\texttt{A'})}\label{sec:adj_pi}
The adjoint of an \texttt{opvar} object \texttt{A} can be calculated using the command
\begin{matlab}
 >> A'
\end{matlab}
For an operator $\mcl{A}:RL^{n_0,n_1}[a,b]\rightarrow RL^{m_0,m_1}[a,b]$, the adjoint $\mcl{A}^*:RL^{m_0,m_1}[a,b]\rightarrow RL^{n_0,n_1}[a,b]$ will be such that, for any $\mbf{x}\in RL^{n_0,n_1}[a,b]$ and $\mbf{y}\in RL^{m_0,m_1}[a,b]$,
\begin{align*}
 \ip{\mcl{A}\mbf{x}}{\mbf{y}}_{RL} = \ip{\mbf{x}}{\mcl{A}^*\mbf{y}}_{RL},
\end{align*}
where for $\mbf{x}=\sbmat{x_0\\\mbf{x}_1}\in\sbmat{\R^{n_0}\\L_2^{n_1}[a,b]}=RL^{n_0,n_1}[a,b]$ and $\mbf{y}=\sbmat{y_0\\\mbf{y}_1}\in\sbmat{\R^{n_0}\\L_2^{n_1}[a,b]}=RL^{n_0,n_1}[a,b]$,
\begin{align*}
    \ip{\mbf{x}}{\mbf{y}}_{RL}:=\ip{x_0}{y_0} + \ip{\mbf{x}_1}{\mbf{y}_1}_{L_2} = x_0^T y_0 + \int_{a}^{b}[\mbf{x}_1(s)]^T\mbf{y}_1(s)ds
\end{align*}

\paragraph*{Example}
Computing the adjoint of the \texttt{opvar} object \texttt{A} corresponding to operator $\mcl{A}$ defined as in Equation~\eqref{eq:opvar_ops_ex}, we find
\begin{matlab}
\begin{verbatim}
 >> AT = A'
 AT =
            [1,2] | [10*s] 
           [0,-1] | [-1] 
       ------------------
       [-s+1,s+1] | AT.R 

 AT.R =
     [2] | [-s+th] | [-s+th] 
\end{verbatim}
\end{matlab}
suggesting that, for $\mbf{x}=\sbmat{x_0\\\mbf{x}_1}\in \sbmat{\R^2\\L_2[-1,1]}$,
\begin{align*}
    \bl(\mcl{A}^*\mbf{x}\br)(s)&=
    \left[\begin{array}{ll}
      \sbmat{1&2\\0&-1}x_0   &\!+\ \int_{-1}^{1}\sbmat{10s\\-1}\mbf{x}_1(s)ds \\
      \sbmat{1-s&s+1}x_0   &\!+\ 2\mbf{x}_1(s) + \int_{-1}^{1}(\theta-s)\mbf{x}_1(\theta)d\theta
      + \int_{s}^{1}(\theta-s)\mbf{x}_1(\theta) d\theta
    \end{array}\right].
\end{align*}

\subsection{Inverse (\texttt{inv\_opvar(A)})}\label{sec:inv_pi}
The inverse of an opvar object \texttt{A} can be numerically calculated, using the function 
\begin{matlab}
 >> inv\_opvar(A)
\end{matlab}
See Lemma \ref{lem:inverse} for details on Inversion formulae.

\paragraph*{Example}
Computing the inverse of the \texttt{opvar} object \texttt{A} corresponding to operator $\mcl{A}$ defined as in Equation~\eqref{eq:opvar_ops_ex}, we find
\begin{matlab}
\begin{verbatim}
 >> Ainv = inv_opvar(A)
 Ainv =
                    [-0.2,-0.4] | [  0.3*s + 0.2]
                      [1.2,0.4] | [ -0.3*s - 0.7]
       ------------------------------------------
       [ 0.3*s+0.7,1.35*s+0.65] | AT.R 

 Ainv.R =
     [0.5] | [-1.2*s*th-0.675*s-0.3*th-0.575] | [-1.2*s*th-0.675*s-0.3*th-0.575]
\end{verbatim}
\end{matlab}
suggesting that, for $\mbf{x}=\sbmat{x_0\\\mbf{x}_1}\in \sbmat{\R^2\\L_2[-1,1]}$
\begin{align*}
    \bl(\mcl{A}^{-1}\mbf{x}\br)(s)&=
    \left[\begin{array}{ll}
      \sbmat{-0.2&-0.4\\1.2&0.4}x_0   &\!+\ \int_{-1}^{1}\sbmat{0.3s+0.2\\-0.3s-0.7}\mbf{x}_1(s)ds \\
      \sbmat{0.3s+0.7&1.35s+0.65}x_0   &\!+\ 0.5\mbf{x}_1(s) + \int_{-1}^{1}(-1.2s\theta-0.675s-0.3\theta-0.575)\mbf{x}_1(\theta)d\theta \\
      & \quad + \int_{s}^{1}(-1.2s\theta-0.675s-0.3\theta-0.575)\mbf{x}_1(\theta) d\theta
    \end{array}\right].
\end{align*}


\begin{boxEnv}{\textbf{Note}}
	An inverse function has not been defined for \texttt{opvar2d}, \texttt{dopvar}, or \texttt{dopvar2d} objects.
\end{boxEnv}


\section{Matrix Operations on \texttt{opvar} Objects}\label{sec:opvar_matrix_ops}

In this section, we show how matrix operations on \texttt{opvar} objects can be performed. In particular, in Subsection~\ref{sec:concat_pi} we show how \texttt{opvar]} objects can be concatenated, and in Subsection~\ref{sec:opvar_subsref}, we show how desired rows and columns of \texttt{opvar} objects can be extracted. Although we explain these operations only for \texttt{opvar} objects, they can also be applied to \texttt{dopvar}, \texttt{opvar2d}, and \texttt{dopvar2d} objects. For the purpose of illustration, we once more use the 4-PI operators $\mcl{A},\mcl{B}:\sbmat{\R^2\\L_2[-1,1]}\rightarrow \sbmat{\R^2\\L_2[-1,1]}$ defined in Equation~\eqref{eq:opvar_ops_ex}, represented in PIETOOLS by the \texttt{opvar} objects \texttt{A} and \texttt{B},
\begin{matlab}
\begin{verbatim}
 >> A
 A =
          [1,0] | [-s+1] 
         [2,-1] | [s+1] 
      ----------------
      [10*s,-1] | A.R 

 A.R =
    [2] | [s-th] | [s-th] 
 >> B
 B =
         [1,0] | [0] 
         [0,3] | [0] 
      ---------------
      [5*s,-s] | B.R 

 B.R =
    [s^2] | [0] | [th] 
\end{verbatim}
\end{matlab}

\subsection{Concatenation (\texttt{[A,B]})}\label{sec:concat_pi}
Just like matrices, multiple \texttt{opvar} objects can be concatenated, provided the dimensions match.
In particular, two \texttt{opvar} objects \texttt{A} and \texttt{B} can be horizontally or vertically concatenated by respectively using the command
\begin{matlab}
>> [A B]  \% for horizontal concatenation\\ 
>> [A; B] \% for vertical concatenation
\end{matlab}
Note that concatenation of \texttt{opvar} objects is allowed only if their spatial domain is the same (\texttt{A.I=B.I}), and the variables involved in each are identical (\texttt{A.var1=B.var1} and \texttt{A.var2=B.var2}). Moreover, \texttt{A} and \texttt{B} can be horizontally concatenated only if they have the same row dimensions (\texttt{A.dim(:,1)=B.dim(:,2)}, and they can be concatenated vertically only if they have the same column dimensions (\texttt{A.dim(:,2)=B.dim(:,2)}). 

\paragraph*{Example}
Horizontally concatenating the \texttt{opvar} objects \texttt{A} and \texttt{B} corresponding to operators $\mcl{A},\mcl{B}$ defined as in Equation~\eqref{eq:opvar_ops_ex}, we find
\begin{matlab}
\begin{verbatim}
 >> C = [A, B]
 C =
             [1,0,1,0] | [-s+1,0] 
            [2,-1,0,3] | [s+1,0] 
      -----------------------
      [10*s,-1,5*s,-s] | C.R 

 C.R = 
     [2,s^2] | [s-th,0] | [s-th,th] 
\end{verbatim}
\end{matlab}
Note that, since $\mcl{A},\mcl{B}:\sbmat{\R^2\\L_2[-1,1]}\rightarrow \sbmat{\R^2\\L_2[-1,1]}$, we have $\mcl{C}:\sbmat{\R^4\\L_2^2[-1,1]}\rightarrow \sbmat{\R^2\\L_2[-1,1]}$, \textbf{not} $\mcl{C}:\sbmat{\R\\L_2[-1,1]\\\R\\L_2[-1,1]}\rightarrow \sbmat{\R^2\\L_2[-1,1]}$. Similarly, taking the vertical concatenation,
\begin{matlab}
\begin{verbatim}
 >> C = [A; B]
 C =
          [1,0] | [-s+1] 
         [2,-1] | [s+1] 
          [1,0] | [0] 
          [0,3] | [0] 
      ----------------
      [10*s,-1] | C.R 
       [5*s,-s] |   

 C.R =
       [2] | [s-th] | [s-th] 
     [s^2] |    [0] |   [th] 
\end{verbatim}
\end{matlab}
the resulting operator $\mcl{C}$ will map $\sbmat{\R^2\\L_2[-1,1]}\rightarrow \sbmat{\R^4\\L_2^2[-1,1]}$, as PIETOOLS cannot represent operators mapping e.g. $\sbmat{\R^2\\L_2[-1,1]}\rightarrow \sbmat{\R\\L_2[-1,1]\\\R\\L_2[-1,1]}$.

\subsection{Subs-indexing (\texttt{A(i,j)})}\label{sec:opvar_subsref}

4-PI operators can also be sliced the way matrices are sliced in matrices. The index slicing is performed in the same manner as matrices. 
\begin{matlab}
>> T(row\_ind, col\_ind)
\end{matlab}
Indexing 4-PI operators is slightly different from matrix indexing due to presence of multiple components. These components can be visualized as being stacked as in a matrix:
\begin{align*}
\texttt{B}=\left[\begin{array}{c|c}
\texttt{P}&\texttt{Q1}\\\hline
\texttt{Q2}&\texttt{Ri}
\end{array}\right]
\end{align*}
Then, row indices specified in \texttt{row\_ind} correspond to the rows in this big matrix. Column indices, \texttt{col\_ind}, are associated with the columns of this big matrix in similar manner. The retrieved slices are put in appropriate components and a 4-PI operator is returned.
Note the bottom-right block of the big matrix  \texttt{B} has 3 components in \texttt{Ri}. If indices in the slice correspond to rows and columns in this block, then the slice is extracted from all three components and stored in a \texttt{Ri} part of the new sliced PI operator.

\paragraph*{Example}
Extracting row 3 and columns 1, and 3 of the \texttt{opvar} object \texttt{A} corresponding to the operator $\mcl{A},\mcl{B}$ defined as in Equation~\eqref{eq:opvar_ops_ex}, we find
\begin{matlab}
\begin{verbatim}
 >> C = A(3,[1,3])
 C =
          [] | [] 
      -------------
      [10*s] | C.R 
 C.R =
     [2] | [s-th] | [s-th] 
\end{verbatim}
\end{matlab}

 \section{Additional Methods for \texttt{opvar} Objects}\label{sec:additional_methods}
There are some additional functions included in PIETOOLS that can be used in debugging or as the user sees fit. In this section, we compile the list of those functions, without going into details or explanation. However, users can find additional information by using \texttt{help} command in MATLAB.
	
\vspace{5mm}
\begin{center}
\begin{tabular}{|p{4.5cm}|p{10cm}|}
		\hline
		\textbf{Function Name} & \textbf{Description}\\\hline
        \texttt{A==B} & The function checks whether the variables, domain, dimensions and parameters of the operators \texttt{A} and \texttt{B} are equal, and returns a binary value 1 if this is the case, or 0 if it is not. The function can also be used to check if an operator \texttt{A} has all parameters equal to zero operator by calling \texttt{A==0}.
        \\\hline
    	\texttt{isvalid(P)}& The function returns a logical value. 0 is everything is in order, 1 if the object has incompatible dimensions, 2 if property \texttt{P} is not a matrix, 3 if properties \texttt{Q1, Q2} or \texttt{R0} are not polynomials in $s$, 4 if properties \texttt{R1} or \texttt{R2} are not polynomials in $s$ and $\theta$. For \texttt{opvar2d} objects, returns boolean value \texttt{true} or \texttt{false} to indicate if \texttt{P} is appropriate or not.
		\\\hline
		\texttt{degbalance(T)}& Estimates polynomial degrees needed to create an \texttt{opvar} object \texttt{Q} associated to a positive PI operator $\mcl{Q}\succeq 0$ in \texttt{poslpivar}, such that \texttt{T=Q} has at least one solution.
        \\\hline
		\texttt{getdeg(T)}& Returns highest and lowest degree of $s$ and $\theta$ in the components of the opvar object \texttt{T}.
        \\\hline
		\texttt{rand\_opvar(dim, deg)}& Creates a random opvar object of specified dimensions \texttt{dim} and polynomial degrees \texttt{deg}.
        \\\hline
		\texttt{show(T,opts)}& Alternative display format for opvar objects with optional argument to omit selected properties from display output. \textbf{Not defined for \texttt{opvar2d} objects.}
        \\\hline
		\texttt{opvar\_postest(T)}& Numerically test for sign definiteness of \texttt{T}. Returns -1 if negative definite, 0 if indefinite and 1 if positive definite. Use \texttt{opvar\_postest\_2d} for \texttt{opvar2d} objects.
        \\\hline
		\texttt{diff\_opvar(T)}& Returns composition of derivative operator with opvar \texttt{T} as described in Lem. \ref{lem:diff_PI}. Use \texttt{diff(T)} for \texttt{opvar2d} objects.
        \\\hline
\end{tabular}
\end{center}

\part{Examples and Applications}

\chapter{PIETOOLS Demonstrations}\label{ch:demos}

In this Chapter, we illustrate several applications of PIE simulation and LPI programming, and how each of these problems can be implemented in PIETOOLS. Each of these problems has also been implemented as a \texttt{DEMO} file in PIETOOLS, which can be found in the \texttt{PIETOOLS\_demos} directory. 

\section{DEMO 1: Simple Stability, Simulation and Control Problem}\label{sec:demos:1}
See Chapter~\ref{ch:scope} for a description.

\section{DEMO 2: Estimating the Volterra Operator Norm}\label{sec:demos:volterra}

The Volterra integral operator $\mcl{T}:L_2[0,1]\rightarrow L_2[0,1]$ is perhaps the simplest example of a 3-PI operator, defined as
\begin{align*}
    \bl(\mcl{T}\mbf{x}\br)(s)&=\int_{0}^{s}\mbf{x}(\theta)d\theta,   &   s&\in[0,1],
\end{align*}
for any $\mbf{x}\in L_2[0,1]$. In PIETOOLS, this operator can be easily declared as
\begin{matlab}
\begin{verbatim}
 a=0;    b=1;
 opvar Top;
 Top.R.R1 = 1;   Top.I = [a,b];
\end{verbatim}
\end{matlab}
Then, an upper bound on the norm of this operator can be computed by solving the LPI
\begin{align*}
    &\min_{\gamma\geq 0}\ \gamma,    \\
    \text{s.t.}\qquad &\mcl{T}^*\mcl{T}\leq\gamma,
\end{align*}
so that case $C=\sqrt{\gamma}$ for any feasible value $\gamma$ is an upper bound on the norm of $\mcl{T}$. This optimization problem is an LPI, that can be declared and solved in PIETOOLS as
\begin{matlab}
\begin{verbatim}
 % First, define dpvar gam and set up an optimization problem
 vars = [Top.var1;Top.var2];     % Free vars in optimization problem
 dpvar gam;                      % Decision var in optimization problem
 prob = sosprogram(vars,gam);

 % Next, set gam as objective function min{gam}
 prob = sossetobj(prob, gam);

 % Then, enforce the constraint Top'*Top-gam<=0
 opts.psatz = 1;                             % Allow Top'*Top-gam>0 outside of [a,b]
 prob = lpi_ineq(prob,-(Top'*Top-gam),opts); % lpi_ineq(prob,Q) enforces Q>=0

 % Finally, solve and retrieve the solution
 prob = sossolve(prob);
 operator_norm = sqrt(double(sosgetsol(prob,gam)));
\end{verbatim}
\end{matlab}
This code can also be run by calling ``volterra\_operator\_norm\_DEMO''. We obtain an upper bound $C=0.68698$ on the induced norm $\|\mcl{T}\|$ of the Volterra operator. The exact value of the induced norm of this operator is known to be equal to $\|\mcl{T}\|=\frac{2}{\pi}= 0.6366...$.

\section{DEMO 3: Solving the Poincar\'e Inequality}\label{sec:demos:poincare}

In a one dimensional domain $\Omega=[a,b]$, the Poincar\'e inequality imposes a bound on the norm of a function $x(s)$ in terms of the spatial derivative $\partial_s x(s)$ of this function,
\begin{align*}
    \|x\|_{L_2}&\leq C\|\partial_{s}x\|_{L_2},  \hspace*{2.0cm}   \forall x\in W_1[a,b],
    \intertext{where}
    W_1[a,b]&:=\{x\in L_2[a,b]\mid \partial_s x\in L_2[a,b], x(a)=x(b)=0\}.
\end{align*}
In PIETOOLS, we can find a value $C$ that satisfies this inequality, by solving the optimization problem
\begin{align*}
    &\min_{\gamma\geq 0}\ \gamma, \\
    \text{s.t.}\qquad &\ip{x}{x}_{L_2}-\gamma\ip{\partial_{s}x}{\partial_{s}x}_{L_2}\leq 0, & \forall x&\in W_1[a,b]
\end{align*}
in which case $C=\sqrt{\gamma}$ satisfies the Poincar\'e inequality.
To declare this problem, we first note that we can represent $x$ and $\partial_{s}x$ in terms of a fundamental state $\partial_{s}^2 x\in L_2[a,b]$, which is free of the boundary conditions and continuity constraints imposed upon $x$ and $\partial_{s}x$. In particular, we first declare a PDE
\begin{align*}
    \dot{x}(t,s)&=\partial_{s}x(t,s), &   s&\in[a,b], \\
    x(t,a)&=x(t,b)=0,
\end{align*}
as a \texttt{pde\_struct} object by calling
\begin{matlab}
\begin{verbatim}
 % % Initialize the PDE structure and spatial variable s in [a,b]
 pvar s theta;
 pde_struct PDE;
 a = 0;      b = 1;

 % % Declare the state variables x(t,s)
 PDE.x{1}.vars = s;
 PDE.x{1}.dom = [a,b];
 PDE.x{1}.diff = 2;          % Let x be second order differentiable wrt s.

 % % Declare the PDE \dot{x}(t,s) = \partial_{s} x(t,s)
 PDE.x{1}.term{1}.D = 1;     % Order of the derivative wrt s

 % % Declare the boundary conditions x(t,a) = x(t,b) = 0
 PDE.BC{1}.term{1}.loc = a;      % Evaluate x at s=a
 PDE.BC{2}.term{1}.loc = b;      % Evaluate x at s=b

 % % Initialize the system
 PDE = initialize(PDE);
\end{verbatim}
\end{matlab}
Here, we explicitly indicate that the function $x$ must be second order differentiable, as we wish to derive a PIE representation in terms of the fundamental state $x_{\text{f}}:=\partial_{s}^2 x$. To obtain this PIE representation, we call \texttt{convert},
\begin{matlab}
\begin{verbatim}
 PIE = convert(PDE,'pie');
 H2 = PIE.T;     % H2 x_{ss} = x
 H1 = PIE.A;     % H1 x_{ss} = x_{s}
\end{verbatim}
\end{matlab}
arriving at an equivalent representation of the PDE as
\begin{align*}
    \mcl{H}_2\dot{x}_{\text{f}}(t,s)&=\mcl{H}_{1}x_{\text{f}}(t,s), &   s&\in[a,b].
\end{align*}
In this representation, the fundamental state $x_{\text{f}}:=\partial_{s}^2 x$ is free of any boundary conditions and continuity constraints. Moreover, we note that
\begin{align*}
    \mcl{H}_{2}x_{\text{f}}(t,s)&=x(t,s),\qquad \text{and},\qquad
    \mcl{H}_1 x_{\text{f}}(t,s)&=\partial_{s}x(t,s).
\end{align*}
As such, the Poincar\'e inequality optimization problem can be equivalently represented as
\begin{align*}
    &\min_{\gamma\geq 0}\ \gamma, \qquad
    \text{s.t.} \quad \ip{\mcl{H}_{2}x_{\text{f}}}{\mcl{H}_{2}x_{\text{f}}}-\gamma\ip{\mcl{H}_1 x_{\text{f}}}{\mcl{H}_1 x_{\text{f}}}\leq 0,   \qquad \forall x_{\text{f}}\in L_2[a,b]
\end{align*}
giving rise to an LPI
\begin{align*}
    &\min_{\gamma\geq 0}\ \gamma, \qquad
    \text{s.t.}\quad \mcl{H}_{2}^*\mcl{H}_{2}-\gamma\mcl{H}_1^*\mcl{H}_1\leq 0.
\end{align*}
We declare and solve this LPI in PIETOOLS as
\begin{matlab}
\begin{verbatim}
 % % First, define dpvar gam and set up an optimization problem
 dpvar gam;
 vars = [H2.var1; H2.var2];
 prob = sosprogram(vars,gam);

 % % Set gam as objective function to minimize
 prob = sossetobj(prob, gam);

 % % Set up the constraint H2'*H2-gam H1'*H1<=0
 opts.psatz = 1;     % allow H2'*H2 > gam H1'*H1 outside of [a,b]
 prob = lpi_ineq(prob,-(H2'*H2-gam*H1'*H1),opts);

 % Solve and retrieve the solution
 prob = sossolve(prob);
 poincare_constant = sqrt(double(sosgetsol(prob,gam)));
\end{verbatim}
\end{matlab}
This code can also be run calling the function ``poincare\_inequality\_DEMO'', arriving at a constant $C=0.42664$ that satisfies the Poincar\'e inequality on the domain $[a,b]=[0,1]$. On this domain, a minimal value of $C$ is known to be $C=\frac{1}{\pi}=0.31831...$.

\section{DEMO 4: Finding an Optimal Stability Parameter}\label{sec:demos:stability}

We consider a reaction-diffusion equation on an interval $[a,b]$,
\begin{align*}
    \dot{\mbf{x}}(t,s)&=\lambda \mbf{x}(t,s) + \partial_{s}^2\mbf{x}(t,s),  &   s&\in[a,b],
    \mbf{x}(t,a)&=\mbf{x}(t,b)=0,
\end{align*}
for some $\lambda\in\R$. On the interval $[a,b]=[0,1]$, this system is know to be stable whenever $\lambda\leq\pi^2=9.8696...$. In PIETOOLS, we can numerically estimate this limit using the function \texttt{stability\_PIETOOLS}. In particular, we may equivalent represent the PDE as a PIE of the form
\begin{align*}
    \bl(\mcl{T}\dot{\mbf{x}}_{\text{f}}\br)(t,s&=\bl(\mcl{A}(\lambda)\mbf{x}_{\text{f}}\br)(t,s),   &   s&\in[a,b],
\end{align*}
where $\mbf{x}_{\text{f}}(t,s):=\partial_s^2\mbf{x}(t,s)$. Then, a maximal value of $\lambda$ for which the PIE is stable can be estimated by solving the optimization problem
\begin{align*}
    &\max_{\lambda\in\R,\mcl{P}\in\Pi}\ \lambda,  \\
    \text{s.t.}\qquad &\mcl{P}\succ 0,   \\
    &\mcl{T}^*\mcl{P}\mcl{A}(\lambda) + \mcl{A}^*(\lambda)\mcl{P}\mcl{T}\preccurlyeq 0.
\end{align*}
Unfortunately, both $\lambda$ and $\mcl{P}$ are decision variables in this optimization program, and so the product $\mcl{P}\mcl{A}(\lambda)$ is not linear in the decision variables. As such, this problem cannot be directly implemented as a convex optimization program. However, for any fixed value of $\lambda$, stability of the PIE can be verified by testing feasibility of the LPI
\begin{align*}
    &\mcl{P}\succ 0,
    &\mcl{T}^*\mcl{P}\mcl{A}(\lambda) + \mcl{A}^*(\lambda)\mcl{P}\mcl{T}\preccurlyeq 0,
\end{align*}
an optimization program that has already been implemented in \texttt{stability\_PIETOOLS}. Therefore, to estimate an upper bound on the value of $\lambda$ for which the PDE is stable, we can test stability for given values of $\lambda$, and perform bisection over some domain $\lambda\in[\lambda_{\min},\lambda_{\max}]$ to find an optimal value. For our demonstration, we use $\lambda_{\min}=0$ and $\lambda_{\max}=20$, testing stability for $8$ values of $\lambda$ between these upper and lower bounds:
\begin{matlab}
\begin{verbatim}
 %%% Set bisection limits for lam.
 lam_min = 0;        lam_max = 20;
 lam = 0.5*(lam_min + lam_max);
 n_iters = 8;
\end{verbatim}
\end{matlab}
We initialize a PDE structure in Terms-format, so that we may easily update the value of $\lambda$:
\begin{matlab}
\begin{verbatim}
 %%% Initialize a PDE structure.
 a = 0;  b = 1;
 pvar s
 pde_struct PDE;
 PDE.x{1}.vars = s;
 PDE.x{1}.dom = [a,b];

 % Set the PDE \dot{x}(t,s) = lam*x(t,s) + x_{ss}(t,s);
 PDE.x{1}.term{1}.C = lam;
 PDE.x{1}.term{2}.D = 2;

 % Set the BCs x(t,a) = x(t,b) = 0;
 PDE.BC{1}.term{1}.loc = a;
 PDE.BC{2}.term{1}.loc = b;
\end{verbatim}
\end{matlab}
Finally, we fix settings for the LPI program, using \texttt{veryheavy} settings to achieve relatively high accuracy
\begin{matlab}
\begin{verbatim}
 %%% Initialize settings for solving the LPI
 settings = lpisettings('veryheavy');
 if strcmp(settings.sos_opts.solver,'sedumi')
     settings.sos_opts.params.fid = 0;   % Suppress output in command window
 end
\end{verbatim}
\end{matlab}
Having initialized everything, we perform bisection on the value of $\lambda$. In particular, for a given value of $\lambda\in[\lambda_{\min},\lambda_{\max}]$, we update the PDE structure \texttt{PDE}, compute the associated PIE structure \texttt{PIE}, and test stability of this PIE using \texttt{stability\_PIETOOLS}. Then, we check the value of \texttt{feasratio} in the output optimization program structure, which should be close to 1 if the LPI was successfully solved, and thus the system was found to be stable. 
\begin{itemize}
    \item If the system is stable, then it is stable for any value of $\lambda$ smaller than the value \texttt{lam} used in the test. As such, we update the value of $\lambda_{\min}\leftarrow\lambda$, and repeat the test with a greater value $\lambda=\frac{1}{2}(\lambda_{\min}+\lambda_{\max})$.
    \item If stability could not be verified, then stability can also not be verified for any value of $\lambda$ greater than the value \texttt{lam} used in the test. As such, we update the value of $\lambda_{\max}\leftarrow\lambda$, and repeat the test with a greater value $\lambda=\frac{1}{2}(\lambda_{\min}+\lambda_{\max})$.
\end{itemize}
This algorithm can be implemeted as
\begin{matlab}
\begin{verbatim}
 %%% Perform bisection on the value of lam
 for iter = 1:n_iters
     % Update the value of lam in the PDE.
     PDE.x{1}.term{1}.C = lam;
    
     % Update the PIE.
     PIE = convert(PDE,'pie');
    
     % Re-run the stability test.
     prog = PIETOOLS_stability(PIE,settings);

     % Check if the system is stable
     if prog.solinfo.info.dinf || prog.solinfo.info.pinf ...
            || abs(prog.solinfo.info.feasratio - 1)>0.3
         % Stability cannot be verified, decreasing the value of lam...
         lam_max = lam;
         lam = 0.5*(lam_min + lam_max);
     else
         % The system is stable, trying a larger value of lam...
         lam_success = lam;
         lam_min = lam;
         lam = 0.5*(lam_min + lam_max);
     end    
 end
\end{verbatim}
\end{matlab}
This code can also be called using ``stability\_parameter\_bisection\_DEMO''. Running this demo, we find that stability can be verified whenever $\lambda\leq 9.8438$.

\section{DEMO 5: Constructing and Simulating an Optimal Estimator}\label{sec:demos:estimator}

We consider a reaction-diffusion PDE, with an observed output $y$,
\begin{align}\label{eq:demos:estimator:PDE}
    && \dot{\mbf{x}}(t,s)&=\partial_{s}^2\mbf{x}(t,s) + 4\mbf{x}(t,s) + w(t), & &&    s&\in[0,1], &&\nonumber\\
    \text{with BCs}& & 0&=\mbf{x}(t,0)=\partial_{s}\mbf{x}(t,1),  \nonumber\\
    \text{and outputs}& & z(t)&=\int_{0}^{1}\mbf{x}(t,s)ds + w(t),   \nonumber\\
    &&y(t)&=\mbf{x}(t,1). 
\end{align}
This PDE can be easily declared in PIETOOLS as
\begin{matlab}
\begin{verbatim}
 pvar s t
 PDE = sys();
 x = state('pde');    w = state('in');
 y = state('out');    z = state('out');
 eqs = [diff(x,t) == diff(x,s,2) + 4*x + w;
        z == int(x,s,[0,1]) + w;
        y == subs(x,s,1);
        subs(x,s,0) == 0;
        subs(diff(x,s),s,1) == 0];
 PDE = addequation(PDE,eqs);
 PDE = setObserve(PDE,y);
\end{verbatim}  
\end{matlab}
at which point an equivalent PIE representation can be derived by calling \texttt{convert}:
\begin{matlab}
\begin{verbatim}
 PIE = convert(PDE,'pie');        PIE = PIE.params;
 T = PIE.T;    
 A = PIE.A;      C1 = PIE.C1;      C2 = PIE.C2;
 B1 = PIE.B1;    D11 = PIE.D11;    D21 = PIE.D21;
\end{verbatim}
\end{matlab}
Then, the PDE~\eqref{eq:demos:estimator:PDE} can be equivalently represented by a PIE
\begin{align}\label{eq:demos:estimator:PIE}
    \bl(\mcl{T}\dot{\mbf{v}}\br)(t,s)&=\bl(\mcl{A}\mbf{v}\br)(t,s)+\bl(\mcl{B}_1 w\br)(t,s), &   s&\in[0,1]  && \nonumber\\
    z(t)&=\bl(\mcl{C}_1\mbf{v}\br)(t)+\bl(\mcl{D}_{11} w)(t), \nonumber\\
    y(t)&=\bl(\mcl{C}_2\mbf{v}\br)(t)+\bl(\mcl{D}_{12} w)(t),
\end{align}
where we define $\mbf{v}:=\partial_{s}^2\mbf{x}$, and where $\mbf{x}=\mcl{T}\mbf{v}$.

We consider the problem of designing an optimal estimator for the PIE~\eqref{eq:demos:estimator:PIE}. In particular, we construct an estimator of the form
\begin{align}\label{eq:demos:estimator:PIE_estimator}
	\mcl{T} \dot{\hat{\mbf{v}}}(t) &=\mcl{A}{\mbf{\hat{v}}}(t)+\mathcal{L}\bl(y(t)-\hat{y}(t)\br), \nonumber\\
	\hat{z}(t) &= \mcl{C}_1\mbf{\hat{v}}(t),   \nonumber\\
	\hat{y}(t) &= \mcl{C}_2\mbf{\hat{v}}(t),   
\end{align}
so that the error $\mbf{e}(t,s):=\hat{\mbf{v}}(t,s)-\mbf{v}(t,s)$ in the state and $\tilde{z}(t):=\hat{z}(t)-z(t)$ in the output satisfy
\begin{align}\label{eq:demos:estimator:PIE_error}
	\mcl T \dot{\mbf{e}}(t) &=(\mcl{A}-\mcl{L}\mcl{C}_2)\mbf{e}(t)-(\mcl{B}_1+\mcl{L}\mcl{D}_{21})w(t) \nonumber\\
	\tilde{z}(t) &= \mcl{C}_1\mbf{e}(t) - \mcl{D}_{11} w(t).
\end{align}
The goal of $H_{\infty}$-optimal estimation, then, is to determine a value for the observer operator $\mcl{L}$ that minimizes the gain $\gamma:=\frac{\|\tilde{z}\|_{L_2}}{\|w\|_{L_2}}$ from disturbances $w$ to errors $\tilde{z}$ in the output. See also Section~\ref{sec:LPI_examples:estimation}. To construct such an operator, we solve the LPI,
\begin{align}\label{eq:demos:estimator:LPI}
	&\min\limits_{\gamma,\mcl{P},\mcl{Z}} ~~\gamma&\notag\\
	&\mcl{P}\succ0, &
	&\hspace{-1ex}Q:=\bmat{-\gamma I& -\mcl D_{11}^{\top}&-(\mcl P\mcl B_1+\mcl Z\mcl D_{21})^*\mcl T\\(\cdot)^*&-\gamma I&\mcl C_1\\(\cdot)^*&(\cdot)^*&(\mcl P\mcl A+\mcl Z\mcl C_2)^*\mcl T+(\cdot)^*}\preccurlyeq 0&
\end{align}
so that, for any solution $(\gamma,\mcl{P},\mcl{Z})$ to this problem, letting $\mcl{L}:=\mcl{P}^{-1} \mcl{Z}$, the estimation error will satisfy $\norm{\hat{z}-z} \leq \gamma \norm{w}$. 

In Chapter~\ref{ch:LPIs}, we showed in detail how the LPI~\eqref{eq:demos:estimator:LPI} can be declared for a PIE of the form~\eqref{eq:demos:estimator:PIE}. The code declaring and solving this LPI can also be found in the demo file ``Hinf\_Optimal\_Estimator\_DEMO''. We will therefore not discuss this implementation here. Instead, we compute an optimal value of the observer operator using one of the pre-defined \texttt{executive} files, as
\begin{matlab}
\begin{verbatim}
 settings = lpisettings('heavy');
 [prog, Lval, gam_val] = PIETOOLS_Hinf_estimator(PIE, settings); 
\end{verbatim}
\end{matlab}
returning a value \texttt{Lval} of the operator $\mcl{L}$ such that the $L_2$-gain $\frac{\|\tilde{z}\|_{L_2}}{\|w\|_{L_2}}$ is bounded from above by \texttt{gam\_val}.

Given the observer operator $\mcl{L}$, we can construct the Estimator~\eqref{eq:demos:estimator:PIE_estimator}, obtaining a PIE
\begin{align*}
    \left(\bmat{\mcl{T}&0\\0&\mcl{T}}\bmat{\dot{\mbf{v}}\\\dot{\hat{\mbf{v}}}}\right)(t,s)&=\left(\bmat{\mcl{A}&0\\-\mcl{L}\mcl{C}_{2}&\mcl{A}+\mcl{L}\mcl{C}_{2}}\bmat{\mbf{v}\\\hat{\mbf{v}}}\right)(t,s)+\left(\bmat{\mcl{B}_{1}\\\mcl{L}\mcl{D}_{21}}w\right)(t) \\
    \bmat{z\\\hat{z}}(t)&=\left(\bmat{\mcl{C}_1&0\\0&\mcl{C}_1}\bmat{\mbf{v}\\\hat{\mbf{v}}}\right)(t)+\left(\bmat{\mcl{D}_{11}\\0}w\right)(t)
\end{align*}
We can declare this PIE in PIETOOLS as
\begin{matlab}
\begin{verbatim}
 % Construct the operators defining the PIE.
 T_CL = [T, 0*T; 0*T, T];
 A_CL = [A, 0*A; -Lval*C2, A+Lval*C2];   B_CL = [B1; Lval*D21];
 C_CL = [C1, 0*C1; 0*C1, C1];            D_CL = [D11; 0*D11];

 % Declare the PIE.
 PIE_CL = pie_struct();
 PIE_CL.vars = PIE.vars;
 PIE_CL.dom = PIE.dom;
 PIE_CL.T = T_CL;
 PIE_CL.A = A_CL;              PIE_CL.B1 = B_CL;
 PIE_CL.C1 = C_CL;             PIE_CL.D11 = D_CL;
 PIE_CL = initialize(PIE_CL);
\end{verbatim}    
\end{matlab}
Then, we can use PIESIM to simulate the evolution of the PDE state $\mbf{x}(t)=\bl(\mcl{T}\mbf{v}\br)(t)$ and its estimate $\hat{\mbf{x}}(t)=\bl(\mcl{T}\hat{\mbf{v}}\br)(t)$ associated to the system defined by \texttt{PIE\_CL}. For this, we first declare initial conditions $\sbmat{\mbf{v}(0,s)\\\hat{\mbf{v}}(0,s)}$ and values of the disturbance $w(t)$ as
\begin{matlab}
\begin{verbatim}
 % Declare initial conditions for the state components of the PIE
 syms st sx real
 uinput.ic.PDE = [-10*sx;    % Actual initial PIE state value
                  0];        % Estimated initial PIE state value
 
 % Declare the value of the disturbance w(t)
 uinput.w = 2*sin(pi*st);
\end{verbatim}
\end{matlab}
Here we use symbolic objects \texttt{st} and \texttt{sx} to represent a temporal variable $t$ and spatial variable $s$ respectively. Note that, since we will be simulating the PIE directly, the initial conditions \texttt{uinput.ic.PDE} will also correspond to the initial values of the PIE state $\sbmat{\mbf{v}(0,s)\\\hat{\mbf{v}}(0,s)}$, not to those of the PDE state. Here, we let the initial PIE state $\mbf{v}(0)$ be parabolic, and start with an estimate of this state that is just $\hat{\mbf{v}}(0)=0$. Given these initial conditions, we then simulate the evolution of the PDE state $\sbmat{\mbf{x}\\\mbf{\hat{x}}}=\sbmat{\mcl{T}&0\\0&\mcl{T}}\sbmat{\dot{\mbf{v}}\\\dot{\hat{\mbf{v}}}}$ using PIESIM as
\begin{matlab}
\begin{verbatim}
 % Set options for the discretization and simulation:
 opts.plot = 'no';   % Do not plot the final solution
 opts.N = 8;         % Expand using 8 Chebyshev polynomials
 opts.tf = 1;        % Simulate up to t = 1;
 opts.dt = 1e-3;     % Use time step of 10^-3
 opts.intScheme = 1; % Time-step using Backward Differentiation Formula (BDF) 
 ndiff = [0,0,2];    % The PDE state involves 2 twice differentiable states

 % Simulate the solution to the PIE with estimator.
 [solution,grid] = PIESIM(PIE_CL,opts,uinput,ndiff);

 % Extract actual and estimated solution at each time step.
 x_act = reshape(solution.timedep.pde(:,1,:),opts.N+1,[]);
 x_est = reshape(solution.timedep.pde(:,2,:),opts.N+1,[]);
 tval = solution.timedep.dtime;
\end{verbatim}
\end{matlab}
Here, we use \texttt{opts.N=8} Chebyshev polynomials in the expansion of the solution, and simulate up to \texttt{opts.tf=1}, taking time steps of \texttt{opts.dt=1e-3}. Having obtained the solution, we then collect the actual values $\mbf{x}(t,s)$ of the PDE state at each time step and each grid point in \texttt{x\_act}, and the estimated values $\mbf{\hat{x}}(t,s)$ of the PDE state at each time step and each grid point in \texttt{x\_est}. Plotting the obtained values at several grid points, as well as the error in estimated state, we obtain a graph as in Figure~\ref{fig:demos:estimator}.

\begin{figure}[H]
	\centering
	\includegraphics[scale=0.75]{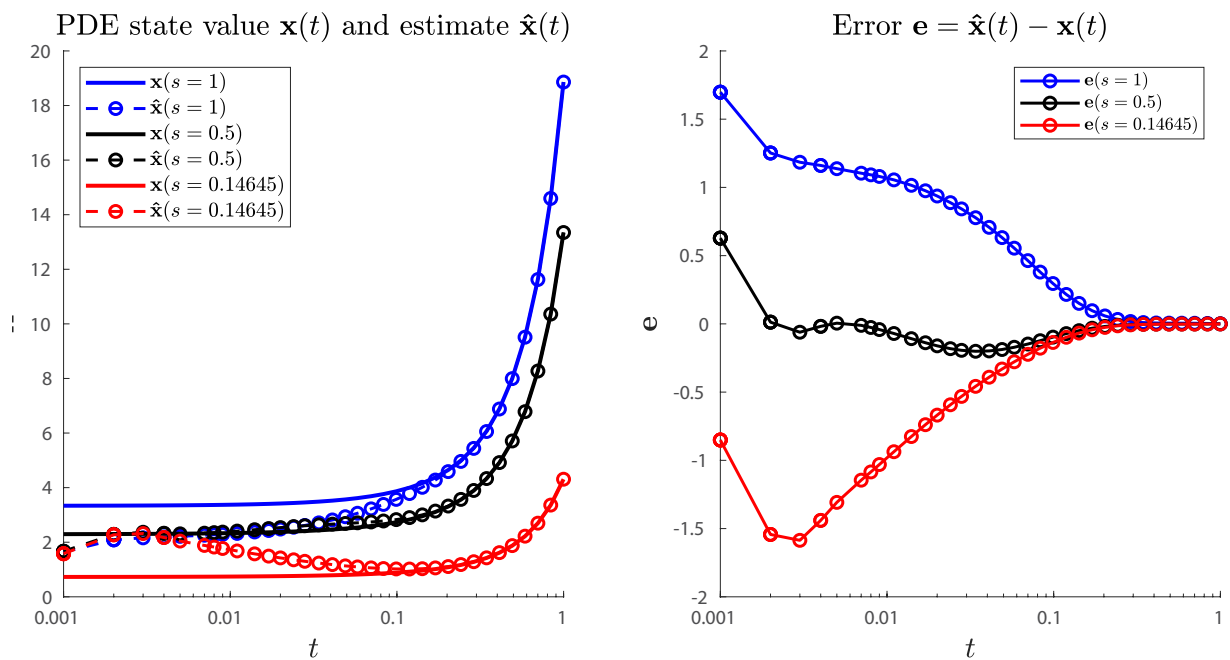}
	\caption{Simulated value of PDE state $\mbf{x}(t,s)$ and estimated state $\mbf{\hat{x}}(t,s)$, along with the error $\mbf{e}=\hat{\mbf{x}}(t,s)-mbf{x}(t,s)$ associated to the PDE~\eqref{eq:demos:estimator:PDE} at several grid points $s\in[0,1]$, using the Estimator~\eqref{eq:demos:estimator:PIE_estimator} with operator $\mcl{L}$ computed by solving the $H_{\infty}$-optimal estimator LPI. See the demo file ``Hinf\_Optimal\_Estimator\_DEMO''.}\label{fig:demos:estimator}
\end{figure}

As Figure~\ref{fig:demos:estimator} shows, the value of the estimate state $\mbf{\hat{x}}$ at each of the grid points quickly converge to those of the actual state $\mbf{x}$, despite starting off with a rather poor estimate of $\mbf{\hat{x}}=0$. Within a time of 1, the values of the actual and estimated state become indistinguishable (at the displayed scale), with the error converging to zero. This is despite the fact that the value of the state itself continues to increase, as the PDE~\eqref{eq:demos:estimator:PDE} is unstable.

The full code constructing the optimal estimator, simulating the PDE state and its estimate, and plotting the results has been included in PIETOOLS as the demo file\\
``Hinf\_Optimal\_Estimator\_DEMO''.

\section{DEMO 6: \texorpdfstring{$\hinf$}{Hinf}-optimal Controller synthesis for PDEs}\label{sec:demo:control}
We consider an unstable reaction-diffusion PDE, with output $z$, disturbance $w$ and control input $u$
\begin{align}\label{eq:demos:control:PDE}
    && \dot{\mbf{x}}(t,s)&=\partial_{s}^2\mbf{x}(t,s) + 4\mbf{x}(t,s) + sw(t)+su(t), & &&    s&\in[0,1], &&\nonumber\\
    \text{with BCs}& & 0&=\mbf{x}(t,0)=\partial_{s}\mbf{x}(t,1),  \nonumber\\
    \text{and outputs}& & z(t)&=\bmat{\int_{0}^{1}\mbf{x}(t,s)ds + w(t)\\u(t)}. 
\end{align}
This PDE can be easily declared in PIETOOLS as
\begin{matlab}
\begin{verbatim}
pvar s t
lam = 4;
PDE = sys();
x = state('pde');   w = state('in');
z = state('out', 2);   u = state('in');
eqs = [diff(x,t) == diff(x,s,2) + lam*x + s*w + s*u;
    z == [int(x,s,[0,1]) + w; u];
    subs(x,s,0)==0;
    subs(diff(x,s),s,1)==0];
PDE = addequation(PDE,eqs);
PDE = setControl(PDE,u);
display_PDE(PDE);
\end{verbatim}  
\end{matlab}
Next, we convert the PDE system to a PIE by using the following code and extract relevant PI operators for easier access:
\begin{matlab}
\begin{verbatim}
PIE = convert(PDE,'pie');       PIE = PIE.params;
T = PIE.T;
A = PIE.A;      C1 = PIE.C1;    B2 = PIE.B2;
B1 = PIE.B1;    D11 = PIE.D11;  D12 = PIE.D12;
\end{verbatim}
\end{matlab}
Then, the PDE~\eqref{eq:demos:control:PDE} can be equivalently represented by a PIE
\begin{align}\label{eq:demos:control:PIE}
    \bl(\mcl{T}\dot{\mbf{v}}\br)(t,s)&=\bl(\mcl{A}\mbf{v}\br)(t,s)+\bl(\mcl{B}_1 w\br)(t,s)+\bl(\mcl{B}_2 u\br)(t,s), &   s&\in[0,1]  && \nonumber\\
    z(t)&=\bl(\mcl{C}_1\mbf{v}\br)(t)+\bl(\mcl{D}_{11} w)(t)+\bl(\mcl{D}_{12} u)(t),
\end{align}
where we define $\mbf{v}:=\partial_{s}^2\mbf{x}$ and  $\mbf{x}=\mcl{T}\mbf{v}$.

We now attempt to find a state-feedback $\hinf$-optimal controller for the PIE~\eqref{eq:demos:control:PIE}. In particular, we use $u(t) = \mcl K \mbf v(t)$, where $\mcl K:L_2\to \R$ is a PI operator of the form
\[\mcl K \mbf v = \int_a^b K(s)\mbf v(s) ds.\] 
Then we get the closed loop system
\begin{align}\label{eq:demos:control:PIE_CL}
	\mcl{T} \dot{\mbf{v}}(t) &=(\mcl{A}+\mcl B_2\mcl K){\mbf{v}}(t)+\mcl B_1 w(t), \nonumber\\
	z(t) &= (\mcl{C}_1+\mcl D_{12}\mcl K)\mbf{z}(t)+\mcl D_{11} w(t).
\end{align}

Next, we solve the LPI for $\hinf$-optimal control to find a $\mcl K$ that minimizes the gain $\gamma:=\frac{\|z\|_{L_2}}{\|w\|_{L_2}}$ from disturbances $w$ to errors $\tilde{z}$ in the output. See also Section~\ref{sec:LPI_examples:control}. To find such an operator, we solve the LPI,
\begin{flalign}\label{eq:demos:control:LPI}
	\min\limits_{\gamma,\mcl{P},\mcl{Z}} ~~\gamma~s.t.\quad&\mcl{P}\succ0&\notag\\
	&\bmat{-\gamma I & \mcl{D}_{11}&\mcl{T}(\mcl{P}\mcl{C}_1+\mcl{Z}\mcl{D}_{12})\\(\cdot)^*&-\gamma I&\mcl{B}_1^*\\(\cdot)^*&(\cdot)^*&\mcl{T}(\mcl{A}\mcl{P}+\mcl{B}_2\mcl{Z})^*+(\mcl{A}\mcl{P}+\mcl{B}_2\mcl{Z})\mcl{T}^*}\preccurlyeq 0&
\end{flalign}
so that, for any solution $(\gamma,\mcl{P},\mcl{Z})$ to this problem, letting $\mcl{K}:=\mcl Z\mcl{P}^{-1}$, we have $\norm{z} \leq \gamma \norm{w}$. 

In Chapter~\ref{ch:LPIs}, we showed in detail how the LPI~\eqref{eq:demos:control:LPI} can be declared for a PIE of the form~\eqref{eq:demos:control:PIE}. The code declaring and solving this LPI can also be found in the demo file ``Hinf\_optimal\_control\_DEMO''. Here we use pre-defined \texttt{executive} files to solve the above LPI, which is a simple function call as shown below
\begin{matlab}
\begin{verbatim}
 settings = lpisettings('heavy');
 [prog, Kval, gam_val] = PIETOOLS_Hinf_control(PIE, settings); 
\end{verbatim}
\end{matlab}
The executive, in this particular case, returns a value \texttt{Kval} of the operator $\mcl{K}$ such that the $L_2$-gain $\frac{\|z\|_{L_2}}{\|w\|_{L_2}}$ is bounded from above by \texttt{gam\_val}. Next, we construct the closed loop PIE using the function \texttt{closedLoop\_PIE},
\begin{matlab}
\begin{verbatim}
    PIE_CL = closedLoopPIE(PIE,Kval);
\end{verbatim}
\end{matlab}
Note, closed loop system can be constructed manually, similar to the method used in \ref{sec:demos:estimator}.

Then, we can use PIESIM to simulate the evolution of the PDE state $\mbf{x}(t)=\bl(\mcl{T}\mbf{v}\br)(t)$ associated to the system defined by \texttt{PIE\_CL}. As described before we can set initial condition and disturbance using the commands below
\begin{matlab}
\begin{verbatim}
 % Declare initial conditions for the state components of the PIE
 syms st sx real
 uinput.ic.PDE = -10*sx;    % IC PIE
 uinput.w = exp(-st);   % disturbance
\end{verbatim}
\end{matlab}
Note that, since we will be simulating the PIE directly, the initial conditions \texttt{uinput.ic.PDE} will also correspond to the initial values of the PIE state $\mbf{v}(0,s)$, not to those of the PDE state. 
In this example, we have used multiple initial conditions and disturbance inputs. For brevity, the code corresponding to all the different initial conditions and disturbances is not presented. 

Next, we set the parameters related to the numerical scheme and simulation as shown below
\begin{matlab}
\begin{verbatim}
 % Set options for the discretization and simulation:
 opts.plot = 'no';   % Do not plot the final solution
 opts.N = 8;         % Expand using 8 Chebyshev polynomials
 opts.tf = 1;        % Simulate up to t = 1;
 opts.dt = 1e-3;     % Use time step of 10^-3
 opts.intScheme = 1; % Time-step using Backward Differentiation Formula (BDF) 
 ndiff = [0,0,1];    % The PDE state involves 2 twice differentiable states
\end{verbatim}
\end{matlab}
These parameters are same as the ones described in \ref{sec:demos:estimator}, and hence the description is omitted here.

We run multiple simulations for both open (without controller) and closed loop PIEs (with controller) the results of which are plotted in Figure~\ref{fig:demos:control_ol} and Figure \ref{fig:demos:control_cl}, respectively.

\begin{figure}[H]
	\centering
	\includegraphics[width=\textwidth]{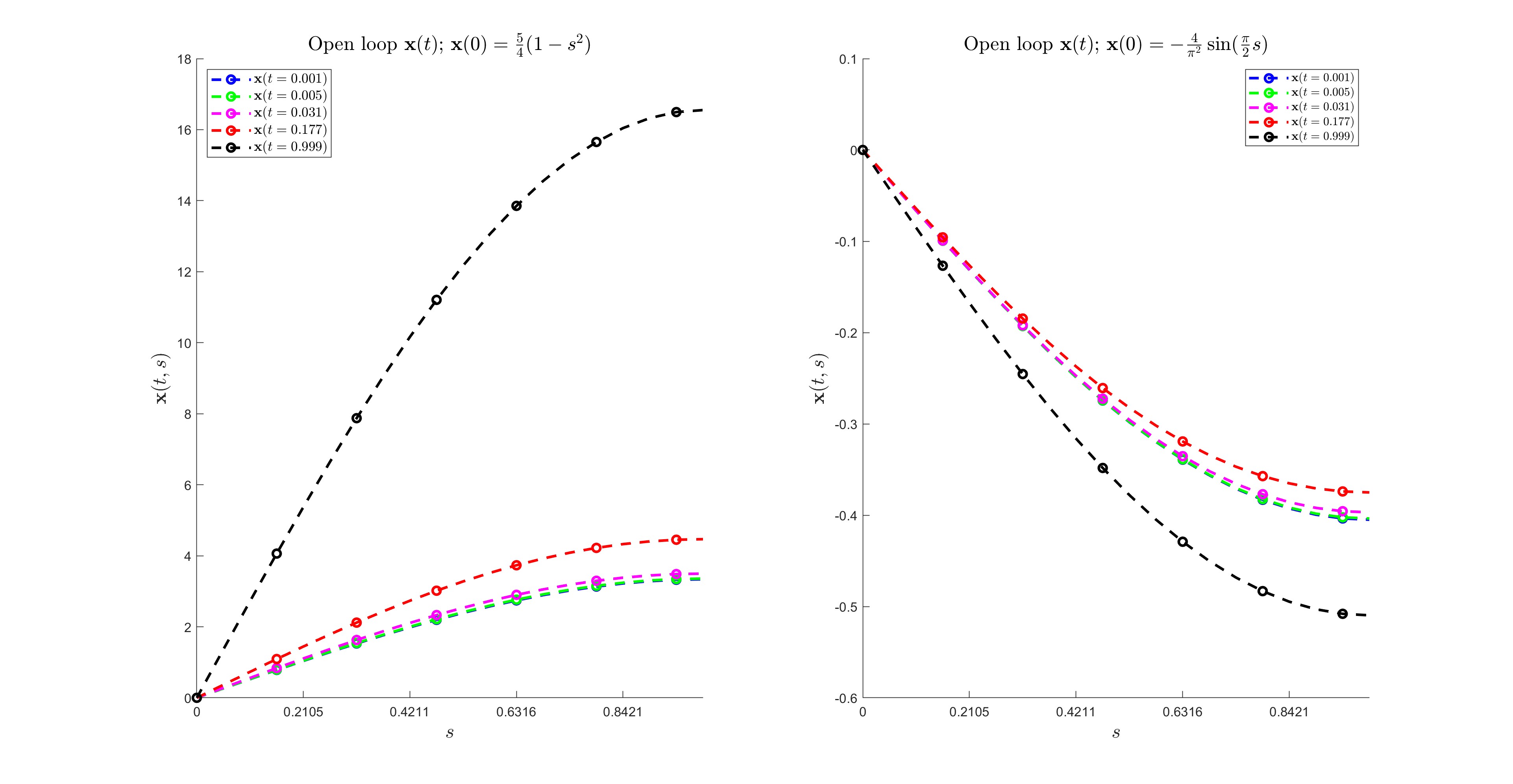}
	\caption{Simulated value of PDE state $\mbf{x}(t,s)$ at several grid points $s\in[0,1]$, without controller input, for different initial conditions and a bounded $L_2$ disturbance $w$.}\label{fig:demos:control_ol}
\end{figure}

\begin{figure}[H]
	\centering
	\includegraphics[width=\textwidth]{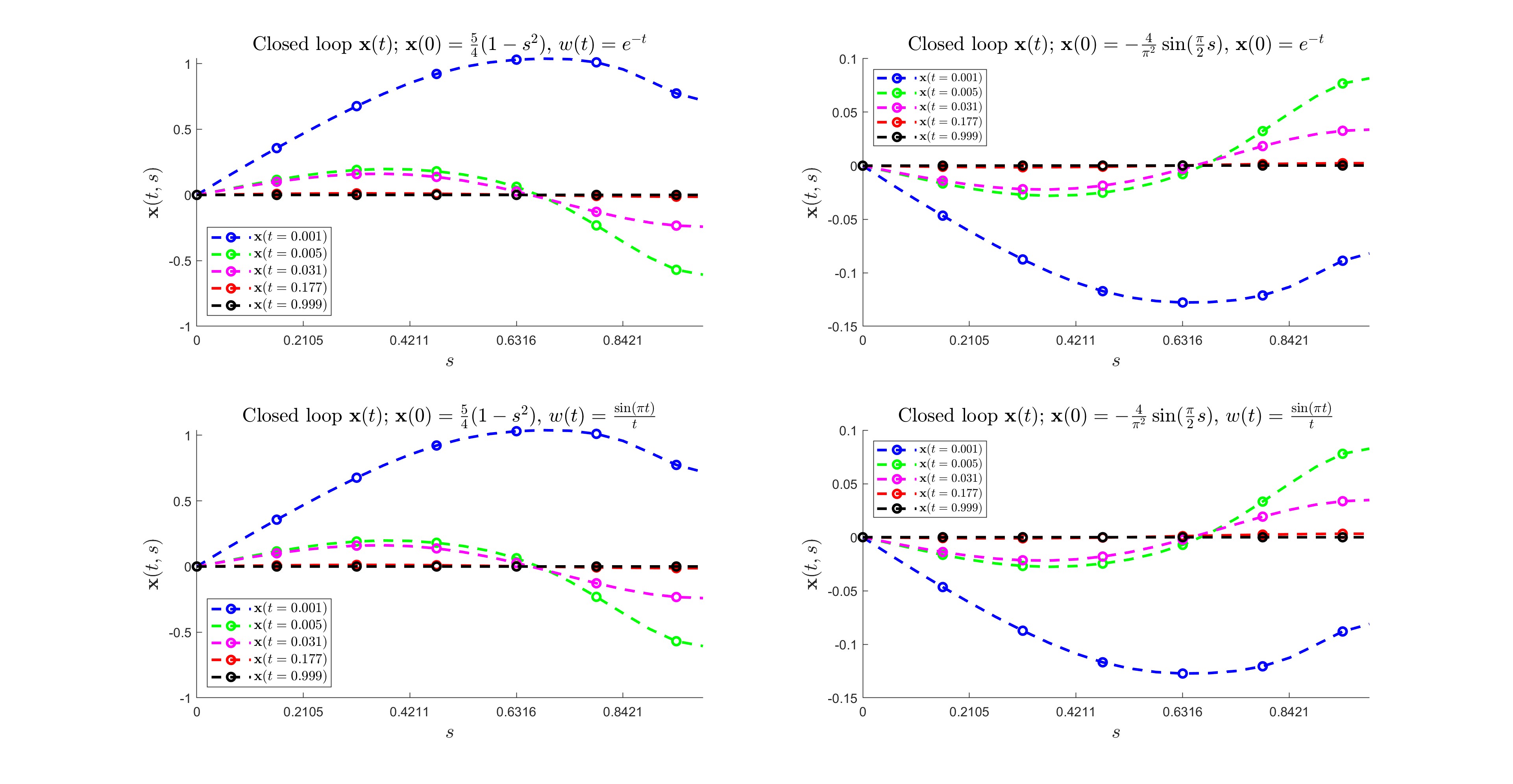}
	\caption{Simulated value of PDE state $\mbf{x}(t,s)$ at several grid points $s\in[0,1]$, with $\hinf$-optimal state feedback controller, for different initial conditions and a bounded $L_2$ disturbance $w$. $L_2$-gain $= 1.6205$.}\label{fig:demos:control_cl}
\end{figure}

Lastly, for the initial condition $\mbf v(0,s) = -10s$, we plot the regulated output $z(t)$ when subjected to two different disturbance inputs, namely, $w(t) = \exp(-t)$ and $w(t) = \frac{\sin(\pi t)}{t}$. This can be obtained by numerical integration of the solution as shown in the code below, which can they be plotted to obtain Figure \ref{fig:demos:control_output}.
\begin{matlab}
\begin{verbatim}
tval = linspace(0,1,2000); % grid points in time
XX = linspace(0,1,20); % grid points in space
w1_tval = subs(sin(pi*st)./(st+eps),tval); 
w2_tval = subs(exp(-st),tval);
z_quadrature = double(subs(C1.Q1(1,1),s,XX)); 
k_quadrature = double(subs(Kval.Q1,s,XX));
ZZ1 = trapz(z_quadrature,x_CL_a)+double(w1_tval); %z1 for w1 disturbance
ZZ2 = trapz(z_quadrature,x_CL_aa)+double(w2_tval);%z1 for w2 disturbance
ZZ3 = trapz(k_quadrature,x_CL_a); %u(t) for w1 disturbance
ZZ4 = trapz(k_quadrature,x_CL_aa); %u(t) for w1 disturbance
\end{verbatim}
\end{matlab}
\begin{figure}[H]
	\centering
	\includegraphics[width=\textwidth]{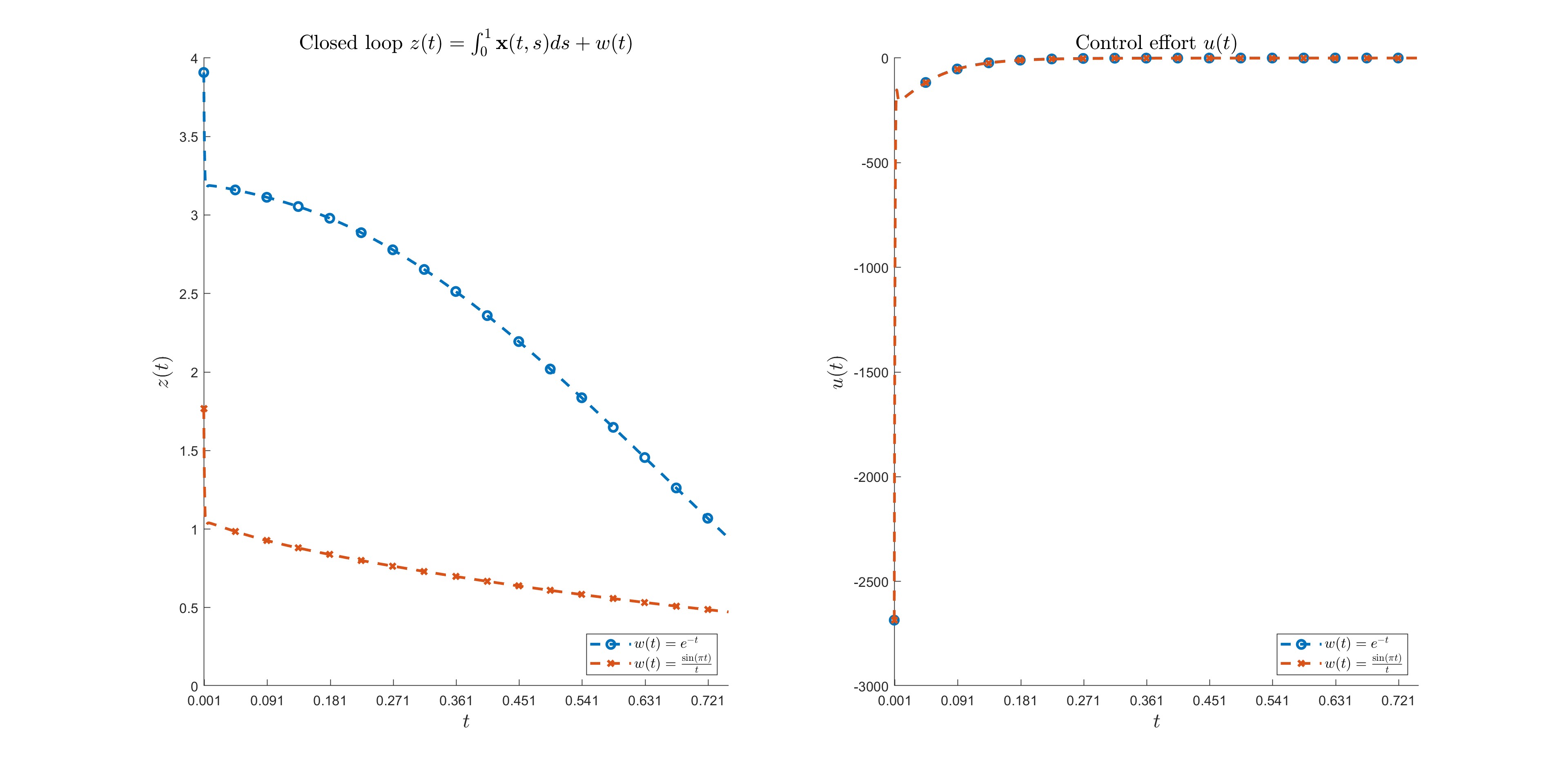}
	\caption{Simulated value of $Z(t)$, with $\hinf$-optimal state feedback controller, for different bounded $L_2$ disturbances, $w$. $L_2$-gain $= 1.6205$.}\label{fig:demos:control_output}
\end{figure}

The full code for designing the optimal controller, simulating the PDE state, and plotting the results has been included in PIETOOLS as the demo file\\
``Hinf\_optimal\_control\_DEMO''

\section{DEMO 7: Observer-based Controller design and simulation for PDEs}\label{sec:demo:observer-control}
In Sections \ref{sec:demos:estimator} and \ref{sec:demo:control}, we showed the procedure to find optimal estimator and controller for a PDE. The controller was designed based on the information of the complete PDE state, which, in practice, is unrealistic. So, in this section, we use the observer and controller obtained in the previous sections and couple them by changing the control law to be dependent on the estimated state instead of PDE state, i.e., we set $u = \mcl K \hat{\mbf v}$, where $\mcl K$ is the control gains obtained in \ref{sec:demo:control} and $\hat{\mbf v}$ is the state estimate from the observer found in \ref{sec:demos:estimator}.

For demonstration, we reuse the unstable reaction-diffusion PDE, with output $z$, disturbance $w$ and control input $u$
\begin{align}\label{eq:demos:obs-control:PDE}
    && \dot{\mbf{x}}(t,s)&=\partial_{s}^2\mbf{x}(t,s) + 4\mbf{x}(t,s) + sw(t)+su(t), & &&    s&\in[0,1], &&\nonumber\\
    \text{with BCs}& & 0&=\mbf{x}(t,0)=\partial_{s}\mbf{x}(t,1),  \nonumber\\
    \text{and outputs}& & z(t)&=\bmat{\int_{0}^{1}\mbf{x}(t,s)ds + w(t)\\u(t)}\notag\\
    && y(t) &= \mbf x(t,1). 
\end{align}
Recall, this PDE can be defined using the code
\begin{matlab}
\begin{verbatim}
pvar s t;
lam = 4;
PDE = sys();
x = state('pde');   w = state('in');
z = state('out', 2);   u = state('in'); y = state('out);
eqs = [diff(x,t) == diff(x,s,2) + lam*x + s*w + s*u;
    z == [int(x,s,[0,1]) + w; u];
    y == subs(x,s,1);
    subs(x,s,0)==0;
    subs(diff(x,s),s,1)==0];
PDE = addequation(PDE,eqs);
PDE = setControl(PDE,u);
PDE = setObserve(PDE,u);
display_PDE(PDE);
\end{verbatim}  
\end{matlab}
Next, we convert the PDE system to a PIE by using the following code and extract relevant PI operators for easier access:
\begin{matlab}
\begin{verbatim}
PIE = convert(PDE,'pie');       PIE = PIE.params;
T = PIE.T;
A = PIE.A;      C1 = PIE.C1;    B2 = PIE.B2;
B1 = PIE.B1;    D11 = PIE.D11;  D12 = PIE.D12;
C2 = PIE.C2;    D21 = PIE.D21;  D22 = PIE.D22;
\end{verbatim}
\end{matlab}

Details on observer design and controller design can be found in Sections \ref{sec:demos:estimator} and \ref{sec:demo:control} respectively, and therefore, will not be repeated here. For now, we proceed assuming the observer gains $\mcl L$ and the controller gains $\mcl K$ are known.

Recall, from Equations \eqref{eq:demos:estimator:PIE_estimator} and \eqref{eq:demos:control:PIE}, the combined system is given by
\begin{align}\label{eq:demos:obscon:PIE}
\bl(\mcl{T}\dot{\mbf{v}}\br)(t)&=\bl(\mcl{A}\mbf{v}\br)(t)+\bl(\mcl{B}_1 w\br)(t)+\bl(\mcl{B}_2 u\br)(t), \nonumber\\
    z(t)&=\bl(\mcl{C}_1\mbf{v}\br)(t)+\bl(\mcl{D}_{11} w)(t)+\bl(\mcl{D}_{12} u)(t),\notag\\
    \mcl{T} \dot{\hat{\mbf{v}}}(t) &=\mcl{A}{\mbf{\hat{v}}}(t)+\mathcal{L}\bl(\mcl{C}_2\mbf{v}(t)-\mcl{C}_2\mbf{\hat{v}}(t)\br), \nonumber\\
	\hat{z}(t) &= \mcl{C}_1\mbf{\hat{v}}(t).
\end{align}
Using the observer-controller coupling, $u = \mcl K \hat{\mbf v}$, we get the closed loop PIE
\begin{align}\label{eq:demos:obcon:PIE_CL}
    \left(\bmat{\mcl{T}&0\\0&\mcl{T}}\bmat{\dot{\mbf{v}}\\\dot{\hat{\mbf{v}}}}\right)(t,s)&=\left(\bmat{\mcl{A}&\mcl B_2\mcl K\\-\mcl{L}\mcl{C}_{2}&\mcl{A}+\mcl{L}\mcl{C}_{2}}\bmat{\mbf{v}\\\hat{\mbf{v}}}\right)(t,s)+\left(\bmat{\mcl{B}_{1}\\\mcl{L}\mcl{D}_{21}}w\right)(t) \notag\\
    \bmat{z\\\hat{z}}(t)&=\left(\bmat{\mcl{C}_1&\mcl D_{12}\mcl K\\0&\mcl{C}_1}\bmat{\mbf{v}\\\hat{\mbf{v}}}\right)(t)+\left(\bmat{\mcl{D}_{11}\\0}w\right)(t).
\end{align}

We can construct the above closed loop system using the code
\begin{matlab}
\begin{verbatim}
T_CL = [T, 0*T; 0*T, T];
A_CL = [A, B2*Kval; -Lval*C2, A+Lval*C2];   B_CL = [B1; Lval*D21];
C_CL = [C1, D12*Kval; 0*C1, C1];            D_CL = [D11; 0*D11];
PIE_CL = pie_struct();
PIE_CL.vars = PIE.vars;
PIE_CL.dom = PIE.dom;
PIE_CL.T = T_CL;
PIE_CL.A = A_CL;        PIE_CL.B1 = B_CL;
PIE_CL.C1 = C_CL;       PIE_CL.D11 = D_CL;
PIE_CL = initialize(PIE_CL);
\end{verbatim}
\end{matlab}

Once, we have the closed loop PIE object, we can use \texttt{PIESIM} to simulate the system with different initial conditions and disturbances. Following the process presented in Section \ref{sec:demo:control} and \ref{sec:demos:estimator}, we use the initial condition $\mbf v(0,s) = -10s$ and $\hat{\mbf v}(0,s)=0$. Using the same options and disturbance as before, we simulate the open loop system (without control or observer) using the code
\begin{matlab}
\begin{verbatim}
syms st sx real
opts.plot = 'no';   % Do not plot the final solution
opts.N = 8;         % Expand using 8 Chebyshev polynomials
opts.tf = 1;        % Simulate up to t = 1;
opts.dt = 1e-3;     % Use time step of 10^-3
opts.intScheme=1;   % Time-step using Backward Differentiation Formula (BDF)
ndiff = [0,0,1];    % The PDE state involves 1 second order differentiable state variables

% Simulate the solution to the PIE without controller for different IC.
uinput.ic.PDE = [-10*sx];   % IC PIE 
uinput.w = exp(-st); % disturbance
[solution_OL,grid] = PIESIM(PIE,opts,uinput,ndiff);
\end{verbatim}
\end{matlab}
Then, we can access the time-varying PDE solution from \texttt{solution\_OL} for post-processing. In Figure \ref{fig:demos:obcon_ol} we post the solution at different time values to show that the open loop system is unstable.

\begin{figure}[H]
	\centering
	\includegraphics[width=\textwidth]{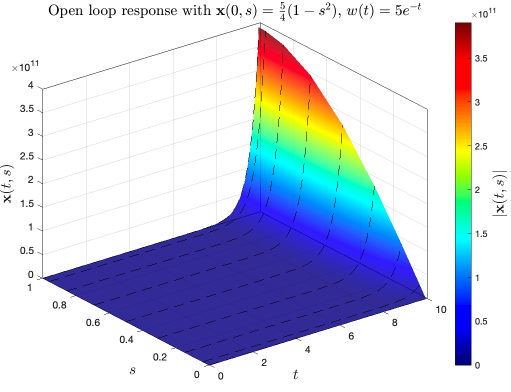}
	\caption{Simulated value of PDE state $\mbf{x}(t,s)$ at several grid points $s\in[0,1]$, without controller input, with some initial condition and a bounded $L_2$ disturbance $w$.}\label{fig:demos:obcon_ol}
\end{figure}

Likewise, we can simulate and extract the closed loop system solution using the code
\begin{matlab}
\begin{verbatim}
% Simulate the solution to the PIE with controller for different IC and disturbance.
ndiff = [0,0,2]; 
uinput.ic.PDE = [-10*sx; 0];    % IC PIE and observed state
uinput.w = exp(-st);   % disturbance
[solution_CL_a,grid] = PIESIM(PIE_CL,opts,uinput,ndiff);
uinput.ic.PDE = [sin(sx*pi/2); 0];    % IC PIE
uinput.w = 5*sin(pi*st)./(st+eps);   % disturbance
[solution_CL_b,grid] = PIESIM(PIE_CL,opts,uinput,ndiff);
% Extract actual solution at each time step.
tval = solution_CL_a.timedep.dtime;
x_CL_a = reshape(solution_CL_a.timedep.pde(:,1,:),opts.N+1,[]);
hatx_CL_a = reshape(solution_CL_a.timedep.pde(:,2,:),opts.N+1,[]);
x_CL_b = reshape(solution_CL_b.timedep.pde(:,1,:),opts.N+1,[]);
hatx_CL_b = reshape(solution_CL_b.timedep.pde(:,2,:),opts.N+1,[]);
\end{verbatim}
\end{matlab}
where, \textbf{note} that \texttt{ndiff} is changed to $[0,0,2]$ because the closed loop PIE system has two distributed states corresponding to two PDE states (PDE state and observer state) that are twice differentiable in space. Using the solutions stored in \texttt{x\_CL\_a(b)}, we can plot the response of the closed loop system for different initial conditions and disturbances as shown below in Figure \ref{fig:demos:obcon_cl}.
\begin{figure}[H]
	\centering
	\includegraphics[width=\textwidth]{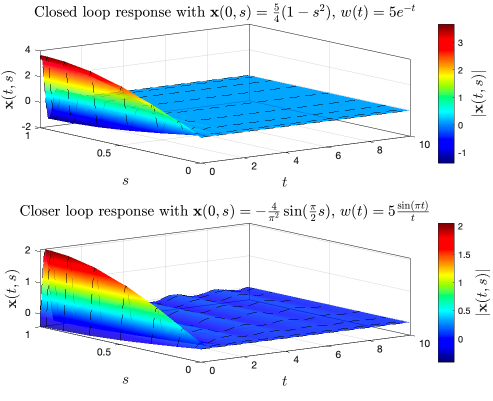}
	\caption{Simulated value of PDE state $\mbf{x}(t,s)$ at several grid points $s\in[0,1]$, with $\hinf$-optimal state feedback controller, for different initial conditions and bounded $L_2$ disturbances $w$. $L_2$-gain $= 1.6205$ for controller, $L_2$-gain $= 1.005$ for observer.}\label{fig:demos:obcon_cl}
\end{figure}

Lastly, for the initial condition $\mbf v(0,s) = -10s$ and $\hat{\mbf v}(0)=0$, we plot (see Figure \ref{fig:demos:obcon_out} the regulated output $z(t)$ when subjected to two different disturbance inputs, namely, $w(t) = \exp(-t)$ and $w(t) = 5\frac{\sin(\pi t)}{t}$, where $z$ is obtained by numerical integration of the solution as described in Section \ref{sec:demo:control}. 

\begin{figure}[H]
	\centering
	\includegraphics[width=\textwidth]{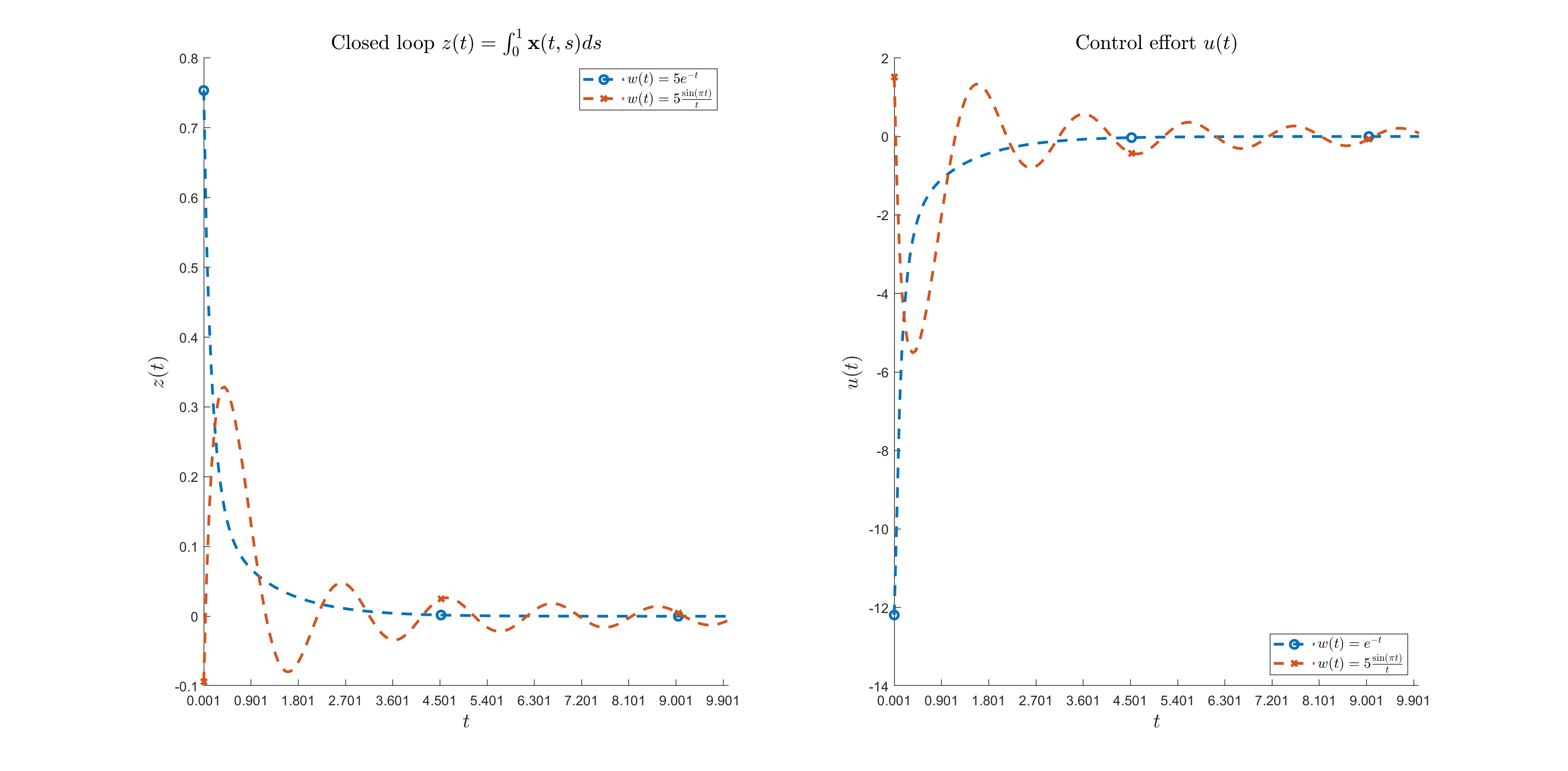}
	\caption{Simulated value of $Z(t)$, with $\hinf$-optimal state feedback controller, for different bounded $L_2$ disturbances, $w$.}\label{fig:demos:obcon_out}
\end{figure}

\chapter{Libraries of PDE and TDS Examples in PIETOOLS}\label{ch:examples}

In Chapters~\ref{ch:PDE_DDE_representation},~\ref{ch:alt_PDE_input} and~\ref{ch:alt_DDE_input}, we have shown how partial differential equations and time-delay systems can be declared in PIETOOLS through different input formats. To help get started with each of these input formats, PIETOOLS includes a variety of example pre-defined PDE and TDS systems, including common examples and particular models from the literature. These examples are collected in the \texttt{PIETOOLS\_examples} folder, and can be accessed calling the function \texttt{examples\_PDE\_library\_PIETOOLS} and the scripts \texttt{examples\_DDE\_library\_PIETOOLS} and \texttt{examples\_NDSDDF\_library\_PIETOOLS}. In this Chapter, we illustrate how this works, focusing on PDE examples in Section~\ref{sec:examples:PDE}, and TDS examples in Section~\ref{sec:examples:TDS}.

\section{A Library of PDE Example Problems}\label{sec:examples:PDE}

To help get started with analysing and simulating PDEs in PIETOOLS, a variety of PDE models have been included in the directory \texttt{PIETOOLS\_examples/Examples\_Library}. These systems include common PDE models, as well as examples from the literature, and are defined through separate MATLAB functions. Each of these functions takes two arguments
\begin{enumerate}
    \item \texttt{GUI}: A binary index (0 or 1) indicating whether the example should be opened in the graphical user interface;
    \item \texttt{params}: A string for specifying allowed parameters in the system.
\end{enumerate}
For example, calling \texttt{help PIETOOLS\_PDE\_Ex\_Reaction\_Diffusion\_Eq}, PIETOOLS indicates that this function declares a reaction-diffusion PDE
\begin{align*}
    \dot{\mbf{x}}(t,s)&=\lambda \mbf{x}(t,s) + \partial_{s}^2\mbf{x}(t,s),  &   s&\in[0,1]
    \mbf{x}(t,0)&=\mbf{x}(t,1)=0,
\end{align*}
where the value of the parameter $\lambda$ can be set. Then, calling
\begin{matlab}
\begin{verbatim}
 >> PDE = PIETOOLS_PDE_Ex_Reaction_Diffusion_Eq(0,{'lam=10;'})    
\end{verbatim}
\end{matlab}
we obtain a \texttt{pde\_struct} object \texttt{PDE} representing the reaction-diffusion equation with $\lambda=10$. Calling
\begin{matlab}
\begin{verbatim}
 >> PDE = PIETOOLS_PDE_Ex_Reaction_Diffusion_Eq(1,{'lam=10;'})    
\end{verbatim}
\end{matlab}
The PDE will also be loaded in the GUI, though a default value of $\lambda=9.86$ will be used, as the GUI will always load a pre-defined file, which cannot be adjusted from the command line.

To simplify the process of extracting PDE examples, PIETOOLS includes a function \texttt{examples} \texttt{\_PDE\_library\_PIETOOLS()}. In this function, each of the pre-defined PDEs is assigned an index, allowing desired PDEs to be extracted by calling \texttt{examples\_PDE\_library\_PIETOOLS} with the associated index. For example, scrolling through this function we find that the reaction-diffusiong equation is the fifth system in the list, and therefore, we can obtain a \texttt{pde\_struct} object defining this system by calling the library function with argument ``5'', returning
\begin{matlab}
\begin{verbatim}
 >> PDE = examples_PDE_library_PIETOOLS(5);  
  --- Extracting ODE-PDE example 5 ---
  
  (d/dt) x(t,s) = 9.86 * x(t,s) + (d^2/ds^2) x(t,s);
  0 = x(t,0);
  0 = x(t,1);
 
 Would you like to run the executive associated to this problem? (y/n) 
 -->
\end{verbatim}
\end{matlab}
We note that the function asks whether an executive should be run for the considered PDE. This is because, for each of the PDE examples, an associated LPI problem has also been declared, matching one of the \texttt{executive} files (see also Chapter~\ref{ch:LPI_examples}). For the reaction-diffusion equation, the proposed executive is the \texttt{PIETOOLS\_stability} function, testing stability of the PDE. Entering \texttt{yes} in the command line window, this executive will be automatically run, whilst entering \texttt{no} will stop the function, and just return the \texttt{pde\_struct} object \texttt{PDE}.

Using the \texttt{examples\_PDE\_library\_PIETOOLS} function, parameters in the PDE can also be adjusted, calling e.g.
\begin{matlab}
\begin{verbatim}
 >> PDE = examples_PDE_library_PIETOOLS(5,'lam=10;');  
  --- Extracting ODE-PDE example 5 ---
  
  (d/dt) x(t,s) = 10 * x(t,s) + (d^2/ds^2) x(t,s);
  0 = x(t,0);
  0 = x(t,1);
\end{verbatim}
\end{matlab}
Similarly, if multiple parameters can be specified, we specify each of these parameters separately. For example, we note that example 7 corresponds to a PDE
\begin{align*}
    \dot{\mbf{x}}(t,s)&= c(s) \mbf{x}(t,s) + b(s)\partial_{s}\mbf{x}(t,s) + a(s)\partial_{s}^2\mbf{x}(t,s),  &   s&\in[0,1],
    \mbf{x}(t,0)&=\partial_{s}\mbf{x}(t,1)=0,
\end{align*}
where the values of the functions $a(s)$, $b(s)$ and $c(s)$ can be specified. As such, we can declare this PDE for $a=1$, $b=2$ and $c=3$ by calling
\begin{matlab}
\begin{verbatim}
 >> PDE = examples_PDE_library_PIETOOLS(7,'a=1;','b=2;','c=3;');
  --- Extracting ODE-PDE example 7 ---
  
  (d/dt) x(t,s) = 3 * x(t,s) + 2 * (d/ds) x(t,s) + (d^2/ds^2) x(t,s);
  0 = x(t,0);
  0 = (d/ds) x(t,1);
\end{verbatim}
\end{matlab}
Finally, we can also open the PDE in the GUI by calling
\begin{matlab}
\begin{verbatim}
 >> examples_PDE_library_PIETOOLS(7,'GUI');
  --- Extracting ODE-PDE example 7 ---
\end{verbatim}
\end{matlab}
or extract the PDE as a \texttt{pde\_struct} (terms-format) and open it in the GUI by calling
\begin{matlab}
\begin{verbatim}
 >> PDE = examples_PDE_library_PIETOOLS(7,'TERM','GUI');
  --- Extracting ODE-PDE example 7 ---
\end{verbatim}
\end{matlab}
In each case, the function will still ask whether the executive associated with the PDE should be run as well. Of course, you can also convert the PDE to a PIE yourself using \texttt{convert}, and then run any desired executive manually, assuming this executive makes sense (e.g. there's no sense in computing an $H_{\infty}$-gain if your system has no outputs).


\section{Libraries of DDE, NDS, and DDF Examples}\label{sec:examples:TDS}

Aside from the PDE examples, a list of TDS examples is also included in PIETOOLS, in DDE, NDS, and DDF format. Unlike the PDE problems, however, these TDS examples are all not separated into distinct functions, but are divided over two scripts: \texttt{examples\_DDE\_library\_PIETOOLS} and \texttt{examples\_NDSDDF\_library\_PIETOOLS}. In each of these scripts, most examples are commented, and only one example should be uncommented at any time. This example can then be extracted by calling the script, adding a structure \texttt{DDE}, \texttt{NDS} or \texttt{DDF} to the MATLAB workspace. To extract a different example, the desired example must be uncommented, and all other examples must be commented, at which point the script can be called again to obtain a structure representing the desired system.

We expect to update the DDE and NDS/DDF example files in a future release to match the format used for the PDE example library.

\subsection{DDE Examples} 
We have compiled a list of 23 DDE numerical examples, grouped into: stability analysis problems; input-output systems; estimator design problems; and feedback control problems. These examples are drawn from the literature and citations are used to indicate the source of each example. For each group, the relevant flags have been included to indicate which executive mode should be called after the example has been loaded.

\subsection{NDS and DDF Examples} 
There are relatively few DDFs which do not arise from a DDE or NDS. Hence, we have combined the DDF and NDS example libraries into the script \texttt{examples\_NDSDDF\_library\_PIETOOLS}. The Neutral Type systems are listed first, and currently consist only of stability analysis problems - of which we include 13. As for the DDE case, the library is a script, so the user must uncomment the desired example and call the script from the root file or command window. For the NDS problems, after calling the example library, in order to convert the NDS to a DDF or PIE, the user can use the following commands:
\begin{matlab}
\begin{verbatim}
 >> NDS = initialize\_PIETOOLS\_NDS(NDS);
 >> DDF = convert\_PIETOOLS\_NDS(NDS,'ddf');
 >> DDF = minimize\_PIETOOLS\_DDF(DDF);
 >> PIE = convert\_PIETOOLS\_DDF(DDF,'pie');
\end{verbatim}
\end{matlab}
In contrast to the NDS case, we only include 3 DDF examples. The first two are difference equations which cannot be represented in either the NDS or DDE format. The third is a network control problem, which is also included in the DDE library in DDE format.

\chapter{Standard Applications of LPI Programming}\label{ch:LPI_examples}
In Chapter~\ref{ch:LPIs}, we showed how general LPI optimization programs can be declared and solved in PIETOOLS. In this chapter, we provide several applications of LPI programming for analysis, estimation, and control of PIEs. Recall that such PIEs take the form
\begin{align}\label{eq:standardizedPIE_LPIforAnalysisofPIEs}
	\mcl T \dot{\mbf v}(t)+\mcl{T}_{w} \dot{w}(t)+\mcl{T}_{u}\dot{u}(t)&=\mcl A\mbf v(t)+\mcl{B}_1w(t)+\mcl B_2u(t), \quad \mbf v(0)=\mbf v_{\text{I}}\notag\\
	z(t) &= \mcl{C}_1\mbf v(t) + \mcl{D}_{11}w(t) + \mcl{D}_{12}u(t),\notag\\
	y(t) &= \mcl{C}_2\mbf{v}(t) + \mcl{D}_{21}w(t) + \mcl{D}_{22}u(t),
\end{align}
where $\mbf{v}=\sbmat{v_0\\\mbf{v}_1\\\mbf{v}_2\\\mbf{v}_3}\in\sbmat{\R^{n_0}\\L_2^{n_1}[a,b]\\L_2^{n_2}[c,d]\\L_2^{n_3}[[a,b]\times[c,d]]}$, and where $\mcl{T}$ through $\mcl{D}_{22}$ are all PI operators. In Section~\ref{sec:LPI_examples:analysis}, we provide several LPIs for stability analysis and $H_{\infty}$-gain estimation of such PIEs, setting $u=0$. Then, in Section~\ref{sec:LPI_examples:estimation}, we present an LPI for $H_{\infty}$-optimal estimation of PIEs of the form~\eqref{eq:standardizedPIE_LPIforAnalysisofPIEs}, followed by an LPI for $H_{\infty}$-optimal full-state feedback control in Section~\ref{sec:LPI_examples:control}. We note that, almost all of these LPIs have already been implemented as \texttt{executive} functions in PIETOOLS, and we will refer to these executives when applicable.

\section{LPIs for Analysis of PIEs}\label{sec:LPI_examples:analysis}
Using LPIs, several properties of a PIE as in Equation~\eqref{eq:standardizedPIE_LPIforAnalysisofPIEs} may be tested, as listed in this section. In particular, the LPIs listed below are extensions of classical results used in analysis of ODEs using LMIs. For most of these LPIs, PIETOOLS includes an executive function that may be run to solve it for a given PIE.

\subsection{Operator Norm}
For a PI operator $\mcl{T}$, an upper bound $\sqrt{\gamma}$ on the operator norm $\norm{\mcl{T}}$ can be computed by solving the LPI 
\begin{align}\label{eq:PI_op_norm}
    & \min\limits_{\gamma,\mcl{P}} ~~\gamma&\notag\\
    &\mcl{T}^*\mcl{T}-\gamma\preccurlyeq 0&
\end{align}
This LPI has not been implemented as an executive in PIETOOLS, but has been implemented in the demo file \texttt{volterra\_operator\_norm\_DEMO} (see also Section~\ref{sec:demos:volterra}).

\subsection{Stability}
For given PI operators $\mcl{T}$ and $\mcl{A}$, stability of the PIE
\begin{align*}
    \mcl{T}\dot{\mbf{v}}(t)=\mcl{A}\mbf{v}(t)
\end{align*}
can be tested by solving the LPI
\begin{align}\label{eq:stab_lpi}
	&\mcl{P}\succ0&\notag\\
	&\mcl{T}^*\mcl{P}\mcl{A}+\mcl{A}^*\mcl{P}\mcl{T}\preccurlyeq 0&.
\end{align}
If there exists a PI operator $\mcl{P}$ such that this LPI is feasible, then the PIE is stable. 
Given a structure \texttt{PIE}, this LPI may be solved for the associated PIE by calling
\begin{matlab}
\begin{verbatim}
 >> [prog, Pop] = PIETOOLS_stability(PIE, settings);
\end{verbatim}
\end{matlab}
Here \texttt{prog} will be an SOS program structure describing the solved problem and \texttt{Pop} will be a \texttt{dopvar} object describing the (unsolved) decision operator $\mcl{P}$ from which the solved operator can be derived using
\begin{matlab}
\begin{verbatim}
 >> Pop = getsol_lpivar(prog,Pop);
\end{verbatim}
\end{matlab}
See Chapter~\ref{ch:LPIs} for more information on the operation of this function and the \texttt{settings} input.

\subsection{Dual Stability}
For given PI operators $\mcl{T}$ and $\mcl{A}$, stability of the PIE
\begin{align*}
    \mcl{T}\dot{\mbf{v}}(t)=\mcl{A}\mbf{v}(t)
\end{align*}
can also be tested by solving the LPI
\begin{align}\label{eq:dual-stab_lpi}
	&\mcl{P}\succ0&\notag\\
	&\mcl{T}\mcl{P}\mcl{A}^*+\mcl{A}\mcl{P}\mcl{T}^*\preccurlyeq 0&
\end{align}
If there exists a PI operator $\mcl{P}$ such that this LPI is feasible, then the PIE is stable. 
Given a structure \texttt{PIE}, this LPI may be solved for the associated PIE by calling
\begin{matlab}
\begin{verbatim}
 >> [prog, Pop] = PIETOOLS_stability_dual(PIE, settings);
\end{verbatim}
\end{matlab}
Here \texttt{prog} will be an SOS program structure describing the solved problem and \texttt{Pop} will be a \texttt{dopvar} object describing the (unsolved) decision operator $\mcl{P}$.

\subsection{Input-Output Gain}
Consider a system of the form
\begin{align}\label{eq:standardizedPIE_noControl}
	\mcl T \dot{\mbf v}(t)&=\mcl A\mbf v(t)+\mcl{B}_1w(t), \quad \mbf v(0)=\mbf{0}\notag\\
	z(t) &= \mcl{C}_1\mbf v(t) + \mcl{D}_{11}w(t), 
\end{align}
where $w\in L_2^{n_w}[0,\infty)$. Then, $z\in L_2^{n_z}[0,\infty)$, and $\norm{z}_{L_2}\leq {\gamma}\norm{w}_{L_2}$, if the following LPI is feasible
\begin{align}\label{eq:hinf_lpi}
	&\min\limits_{\gamma,\mcl{P}} ~~\gamma&\notag\\
	&\mcl{P}\succ0&\notag\\
	&\bmat{-\gamma I & \mcl{D}_{11}^*&\mcl{B}_1^*\mcl{PT}\\(\cdot)^*&-\gamma I&\mcl{C}_1\\(\cdot)^*&(\cdot)^*&\mcl{T}^*\mcl{P}\mcl{A}+\mcl{A}^*\mcl{P}\mcl{T}}\preccurlyeq 0&
\end{align}
Given a structure \texttt{PIE}, this LPI may be solved for the associated PIE by calling
\begin{matlab}
\begin{verbatim}
 >> [prog, Pop, gam] = PIETOOLS_Hinf_gain(PIE, settings);
\end{verbatim}
\end{matlab}
Here \texttt{prog} will be an SOS program structure describing the solved problem, and \texttt{gam} will be the smallest value of $\gamma$ for which the LPI was found to be feasible, offering a bound on the $L_2$-gain from $w$ to $z$ of the system. The output \texttt{Pop} will be a \texttt{dopvar} object describing the (unsolved) decision operator $\mcl{P}$. 
				
\subsection{Dual Input-Output Gain}
For a System~\eqref{eq:standardizedPIE_noControl} with $w\in L_2^{n_w}[0,\infty)$, an upper bound $\gamma$ on the $L_2$-gain from $w$ to $z$ can also be obtained by solving the LPI	
\begin{align}\label{eq:dual_hinf_lpi}
	&\min\limits_{\gamma,\mcl{P}} ~~\gamma&\notag\\
	&\mcl{P}\succ0&\notag\\
	&\bmat{-\gamma I & \mcl{D}_{11}&\mcl{T}\mcl{P}\mcl{C}_1\\(\cdot)^*&-\gamma I&\mcl{B}_1^*\\(\cdot)^*&(\cdot)^*&\mcl{T}\mcl{P}\mcl{A}^*+\mcl{A}\mcl{P}\mcl{T}^*}\preccurlyeq 0&
\end{align}
Given a structure \texttt{PIE}, this LPI may be solved for the associated PIE by calling
\begin{matlab}
\begin{verbatim}
 >> [prog, Pop, gam] = PIETOOLS_Hinf_gain_dual(PIE, settings);
\end{verbatim}
\end{matlab}
Here \texttt{prog} will be an SOS program structure describing the solved problem, and \texttt{gam} will be the found optimal value for $\gamma$. The output \texttt{Pop} will be a \texttt{dopvar} object describing the (unsolved) decision operator $\mcl{P}$.

\subsection{Positive Real Lemma}
For a PIE of the form
\begin{align*}
	\mcl T \dot{\mbf v}(t)&=\mcl A\mbf v(t)+\mcl{B}_1w(t), \quad \mbf v(0)=\mbf{0}\notag\\
	z(t) &= \mcl{C}_1\mbf v(t) + \mcl{D}_{11}w(t), 
\end{align*}
we can test whether the system is passive by solving the LPI
{
\begin{align}\label{eq:positive_real_lpi}
	&\mcl{P}\succ0&    \notag\\
	&\bmat{-\mcl{D}_{11}^*-\mcl{D}_{11}&\mcl{B}_1^*\mcl{PT}-\mcl{C}_1\\(\cdot)^*&\mcl{T}^*\mcl{P}\mcl{A}+\mcl{A}^*\mcl{P}\mcl{T}}\preccurlyeq 0&
\end{align}
}
If there exists a PI operator $\mcl{P}$ such that this LPI is feasible, then the system is passive. Note that this LPI has not been implemented as an executive in PIETOOLS.

\section{LPIs for Optimal Estimation of PIEs}\label{sec:LPI_examples:estimation}
For the following PIE
\begin{align}
	\mcl T \dot{\mbf v}(t)+\mcl{T}_{w} \dot{w}(t)&=\mcl A\mbf v(t)+\mcl{B}_1w(t),\notag\\
	z(t) &= \mcl{C}_1\mbf v(t) + \mcl{D}_{11}w(t),\notag\\
	y(t) &= \mcl{C}_2\mbf{v}(t) + \mcl{D}_{21}w(t),
\end{align}
a state estimator has the following structure:
\begin{align} 
	\mcl T \dot{\mbf{\hat{v}}}(t) &=\mcl A\mbf{\hat{v}}(t)+\mathcal{L}(\hat{y}(t)-y(t)),\notag\\
	\hat{z}(t) &= \mcl{C}_1\mbf{\hat{v}}(t) \notag\\
	\hat{y}(t) &= \mcl{C}_2\mbf{\hat{v}}(t).
\end{align}
so that the errors $\mbf{e}:=\hat{\mbf{v}}-\mbf{v}$ and $\tilde{z}:=\hat{z}-z$ in respectively the state and regulated output estimates satisfy
\begin{align*}
    \mcl{T}\dot{\mbf{e}}-\mcl{T}_w w(t)&=(\mcl{A}+\mcl{L}\mcl{C}_2)\mbf{e}(t) - (\mcl{B}_1+\mcl{L}\mcl{D}_{21})w(t), \\
    \tilde{z}(t)&=\mcl{C}_1\mbf{e}(t) - D_{11}w(t)
\end{align*}
The $H_{\infty}$-optimal estimation problem amounts to synthesizing $\mathcal{L}$ such that the estimation error $\tilde{z}:={\hat{z}-z}$ admits $\norm{\tilde{z}} \leq \gamma \norm{w}$ for a particular $\gamma >0$. To establish such an estimator, we can solve the LPI
\begin{align}\label{eq:esti_lpi}
	&\min\limits_{\gamma,\mcl{P},\mcl{Z}} ~~\gamma&\notag\\
	&\mcl{P}\succ0&\notag\\
	&\bmat{\mcl T_{w}^*(\mcl P\mcl B_1+\mcl Z\mcl D_{21})+(\cdot)^*& 0 &(\cdot)^*\\ 0  &0 & 0 \\ -(\mcl P\mcl A+\mcl Z\mcl C_2)^*\mcl T_{w}& 0 &0}\hspace{-1ex}+\hspace{-1ex}\bmat{-\gamma I& -\mcl D_{11}^{\top}&-(\mcl P\mcl B_1+\mcl Z\mcl D_{21})^*\mcl T\\(\cdot)^*&-\gamma I&\mcl C_1\\(\cdot)^*&(\cdot)^*&(\mcl P\mcl A+\mcl Z\mcl C_2)^*\mcl T+(\cdot)^*}\preccurlyeq 0&
\end{align}
Then, if this LPI is feasible for some $\gamma>0$ and PI operators $\mcl{P}$ and $\mcl{Z}$, 
then, letting $\mcl{L}:=\mcl{P}^{-1} \mcl{Z}$, the estimation error will satisfy $\norm{z-\hat{z}} \leq \gamma \norm{w}$. 
Given a structure \texttt{PIE}, this LPI may be solved for the associated PIE by calling
\begin{matlab}
\begin{verbatim}
 >> [prog, Lop, gam, Pop, Zop] = PIETOOLS_Hinf_estimator(PIE, settings);
\end{verbatim}
\end{matlab}
Here \texttt{prog} will be an SOS program structure describing the solved problem, \texttt{gam} will be the found optimal value for $\gamma$, and \texttt{Lop} will be an \texttt{opvar} object describing the optimal estimator $\mcl{L}$. Outputs \texttt{Pop} and \texttt{Zop} will be \texttt{dopvar} objects describing the (unsolved) decision operators $\mcl{P}$ and $\mcl{Z}$, from which the solved operators can be derived using
\begin{matlab}
\begin{verbatim}
 >> Pop = getsol_lpivar(prog,Pop);        Zop = getsol_lpivar(prog,Zop);
\end{verbatim}
\end{matlab}
See Chapter~\ref{ch:LPIs} for more information on the operation of this function and the \texttt{settings} input.


\section{LPIs for Optimal Control of PIEs}\label{sec:LPI_examples:control}

In this section, we discuss the synthesis of $H_{\infty}$ optimal control of a PIE of the form
\begin{align}\label{eq:LPI_examples:control_PIE}
	\mcl T \dot{\mbf v}(t)&=\mcl A\mbf v(t)+\mcl{B}_1w(t)+\mcl B_2u(t), \qquad \mbf v(0)=\mbf v_0\notag\\
	z(t) &= \mcl{C}_1\mbf v(t) + \mcl{D}_{11}w(t) + \mcl{D}_{12}u(t),
\end{align}
where $w\in L_2[0,\infty)$. The problem of synthesizing an $H_{\infty}$-optimal controller amounts to determining a PIE operator $\mcl{K}$ such that, using the full-state feedback law $u(t) = \mathcal{K}\mbf{v}(t)$, the regulated output ${z}$ admits $\norm{z}_{L_2} \leq \gamma \norm{w}_{L_2}$ for a particular $\gamma >0$. To establish such a controller, we can solve the LPI
\begin{align}\label{eq:cont_lpi}
	&\min\limits_{\gamma,\mcl{P},\mcl{Z}} ~~\gamma&\notag\\
	&\mcl{P}\succ0&\notag\\dj
	&\bmat{-\gamma I& \mcl D_{11}& (\mcl{C}_1\mcl P+D_{12}\mcl{Z})\mcl T^*\\
		\mcl D_{11}^* & -\gamma I & \mcl B_1^*\\
		()^* & \mcl B_1& ()^*+\left(\mcl{AP}+\mcl{B}_2\mcl{Z}\right)\mcl{T}^*}\preccurlyeq 0&
\end{align}
If this LPI is feasible for some $\gamma>0$ and PI operators $\mcl{P}$ and $\mcl{Z}$, then, letting $\mcl{K}:=\mcl{Z}\mcl{P}^{-1}$, the $L_2$-gain for the controlled system with $u=\mcl{K}\mbf{v}$ will be such that $\norm{z} \leq \gamma \norm{w}$.  
Given a structure \texttt{PIE}, this LPI may be solved for the associated PIE by calling
\begin{matlab}
\begin{verbatim}
 >> [prog, Kop, gam, Pop, Zop] = PIETOOLS_Hinf_control(PIE, settings);
\end{verbatim}
\end{matlab}
Here \texttt{prog} will be an SOS program structure describing the solved problem, \texttt{gam} will be the found optimal value for $\gamma$, and \texttt{Kop} will be an \texttt{opvar} object describing the optimal feedback $\mcl{K}$. Outputs \texttt{Pop} and \texttt{Zop} will be \texttt{dopvar} objects describing the (unsolved) decision operators $\mcl{P}$ and $\mcl{Z}$, from which the solved operators can be derived using
\begin{matlab}
\begin{verbatim}
 >> Pop = getsol_lpivar(prog,Pop);        Zop = getsol_lpivar(prog,Zop);
\end{verbatim}
\end{matlab}
See Chapter~\ref{ch:LPIs} for more information on the operation of this function and the \texttt{settings} input.

\part{Appendices}

\appendix
%

\chapter{PI Operators and their Properties}\label{appx:PI_theory}

In this appendix, we discuss in a bit more detail the crucial properties of PI operators that PIETOOLS relies on for implementation and analysis of PIEs. For this, in Section~\ref{sec:PI_theory:PI_defs}, we first recap the definitions of PI operators as presented in Chapter~\ref{ch:PIE}, also introducing some notation that we will continue to use throughout the appendix. In Section~\ref{sec:PI_theory:Addition},~\ref{sec:PI_theory:Composition}, and~\ref{sec:PI_theory:Adjoint}, we show that respectively the sum, composition, and adjoint of PI operators can be expressed as PI operators. In Section~\ref{sec:PI_theory:Inverse}, and~\ref{sec:PI_theory:Derivative}, we then show how the respectively the inverse of a PI operator, and the composition of a PI operator with a differential operator can be computed. Finally, in Section~\ref{sec:PI_theory:Positive_PI}, we show how a cone of positive PI operators can be parameterized by positive matrices, allowing an LPI constraint $\mcl{P}\succcurlyeq 0$ to be posed as an LMI $P\succcurlyeq 0$.

For more information on PI operators, and full proofs of each of the results, we refer to e.g.~\cite{peet_2020Aut}~\cite{shivakumar_2019CDC} (1D) and~\cite{jagt_2021PIEACC} (2D).

\section{PI Operators on Different Function Spaces}\label{sec:PI_theory:PI_defs}

Recall that we denote the space of square integrable functions on a domain $\Omega$ as $L_2[\Omega]$, with inner product
\begin{align*}
    \ip{\mbf{x}}{\mbf{y}}_{L_2}=\int_{\Omega}\bl[\mbf{x}(s)\br]^T\mbf{y}(s)ds.
\end{align*}
In defining the different PI operators, we will restrict ourselves to domains of 1D or 2D hypercubes. In 1D, such a hypercube is simply an interval $[a,b]$, for which we define the following PI operator:
    
\begin{defn}[3-PI Operator]\label{appx_def:3-PI}
 For given parameters 
 \begin{align*}
    R:=\{R_0,R_1,R_2\}\in\left\{L_2^{m\times n}[a,b],L_2^{m\times n}\bl[[a,b]^2\br],L_2^{m\times n}\bl[[a,b]^2\br]\right\}=:\mcl{N}_{1D}^{m\times n}[a,b],
 \end{align*}
 we define the associated 3-PI operator $\mcl{P}[R]:=\mcl{P}_{\{R_0,R_1,R_2\}}:L_2^{n}[a,b]\rightarrow L_2^{m}[a,b]$ as
 \begin{align}
  \bl(\mcl{P}[R] \mbf{x}\br)(s) := R_0(s) \mbf{x}(s) +\int_{a}^{s} R_1(s,\theta)\mbf{x}(\theta)d \theta +\int_s^b R_2(s,\theta)\mbf{x}(\theta)d \theta,
 \end{align} 
 for any $\mbf{x}\in L_2^{n}[a,b]$.
\end{defn}

We note that 3-PI operators can be seen as (one possible) generalization of matrices to infinite dimensional vector spaces. In particular, suppose we have a matrix $P\in\R^{m\times n}$, which we decompose as $P=D+L+U$, where $D$ is diagonal, $L$ is strictly lower triangular, and $U$ is strictly upper-triangular. Then, for any $x\in\R^n$, the $i$th element of the product $Px$ is given by:
\begin{align*}
 \bl(Px\br)_{i} &= D_{ii} x_i + \sum_{j=1}^{i-1} L_{ij}x_j + \sum_{j=i+1}^{n} U_{ij} x_j.
 \intertext{
 Compare this to the value of $\mcl{P}[R]\mbf{x}$ at a position $s\in[a,b]$ for some $\mbf{x}\in L_2[a,b]$ and 3-PI parameters $R$:}
 \bl(\mcl{P}[R]\mbf{x}\br)(s) &= R_0(s)\mbf{x}(s) + \int_{a}^{s} R_1(s,\theta)\mbf{x}(\theta)d\theta + \int_{s}^{b} R_2(s,\theta)\mbf{x}(\theta)d\theta.
 \end{align*}
 Replacing row and column indices $(i,j)$ by primary and dummy variables $(s,\theta)$, and performing integration instead of summation, 3-PI operators have a structure very similar to that of matrices, wherein we can recognize a diagonal, lower-triangular, and upper-triangular part. Accordingly, we will occasionally refer to a PI operator of the form $\mcl{P}_{\{R_0,0,0\}}$ as a diagonal 3-PI operator, and to PI operators of the forms $\mcl{P}_{\{0,R_1,0\}}$ and $\mcl{P}_{\{0,0,R_2\}}$ as lower- and upper-triangular PI operators respectively. The similar structure between matrices and PI operator also ensures that matrix operations such as addition and multiplication are valid for PI operators as well, as we will discuss in more detail in the next sections.

 To map functions on a domain $[a,b]\times[c,d]\subset\R^2$, we also define the $9$-PI operator:

\begin{defn}[9-PI Operator]\label{appx_def:9-PI}
 For given parameters
{\small
\begin{align*}
 R&:={\small\left[\!\!\begin{array}{lll}
 R_{00} & R_{01} & R_{02}\\ R_{10} & R_{11} & R_{12}\\ R_{20} & R_{21} & R_{22}
 \end{array}\!\!\right]}    \\
 &\hspace*{0.5cm}\in
 {\small\left[\begin{array}{lll}
    L_{2}^{m\times n}\bl[[a,b]\times[c,d]\br] & L_{2}^{m\times n}\bl[[a,b]\times[c,d]^2\br] & L_{2}^{m\times n}\bl[[a,b]\times[c,d]^2\br] \\ 
    L_{2}^{m\times n}\bl[[a,b]^2\times[c,d]\br] & L_{2}^{m\times n}\bl[[a,b]^2\times[c,d]^2\br] & L_{2}^{m\times n}\bl[[a,b]^2\times[c,d]^2\br]\\
    L_{2}^{m\times n}\bl[[a,b]^2\times[c,d]\br] & L_{2}^{m\times n}\bl[[a,b]^2\times[c,d]^2\br] & L_{2}^{m\times n}\bl[[a,b]^2\times[c,d]^2\br]
 \end{array}\right]}
 =:\mcl{N}_{2D}^{m\times n}\bl[[a,b]\!\times\![c,d]\br]
 \end{align*}
 }
 we define the associated 9-PI operator $\mcl{P}[R]:=\mcl{P}\sbmat{R_{00} & R_{01} & R_{02}\\ R_{10} & R_{11} & R_{12}\\ R_{20} & R_{21} & R_{22}}:L_2^{n}\bl[[a_1,b_1]\times[a_2,b_2]\br]\rightarrow L_2^{m}\bl[[a_1,b_1]\times[a_2,b_2]\br]$ as
{\small
\begin{align}
    \left(\mcl{P}[R]\mbf{x}\right)(s,r)= R_{00}(s,r)\mbf{x}(s,r) &+\hspace*{0.0cm} \int_{c}^{r}\! R_{01}(s,r,\nu)\mbf{x}(s,\nu)d\nu + \int_{r}^{d}\! R_{02}(s,r,\nu)\mbf{x}(s,\nu)d\nu \nonumber\\
    +\int_{a}^{s}\! R_{10}(s,r,\theta)\mbf{x}(\theta,r)d\theta &+ \int_{a}^{s}\!\int_{c}^{r}\! R_{11}(s,r,\theta,\nu)\mbf{x}(\theta,\nu)d\nu d\theta + \int_{a}^{s}\!\int_{r}^{d}\! R_{12}(s,r,\theta,\nu)\mbf{x}(\theta,\nu)d\nu d\theta  \nonumber\\
    +\int_{s}^{b} R_{20}(s,r,\theta)\mbf{x}(\theta,r)d\theta &+ \int_{s}^{b}\!\int_{c}^{r}\! R_{21}(s,r,\theta,\nu)\mbf{x}(\theta,\nu)d\nu d\theta + \int_{s}^{b}\!\int_{r}^{d}\! R_{22}(s,r,\theta,\nu)\mbf{x}(\theta,\nu)d\nu d\theta
 \end{align}
}
for any $\mbf{x}\in L_2^{n}\bl[[a_1,b_2]\times[a_2,b_2]\br]$.
\end{defn}

Note that, similar to how 3-PI operators can be seen as a generalization of matrices, operating on infinite-dimensional states $\mbf{x}(s)$ instead of a finite-dimensional vectors $x_i$, 9-PI operators are a generalization of (a particular class of) tensors, operating on infinite-dimensional states $\mbf{x}(s,r)$ instead of matrix-valued states $x_{ij}$. However, this comparison is not quite as easy to visualize as that between 3-PI operators and matrices, so we will mostly use 3-PI operators to illustrate the different properties of PI operators in the remaining sections.

Finally we define a general class of PI operators, encapsulating 3-PI operators and 9-PI operators, as well as matrices and ``cross-operators''. In particular, we consider operators defined on the set $Z^{\textnormal{n}}\br[[a,b],[c,d]\bl]:={\scriptsize\left[\begin{array}{l}
\R^{n_0}\\ L_2^{n_s}[a,b]\\ L_2^{n_r}[c,d]\\ L_2^{n_2}\bl[[a,b]\times[c,d]\br]
\end{array}\right]}$, where $\text{n}:=\{n_0,n_s,n_r,n_2\}$, with each element being a coupled state of finite-dimensional variables $x_0\in\R^{n_0}$, 1D functions $\mbf{x}_s\in L_2^{n_s}[a,b]$ and $\mbf{x}_r\in L_2^{n_r}[a,b]$, and 2D functions $\mbf{x}_2\in L_2^{n_2}\bl[[a,b]\times[c,d]\br]$.

\begin{defn}[PI Operator]

For any operator $\mcl{R}:Z^{\textnormal{n}}\br[[a,b],[c,d]\bl]\rightarrow Z^{\textnormal{m}}\br[[a,b],[c,d]\bl]$ with $\textnormal{m}:=\{m_0,m_s,m_r,m_2\}$ and $\textnormal{n}:=\{n_0,n_s,n_r,n_2\}$, we say that $\mcl{R}$ is a PI operator, denoted by $\mcl{R}\in\Pi^{\textnormal{m}\times\textnormal{m}}$ if there exist parameters
{\scriptsize
\begin{align*}
 &R:=\left[\!\!\begin{array}{llll}
    R_{00} & R_{0s} & R_{0r} & R_{02}\\
    R_{s0} & R_{ss} & R_{sr} & R_{s2}\\
    R_{r0} & R_{rs} & R_{rr} & R_{s2}\\
    R_{20} & R_{2s} & R_{2r} & R_{22}
 \end{array}\!\!\right] \\
 &\qquad \in
 \left[\!\!\begin{array}{llll}
    \R^{m_0\times n_0} & L_2^{m_0\times n_s}[a,b] & L_2^{m_0\times n_r}[c,d] & L_2^{m_0\times n_2}\bl[[a,b]\times[c,d]\br]\\
    L_2^{m_s\times n_0}[a,b] & \mcl{N}_{1D}^{m_s\times n_s}[a,b] & L_2^{m_s\times n_r}\bl[[a,b]\times[c,d]\br] & \mcl{N}_{1D\leftarrow 2D}^{m_s\times n_2}\bl[[a,b],[c,d]\br]\\
    L_2^{m_r\times n_0}[c,d] & L_2^{m_r\times n_s}\bl[[a,b]\times[c,d]\br] & \mcl{N}_{1D}^{m_r\times n_r}[c,d] & \mcl{N}_{1D\leftarrow 2D}^{m_r\times n_2}\bl[[c,d],[a,b]\br]\\
    L_2^{m_2\times n_0}\bl[[a,b]\times[c,d]\br] & \mcl{N}_{2D\leftarrow 1D}^{m_2\times n_s}\bl[[a,b],[c,d]\br] & \mcl{N}_{2D\leftarrow 1D}^{m_2\times n_r}\bl[[c,d],[a,b]\br] & \mcl{N}_{2D}^{m_2\times n_r}\bl[[a,b]\times[c,d]\br]
 \end{array}\!\!\right]
 =:\mcl{N}^{\textnormal{m}\times \textnormal{n}}\bl[[a,b]\times[c,d]\br]
\end{align*}
}
such that
{\small
\begin{align}
 \mcl{R}&=\left(\mcl{P}[R]\mbf{x}\right)(s,r)    \nonumber\\
 &=\mcl{P}\sbmat{R_{00} & R_{0s} & R_{0r} & R_{02}\\
    R_{s0} & R_{ss} & R_{sr} & R_{s2}\\
    R_{r0} & R_{rs} & R_{rr} & R_{s2}\\
    R_{20} & R_{2s} & R_{2r} & R_{22}}\sbmat{x_0\\ \mbf{x}_{s}\\ \mbf{x}_{r}\\ \mbf{x}_2}   \nonumber\\
    &:=
    \left[\!\begin{array}{llll}
     R_{00}x_0 &\!\! +\ \int_{a}^{b}\! R_{0s}(s)\mbf{x}_s(s)ds &\!\! +\ \int_{c}^{d}\! R_{0r}(r)\mbf{x}_r(r)dr &\!\! +\ \int_{a}^{b}\!\int_{c}^{d}R_{02}\mbf{x}_{2}(s,r)drds \\
     R_{s0}(s)x_0 &\!\! +\ \bl(\mcl{P}[R_{ss}]\mbf{x}_{s}\br)(s) &\!\! +\ \int_{c}^{d}\! R_{sr}(s,r)\mbf{x}_{r}(r)dr &\!\! +\ \bl(\mcl{P}[R_{s2}]\mbf{x}_{2}\br)(s) \\
     R_{r0}(r)x_0 &\!\! +\ \int_{a}^{b}\! R_{rs}(s,r)\mbf{x}_{s}(s)ds &\!\! +\ \bl(\mcl{P}[R_{rr}]\mbf{x}_{r}\br)(r)  &\!\! +\ \bl(\mcl{P}[R_{r2}]\mbf{x}_{2}\br)(r) \\
     R_{20}(s,r)x_0 &\!\! +\ \bl(\mcl{P}[R_{2s}]\mbf{x}_{s}\br)(s,r) &\!\! +\ \bl(\mcl{P}[R_{2r}]\mbf{x}_{r}\br)(s,r)  &\!\! +\ \bl(\mcl{P}[R_{22}]\mbf{x}_{2}\br)(s,r) 
    \end{array}\!\right],
\end{align}
}
for any $\mbf{x}=\sbmat{x_0\\ \mbf{x}_{s}\\ \mbf{x}_{r}\\ \mbf{x}_2}\in {\scriptsize\left[\begin{array}{l}
    \R^{n_0}\\ L_2^{n_s}[a,b]\\ L_2^{n_r}[c,d]\\ L_2^{n_2}\bl[[a,b]\times[c,d]\br]
    \end{array}\right]}=:Z^{\textnormal{n}}$, where for given parameters
{\small
\begin{align*}
P&:=\{P_0,P_1,P_2\}\\
&\qquad \in \left\{L_2^{m_s\times n_2}\bl[[a,b]\!\times\![c,d]\br],L_2^{m_s\times n_2}\bl[[a,b]^2\!\times\![c,d]\br],L_2^{m_s\times n_2}\bl[[a,b]^2\!\times\![c,d]\br]\right\}=:\mcl{N}^{m_s\times n_2}_{1D\leftarrow 2D}\bl[[a,b],[c,d]\br], \\
Q&:=\{Q_0,Q_1,Q_2\}\\
&\qquad \in \left\{L_2^{m_s\times n_2}\bl[[a,b]\!\times\![c,d]\br],L_2^{m_s\times n_2}\bl[[a,b]^2\!\times\![c,d]\br],L_2^{m_s\times n_2}\bl[[a,b]^2\!\times\![c,d]\br]\right\}=:\mcl{N}^{m_s\times n_2}_{2D\leftarrow 1D}\bl[[a,b],[c,d]\br],
\end{align*}
}
we define
{\small
\begin{align*}
 \bl(\mcl{P}[P] \mbf{x}_{2}\br)(s) &:= \int_{c}^{d}\bbbl[P_0(s,r) \mbf{x}_{2}(s,r) +\int_{a}^{s} P_1(s,r,\theta)\mbf{x}_{2}(\theta,r)d \theta +\int_s^b P_2(s,r,\theta)\mbf{x}(\theta,r)d \theta \bbbr]dr,  \nonumber\\
 \bl(\mcl{P}[Q] \mbf{x}_{s}\br)(s,r) &:= Q_0(s,r) \mbf{x}_{s}(s) +\int_{a}^{s} Q_1(s,r,\theta)\mbf{x}_{s}(\theta)d \theta +\int_s^b Q_2(s,r,\theta)\mbf{x}_{s}(\theta)d \theta,
\end{align*}
}
for any $\mbf{x}_2\in L_2^{n_2}\bl[[a,b]\times[c,d]\br]$ and $\mbf{x}_s\in L_2^{n_s}[a,b]$.

\end{defn}

\section{Addition of PI Operators}\label{sec:PI_theory:Addition}

An obvious but crucial property of PI operators is that the sum of two PI operators (of appropriate dimensions) is again a PI operator.

\begin{lem}
 For any PI parameters $Q,R\in\mcl{N}^{\text{m}\times\text{n}}\bl[[a,b],[c,d]\br]$, there exist unique parameters $P\in\mcl{N}^{\text{m}\times\text{n}}\bl[[a,b],[c,d]\br]$ such that
 \begin{align*}
    \mcl{P}[R]+\mcl{P}[Q]=\mcl{P}[P].
 \end{align*}
 That is, for any $\mbf{x}\in Z^{\textnormal{n}}\bl[[a,b],[c,d]\br]$,
 \begin{align*}
    \bl((\mcl{P}[Q]+\mcl{P}[R])\mbf{x}\br)(s)=\bl(\mcl{P}[P]\mbf{x})(s).
 \end{align*}
\end{lem}
\begin{proof}
    We outline the proof for 3-PI operators, for which it is easy to see that, by linearity of the integral,
    \begin{align*}
        &\bl(\mcl{P}_{\{R_0,R_1,R_2\}}\mbf{x}\br)(s)+\bl(\mcl{P}_{\{Q_0,Q_1,Q_2\}}\mbf{x}\br)(s)    \\
        &\qquad=[R_0(s)+Q_0(s)]\mbf{x}(s)+\int_{a}^{s}[R_1(s,\theta)+Q_1(s,\theta)]\mbf{x}(\theta)d\theta +\int_{s}^{b}[R_2(s,\theta)+Q_2(s,\theta)]\mbf{x}(\theta)d\theta  \\
        &\qquad\qquad= \bl(\mcl{P}_{\{R_0+Q_0,R_1+Q_1,R_2+Q_2\}}\mbf{x}\br)(s).
    \end{align*}
    Similar results can be easily derived for more general PI operators. For a full proof, we refer to~\cite{peet_2020Aut}.
\end{proof}
Comparing the addition operation for $3$-PI operators to that for matrices $A,B\in\R^{m\times n}$, we can draw direct parallels. In particular, where the sum $C=A+B$ of two matrices is simply computed by adding the elements $[C]_{ij}=[A]_{ij}+[B]_{ij}$ for each row $i$ and column $j$, the sum of two 3-PI operators is computed by simply adding the values of the parameters $P(s,\theta)=Q(s,\theta)+R(s,\theta)$ at each position $s$ and $\theta$ within the domain.

\section{Composition of PI Operators}\label{sec:PI_theory:Composition}

In addition to the sum of two PI operators being a PI operator, the composition of two PI operators can also be shown to be a PI operator, as stated in the following lemma:

\begin{lem}
 For any PI parameters $R_1\in\mcl{N}^{\text{m}\times\text{p}}\bl[[a,b],[c,d]\br]$ and $R_2\in\mcl{N}^{\text{p}\times\text{n}}\bl[[a,b],[c,d]\br]$, there exist unique parameters $R_3\in\mcl{N}^{\text{m}\times\text{n}}\bl[[a,b],[c,d]\br]$ such that
 \begin{align*}
    \mcl{P}[R_1]\circ\mcl{P}[R_2]=\mcl{P}[R_3].
 \end{align*}
 That is, for any $\mbf{x}\in Z^{\textnormal{n}}\bl[[a,b],[c,d]\br]$,
 \begin{align*}
    \bbl(\mcl{P}[R_1]\bl(\mcl{P}[R_2]\mbf{x}\br)\bbr)(s)=\bl(\mcl{P}[R_3]\mbf{x})(s).
 \end{align*}
\end{lem}
\begin{proof}
    We once again outline the proof only for 3-PI operators. For this, we define the indicator function
    \begin{align*}
    \mbf{I}(s-\theta)=\begin{cases}1,   &\text{if }s\geq \theta\\ 0.    &\text{else} \end{cases}
    \end{align*}
    allowing us to write, e.g.
    \begin{align*}
     \bl(\mcl{P}_{\{R_0,R_1,R_2\}}\mbf{x}\bl)(s)=R_0(s)\mbf{x}+\int_{a}^{b}\bbl[\mbf{I}(s-\theta)R_1(s,\theta)+\mbf{I}(\theta-s)R_2(s,\theta)\bbr]\mbf{x}(\theta)d\theta.
    \end{align*}
    Furthermore, we have the following relations for the indicator function
    \begin{align*}
        \mbf{I}(s-\eta)\mbf{I}(\eta-\theta)&=\begin{cases}
            \mbf{I}(s-\theta),    &\text{if } \eta\in[\theta,s],  \\
            0,                  &\text{else,}
        \end{cases} \\
        \mbf{I}(s-\eta)\mbf{I}(\theta-\eta)&=\mbf{I}(s-\theta)\mbf{I}(\theta-\eta) + \mbf{I}(\theta-s)\mbf{I}(s-\eta)
    \end{align*}
    Using the first relation, it follows that for any $R_1,Q_1\in L_2^{m_s\times n_s}\bl[[a,b]^2\br]$,
    \begin{align*}
     \bbl(\mcl{P}_{\{0,R_1,0\}}\bl(\mcl{P}_{\{0,Q_1,0\}}\mbf{x}\br)\bbr)(s)&=\int_{a}^{s}R_1(s,\eta)\int_{a}^{\eta}Q_1(\eta,\theta)\mbf{x}(\theta)d\theta d\eta \\
     &= \int_{a}^{b}\int_{a}^{b}\mbf{I}(s-\eta)\mbf{I}(\eta-\theta)R_1(s,\eta)Q_1(\eta,\theta)\mbf{x}(\theta)d\theta d\eta \\
     &= \int_{a}^{s}\left[\int_{\theta}^{s}R_1(s,\eta)Q_1(\eta,\theta) d\eta\right] \mbf{x}(\theta)d\theta  
     = \bl(\mcl{P}_{\{0,P_{11},0\}}\mbf{x}\br)(s),
    \end{align*}
    where $P_{11}(s,\theta):=\int_{\theta}^{s}R_1(s,\eta)Q_1(\theta,\eta)d\eta$. Similarly, we can show that
    \begin{align*}
     \bbl(\mcl{P}_{\{0,R_1,0\}}\bl(\mcl{P}_{\{0,0,Q_2\}}\mbf{x}\br)\bbr)(s)&=
     \int_{a}^{s}\left[\int_{a}^{\theta}R_1(s,\eta)Q_2(\eta,\theta)d\eta\right]\mbf{x}(\theta)d\theta + \int_{s}^{b}\left[\int_{a}^{s}R_1(s,\eta)Q_2(\eta,\theta)d\eta\right]\mbf{x}(\theta)d\theta \\
     \bbl(\mcl{P}_{\{0,0,R_2\}}\bl(\mcl{P}_{\{0,Q_1,0\}}\mbf{x}\br)\bbr)(s)&=
     \int_{a}^{s}\left[\int_{s}^{b}R_2(s,\eta)Q_1(\eta,\theta)d\eta\right]\mbf{x}(\theta)d\theta + \int_{s}^{b}\left[\int_{\theta}^{b}R_2(s,\eta)Q_1(\eta,\theta)d\eta\right]\mbf{x}(\theta)d\theta \\
     \bbl(\mcl{P}_{\{0,0,R_2\}}\bl(\mcl{P}_{\{0,0,Q_2\}}\mbf{x}\br)\bbr)(s)&=
     \int_{s}^{b}\left[\int_{s}^{\theta}R_2(s,\eta)Q_2(\theta,\eta)d\eta\right] \mbf{x}(\theta) d\theta = \bl(\mcl{P}_{\{0,0,P_{22}\}}\mbf{x}\br)(s),
    \end{align*}
    proving that the composition of lower-triangular and upper-triangular partial integrals can always be expressed as partial integrals as well. It is also easy easy to see that
    \begin{align*}
        \bbl(\mcl{P}_{\{R_0,0,0\}}\bl(\mcl{P}_{\{0,Q_1,Q_2\}}\mbf{x}\br)\bbr)(s)&=\int_{a}^{s}R_0(s)Q_1(s,\theta)\mbf{x}(\theta)d\theta + \int_{s}^{b}R_0(s)Q_2(s,\theta)\mbf{x}(\theta)d\theta = \bl(\mcl{P}_{\{0,P_{01},P_{02}\}}\mbf{x}\br)(s), \\
        \bbl(\mcl{P}_{\{0,R_1,R_2\}}\bl(\mcl{P}_{\{Q_0,0,0\}}\mbf{x}\br)\bbr)(s)&=\int_{a}^{s}R_1(s,\theta)Q_0(\theta)\mbf{x}(\theta)d\theta + \int_{s}^{b}R_2(s,\theta)Q_2(\theta)\mbf{x}(\theta)d\theta = \bl(\mcl{P}_{\{0,P_{10},P_{20}\}}\mbf{x}\br)(s),
    \end{align*}
    from which it follows that the composition of any 3-PI operators can be expressed as a 3-PI operator as well. Moreover, since we can repeat these steps along any spatial directions, this result extends to more general (2D) PI operators as well. For a full proof, we refer to~\cite{}.
    
\end{proof}

We note again the similarity to matrices: just like the product of two lower triangular matrices $L_1,L_2$ is a lower triangular matrix $L_3$, the composition of two lower-triangular 3-PI operators $\mcl{P}_{\{0,R_1,0\}},\mcl{P}_{\{0,Q_1,0\}}$ is also a lower-triangular 3-PI operator $\mcl{P}_{\{0,P_{11},0\}}$. Similarly, the product of two upper-triangular 3-PI operators $\mcl{P}_{\{0,0,R_2\}},\mcl{P}_{\{0,0,Q_2\}}$ is also an upper-triangular 3-PI operator $\mcl{P}_{\{0,0,P_{22}\}}$, but the composition of lower- and upper-triangular PI operators need not be lower- or upper-triangular -- just as with matrices. Finally, the composition of a diagonal operator $\mcl{P}_{\{R_0,0,0\}}$ with a lower- or upper-diagonal PI operator is also respectively lower- or upper-diagonal.

\section{Adjoint of PI Operators}\label{sec:PI_theory:Adjoint}

To define the adjoint of a PI operator $\mcl{R}\in\Pi^{\textnormal{m}\times\text{normal{n}}}$, we first recall the definition of the function space that these operators map: $Z^{\textnormal{n}}\br[[a,b],[c,d]\bl]:={\scriptsize\left[\begin{array}{l}
\R^{n_0}\\ L_2^{n_s}[a,b]\\ L_2^{n_r}[c,d]\\ L_2^{n_2}\bl[[a,b]\times[c,d]\br]
\end{array}\right]}$, where $\text{n}:=\{n_0,n_s,n_r,n_2\}$, with each element being a coupled state of finite-dimensional variables $x_0\in\R^{n_0}$, 1D functions $\mbf{x}_s\in L_2^{n_s}[a,b]$ and $\mbf{x}_r\in L_2^{n_r}[a,b]$, and 2D functions $\mbf{x}_2\in L_2^{n_2}\bl[[a,b]\times[c,d]\br]$. We endow this space with the inner product
\begin{align*}
    \ip{\mbf{x}}{\mbf{y}}_{Z}&=\ip{x_0}{y_0}+\ip{\mbf{x}_s}{\mbf{y}_s}_{L_2}+\ip{\mbf{x}_r}{\mbf{y}_r}_{L_2}+\ip{\mbf{x}_2}{\mbf{y}_2}_{L_2}    \\
    &=x_0^T y_0 + \int_{a}^{b}[\mbf{x}_s(s)]^T\mbf{y}(s)ds + \int_{c}^{d}[\mbf{x}_r(r)]^T\mbf{y}(r)dr + \int_{a}^{b}\int_{c}^{d}[\mbf{x}_2(s,r)]^T\mbf{y}(s,r)dr ds
\end{align*}
Defining this inner product, we can also define the adjoint of PI operators.

\begin{lem}
 For any PI parameters $R\in\mcl{N}^{\text{m}\times\text{n}}\bl[[a,b],[c,d]\br]$, there exist unique parameters $Q\in\mcl{N}^{\text{n}\times\text{m}}\bl[[a,b],[c,d]\br]$ such that
 \begin{align*}
    \bl(\mcl{P}[R]\br)^*=\mcl{P}[Q],
 \end{align*}
 where $\mcl{P}^*$ denotes the adjoint of a PI operator $\mcl{P}$. 
 That is, for any $\mbf{x}\in Z^{\textnormal{n}}\bl[[a,b],[c,d]\br]$ and $\mbf{y}\in Z^{\textnormal{m}}\bl[[a,b],[c,d]\br]$,
 \begin{align*}
    \ip{\mcl{P}[R]\mbf{x}}{\mbf{y}}_{Z}=\ip{\mbf{x}}{\mcl{P}[Q]\mbf{y}}_{Z}.
 \end{align*}
\end{lem}
\begin{proof}
 We outline the proof only for 4-PI operators. In particular, let $\textnormal{n}=\{n_0,n_1,0,0\}$ and $\textnormal{m}=\{m_0,m_1,0,0\}$, and let $B=\sbmat{P&Q_1\\Q_2&\{R_0,R_1,R_2\}}$ define a 4-PI operator $\mcl{P}[B]\in\Pi^{\textnormal{n}\times\textnormal{m}}$. Define $\hat{B}=\sbmat{\hat{P}&\hat{Q}_1\\\hat{Q}_2&\{\hat{R}_0,\hat{R}_1,\hat{R}_2\}}$, where
 \begin{align*}
     \hat{P}&=P^T,           &   \hat{Q}_1(s)&=Q_2^T(s),    &   \hat{R}_1(s,\theta)&=R_2^T(\theta,s),   \\
     \hat{Q}_2(s)&=Q_1^T(s), &      \hat{R}_0(s)&=R_0^T(s),  &   \hat{R}_2(s,\theta)&=R_1^T(\theta,s), 
 \end{align*}
 Then, for arbitrary $\mbf{x}=\sbmat{x_0\\\mbf{x}_1}\in Z^{\textnormal{n}}$ and $\mbf{y}=\sbmat{y_0\\\mbf{y}_1}\in Z^{\textnormal{m}}$, we note that
 \begin{align*}
     \ip{\mcl{P}[B]\mbf{x}}{\mbf{y}}_{Z}&=[Px_0]^T y_0 + \left[\int_{a}^{b}Q_1(s)\mbf{x}_1(s)ds\right]^Ty_0 + \int_{a}^{s}[Q_2(s)x_0]^T\mbf{y}_1(s)ds +\int_{a}^{b}[R_0(s)\mbf{x}_1(s)]^T\mbf{y}_1(s)ds    \\
    &\qquad + \int_{a}^{b}\left[\left(\int_{a}^{b}\mbf{I}(s-\theta)R_1(s,\theta)+\mbf{I}(\theta-s)R_2(s,\theta)\right)\mbf{x}_1(\theta)d\theta\right]^T\mbf{y}_1(s)ds    \\
    &=x_0 [P^T y_0] + x_0^T\left[\int_{a}^{b}Q_2^T(s)\mbf{y}_1(s)ds\right] + \int_{a}^{b}\mbf{x}_1^T(s)\left[Q_1^T(s)y_0\right]ds +\int_{a}^{b}\mbf{x}_1^T(s)\left[R_0^T(s)\mbf{y}_1(s)\right]ds \\
    &\qquad +\int_{a}^{b}\mbf{x}_1^T(s)\left[\left(\int_{a}^{b}\mbf{I}(-\theta-s)R_1(\theta,s)+\mbf{I}(s-\theta)R_2(\theta,s)\right)\mbf{y}_1(\theta)d\theta\right]ds   \\
    &=\ip{\mbf{x}}{\mcl{P}[\hat{B}]\mbf{y}}_{Z}.
 \end{align*}
 
\end{proof}
We note again the similarities with matrices: Just like the adjoint of a matrix can be determined by switching the rows and columns, the adjoint of a 3-PI operator is determined by switching the primary and dual variables $(s,\theta)$, as well as switching the lower- and upper-triangular parts.

\section{Inversion of PI operators}\label{sec:PI_theory:Inverse}
In this section, we focus on constructing the controller from the feasible solutions $\mcl{P}$ and $\mcl{Z}$ to the LPIs described in previous sections. The controller gains $\mcl{K}$ is given by the relation $\mcl{K}=\mcl{Z}\mcl{P}^{-1}$. However, $\mcl{P}^{-1}$, although is a 4-PI operator, may not have polynomial parameters. Hence the inverse is calculated numerically. First, we propose a numerical method to find the inverse of $\mcl{P}$. Then, we use the numerical inverse to construct the controller gains. To find the inverse of a 4-PI operator, $\fourpi{P}{Q}{Q^T}{R_i}$, we first find the inverse of 3-PI operators of the form $\threepi{I,H_1,R_2}$ and then express the inverse of $\fourpi{P}{Q}{Q^T}{R_i}$ in terms of $\threepi{I,H_1,H_2}^{-1}$ where $H_i$ are dependent on the parameters $\{P,Q,R_i\}$.
\subsection{Inversion of 3-PI operators}
	First, we note that any matrix-valued polynomial $H(s,\theta)$ can be factored as $F(s)G(\theta)$. Then, for any given 3-PI operator $\threepi{I,H_1,H_2}$ with matrix-valued polynomial parameters $H_1$ and $H_2$, we have
	\begin{align*}
	\threepi{I,H_1,H_2} &= \threepi{I,-F_1G_1,-F_2G_2},~\text{where}\\
	H_i(s,\theta) &= -F_i(s)G_i(\theta),
	\end{align*}
	for some matrix-valued polynomials $F_i$ and $G_i$. We can now find an inverse for $\threepi{I,H_1,H_2}$ using the following result. 
	\begin{lem}\label{lem:3piINV}
	Suppose $F_1: [a,b]\to \R^{n\times p}$, $G_1:[a,b]\to \R^{p\times n}$, $F_2:[a,b]\to\R^{n\times q}$, $G_2:[a,b]\to \R^{q\times n}$ and $U$ is the unique function that satisfies the equation 
	\begin{align*}
	U(s) &= I_{(p+q)} +\int_a^s U(t)\bmat{G_1(t)F_1(t) & G_1(t)F_2(t)\\-G_2(t)F_1(t)&-G_2(t)F_2(t)} dt, 
	\end{align*} 
	where $U$ is partitioned as
	\[
	U = \bmat{U_{11}&U_{12}\\U_{21}&U_{22}}, \quad U_{11}(s)\in \R^{p\times p}, U_{22}(s)\in \R^{q\times q}.
	\]
	Then, the 3-PI operator $\threepi{I,-F_1G_1,-F_2G_2}$ is invertible if and only if $U_{22}(b)$ is invertible and 
	\begin{align*}
	(\threepi{I,-F_1G_1,-F_2G_2})^{-1} &= \threepi{I,L_1,L_2},
	\end{align*}
	where
	\begin{subequations}\label{eq:Li_inv}
	\begin{align}
	L_1(s,t) &= \bmat{F_1(s)&F_2(s)}U(s)V(t)\bmat{G_1(t)\\-G_2(t)}-L_2(s,t),\\
	L_2(s,t) &= -\bmat{F_1(s)&F_2(s)}U(s)PV(t)\bmat{G_1(t)\\-G_2(t)},
	\end{align}
	\end{subequations}
	\begin{align*}
	P &= \bmat{0_{p\times p}&0_{p\times q}\\U_{22}(b)^{-1}U_{21}(b)&I_{q}},	
	\end{align*}
	and $V$ is the unique function satisfying the equation 
	\[
	V(t) = I_{(p+q)} - \int_a^t \bmat{G_1(s)F_1(s) & G_1(s)F_2(s)\\-G_2(s)F_1(s)&-G_2(s)F_2(s)}V(s) ds.
	\]	
	\end{lem}
	\begin{proof}
	Proof can be found in \cite[Chapter IX.2]{gohberg2013classes}.
	\end{proof}
	Using Lemma 2.2. of \cite[Chapter IX.2]{gohberg2013classes}, we can use an iterative process and numerical integration to approximate $U$ and $V$ functions in the above result to find an inverse for the 3-PI operators of the form $\threepi{I,R_1,R_2}$ where $R_i$ are matrix-valued polynomials. By extension, given an $R_0$ invertible, we can obtain the inverse of a general 3-PI operator as shown below.
	\begin{cor}\label{cor:3piINV}
	Suppose $R_0:[a,b]\to \R^{n\times n}$, $R_1,R_2: [a,b]^2\to \R^{n\times n}$, with $R_0$ invertible on $[a,b]$. Then, the inverse of the 3-PI operator, $\threepi{R_i}$, is given by $\threepi{\hat R_0,\hat R_1,\hat R_2}$ where
	\begin{align*}
	&\hat R_0(s) = R_0(s)^{-1}, \quad \hat{R}_i(s,\theta) = L_i(s,\theta)\hat R_0(\theta), \quad i\in \{1,2\},
	\end{align*}
	where $L_1$ and $L_2$ are as defined in \eqref{eq:Li_inv} for functions $F_i$ and $G_i$ such that $F_i(s)G_i(\theta)= R_0(s)^{-1}R_i(s,\theta)$.
	\end{cor}
	\begin{proof} 
	Let $R_i$ be as stated above.
	\begin{align*}
	&(\threepi{R_0(s),R_1(s,\theta),R_2(s,\theta)})^{-1} \\
	&= (\threepi{R_0(s),0,0} \threepi{I,R_0(s)^{-1}R_1(s,\theta),R_0(s)^{-1}R_2(s,\theta)})^{-1} \\
	&= (\threepi{I,R_0(s)^{-1}R_1(s,\theta),R_0(s)^{-1}R_2(s,\theta)})^{-1}(\threepi{R_0(s),0,0})^{-1} \\
	&= \threepi{I,L_1(s,\theta),L_2(s,\theta)}\threepi{R_0(s)^{-1},0,0} \\
	&= \threepi{R_0(s)^{-1},L_1(s,\theta)R_0(\theta)^{-1},L_2(s,\theta)R_0(\theta)^{-1}}.
	\end{align*}
	where $L_i$ are obtained from the Lemma \ref{lem:3piINV} and the composition of PI operators is performed using the formulae in \cite{shivakumar_2019CDC}.

	\end{proof}
	Note the above expressions for the inverse are exact, however, in practice, $R_0^{-1}$ may not have an analytical expression (or very hard to determine). Thus, finding $F_i$ and $G_i$ such that $F_i(s)G_i(\theta)= R_0(s)^{-1}R_i(s,\theta)$ may not be possible. To overcome this problem, we approximate $R_0^{-1}$ by a polynomial which guarantees that $R_0^{-1}R_i$ are polynomials and can be factorized into $F_i$ and $G_i$. Using this approach, we can find an approximate inverse for $\threepi{R_i}$ using Lemma \ref{lem:3piINV}. 

	\subsection{Inversion of 4-PI operators}
	Given, $R_0, R_1, R_2$ with $R_0$ invertible, we proposed a way to find the inverse of the operator $\threepi{R_i}$. Now, we use this method to find the inverse of a 4-PI operator $\fourpi{P}{Q}{Q^T}{R_i}$. Given $P$, $Q$ and $R_i$ with invertible $P$ and $R_0$, define the 3-PI operator $\threepi{H_i}$ with parameters $H_i$ 
	\begin{align*}
	H_0(s) = R_0(s), \quad H_i(s,\theta) &=R_i(s,\theta)-Q(s)^TP^{-1}Q(\theta),
	&\qquad\qquad \quad i\in\{1,2\}.
	\end{align*}

	Next we suppose that $\threepi{H_i}$ is invertible. We will use the inverse of this 3-PI operator to find the inverse of the 4-PI operator as follows.
	\begin{cor}\label{lem:inverse}
	Suppose $P\in \R^{m\times m}$, $Q:[a,b]\to \R^{m\times n}$, $R_0:[a,b]\to\R^{n\times n}$ and $R_1,R_2:[a,b]^2\to \R^{n\times n}$ are matrices and matrix-valued polynomials with $P$ and $\threepi{H_i}$ invertible such that $(\threepi{H_i})^{-1}=\threepi{L_i}$, where
	\begin{align*}
	H_0(s) &= R_0(s), \quad H_i(s,\theta) =R_i(s,\theta)-Q(s)^TP^{-1}Q(\theta),
	\end{align*}
	for $~i\in \{1,2\}$.
	Then, the inverse of $\fourpi{P}{Q}{Q^T}{R_i}$ is given by
	\begin{align*}
	\left(\fourpi{P}{Q}{Q^T}{R_i}\right)^{-1} &= \fourpi{\hat P}{\hat Q}{\hat Q^T}{\hat R_i}, 
	\end{align*}
	where $\hat P, \hat Q, \hat R_i$ are defined as
	\begin{align}\label{eq:4piINV}
	\hat R_0(s) &= R_0(s)^{-1},\quad \hat R_i(s,\theta) = L_i(s,\theta), \quad\text{for} ~i\in \{1,2\},\notag\\
	\hat P &= P^{-1}+P^{-1}\left(\int_a^b (Q(s)\hat R_0(s)Q(s)^T\right.\notag\\
	&\qquad\left. +\int_s^bQ(\theta)L_1(\theta,s)Q(\theta)^Td\theta \right.\notag\\&\quad\qquad\left.+\int_a^sQ(\theta)L_2(\theta,s)Q(\theta)^Td\theta \right)ds P^{-1}\notag\\
	\hat Q(s) &=-P^{-1}\left(Q(s)\hat R_0(s)+\int_s^bQ(\theta)L_1(\theta,s)d\theta \right.\notag\\
	&\qquad\left.+\int_a^sQ(\theta)L_2(\theta,s)d\theta \right). 
	\end{align}
	\end{cor}
	\begin{proof}
	This result can be proved by performing the composition of $\mcl P$ and its inverse $\mcl P^{-1}$ where the inverse is provided by the formulae stated above. Using composition formulae for 4-PI operators (See \cite{shivakumar_2019CDC}), we show that the resulting operator is an identity. 
	\end{proof}



\section{Composition of Differential and PI operator}\label{sec:PI_theory:Derivative}

Given the well-known relationship between integrals and derivatives (think e.g. the fundamental theorem of calculus), it is natural to assume that the composition of a differential operator and a PI operator may be expressed as a PI operator as well. Unfortunately, this is not true in general, as e.g. the operator $\mcl{P}$ defined as $\bl(\mcl{P}\mbf{v}\br)(s)=P(s)\mbf{v}(s)$ is a PI operator, but there clearly does not exist a PI operator $\mcl{Q}$ such that $\partial_{s}\bl(\mcl{P}\mbf{v}\br)(s)=\bl(\mcl{Q}\mbf{v}\br)(s)$. Nevertheless, if the function $\mbf{v}$ is differentiable, i.e. $\mbf{v}\in H_1$, then we can always express the derivative $\bl(\mcl{P}\mbf{v}\br)(s)$ for a PI operator $\mcl{P}$ in terms of $\mbf{v}(s)$ and $\partial_{s}\mbf{v}(s)$ as $\partial_{s}\bl(\mcl{P}\mbf{v}\br)(s)=\bl(\mcl{Q}\sbmat{\mbf{v}\\\partial_{s}\mbf{v}}\br)(s)$, as we show in the next lemma.


\begin{lem}\label{lem:diff_PI}
	Suppose $\fourpi{P}{Q_1}{Q_2}{R_i}:\R^m\times H_1^n\to\R^p\times L_2^q$, and define $\partial_s \fourpi{P}{Q_1}{Q_2}{R_i}: \R^m\times H_1^n\times L_2^n\to \R^p \times L_2^q$ as
	\begin{align}\label{eq:diff_PI}
		\partial_s\fourpi{P}{Q_1}{Q_2}{R_i}= \fourpi{0}{0}{\bar{Q}_2}{\bar{R}_i}
	\end{align}
	where $\bar{Q}_2(s) = \partial_s Q_2(s)$, $\bar{R}_0(s) = \bmat{\partial_s R_0(s)+R_1(s,s)-R_2(s,s) & R_0(s)}$ and $\bar{R}_i(s,\theta)=\bmat{\partial_s R_i(s,\theta)&0}$ for $i\in\{1,2\}$. Then, for any $x\in\R^m$, $\mbf{x}\in H_1^n$,
	\begin{align*}
		\partial_s\left(\fourpi{P}{Q_1}{Q_2}{R_i}\bmat{x\\\mbf{x}}\right)
		=\left(\partial_s\fourpi{P}{Q_1}{Q_2}{R_i}\right)
		\bmat{x\\\mbf{x}\\\partial_s \mbf{x}}
	\end{align*}
\end{lem}

\begin{proof}
    The result can be easily derived using the Leibniz integral rule, stating that for any $F\in H_1[a,b]^2$,
    \begin{align*}
        \frac{d}{ds}\left(\int_{L(s)}^{U(s)}F(s,\theta) d\theta\right) = F(s,U(s))\frac{d}{ds}U(s) - F(s,L(s))\frac{d}{ds} + \int_{L(s)}^{U(s)}\frac{d}{ds}F(s,\theta)d\theta.
    \end{align*}
    For a similar result for PI operators in 2D, we refer to~\cite{jagt_2021PIEACC}.
\end{proof}


\section{Matrix Parametrization of Positive Definite PI Operators}\label{sec:PI_theory:Positive_PI}

In order to be able to solve optimization programs involving PI operator $\mcl{P}$, we need to be able to enforce positivity constraints $\mcl{P}\succcurlyeq 0$. For this, we parameterize PI operators by positive matrices, expanding them as $\mcl{P}=\mcl{Z}^*P\mcl{Z}$ for a fixed operator $\mcl{Z}$, and positive semidefinite matrix $P\succcurlyeq 0$. The following theorem provides a sufficient condition for positivity of a 4-PI operator defined as
\begin{align}\label{eq:4pi}
    \bl(\mcl{P}\sbmat{P,&Q_1\\Q_2,&\{R_i\}}\mbf{x}\br)(s)=
    \left[\begin{array}{ll}
        Px_0        \hspace*{-0.1cm}~& +\ \int_{a}^{b}Q_1(s)\mbf{x}_1(s)ds  \\
        Q_2(s)x_0   \hspace*{-0.1cm}& +\ R_{0}(s)\mbf{x}_1(s) + \int_{a}^{s}R_{1}(s,\theta)\mbf{x}_1(\theta)d\theta + \int_{s}^{b}R_{2}(s,\theta)\mbf{x}_1(\theta)d\theta
    \end{array}\right]
\end{align}
for $\mbf{x}=\bmat{x_0\\\mbf{x}_1}\in \bmat{\R^{n_0}\\L_2^{n_1}[a,b]}$. This result allows us to parameterize a cone of positive PI operators as positive matrices, implement LPI constraints as LMI constraints, allowing us to solve optimization problems with PI operators using semi-definite programming solvers. 
\begin{thm}\label{th:positivity}
	For any functions $Z_1:[a, b]\to\R^{d_1\times n}$, $Z_2:[a, b]\times [a, b]\rightarrow\mathbb{R}^{d_2\times n}$, if $g(s)\geq0$ for all $s\in[a,b]$ and 
	{\small
		\begin{align}
		P&= T_{11}\myint g(s) ds,\nonumber\\
		Q(\eta) &= g(\eta)T_{12}Z_1(\eta)+\myintb{\eta} g(s)T_{13}Z_2(s,\eta)\text{d}s+ \myinta{\eta}g(s) T_{14}Z_2(s,\eta)\text{d}s, \nonumber\\
		R_1(s,\eta) &=g(s)Z_1(s)^{\top}T_{23}Z_2(s,\eta)+g(\eta)Z_2(\eta,s)^{\top}T_{42}Z_1(\eta)+\myintb{s}g(\th)Z_2(\theta,s)^{\top}T_{33}Z_2(\theta,\eta)\text{d}\theta\nonumber\\
		&\hspace{1.3cm}~~+\int_{\eta}^{s}g(\th)Z_2(\theta,s)^{\top}T_{43}Z_2(\theta,\eta)\text{d}\theta+\myinta{\eta}g(\th)Z_2(\theta,s)^{\top}T_{44}Z_2(\theta,\eta)\text{d}\theta,\nonumber\\
		R_2(s,\eta) &=g(s)Z_1(s)^{\top}T_{32}Z_2(s,\eta)+g(\eta)Z_2(\eta,s)^{\top}T_{24}Z_1(\eta)+\myintb{\eta}g(\th)Z_2(\theta,s)^{\top}T_{33}Z_2(\theta,\eta)\text{d}\theta\notag\\
		&\hspace{1.3cm}~~+\int_{s}^{\eta}g(\th)Z_2(\theta,s)^{\top}T_{34}Z_2(\theta,\eta)\text{d}\theta+\myinta{s}g(\th)Z_2(\theta,s)^{\top}T_{44}Z_2(\theta,\eta)\text{d}\theta,\nonumber\\
		R_0(s) &= g(s)Z_1(s)^{\top} T_{22} Z_1(s).
		\label{eq:TH}
		\end{align}}
	where
	\begin{align*}
	T= \begin{bmatrix}
	T_{11} & T_{12} & T_{13} & T_{14}\\
	T_{21} & T_{22} & T_{23} & T_{24}\\
	T_{31} & T_{32} & T_{33} & T_{34}\\
	T_{41} & T_{42} & T_{43} & T_{44}
	\end{bmatrix}\succcurlyeq 0,
	\end{align*}\\
	then the operator $\fourpi{P}{Q_1}{Q_2}{R_i}$ ~as defined in \eqref{eq:4pi} is positive semidefinite, i.e. $\ip{\mbf{x}}{\fourpi{P}{Q_1}{Q_2}{R_i}\mbf{x}}\ge 0$ for all $\mbf{x}\in \R^m\times L_2^n[a,b]$.
\end{thm}
To see the PIETOOLS implementation, check Section \ref{sec:poslpivar}. For an extension of this result to 2D PI operators, we refer to~\cite{jagt_2021PIEACC}.

\chapter{Troubleshooting}
This appendix is dedicated to tackling issues related to installation and setting up of PIETOOLS 2022. Additionally, we discuss some common issues users may run into while solving LPIs.
	
\section{Troubleshooting: Installation}
When installing the PIETOOLS toolbox using the install script, in rare circumstances, the user may run into one of following errors. In case of an error, a message is displayed explaining the issue, which can be one of the following.
\begin{enumerate}
	\item \textit{An error appeared when trying to create the folder ...}\\
	Check if the folder already has a folder named \textbf{PIETOOLS\_2022}. If that does not resolve the issue try running MATLAB from an administrator account which has the authority to create or modify folders. As a last resort, try installing in a different folder. If that does not fix this, contact us.
	
	\item \textit{The installation directory ``   '' already exists...}\\
	Obvious error. However, if there is no folder with the name PIETOOLS\_2022, check for a hidden folder.
	
	
	\item \textit{`tbxmanager' or `SeDuMi' were not downloaded or installed...}\\
	Check your internet connection. Verify that MATLAB is allowed to download files using the internet connection. Check if the websites for tbxmanager and SeDuMi are operational. 
	
	\item \textit{`PIETOOLS' was not downloaded or installed...}\\
	Check the suggestions for the previous error. If that does not fix the issue, contact us.
	
	\item \textit{Could not modify the initialization file ``startup.m''...}\\
	Try running MATLAB from an administrator account which has the authority to create or modify folders. If that does not fix the issue, manually add SeDuMi or the relevant SDP solver (like MOSEK, sdpt3, sdpnal) to your MATLAB path, and extract the files in \textbf{PIETOOLS\_2022.zip}. Also add PIETOOLS to your MATLAB path.
	
	\item \textit{Could not save the path to a default location...}\\
	Try running MATLAB from an administrator account which has the authority to create or modify folders. If that does not fix it, just add PIETOOLS and SeDuMi to your MATLAB path and skip this step.
	
\end{enumerate}

\section{Troubleshooting: Solving LPIs}
PIETOOLS can be used for solving LPI optimization problems and users may run into errors while setting up and solving them. For any errors in setting up an LPI, refer to the function and script headers to ensure that input-output formats are correct. Ensure that PI objects, defined in the LPI problem, are well-defined and valid PI operators. You can use the \texttt{isvalid} function to check if a PI operator is well defined. If that does not fix the problem, feel free to contact us.\\
PIETOOLS 2022 relies on recently released SOSTOOLS 4.00 with solver SeDuMi to solve optimization problems. If you are unfamiliar with SOSTOOLS, and are unsure how to intepret the results of a solved optimization problem, please check the SOSTOOLS manual available at \url{http://www.cds.caltech.edu/sostools/}, or the SeDuMi manual avaialable at \url{https://sedumi.ie.lehigh.edu/?page_id=58} for more information. A brief overview of how to interpret results, and what to do in case of error, is provided below:
\begin{enumerate}
	\item \textbf{How do I interpret the results of a solved optimization problem?}\\
	A general rule of thumb is to look at: pinf, dinf, feasratio and Residual norm. pinf and dinf should be 0, while feasratio is in between -1 and 1 (preferably closer to 1). The lower the residual norm the better. Refer to SeDuMi manual to interpret other output parameters and more details~\cite{sedumi}.
	
	\item \textbf{What if pinf is 1?}\\
	Verify if the LPI constraints are in fact feasible. Verify if the sign-definiteness of the PI operator is on a compact interval (use psatz term if local sign-definite is needed) or the entire real line. Use 'getdeg' function to check if your LPI constraint has high degree polynomials. If yes, make sure that all opvar variables used in lpi\_eq function have high enough degrees to match it. If this does not resolve the issue contact us and attach the files that you are trying to run along with a snapshot of the error/output.
	
	\item \textbf{What if dinf is 1 and feasratio is -1?}\\
	This issue typically occurs when the objective function is unbounded from below and becomes $-\infty$. Check if the objective function is bounded below. If this does not resolve the issue contact us by email and attach the files that you are trying to run along with a snapshot of the error/output.
\end{enumerate}

\section{Contact Details}
To resolve issues, report bugs or to collaborate on any development work regarding PIETOOLS, please contact us through email and we will get back to you as quickly as possible. In case of issues with installation, solving problems or bugs identified, please include the script file that generates the error along with images of the error generated in MATLAB. You can reach us through email at: \url{sshivak8@asu.edu}, \url{ad2079@cam.ac.uk}, \url{djagt@asu.edu} and \url{mpeet@asu.edu}

\bibliographystyle{plain}
\bibliography{user_manual}

\appendix

\end{document}